\documentclass[12pt,leqno,twoside]{amsproc}
\usepackage{amsfonts}
\usepackage{amsmath}
\usepackage{amssymb}
\usepackage{amsthm}
\usepackage{mathtools}
\usepackage{epsf}
\usepackage{graphics}
\usepackage{graphicx}
\usepackage{latexsym}
\usepackage{psfrag}
\usepackage{tikz}
\usepackage{verbatim}
\usepackage{hyperref}
\usetikzlibrary{positioning}
\usetikzlibrary{arrows}
\usetikzlibrary{decorations.pathreplacing}

\newtheorem{thm}{Theorem}[section]
\newtheorem{lem}{Lemma}[section]
\newtheorem{cor}[thm]{Corollary}
\newtheorem{definition}{Definition}[section]

\newtheorem{rem}{Remark}[section]

\newtheorem{ass}{Assumption}[section]

\def\g{\gamma}
\def\G{\Gamma}
\def\d{\delta}
\def\e{\epsilon}
\def\l{\lambda}

\def\Z{{\mathbb Z}}
\def\R{{\mathbb R}}
\def\C{{\mathbb C}}

\def\EE{{\mathcal E}}
\def\LL{{\mathcal L}}

\def\supp{{\text{Supp}}}

\title{Universal edge fluctuations of discrete interlaced particle systems.}
\author{Erik Duse and Anthony Metcalfe}
\date{\today}

\begin{document}

\begin{abstract}
We impose the uniform probability measure on the set of all discrete
Gelfand-Tsetlin patterns of depth $n$ with the particles on row $n$
in deterministic positions. These systems equivalently describe
a broad class of random tilings models, and are closely related to
the eigenvalue minor processes of
a broad class of random Hermitian matrices. They have a determinantal
structure, with a known correlation kernel. We rescale the systems by
$\frac1n$, and examine the asymptotic behaviour, as $n \to \infty$, under
weak asymptotic assumptions for the (rescaled) particles on row $n$: The
empirical distribution of these converges weakly to a probability measure
with compact support, and they otherwise satisfy mild regulatory restrictions.

We prove that the correlation kernel of particles in the neighbourhood of
`typical edge points' convergences to the extended Airy kernel. To do this,
we first find an appropriate scaling for the fluctuations of the particles.
We give an explicit parameterisation of the asymptotic edge, define an
analogous non-asymptotic edge curve (or finite $n$-deterministic equivalent),
and choose our scaling such that that the particles fluctuate around this with
fluctuations of order $O(n^{-\frac13})$ and $O(n^{-\frac23})$ in the
tangent and normal directions respectively. While the final results are
quite natural, the technicalities involved in studying such a broad class
of models under such weak asymptotic assumptions are unavoidable and extensive.
\end{abstract}

\maketitle

\section{Introduction}

\subsection{Overview of the model, the asymptotic assumptions, and results}
\label{secoverview}

In this paper we consider universal edge behaviour of certain random systems
of discrete interlaced particles referred to as {\em Gelfand-Tsetlin patterns}.
A discrete Gelfand-Tsetlin pattern of depth $n$ is an $n$-tuple,
$(y^{(1)},y^{(2)},\ldots,y^{(n)}) \in  \Z \times \Z^2 \times \cdots \times \Z^n$,
which satisfies the interlacing constraint,
\begin{equation}
\label{eqInt}
y_1^{(r+1)} \; \ge \; y_1^{(r)} \; > \; y_2^{(r+1)} \; \ge \; y_2^{(r)}
\; > \cdots \ge \; y_r^{(r)} \; > \; y_{r+1}^{(r+1)},
\end{equation}
for all $r \in \{1,\ldots,n-1\}$, denoted $y^{(r+1)} \succ y^{(r)}$. Equivalently
this can be considered as an interlaced configuration of $\frac12 n (n+1)$
particles in $\Z \times \{1,2,\ldots,n\}$ by placing a particle at position
$(u,r) \in \Z \times \{1,2,\ldots,n\}$ whenever $u$ is an element of $y^{(r)}$. A
Gelfand-Tsetlin pattern of depth $4$ is shown on the
left of figure \ref{figGTAsyShape}.

\begin{figure}[t]
\centering
\begin{tikzpicture}

\draw (0,4.5) node {$y_4^{(4)}$};
\draw (2,4.5) node {$y_3^{(4)}$};
\draw (4,4.5) node {$y_2^{(4)}$};
\draw (6,4.5) node {$y_1^{(4)}$};
\draw (1,3) node {$y_3^{(3)}$};
\draw (3,3) node {$y_2^{(3)}$};
\draw (5,3) node {$y_1^{(3)}$};
\draw (2,1.5) node {$y_2^{(2)}$};
\draw (4,1.5) node {$y_1^{(2)}$};
\draw (3,0) node {$y_1^{(1)}$};

\draw (.45,3.75) node [rotate=-55] {$<$};
\draw (1.45,3.75) node [rotate=55] {$\le$};
\draw (2.45,3.75) node [rotate=-55] {$<$};
\draw (3.45,3.75) node [rotate=55] {$\le$};
\draw (4.45,3.75) node [rotate=-55] {$<$};
\draw (5.45,3.75) node [rotate=55] {$\le$};
\draw (1.45,2.25) node [rotate=-55] {$<$};
\draw (2.45,2.25) node [rotate=55] {$\le$};
\draw (3.45,2.25) node [rotate=-55] {$<$};
\draw (4.45,2.25) node [rotate=55] {$\le$};
\draw (2.45,.75) node [rotate=-55] {$<$};
\draw (3.45,.75) node [rotate=55] {$\le$};


\draw (7.5,3.5)
--++ (5,0)
--++ (0,-2.5)
--++ (-2.5,0)
--++ (-2.5,2.5);

\draw (7.5,3.9) node {$(a,1)$};
\draw (12.5,3.9) node {$(b,1)$};
\draw (12.5,.6) node {$(b,0)$};
\draw (10,.6) node {$(a+1,0)$};

\end{tikzpicture}
\caption{Left: A visualisation of a Gelfand-Tsetlin pattern of depth $4$.
Right: $\{(\chi,\eta) \in [a,b] \times [0,1] : b \ge \chi \ge \chi + \eta - 1 \ge a \}$.
Assumption \ref{assWeakConv} implies that the bulk of the rescaled particles of the
Gelfand-Tsetlin patterns lie asymptotically in this region as $n \to \infty$.}
\label{figGTAsyShape}
\end{figure}
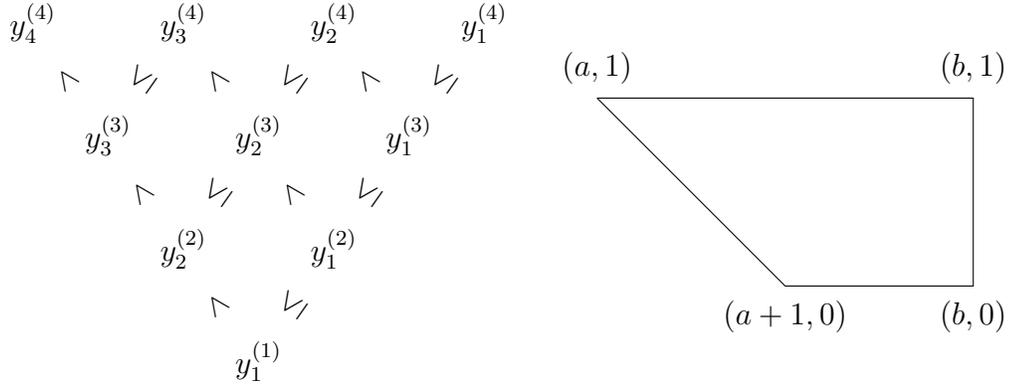

For each $n\ge1$, fix $x^{(n)} \in \Z^n$ with $x_1^{(n)} > x_2^{(n)} > \cdots > x_n^{(n)}$.
Consider the uniform probability measure, $\nu_n$, on the set of discrete Gelfand-Tsetlin
patterns of depth $n$ with the particles on row $n$ in the deterministic positions defined
by $x^{(n)}$:
\begin{equation}
\label{eqnun}
\nu_n[(y^{(1)},y^{(2)},\ldots,y^{(n)})]
:= \frac1{Z_n} \cdot \left\{
\begin{array}{rcl}
1 & ; &
\text{when} \; x^{(n)} = y^{(n)} \succ y^{(n-1)} \succ \cdots \succ y^{(1)}, \\
0 & ; & \text{otherwise},
\end{array}
\right.
\end{equation}
where $Z_n > 0$ is a normalisation constant. This measure, and the equivalent
description of Gelfand-Tsetlin patterns given above, induces a random point
process on interlaced configurations of particles in $\Z \times \{1,2,\ldots,n\}$.
In \cite{Duse15a}, we show that this process is determinantal.  See Johansson,
\cite{Joh06b}, for an introduction to such processes. Equation
(\ref{eqKnrusvFixTopLine}), below, recalls our expression for the correlation
kernel of this process, denoted $K_n : (\Z \times \{1,2,\ldots,n\})^2 \to \C$.
Loosely speaking, $K_n$ is a function on pairs of particle positions which
conveniently encodes the densities and correlations of the particles. $K_n$ was
also, independently, obtained by Petrov, \cite{Pet14}.

As discussed in sections 1.1 and 1.2 of Duse and Metcalfe, \cite{Duse15a},
our motivation for studying these processes is that they are an equivalent
description of uniform random tilings of `half-hexagons' with lozenges,
and of perfect matchings of dimer configurations of honeycomb lattices.
The set of `half-hexagons' is a class of polygons with quite general boundaries.
The boundary is determined by the (arbitrary) choice of the above
$x^{(n)} \in \Z^n$, and particular choices of $x^{(n)}$ recover well-studied
models. For example, if we fix $p,q,r \in \Z^+$, and choose $n = p+r$ and
$x^{(n)} = (p+q+r,p+q+r-1,\ldots,p+q+1,p,p-1,\ldots,1)$, then we recover
the uniform random tiling of a hexagon with sides of length $p,q,r,p,q,r$.
Cohn et al, \cite{Coh98}, studied the asymptotic shape of a `typical'
such tiling as $n \to \infty$, under the assumption that $\frac{p}q$ and
$\frac{p}r$ converge. Analogous limit shapes for other random tilings
models have been studied, for example, in \cite{Ken06} and \cite{Ken07}.

As we will discuss in more detail below, we rescale our systems by
$\frac1n$, and examine the asymptotic behaviour, as $n \to \infty$, under
weak asymptotic assumptions: We assume that the empirical distribution
of $\frac1n x^{(n)}$ converges weakly to a probability measure, $\mu$, with
compact support (assumption \ref{assWeakConv}), and that $\frac1n x^{(n)}$ otherwise
satisfies only mild regulatory restrictions (assumptions \ref{assIsol}
and \ref{asscases}). We avoid only that degenerate case where $\mu$ is
Lebesgue measure on a single interval of length 1, and allow all other $\mu$
which can be obtained via the weak convergence.
Note, as $n \to \infty$, the interlacing constraint
implies that the bulk of the rescaled particles of the Gelfand-Tsetlin
patterns lie asymptotically in the polygon shown on the right of figure
\ref{figGTAsyShape}. The technicalities involved in studying the asymptotic
behaviours of these systems, under such weak asymptotic assumptions,
are unavoidable and extensive. A positive aspect of this is that we uncovered
many unexpected situations. Indeed, we have written 4 papers on these rich systems of models.
Papers \cite{Duse15a, Duse15b} explore the possible global asymptotic shapes, and
this paper and \cite{Duse15d} examine the local asymptotic fluctuations of
the particles in neighbourhoods of the possible edges. Paper \cite{Duse15a}
examines `classic' global asymptotic shapes, and \cite{Duse15b} finds novel global
asymptotic behaviours. Figure \ref{figexs} depicts some example asymptotic shapes
obtained using the results of those papers. In \cite{Duse15d}, we find novel
local asymptotic edge fluctuations.

In this paper, we examine {\em universal} local asymptotic fluctuations at `typical
edge points'. To do this, we must first find an appropriate scaling for the
fluctuations of the particles. We start with the explicit parameterisation of the
asymptotic edge obtained in \cite{Duse15a} (see equation (\ref{eqchiEEetaEE})),
and a define a natural subset of the asymptotic edge called the set of
{\em typical edge points} (see definition \ref{defEdgeTyp}). This set is
always non-empty, and the difference between it and the whole edge is either
empty or discrete. Next, we use the parameterisation to define an analogous
{\em non-asymptotic edge curve for each $n$}, sometimes referred to in the literature
as the {\em finite $n$-deterministic equivalent}
(see definition \ref{defEdgeNonAsy}). We fix a typical edge point, denoted by
$(\chi,\eta)$, and (for each $n$) we let $(\chi_n,\eta_n)$ denote the analogous
point on the non-asymptotic edge. Our asymptotic assumptions imply that
$(\chi_n,\eta_n) \to (\chi,\eta)$ as $n \to \infty$, but we have no
control of the rate of convergence. Nevertheless, theorems \ref{thmAiry2} and
\ref{thmAiry} essentially prove the following:
\begin{thm}
Let $\{(u_n,r_n)\}_{n\ge1}, \{(v_n,s_n)\}_{n\ge1} \subset \Z \times \{1,2,\ldots,n\}$
be sequences of particle positions chosen as follows: For all $n$ sufficiently large,
both $\frac1n (u_n,r_n)$ and $\frac1n (v_n,s_n)$ fluctuate around $(\chi_n,\eta_n)$,
with fluctuations of order $O(n^{-\frac13})$ and $O(n^{-\frac23})$ respectively in
the tangent and normal directions of the non-asymptotic edge curve. Then the
asymptotic behaviour of $n^\frac13 K_n((u_n,r_n),(v_n,s_n))$ is governed by the
extended Airy kernel as $n \to \infty$.
\end{thm}

A steepest descent analysis of a contour integral expression for
$K_n((u_n,r_n),(v_n,s_n))$ (see equation (\ref{eqKnrnunsnvn1}), below)
is used to prove the asymptotics in theorems \ref{thmAiry2} and
\ref{thmAiry}. The length of this paper reflects the extensive technicalities
needed to do the analysis rigorously under our weak assumptions.
In section \ref{secabotrofn}, we examine the roots of the derivatives of
the appropriate steepest descent functions (see equations (\ref{eqfn},
\ref{eqtildefn}, \ref{eqf})). In section \ref{secsdatak}, we perform the
steepest descent analysis. In particular, we highlight lemmas
\ref{lemDesAsc1-12} and \ref{lemDesAsc1-12Rem}, which prove the
existence of appropriate contours of descent/ascent.
The proofs of these are given in section \ref{secCont}. This is, by far,
the most difficult part of the paper. Indeed, our weak assumptions
necessitate that we need to prove existence of different contours
for 12 distinct cases (see lemma \ref{lemCases}).

Note, while convergence to the extended Airy kernel is the only case which
we consider rigorously in this paper, we will briefly discuss other natural
asymptotic situations in section \ref{secOAS}. We will discuss known
results in the literature, and conjecture analogous asymptotic results for
this model. Unfortunately, the length of this paper necessitates that we do
not attempt to study these situations in greater detail here.

We end this section by comparing our result with analogous results
in the literature. Perhaps the closest result to ours is in Petrov,
\cite{Pet14}. Theorem 8.1 of \cite{Pet14} proves a similar asymptotic
result for the special case where $\mu$ is given by Lebesgue measure
on a disjoint union, $[a_1,b_1] \cup [a_2,b_2] \cup \cdots \cup [a_k,b_k]$,
where $k \ge 2$ and $\sum_i (b_i-a_i) = 1$. By contrast, here,
assumption \ref{assWeakConv} avoids only that degenerate case where $\mu$ is
Lebesgue measure on a single interval of length 1, and allows all other $\mu$
with compact support which can be obtained via the weak convergence.
Petrov further specialises by assuming that
$\frac1n x^{(n)}$ is essentially densely packed in the above disjoint union
of intervals. By contrast, here, assumption \ref{assWeakConv} implies
that the empirical distribution of $\frac1n x^{(n)}$ converges weakly to $\mu$,
and assumptions \ref{assIsol} and \ref{asscases} otherwise impose only
mild regulatory restrictions on $\frac1n x^{(n)}$. The stronger
assumptions of Petrov give a fast rate of convergence, and indeed it can be
shown that $(\chi_n,\eta_n) = (\chi,\eta) + O(n^{-1})$
for all $n$ sufficiently large under these. Petrov, therefore,
can ignore $(\chi_n,\eta_n)$ (the non-asymptotic edge), and fluctuate
simply around $(\chi,\eta)$ (the asymptotic edge). By contrast, we have
no control of the rate of convergence.
The fast rate of convergence in \cite{Pet14} also allows Petrov to avoid
many subtle technical points.

Theorem 8.1 of \cite{Pet14} furthermore chooses a somewhat different
scaling for the sequences of particle positions: Petrov fixes parameters
$(\tau_1,\sigma_1) \in \R^2$ and $(\tau_2,\sigma_2) \in \R^2$,
and fluctuates around the asymptotic edge. For all $n$ sufficiently large,
the fluctuations have order $O(n^{-\frac13})$ and $O(n^{-\frac23})$ respectively
in the tangent direction (of the asymptotic edge) and the direction $(1,0)$,
$\tau_1$ and $\tau_2$ measure the size of the $O(n^{-\frac13})$
fluctuations, and $\sigma_1 - \tau_1^2$ and $\sigma_2 - \tau_2^2$
measure the size of the $O(n^{-\frac23})$ fluctuations. Petrov chooses
this scaling to ensure convergence to the form of the extended Airy
kernel defined in \cite{Pra02}, evaluated at $(\tau_1,\sigma_1)$ and
$(\tau_2,\sigma_2)$. A contour integral expression for that kernel
was obtained in \cite{Bor08}. By contrast, here, we fix parameters
$(u,r) \in \R^2$ and $(v,s) \in \R^2$, and fluctuate around the
non-asymptotic edge. For all $n$ sufficiently large, the fluctuations
have order $O(n^{-\frac13})$ and $O(n^{-\frac23})$ respectively
in the tangent and normal directions of the non-asymptotic edge,
$u$ and $v$ measure the size of the $O(n^{-\frac13})$ fluctuations, and
$r$ and $s$ measure the size of the $O(n^{-\frac23})$ fluctuations (see
equations (\ref{equnrnvnsn2}, \ref{equnrnvnsn3})). The scaling is thus
naturally related to the geometric behaviour of the edge. We use the
scaling in lemma \ref{lemNonAsyRoots} to show that the relevant roots
and derivatives of the steepest descent functions have
well-behaved asymptotic behaviours, and these result
in simple Taylor expansions for the steepest descent functions
give in corollary \ref{corTay}. We then use these in theorems
\ref{thmAiry2} and \ref{thmAiry} to prove convergence to
the form of the extended Airy kernel given in equation (\ref{eqAi}),
evaluated at $(u,r)$ and $(v,s)$. Note that this expression is simpler
than that given in \cite{Pet14}, but equivalent, as can be
seen via a change of variables and the removal of a conjugation factor.

Our results also have interesting connections with those of
Kenyon et al, \cite{Ken06, Ken07}. Those papers study
the global asymptotic shapes of random tilings of a class of polygons.
In papers \cite{Duse15a, Duse15b}, we explore the global asymptotic
shapes of random tilings of a more restricted class of polygons,
but we allow more general boundary/asymptotic conditions which
results in some important differences. For example, the asymptotic
boundaries in \cite{Ken06, Ken07} are shown to be {\em algebraic}, and
this is not necessarily true in \cite{Duse15a}.
Moreover, in \cite{Duse15b}, we find novel global asymptotic behaviours.
Also, in \cite{Duse15a, Duse15b}, we obtain parameterisations of the boundaries.
This enables us to prove the asymptotic fluctuations
seen in this paper and \cite{Duse15d}, in neighbourhoods of the edges.
It is intuitively clear that these universal fluctuations will
also appear in the models of \cite{Ken06, Ken07}, under analogous
conditions. For example, in figure 1 of \cite{Ken07}, the asymptotic
frozen boundary of the polygon in seen to be a cardioid. In \cite{Duse15d},
we consider the asymptotic fluctuations in neighbourhoods of an analogous cusp,
and show that they are governed by a novel point process, which we call
the {\em Cusp Airy} process.

More generally, we believe that our techniques can be of use in other
random tiling models, or random perfect matchings, or systems of
random non-intersecting paths, etc, that can equivalently be described
as interacting particle systems. For example, \cite{Bufe16}
is a recent work concerning the asymptotics of random domino tilings of
rectangular Aztec diamonds. In it, the authors find explicit parameterisations
of the possible asymptotic boundaries using almost identical methods to those
that we used in \cite{Duse15a}, and find analogous results.
It is therefore reasonable to expect that the results we found in
\cite{Duse15b}, this paper, and in \cite{Duse15d}, also have natural
analogues for random domino tilings, and the techniques of those papers
would be sufficient to prove these results. There has been significant
interest in related models. See, for example, \cite{Bor08, Chh15, Def08, Joh03,
Joh05a}.

Other closely connected models arise from random matrices.
For each $n \ge 1$, let $A_n \in \C^{n \times n}$ be a random Hermitian matrix
whose distribution is unitarily invariant. For each $r \le n$, let $\l^{(r)} \in \R^r$
be the eigenvalues (in decreasing order) of the $r^\text{th}$ principal sub-matrix of
$A_n$ consisting of the first $r$ rows and columns. The asymptotic
behaviour of $\lambda^{(n)}$ (the eigenvalues of $A_n$) as $n \to \infty$ has been
studied for many different ensembles of random matrices (i.e. for different
choices of $A_n$), and universal behaviours have been found. See, for example,
\cite{Bai10, Past11} for reviews of known results.
See also the recent work of Hachem et al, \cite{Hachem15},
which studies the asymptotic behaviour of the edge of $\lambda^{(n)}$ as
$n \to \infty$ when $A_n$ is a complex correlated Wishart matrix. In \cite{Hachem15},
the model is shown to be determinantal, the edge asymptotic behaviour is
examined via a closely related saddle point problem to that seen in this
paper, and it is shown that this behaviour is governed by the standard Airy kernel.

Note, for general Hermitian $A_n$, an elementary result from matrix analysis
shows that $(\l^{(1)},\l^{(2)},\ldots,\l^{(n)})$ is a Gelfand-Tsetlin pattern
of depth $n$, where now the particles on each level take positions in $\R$
rather than $\Z$. Such models often display a similar determinantal structure
to the discrete Gelfand-Tsetlin patterns of this paper. Perhaps the
best studied determinantal minor process is that of the
{\em Gaussian Unitary Ensemble} (GUE), where the entries of $A_n$
are random independent Gaussians. See, for example, Mehta, \cite{Mehta04},
and Johansson and Nordenstam, \cite{Joh06a}.

Perhaps the most similar minor process to ours is that
studied by the author Metcalfe in \cite{Met13}. There,
the eigenvalues of $A_n$ are deterministic:
$\l^{(n)} = x^{(n)}$ for some fixed $x^{(n)} \in \R^n$ with
$x_1^{(n)} > x_2^{(n)} > \cdots > x_n^{(n)}$, and the eigenvalue
minor process induces the uniform probability
distribution on the set of Gelfand-Tsetlin patterns of depth $n$ with
the particles on the top row in the deterministic positions defined by
$x^{(n)} \in \R^n$. This is very similar to the measure defined in
equation (\ref{eqnun}). Metcalfe showed that the process
is determinantal, and found a correlation kernel. This kernel can, in fact,
be shown to be a limit of the kernel in equation
(\ref{eqKnrusvFixTopLine}), below, where we scale the discrete particle
positions on each level such that they become continuous.
In \cite{Met13}, Metcalfe proved universal bulk asymptotic behaviour
under the assumption that the empirical distribution
of $x^{(n)}$ converges weakly to a probability measure with
compact support, similar to assumption \ref{assWeakConv}, below.
The edge asymptotic behaviour has not yet been studied.
It is clear to the authors, however, that the techniques of this
paper would be sufficient to study this. Also, it is worth noting that there
are interesting asymptotic situations there that have no
natural analogues here. In particular,
the limit measure may have atoms in \cite{Met13}, which is not possible here.
The related asymptotic situations have also not yet been studied,
and we believe that our techniques would help in studying these.
More generally, we hope that our techniques would help to study the
determinantal minor processes of other ensembles of random matrices.

\subsection{The determinantal structure of discrete Gelfand-Tsetlin patterns}

As in the previous section, define interlacing as in equation (\ref{eqInt}),
fix $x^{(n)} \in \Z^n$ with $x_1^{(n)} > x_2^{(n)} > \cdots > x_n^{(n)}$,
and define $\nu_n$ as in equation (\ref{eqnun}). Recall
that $\nu_n$ induces a random point process on interlaced configurations of particles
in $\Z \times \{1,2,\ldots,n\}$. In section 4.1 of \cite{Duse15a}, we showed that
this process is determinantal, and we found an expression for a correlation kernel,
denoted by $K_n : (\Z \times \{1,2,\ldots,n\})^2 \to \C$. Note, ignoring the
deterministic particles on row $n$, interlacing implies that we need only consider
those particle positions, $(u,r), (v,s) \in \Z \times \{1,2,\ldots,n-1\}$, which satisfy
$u \ge x_n^{(n)}+n-r$ and $v \ge x_n^{(n)}+n-s$. For all such $(u,r),(v,s)$,
\begin{align}
\label{eqKnrusvFixTopLine}
&K_n((u,r),(v,s)) = - \phi_{r,s}(u,v) \\
\nonumber
& + \frac{(n-s)!}{(n-r-1)!} \sum_{k=1}^n 1_{(x_k^{(n)} \ge  u)} \sum_{l=v+s-n}^v
\frac{\prod_{j=u+r-n+1}^{u-1} (x_k^{(n)} - j)}{\prod_{j=v+s-n, \; j \neq l}^v (l - j)} \;
\prod_{i=1, \; i \neq k}^n \bigg( \frac{l - x_i^{(n)}}{x_k^{(n)} - x_i^{(n)}} \bigg),
\end{align}
where
\begin{equation}
\label{eqphirsuv}
\phi_{r,s} (u,v)
:= \left\{
\begin{array}{lll}
0 & ; & \text{when } v < u \text{ or } s \le r, \\
1 & ; & \text{when } v \ge u \text{ and } s = r+1, \\
\frac1{(s-r-1)!} \prod_{j=1}^{s-r-1} (v-u+s-r-j) & ; & \text{when } v \ge u \text{ and } s > r+1.
\end{array}
\right.
\end{equation}

Note, correlation kernels are not uniquely defined. Indeed,
$\mathcal{K}_n : (\Z \times \{1,2,\ldots,n\})^2 \to \C$ is an
equivalent kernel if $\det \left[ \mathcal{K}_n((u_i,r_i),(u_j,r_j)) \right]_{i,j=1}^m
= \det [ K_n((u_i,r_i),(u_j,r_j)) ]_{i,j=1}^m$ for all $m\ge1$, and
for all particle positions
$\{(u_1,r_1),\ldots,(u_m,r_m)\} \subset \Z \times \{1,2,\ldots,n\}$.
An equivalent kernel which will prove useful in our asymptotic analysis
is the following:
\begin{equation}
\label{eqKntilde}
\mathcal{K}_n((u,r),(v,s))
:= K_n((v,s),(u,r)) \; B_n(s,r)^{-1} \; A_{t,n}((v,s),(u,r))^{-1},
\end{equation}
for all $(u,r), (v,s) \in \Z \times \{1,2,\ldots,n\}$,
where $t \in \R$ is that fixed value in equation (\ref{equnrnvnsn}),
$A_{t,n} : (\Z^2)^2 \to \R \setminus \{0\}$ is defined in lemma \ref{lemFnt},
and $B_n : \Z_+^2 \to (0,+\infty)$ is defined by,
\begin{equation}
\label{eqConjFactB}
B_n(r,s) := \frac{(n-s)!}{(n-r)!} \; n^{s-r}.
\end{equation}

\subsection{The asymptotic `shape' of discrete Gelfand-Tsetlin patterns}
\label{sectasoGTp}

In \cite{Duse15a} and \cite{Duse15b} we consider the asymptotic `shape' of
the systems of Gelfand-Tsetlin patterns of the previous sections, under
some natural asymptotic assumptions. In this section, we recall the
relevant asymptotic assumptions, definitions, and results of those papers.
We state these without motivation or proof, and refer the interested reader
to those papers.

\begin{ass}
\label{assWeakConv}
Let $\mu$ be a probability measure on $\R$ with $\mu \le \l$, where $\l$
is Lebesgue measure. Assume that there is a compact interval $[a,b] \subset \R$
with $b-a>1$, $\supp(\mu) \subset [a,b]$ and $\{a,b\} \subset \supp(\mu)$.
Moreover, assume that,
\begin{equation*}
\frac1n \sum_{i=1}^n \delta_{x_i^{(n)}/n} \to \mu,
\end{equation*}
as $n \to \infty$, in the sense of weak convergence of measures.
\end{ass}
Then, rescaling the sides of the Gelfand-Tsetlin patterns by $\frac1n$,
the bulk of the rescaled particles asymptotically lie
in the polygon on the right of figure \ref{figGTAsyShape} as $n \to \infty$, i.e.,
$\{(\chi,\eta) \in [a,b] \times [0,1] : b \ge \chi \ge \chi + \eta - 1 \ge a \}$.
The local asymptotic behaviour of particles near a fixed point, $(\chi,\eta)$, in
this polygon is studied by considering $K_n((u_n,r_n),(v_n,s_n))$ as $n \to \infty$,
where $\{(u_n,r_n)\}_{n\ge1} \subset \Z^2$ and $\{(v_n,s_n)\}_{n\ge1} \subset \Z^2$
satisfy $\frac1n (u_n,r_n) \to (\chi,\eta)$ and $\frac1n (v_n,s_n) \to (\chi,\eta)$
as $n \to \infty$. First note, equation (\ref{eqKnrusvFixTopLine}) and the
Residue theorem give,
\begin{align}
\label{eqKnrnunsnvn1}
&K_n((u_n,r_n),(v_n,s_n)) = - \phi_{r_n,s_n}(u_n,v_n) \\
\nonumber
&+ \frac{(n-s_n)!}{(n-r_n-1)!} \frac{n^{s_n-r_n-1}}{(2\pi i)^2}
\int_{c_n} dw \int_{C_n} dz \; \frac{\prod_{j=u_n+r_n-n+1}^{u_n-1}
(z - \frac{j}n)}{\prod_{j=v_n+s_n-n}^{v_n} (w - \frac{j}n)}
\frac1{w-z} \prod_{i=1}^n \left( \frac{w - \frac{x_i}n}{z - \frac{x_i}n} \right),
\end{align}
for all $n \ge 1$, where we have dropped the superscript from
$x^{(n)} = (x_1^{(n)},x_2^{(n)},\ldots,x_n^{(n)})$,
$C_n$ is any counter-clockwise simple closed contour which
contains all of $\{\frac1n x_j : x_j \ge u_n\}$ but none of
$\{\frac1n x_j : x_j \le u_n+r_n-n \}$, and $c_n$ is any
counter-clockwise simple closed contour which contains
$\frac1n \{v_n+s_n-n,v_n+s_n-n+1,\ldots,v_n\}$ and $C_n$.
Next note, the above integrand equals
$\frac1{w-z} \exp(n f_n(w) - n \tilde{f}_n(z))$, where
\begin{align}
\label{eqfn}
f_n(w)
& := \frac1n \sum_{i=1}^n \log \left( w - \frac{x_i}n \right) -
\frac1n \sum_{j=v_n+s_n-n}^{v_n} \log \left( w - \frac{j}n \right), \\
\label{eqtildefn}
\tilde{f}_n(z)
& := \frac1n \sum_{i=1}^n \log \left( z - \frac{x_i}n \right) -
\frac1n \sum_{j=u_n+r_n-n+1}^{u_n-1} \log \left( z - \frac{j}n \right),
\end{align}
for all $w,z \in \C \setminus \R$, and $\log$ denotes the principal
logarithm. Finally note, since $\frac1n \sum_i \delta_{x_i/n} \to \mu$ weakly
as $n \to \infty$, and $\frac1n (u_n,r_n), \frac1n (v_n,s_n) \to (\chi,\eta)$
as $n \to \infty$, it is natural to define the following asymptotic function:
\begin{equation}
\label{eqf}
f_{(\chi,\eta)} (w)
:= \int_a^b \log (w-x) \mu[dx] - \int_{\chi+\eta-1}^\chi \log (w-x) dx,
\end{equation}
for all $w \in \C \setminus \R$.

Steepest descent analysis, and the above structure, intuitively suggests that the
behaviour of $K_n((u_n,r_n),(v_n,s_n))$
as $n \to \infty$ depends on the roots of $f_{(\chi,\eta)}'$. Recall that
$b \ge \chi \ge \chi+\eta-1 \ge a$, and (see assumption \ref{assWeakConv})
that $\mu$ and $\l-\mu$ are positive measures. Thus, for all $w \in \C \setminus \R$,
it is natural to write,
\begin{equation}
\label{eqf2}
f_{(\chi,\eta)}(w) 
= \int_{S_1} \log (w-x) \mu[dx]
- \int_{S_2} \log (w-x) (\l-\mu)[dx]
+ \int_{S_3} \log (w-x) \mu[dx],
\end{equation}
where $S_i := S_i(\chi,\eta)$ for all $i \in \{1,2,3\}$ are defined by: 
\begin{equation}
\label{eqS1S2S3}
S_1 := \supp(\mu |_{[\chi,b]}),
\hspace{0.5cm}
S_2 := \supp((\l-\mu) |_{[\chi+\eta-1,\chi]}),
\hspace{0.5cm}
S_3 := \supp(\mu |_{[a,\chi+\eta-1]}).
\end{equation}
Then, for all $w \in \C \setminus \R$,
\begin{equation}
\label{eqf'}
f_{(\chi,\eta)}' (w)
= \int_{S_1} \frac{\mu[dx]}{w-x}
- \int_{S_2} \frac{(\l-\mu)[dx]}{w-x}
+ \int_{S_3} \frac{\mu[dx]}{w-x}.
\end{equation}
Thus $f_{(\chi,\eta)}'$ has a unique analytic extension to $\C \setminus S$,
where $S := S_1 \cup S_2 \cup S_3$.

In \cite{Duse15a, Duse15b}, we used the possible behaviours of the
roots of $f_{(\chi,\eta)}'$ in the above domain to examine the asymptotic
shapes. First we defined:
\begin{definition}
The liquid region, $\LL$, is the set of all $(\chi,\eta) \in [a,b] \times [0,1]$
with $b \ge \chi \ge \chi + \eta - 1 \ge a$, for which $f_{(\chi,\eta)}'$ has a root
in $\mathbb{H} := \{ w \in \C : \text{Im}(w) > 0 \}$.
\end{definition}
We showed that $f_{(\chi,\eta)}'$ has a unique root in $\mathbb{H}$
whenever $(\chi,\eta) \in \LL$, and this root is of multiplicity $1$.
Moreover, we showed that the resulting map from $\LL$ to $\mathbb{H}$ is a
homeomorphism, and so $\LL$ is a non-empty, open, connected subset of the
interior of the polygon on the right of figure \ref{figGTAsyShape}. We used the
homeomorphism to study $\partial \LL$. In \cite{Duse15a}, we obtained a complete
parameterisation of $\partial \LL$ for $\mu$ in a broad class. In
\cite{Duse15b}, we examined the highly non-trivial behaviours of $\partial \LL$
that can occur when $\mu$ is not restricted to this class. $\LL$ and $\partial \LL$
for some example measures, $\mu$, studied in \cite{Duse15a, Duse15b},
are given in figure \ref{figexs}.

For the purposes of this paper, we only need the results of \cite{Duse15a},
which we now recall in more detail. We showed that the inverse of the above
homeomorphism, from $\mathbb{H}$ to $\LL$, has a unique continuous extension
to a natural, non-empty, open subset of $\R$ which depends on $\mu$. We denoted
this set by $R \subset \R$, and showed that the extension is an injective
smooth curve, parameterised over $R$. We defined the {\em edge}, denoted
$\EE \subset \partial \LL$, to be the image of this curve, and
referred to this curve as the {\em edge curve}, denoted
\begin{equation*}
(\chi_\EE(\cdot),\eta_\EE(\cdot)) : R \to \EE \subset \partial \LL.
\end{equation*}

In \cite{Duse15a}, we found an alternative definition of $\EE$
which is analogous to that of $\LL$:
\begin{definition}
\label{defEdge}
The edge, $\EE$, is the union $\EE := \EE_\mu^+ \cup \EE_{\l-\mu} \cup
\EE_\mu^- \cup \EE_0 \cup \EE_1 \cup \EE_2$, where,
\begin{itemize}
\item
$\EE_\mu^+$ is the set of all $(\chi,\eta) \in [a,b] \times [0,1]$
with $b \ge \chi \ge \chi + \eta - 1 \ge a$, for which $f_{(\chi,\eta)}'$ has a
repeated root in $(\chi,+\infty) \setminus \supp(\mu)$.
\item
$\EE_{\l-\mu}$ is the set of all $(\chi,\eta)$ for which $f_{(\chi,\eta)}'$ has a
repeated root in $(\chi+\eta-1,\chi) \setminus \supp(\l-\mu)$.
\item
$\EE_\mu^-$ is the set of all $(\chi,\eta)$ for which $f_{(\chi,\eta)}'$ has a
repeated root in $(-\infty,\chi+\eta-1) \setminus \supp(\mu)$.
\item
$\EE_0 \cup \EE_1 \cup \EE_2$ is the set of $(\chi,\eta)$ for which
$f_{(\chi,\eta)}'$ has a root in $\{\chi,\chi+\eta-1\}$.
\end{itemize}
\end{definition}
For clarity, in \cite{Duse15a}, we state that we denoted
$\EE_\mu = \EE_\mu^+ \cup \EE_\mu^-$. This decomposition
is more convenient here. Moreover, we defined $\EE_0$ and $\EE_1$ and $\EE_2$
exactly in \cite{Duse15a}, but do not do so here for brevity.
We showed that the above two definitions of $\EE$ are equivalent:
We started with definition \ref{defEdge}, showed that the sets
in this definition are mutually disjoint, $f_{(\chi,\eta)}'$ has a
unique real-valued repeated root in $\R \setminus \{\chi,\chi+\eta-1\}$ when
$(\chi,\eta) \in \EE_\mu^+ \cup \EE_{\l-\mu} \cup \EE_\mu^-$, and
$f_{(\chi,\eta)}'$ has a unique root in $\{\chi,\chi+\eta-1\}$ when
$(\chi,\eta) \in \EE_0 \cup \EE_1 \cup \EE_2$. We showed that
the resulting map from $\EE$ to $\R$ is injective, has image space $R$, and has
inverse equal to the edge curve discussed above. Therefore the definitions
are trivially equivalent. We also showed that the multiplicity of the
unique root determines the geometric behaviour of the edge curve. Indeed,
denoting the multiplicity by $m = m(\chi,\eta) \ge1$, we showed that:
\begin{itemize}
\item
The edge curve behaves like a parabola in a neighbourhood of $(\chi,\eta)$
when $(\chi,\eta) \in \EE_\mu^+ \cup \EE_{\l-\mu} \cup \EE_\mu^-$ and $m=2$,
and when $(\chi,\eta) \in \EE_0 \cup \EE_1 \cup \EE_2$ and $m=1$.
\item
The edge curve behaves like an algebraic cusp of first order in a
neighbourhood of $(\chi,\eta)$ when $(\chi,\eta) \in \EE_\mu^+ \cup \EE_{\l-\mu} \cup \EE_\mu^-$
and $m=3$, and when $(\chi,\eta) \in \EE_1 \cup \EE_2$ and $m=2$.
\end{itemize}
For clarity we state that $m$ takes no other values.
Examples edge curves, with the above sets clearly labelled, are depicted in
figure \ref{figexs}. Finally, we showed that
$\EE \setminus \{(\chi,\eta) \in \EE_\mu^+ \cup \EE_{\l-\mu} \cup \EE_\mu^- : m = 2\}$
is either empty or discrete. We therefore now define:
\begin{definition}
\label{defEdgeTyp}
The set of typical edge points is
$\{(\chi,\eta) \in \EE_\mu^+ \cup \EE_{\l-\mu} \cup \EE_\mu^- : m = 2\}$.
\end{definition}

\begin{figure}
\centering
\mbox{\includegraphics{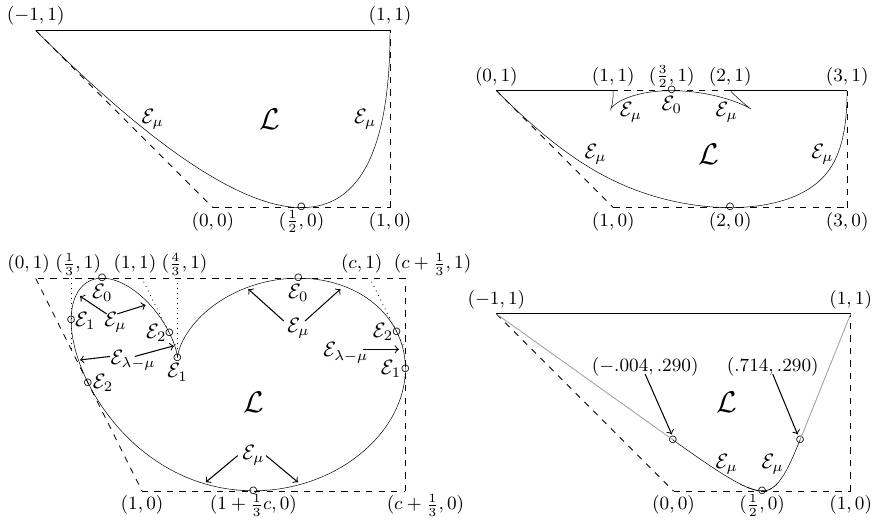}}
\caption{$\LL$ and $\partial \LL$ for some example measures,
$\mu$, with density $\varphi : \R \to [0,1]$. Top left:
$\varphi(x) = \frac12 \; \forall \; x \in [-1,1]$,
all points of $\EE$ are in $\EE_\mu$ with $m=2$. Top right:
$\varphi(x) = \frac12 \; \forall \; x \in [0,1] \cup [2,3]$, the cusps
are in $\EE_\mu$ with $m=3$, all other points of $\EE$ are in
$\EE_\mu$ with $m=2$. Lower left: $\varphi (x) = 1 \; \forall \;
x \in [0,\frac13] \cup [1,\frac43] \cup [c,c+\frac13]$,
where $c := \frac1{12} (23 + \sqrt{217})$, the cusp is in
$\EE_1$ with $m=2$, all other points in $\EE_1 \cup \EE_2$ have $m=1$.
Lower right: $\varphi (x) = \frac{15}{16} (x-1)^2 (x+1)^2 \; \forall \;
x \in [-1,1]$, the grey parts of $\partial \LL$ are not in $\EE$
but follow from the analysis in \cite{Duse15b}. For general $\mu$,
there is always an analogous lower convex part of $\EE$
which is contained in $\{(\chi,\eta) \in \EE_\mu : m = 2\}$, and
the lower tangent point always equals $(\frac12 + \int x \mu[dx],0)$.}
\label{figexs}
\end{figure}

We now consider the subset $\EE_\mu^+ \cup \EE_{\l-\mu} \cup \EE_\mu^- \subset \EE$,
in more detail. Let $C : \C \setminus \supp(\mu) \to \C$ denote the {\em Cauchy}
transform of $\mu$,
\begin{equation}
\label{eqCauTrans}
C(w) := \int_a^b \frac{\mu[dx]}{w-x},
\end{equation}
for all $w \in \C \setminus \supp(\mu)$. In \cite{Duse15a}, we showed that:
\begin{lem}
\label{lemEdge}
Recall that the edge curve, $(\chi_\EE(\cdot), \eta_\EE(\cdot)) : R \to \EE$,
is bijective. Define:
\begin{itemize}
\item
$R_\mu^+ := \{ t \in \R \setminus \supp(\mu) : C(t) > 0 \}$.
\item
$R_{\l-\mu} := \R \setminus \supp(\l-\mu)$.
\item
$R_\mu^- := \{ t \in \R \setminus \supp(\mu) : C(t) < 0 \}$.
\end{itemize}
Then, these are disjoint open subsets of $R$,
and the image spaces of these under the bijection are
(respectively) $\EE_\mu^+$, $\EE_{\l-\mu}$, $\EE_\mu^-$. Moreover,
for all $t \in R_\mu^+ \cup R_{\l-\mu} \cup R_\mu^-$,
\begin{equation}
\label{eqchiEEetaEE}
\chi_\EE(t) = t + \frac{e^{C(t)}-1}{e^{C(t)} C'(t)}
\hspace{0.5cm} \text{and} \hspace{0.5cm}
\eta_\EE(t) = 1 + \frac{(e^{C(t)}-1)^2}{e^{C(t)} C'(t)}.
\end{equation}
Finally, $(\chi_\EE(t),\eta_\EE(t)) \in (a,b) \times (0,1)$ and
$b > \chi_\EE(t) > \chi_\EE(t) + \eta_\EE(t) - 1 > a$
for all $t \in R_\mu^+ \cup R_{\l-\mu} \cup R_\mu^-$, i.e.,
$(\chi_\EE(t),\eta_\EE(t))$ is in the interior of the polygon
shown on the right of figure \ref{figGTAsyShape}.
\end{lem}

Note, equation (\ref{eqchiEEetaEE}) is well-defined whenever
$t \in R_\mu^+ \cup R_\mu^-$, since $R_\mu^+ \cup R_\mu^- \subset \R \setminus \supp(\mu)$.
Indeed, equation (\ref{eqCauTrans}) gives,
\begin{equation}
\label{eqC'Rmu}
C(t) = \int_a^b \frac{\mu[dx]}{t-x}
\hspace{.5cm} \text{and} \hspace{.5cm}
C'(t) = - \int_a^b \frac{\mu[dx]}{(t-x)^2},
\end{equation}
for all $t \in R_\mu^+ \cup R_\mu^-$. Moreover, it is well-defined
whenever $t \in R_{\l-\mu}$. Indeed, since
$\mu = \l$ in $R_{\l-\mu} = \R \setminus \supp(\l-\mu)$, lemma
2.2 of \cite{Duse15a} implies that $e^{C(\cdot)} : \C \setminus \R \to \C$
and $C'(\cdot) : \C \setminus R \to \C$ have the following unique analytic
extensions to $R_{\l-\mu}$:
\begin{equation}
\label{eqC'Rl-mu}
e^{C(t)} = e^{C_I(t)} \bigg( \frac{t-t_2}{t-t_1} \bigg)
\hspace{.5cm} \text{and} \hspace{.5cm}
C'(t) = C_I'(t) - \frac1{t-t_1} + \frac1{t-t_2},
\end{equation}
for all $t \in R_{\l-\mu}$, where $I = (t_2,t_1)$ is any interval
with $t \in I \subset R_{\l-\mu}$, and
$C_I(t) := \int_{[a,b] \setminus I} \frac{\mu[dx]}{t-x}$.
These expressions are independent of the choice of $I$.

Finally, as discussed above, the set of typical edge points is
$\{(\chi,\eta) \in \EE_\mu^+ \cup \EE_{\l-\mu} \cup \EE_\mu^- : m = 2\}$,
and the edge curve behaves like a parabola in a
neighbourhood of each $(\chi,\eta)$ in this set.
Fix the corresponding points $t \in R_\mu^+ \cup R_{\l-\mu} \cup R_\mu^-$
and $(\chi,\eta) \in \EE_\mu^+ \cup \EE_{\l-\mu} \cup \EE_\mu^-$
with $(\chi,\eta) = (\chi_\EE(t),\eta_\EE(t))$.
Define the (un-normalised) orthogonal vectors
$\mathbf{x}(t) := (1,e^{C(t)}-1)$ and $\mathbf{y}(t) := (e^{C(t)}-1,-1)$.
In \cite{Duse15a}, we show that $\mathbf{x}(t)$ and $\mathbf{y}(t)$
are (respectively) tangent and normal to the edge curve
at $(\chi,\eta)$.

\subsection{Motivation and statement of main results}

In this paper, we consider the universal asymptotic behaviour, as
$n \to \infty$, of the systems introduced in the last section,
in the neighbourhood of typical edge points (see definition
\ref{defEdgeTyp}). More specifically, we study the asymptotic
behaviour of $K_n((u_n,r_n),(v_n,s_n))$ as
$n \to \infty$, where $K_n$ is the correlation kernel of the system
(see equation (\ref{eqKnrusvFixTopLine})), and
$\{(u_n,r_n)\}_{n\ge1} \subset \Z^2$ and $\{(v_n,s_n)\}_{n\ge1} \subset \Z^2$
are sequence of particle positions which satisfy:

\vspace{.2cm}

{\em Fix the corresponding
points $t \in R_\mu^+ \cup R_{\l-\mu} \cup R_\mu^-$
and $(\chi,\eta) \in \EE_\mu^+ \cup \EE_{\l-\mu} \cup \EE_\mu^-$ with
$(\chi,\eta) = (\chi_\EE(t),\eta_\EE(t))$, where $t$ is a root of
$f_{(\chi,\eta)}'$ of multiplicity $2$, and take,
\begin{equation}
\label{equnrnvnsn}
(u_n,r_n) = n(\chi,\eta) + o(n)
\hspace{.25cm} \text{and} \hspace{.25cm}
(v_n,s_n) = n(\chi,\eta) + o(n)
\hspace{.25cm} \text{as} \hspace{.25cm} n \to \infty.
\end{equation}}

A steepest descent analysis of a contour integral expression for
$K_n((u_n,r_n),(v_n,s_n))$ (see equation (\ref{eqKnrnunsnvn1}))
is used to consider the asymptotic behaviour. Since $t$ is a root of
$f_{(\chi,\eta)}'$ of multiplicity $2$, equation (\ref{eqKnrnunsnvn1})
and steepest descent analysis intuitively
imply that universal edge asymptotic behaviour should
be observed. The main results of this paper, theorems \ref{thmAiry2} and \ref{thmAiry},
puts this intuition on a rigorous footing. We show, under natural conditions, that
the asymptotic behaviour of $ n^\frac13 K_n((u_n,r_n),(v_n,s_n))$ is governed
by the {\em extended Airy kernel}, $K_\text{Ai} : (\R^2)^2 : \to \R$:
First define $\widetilde{K}_\text{Ai} : (\R^2)^2 : \to \R$ by,
\begin{equation}
\label{eqAitilde}
\widetilde{K}_\text{Ai}((u,r),(v,s)) :=
\frac1{(2\pi i)^2} \int_l dw \int_L dz \; \frac1{w-z}
\frac{\exp(w r  + w^2 u + \frac13 w^3)}
{\exp(z s  + z^2 v + \frac13 z^3)},
\end{equation}
for all $(u,r),(v,s) \in \R^2$, where $l$ and $L$ are the contours
in figure \ref{figAirtCont}. Note that the above integrals are finite
since $w^3 = - |w|^3$ and $z^3 = |z|^3$ for all $w$ on $l$ and $z$ on
$L$ respectively. Next define $\Phi :  (\R^2)^2 \to \R$ by,
\begin{equation}
\label{eqPhi}
\Phi((u,r),(v,s)) := 1_{(u>v)} \; \frac1{2 \sqrt{\pi(u-v)}}
\exp \left( -\frac14 \frac{(s-r)^2}{u-v} \right),
\end{equation}
for all $(u,r),(v,s) \in \R^2$. Finally define
$K_\text{Ai} : (\R^2)^2 : \to \R$ by,
\begin{equation}
\label{eqAi}
K_\text{Ai} := \widetilde{K}_\text{Ai} - \Phi.
\end{equation}

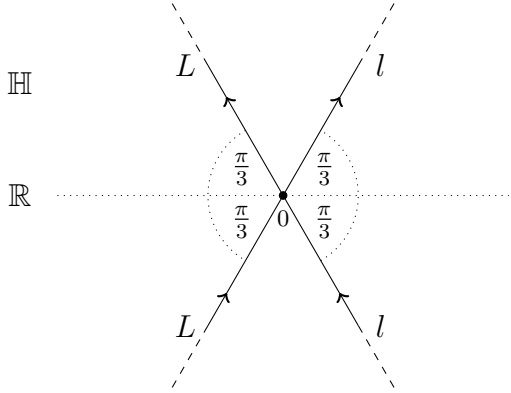
\begin{figure}[t]
\centering
\begin{tikzpicture};

\draw (-3.5,1.5) node {$\mathbb{H}$};
\draw [dotted] (-3,0) --++(6,0);
\draw (-3.5,0) node {$\R$};

\draw [fill] (0,0) circle (.05cm);
\draw (0,-.3) node {\scriptsize $0$};

\draw (0,0) --++ (1,1.732);
\draw [dashed] (1,1.732) --++ (.5,0.866);
\draw[arrows=->,line width=1pt](.75,1.299)--(.755,1.308);
\draw (1.3,1.732) node {$l$};
\draw (0,0) --++ (1,-1.732);
\draw [dashed] (1,-1.732) --++ (.5,-0.866);
\draw[arrows=->,line width=1pt](.755,-1.308)--(.75,-1.299);
\draw (1.3,-1.732) node {$l$};

\draw (0,0) --++ (-1,1.732);
\draw [dashed] (-1,1.732) --++ (-.5,0.866);
\draw[arrows=->,line width=1pt](-.75,1.299)--(-.755,1.308);

\draw (-1.3,1.732) node {$L$};
\draw (0,0) --++ (-1,-1.732);
\draw [dashed] (-1,-1.732) --++ (-.5,-0.866);
\draw[arrows=->,line width=1pt](-.755,-1.308)--(-.75,-1.299);
\draw (-1.3,-1.732) node {$L$};

\draw [dotted,domain=-60:60] plot ({cos(\x)}, {sin(\x)});
\draw [dotted,domain=120:240] plot ({cos(\x)}, {sin(\x)});
\draw (.55,.35) node {$\frac\pi3$};
\draw (-.55,.35) node {$\frac\pi3$};
\draw (-.55,-.35) node {$\frac\pi3$};
\draw (.55,-.35) node {$\frac\pi3$};

\end{tikzpicture}
\caption{The contours $l$ and $L$ of equation (\ref{eqAitilde}).
$l$ is the (straight) lines from $\infty e^{-i \frac\pi3}$ to
$0$, and from $0$ to $\infty e^{i \frac\pi3}$. $l$ is the
lines from $\infty e^{-i \frac{2\pi}3}$ to
$0$, and from $0$ to $\infty e^{i \frac{2\pi}3}$.}
\label{figAirtCont}
\end{figure}

We begin the analysis with a technical assumption that arises from steepest
descent considerations. First, for all $n \ge 1$, define
\begin{equation}
\label{eqPnHn}
P_n := \tfrac1n \{x_1,x_2,\ldots,x_n\} 
\;\; \text{and} \;\;
H_n := \tfrac1n (\Z \setminus \{x_1,x_2,\ldots,x_n\}).
\end{equation}
Above, we again omit the superscript from $x^{(n)} = (x_1^{(n)},x_2^{(n)},\ldots,x_n^{(n)})$
for simplicity. $P_n$ is referred to as the set of {\em particles}, and $H_n$ as
the set of {\em holes}. Note, an element of these sets may act as a pole for the
contour integral expression of equation (\ref{eqKnrnunsnvn1}), and so a problem may
arise in the steepest descent analysis if these are not {\em eventually isolated}
from the root, $t$, in equation (\ref{equnrnvnsn}). We therefore assume:
\begin{ass}
\label{assIsol}
Assume that $d(P_n, \supp(\mu)) \to 0$ and $d(H_n, \supp(\l-\mu)) \to 0$
as $n \to \infty$, where $d$ represents the Hausdorff distance.
\end{ass}
To see that this assumption has the desired effect, recall that
$t \in R_\mu^+ \cup R_{\l-\mu} \cup R_\mu^-$, a union of mutually
disjoint open sets. Thus there exists a fixed $\e>0$ for which:
\begin{align}
\nonumber
&t \in R_\mu^+(\e) := \{s \in R_\mu^+ : (s-\e,s+\e) \subset R_\mu^+\}
\text{ when } t \in R_\mu^+. \\
\label{eqtRmuRlmu}
&t \in R_{\l-\mu}(\e) := \{s \in R_{\l-\mu} : (s-\e,s+\e) \subset R_{\l-\mu}\}
\text{ when } t \in R_{\l-\mu}. \\
\nonumber
&t \in R_\mu^-(\e) := \{s \in R_\mu^- : (s-\e,s+\e) \subset R_\mu^-\}
\text{ when } t \in R_\mu^-.
\end{align}
Also, since $R_\mu^+ \cup R_\mu^- \subset \R \setminus \supp(\mu)$ and
$R_{\l-\mu} = \R \setminus \supp(\l-\mu)$, assumption \ref{assIsol}
implies that $P_n \subset \R \setminus (R_\mu^+(\e) \cup R_\mu^-(\e))$ and
$H_n \subset \R \setminus R_{\l-\mu}(\e)$ for all $n$ sufficiently
large, as desired. Finally note that we can equivalently write,
\begin{equation}
\label{eqPnHnlarge}
P_n \cap R_\mu^+(\e) = \emptyset
\hspace{0.25cm} \text{and} \hspace{0.25cm}
\frac{\Z}n \cap R_{\l-\mu}(\e) \subset P_n
\hspace{0.25cm} \text{and} \hspace{0.25cm}
P_n \cap R_\mu^-(\e) = \emptyset,
\end{equation}
for all $n$ sufficiently large. Indeed, the second part follows
since $H_n = \frac{\Z}n \setminus P_n$ (see equation (\ref{eqPnHn})),
and it implies that particles are {\em eventually
densely packed} in $R_{\l-\mu}(\e)$.

Note, assumption \ref{assIsol} imposes mild regulatory restrictions
on $x^{(n)}$. Our final assumption, assumption \ref{asscases},
similarly imposes mild regulatory restrictions on $x^{(n)}$.
Assumption \ref{asscases} is more subtle, however, and applies only in specific
cases, and so we leave the statement of this to section \ref{secTROft}.
The regularity effect of assumption \ref{asscases} can be seen, for example,
in part (ii) in the proof of lemma \ref{lemJn}. For the relevant cases
of assumption \ref{asscases}, part (ii)
is not necessarily true if assumption \ref{asscases} does not hold.
Note also, assumptions \ref{assIsol} and \ref{asscases} are sufficient
for theorems \ref{thmAiry2} and \ref{thmAiry} to be satisfied for
{\em all typical edge points}, i.e., for all corresponding
points $t$ and $(\chi,\eta)$ chosen as in equation (\ref{equnrnvnsn}).
Though we do not discuss this further, weaker forms of these assumptions
can be used if we are only interested in specific edge points.

Next, we motivate the choice of the $o(n)$ terms in equation
(\ref{equnrnvnsn}). Note, since the convergence in assumption
\ref{assWeakConv} is weak, there is no control of the rate of
convergence. It is therefore not natural to simply consider fluctuations
of the particles around the asymptotic edge curve. Instead, we
consider fluctuations around analogous {\em non-asymptotic edge curves}.
Intuitively, we could define these by replacing all Cauchy transforms
(see equation (\ref{eqCauTrans})) in equation (\ref{eqchiEEetaEE}) with
the following non-asymptotic analogue inspired by assumption
\ref{assWeakConv}: $w \mapsto \frac1n \sum_{i=1}^n (w-\frac{x_i}n)^{-1}
= \frac1n \sum_{x \in P_n} ({w-x})^{-1}$ for all $n \ge 1$ and
$w \in \C \setminus \R$. However, since it is desirable that the
non-asymptotic edge curves and the asymptotic edge curve have
approximately the same domain of definition, we use a modified non-asymptotic
Cauchy transform. First, fixing $\e>0$ as in equation (\ref{eqtRmuRlmu}),
define a new non-asymptotic measure by,
\begin{equation}
\label{eqWeakmun}
\mu_n := \frac1n \sum_{x \in P_n ; x \not\in R_{\l-\mu}(\e)} \delta_x
+ \l |_{R_{\l-\mu}(\e)},
\end{equation}
for all $n \ge 1$. Assumption \ref{assWeakConv} and equation (\ref{eqPnHn})
then imply that $\mu_n \to \mu$ weakly as $n \to \infty$. Next, let
$C_n : \C \setminus \supp(\mu_n) \to \C$ denote the Cauchy Transform of
$\mu_n$:
\begin{equation}
\label{eqCauTransn}
C_n(w) := \frac1n \sum_{x \in P_n ; x \not\in R_{\l-\mu}(\e)} \frac1{w-x}
+ \int_{R_{\l-\mu}(\e)} \frac{dx}{w-x},
\end{equation}
for all $n \ge 1$ and $w \in \C \setminus \supp(\mu_n)$. Note that
$R_\mu(\e)^+ \cup R_\mu(\e)^- \subset \C \setminus \supp(\mu_n)$
for all $n$ sufficiently large, since
$P_n \cap (R_\mu(\e)^+ \cup R_\mu(\e)^-) = \emptyset$ (see equation (\ref{eqPnHnlarge})),
and since $R_\mu(\e)^+ \cup R_\mu(\e)^-$ and $R_{\l-\mu}(\e)$ are disjoint
(see equation (\ref{eqtRmuRlmu})). Therefore $C_n$ is well-defined and analytic
in $R_\mu(\e)^+ \cup R_\mu(\e)^-$. Also, since $\mu_n = \l$ in $R_{\l-\mu}(\e)$
for all $n$ sufficiently large, lemma 2.2 of \cite{Duse15a} shows that $e^{C_n}$
and $C_n'$ have unique analytic extensions to $R_{\l-\mu}(\e)$. Finally, inspired
by equation (\ref{eqchiEEetaEE}), define:
\begin{definition}
\label{defEdgeNonAsy}
Fix $\e>0$ as in equation (\ref{eqtRmuRlmu}), and define $C_n$ as in equation
(\ref{eqCauTransn}). Then, for all $n$ sufficiently large, the non-asymptotic edge curves,
$(\chi_n(\cdot),\eta_n(\cdot)) : R_\mu^+(\e) \cup R_{\l-\mu}(\e) \cup R_\mu^-(\e)\to \R^2$,
are defined by,
\begin{equation*}
\chi_n(s) := s + \frac{e^{C_n(s)}-1}{e^{C_n(s)} C_n'(s)}
\hspace{0.5cm} \text{and} \hspace{0.5cm}
\eta_n(s) := 1 + \frac{(e^{C_n(s)}-1)^2}{e^{C_n(s)} C_n'(s)},
\end{equation*}
for all $s \in R_\mu^+(\e) \cup R_{\l-\mu}(\e) \cup R_\mu^-(\e)$.
Moreover, we define $(\chi_n,\eta_n) := (\chi_n(t),\eta_n(t))$.
\end{definition}

Note, since $\mu_n \to \mu$ weakly as $n \to \infty$,
\begin{equation}
\label{eqNonAsyEdge}
e^{C_n(s)} \to e^{C(s)}
\hspace{0.25cm} \text{and} \hspace{0.25cm}
C_n'(s) \to C'(s)
\hspace{0.25cm} \text{and} \hspace{0.25cm}
(\chi_n(s),\eta_n(s)) \to (\chi_\EE(s),\eta_\EE(s)),
\end{equation}
for all $s \in R_\mu^+(\e) \cup R_{\l-\mu}(\e) \cup R_\mu^-(\e)$.
As observed above, however, we have no control of the rate of convergence.

\begin{rem}
Note, the above definition depends on an arbitrary $\e>0$.
Suppose $\e'>0$ is any other value which also satisfies equation
(\ref{eqtRmuRlmu}), and let $(\chi_n',\eta_n')$ denote the
analogous non-asymptotic  edge obtained using $\e'$. Then,
equations (\ref{eqtRmuRlmu}, \ref{eqPnHnlarge}, \ref{eqCauTransn}),
definition \ref{defEdgeNonAsy}, and Riemann sum approximations
imply that $(\chi_n',\eta_n') = (\chi_n,\eta_n) + O(n^{-1})$
for all $n$ sufficiently large. This error is absorbed into the errors
of order $O(1)$ in equations (\ref{equnrnvnsn2}, \ref{equnrnvnsn3}),
below, and possibly effects the asymptotic behaviour in theorems \ref{thmAiry2}
and \ref{thmAiry} when $(u,r) = (v,s)$. The asymptotic behaviour when
$(u,r) \neq (v,s)$ is unaffected. 
\end{rem} 

Next recall (see end of last section), that the asymptotic edge curve
behaves like a parabola in a neighbourhood of $(\chi,\eta)$ with
tangent vector $\mathbf{x}(t) := (1,e^{C(t)}-1)$ and normal vector
$\mathbf{y}(t) := (e^{C(t)}-1,-1)$. This was proven in lemma 2.9 of
\cite{Duse15a}. Proceeding similarly for the non-asymptotic edge
curves for all $n$ sufficiently large, we can show that these also behave like a parabola in a
neighbourhood of $(\chi_n,\eta_n)$, with tangent vector
$\mathbf{x}_n(t) := (1,e^{C_n(t)}-1)$ and normal vector
$\mathbf{y}_n(t) := (e^{C_n(t)}-1,-1)$. Finally, we choose the
$o(n)$ terms in equation (\ref{equnrnvnsn}) as follows: Fix $(u,r) \in \R^2$
and $(v,s) \in \R^2$, and let $\{(u_n,r_n)\}_{n\ge1}$ and
$\{(v_n,s_n)\}_{n\ge1}$ be sequences in $\Z^2$ which satisfy,
\begin{align}
\label{equnrnvnsn2}
(u_n,r_n) &= n (\chi_n,\eta_n) + n^{\frac23} m_n \mathbf{x}_n \; u
+ n^{\frac13} p_n \mathbf{y}_n \; r + O(1), \\
\label{equnrnvnsn3}
(v_n,s_n) &= n (\chi_n,\eta_n) + n^{\frac23} m_n \mathbf{x}_n \; v
+ n^{\frac13} p_n \mathbf{y}_n \; s + O(1),
\end{align}
for all $n$ sufficiently large, where $\{m_n\}_{n\ge1} = \{m_n(t)\}_{n\ge1}$
and $\{p_n\}_{n\ge1} = \{p_n(t)\}_{n\ge1}$ are those convergent sequences of
real numbers with non-zero limits given in definition \ref{defmnpn}.
In words, $(\chi_n,\eta_n) \to (\chi,\eta)$ as $n \to \infty$ at an
indeterminate rate, and $\{\frac1n (u_n,r_n)\}_{n\ge1}$ and
$\{\frac1n (v_n,s_n)\}_{n\ge1}$ fluctuate around $(\chi_n,\eta_n)$.
The fluctuations are of order $O(n^{-\frac13})$ and $O(n^{-\frac23})$ respectively in
the tangent and normal directions of the non-asymptotic edge curve, $u$
and $v$ measure the size of the $O(n^{-\frac13})$ fluctuations,
and $r$ and $s$ measure the size of the $O(n^{-\frac23})$ fluctuations.

We finally state the main results of this paper:
\begin{thm}
\label{thmAiry2}
Assume assumptions \ref{assWeakConv} and \ref{assIsol}.
Fix any corresponding points $t \in R_\mu^+ \cup R_{\l-\mu} \cup R_\mu^-$
and $(\chi,\eta) \in \EE_\mu^+ \cup \EE_{\l-\mu} \cup \EE_\mu^-$ with
$(\chi,\eta) = (\chi_\EE(t),\eta_\EE(t))$, where $t$ is a root of
$f_{(\chi,\eta)}'$ of multiplicity $2$. Assume assumption
\ref{asscases} if one of the relevant cases is satisfied, and
choose $\{(u_n,r_n)\}_{n\ge1}$ and $\{(v_n,s_n)\}_{n\ge1}$ as in
equations (\ref{equnrnvnsn2}, \ref{equnrnvnsn3}).
Additionally assume that either $(u,r) \neq (v,s)$,
or $(u,r) = (v,s)$ and $r_n = s_n$ for all $n$ sufficiently large.

Let $\mathcal{K}_n$ be the equivalent correlation kernel
defined in equation (\ref{eqKntilde}), and $K_\text{Ai}$ be the extended Airy
kernel defined in equation (\ref{eqAi}). Define $C'(t)$ as in equations
(\ref{eqC'Rmu}, \ref{eqC'Rl-mu}), note that $C'(t) \neq 0$ (see lemma
\ref{lemAnalExt}, below), and define $\beta(t) > 0$ by
$\beta(t) := 2^{\frac13} |f_{(\chi,\eta)}'''(t)|^{-\frac13} |C'(t)|$.
Then, as $n \to \infty$,
\begin{align*}
n^\frac13 \beta(t)^{-1} \mathcal{K}_n((u_n,r_n),(v_n,s_n))
&\to K_\text{Ai}((u,r),(v,s))
\text{ when } t \in R_\mu^+ \cup R_\mu^-, \\
n^\frac13 \beta(t)^{-1} (1_{(u_n = v_n)} - \mathcal{K}_n((u_n,r_n),(v_n,s_n)))
&\to K_\text{Ai}((u,r),(v,s))
\text{ when } t \in R_{\l-\mu}.
\end{align*}
\end{thm}

In other words, we let $(\chi,\eta)$ be any typical edge point (see
definition \ref{defEdgeTyp}), we let $(\chi_n,\eta_n)$ denote the
analogous point on the non-asymptotic edge for each $n$
(see definition \ref{defEdgeNonAsy}), and we choose $\frac1n (u_n,r_n)$
and $\frac1n (v_n,s_n)$ (the rescaled particle positions) to fluctuate around
$(\chi_n,\eta_n)$ as described by equations (\ref{equnrnvnsn2}, \ref{equnrnvnsn3}).
Then, when $t \in R_\mu^+ \cup R_\mu^-$, theorem
\ref{thmAiry2} shows that $n^\frac13 \mathcal{K}_n((u_n,r_n),(v_n,s_n))$
(the rescaled correlation kernel of the particles) converges to the
extended Airy kernel as $n \to \infty$. When $t \in R_{\l-\mu}$, theorem
\ref{thmAiry2} shows that $n^\frac13 (1_{(u_n = v_n)} - \mathcal{K}_n((u_n,r_n),(v_n,s_n)))$
(the rescaled correlation kernel of the `holes') converges to the
extended Airy kernel as $n \to \infty$.

Note that theorem \ref{thmAiry2} does not give the asymptotic behaviour
when $(u,r) = (v,s)$ and $r_n \neq s_n$. The following theorem completes
the result by giving the asymptotic behaviour for all cases.
We prove the existence of a term, $\alpha_n$,
which has a well-defined asymptotic behaviour for the cases of theorem
\ref{thmAiry2}, but (possibly) has no well-defined asymptotic behaviour
when $(u,r) = (v,s)$ and $r_n \neq s_n$. Theorem \ref{thmAiry2} is, in
fact, a trivial corollary of theorem \ref{thmAiry}. We will prove theorem
\ref{thmAiry}, and leave the deduction of theorem \ref{thmAiry2} to the
interested reader:
\begin{thm}
\label{thmAiry}
Assume assumptions \ref{assWeakConv} and \ref{assIsol}.
Fix any corresponding points $t \in R_\mu^+ \cup R_{\l-\mu} \cup R_\mu^-$
and $(\chi,\eta) \in \EE_\mu^+ \cup \EE_{\l-\mu} \cup \EE_\mu^-$ with
$(\chi,\eta) = (\chi_\EE(t),\eta_\EE(t))$, where $t$ is a root of
$f_{(\chi,\eta)}'$ of multiplicity $2$. Assume assumption
\ref{asscases} if one of the relevant cases is satisfied, and
choose $\{(u_n,r_n)\}_{n\ge1}$ and $\{(v_n,s_n)\}_{n\ge1}$ as in
equations (\ref{equnrnvnsn2}, \ref{equnrnvnsn3}).

Define $\widetilde{K}_\text{Ai} : (\R^2)^2 : \to \R$
as in equation (\ref{eqAitilde}), $A_{t,n} : (\Z^2)^2 \to \R \setminus \{0\}$
as in lemma \ref{lemFnt}, $B_n : \Z_+^2 \to (0,+\infty)$ as in equation
(\ref{eqConjFactB}), $\beta(t) > 0$ as in theorem \ref{thmAiry2},
and $\alpha_n \ge 0$ by,
\begin{equation*}
\alpha_n
:= \begin{dcases}
\frac{1_{(u_n \ge v_n, r_n > s_n)}}
{|A_{t,n} ((v_n,s_n),(u_n,r_n))| B_n(s_n,r_n)}
\frac{(u_n+r_n-v_n-s_n-1)!}{(r_n-s_n-1)! (u_n-v_n)!}
& ; \; t \in R_\mu^+, \\
\frac{1_{(u_n \ge v_n, u_n+r_n \le v_n+s_n, r_n \le s_n)}}
{|A_{t,n} ((v_n,s_n),(u_n,r_n))| B_n(s_n,r_n)}
\frac{(s_n-r_n)!}{(u_n-v_n)! (v_n+s_n-u_n-r_n)!}
& ; \; t \in R_{\l-\mu}, \\
\frac{1_{(u_n+r_n \le v_n+s_n, r_n > s_n)}}
{|A_{t,n} ((v_n,s_n),(u_n,r_n))| B_n(s_n,r_n)}
\frac{(v_n-u_n-1)!}{(r_n-s_n-1)! (v_n+s_n-u_n-r_n)!}
& ; \; t \in R_\mu^-.
\end{dcases}
\end{equation*}
Then, as $n \to \infty$,
\begin{align*}
\mathcal{K}_n((u_n,r_n),(v_n,s_n))
&= - \alpha_n + n^{-\frac13} \; \beta(t) \; \widetilde{K}_\text{Ai}((u,r),(v,s))
+ o(n^{-\frac13})
\text{ when } t \in R_\mu^+ \cup R_\mu^-, \\
\mathcal{K}_n((u_n,r_n),(v_n,s_n))
&= + \alpha_n - n^{-\frac13} \; \beta(t) \; \widetilde{K}_\text{Ai}((u,r),(v,s))
+ o(n^{-\frac13})
\text{ when } t \in R_{\l-\mu}.
\end{align*}
Moreover, for all $n$ sufficiently large,
$\alpha_n = 0$ when $(u,r) = (v,s)$ and $r_n = s_n$ and $t \in R_\mu^+ \cup R_\mu^-$,
$\alpha_n = 1_{(u_n = v_n)}$ when $(u,r) = (v,s)$ and $r_n = s_n$ and $t \in R_{\l-\mu}$,
and $\alpha_n = O(1)$ when $(u,r) = (v,s)$ and $r_n \neq s_n$. Finally,
\begin{equation*}
\alpha_n = n^{-\frac13} \; 1_{(u>v)} \; \beta(t) \; \Phi((u,r),(v,s))
+ o(n^{-\frac13}),
\end{equation*}
as $n \to \infty$ when $(u,r) \neq (v,s)$,
where $\Phi :  (\R^2)^2 : \to \R$ is defined in equation (\ref{eqPhi}).
\end{thm}

Note that theorems \ref{thmAiry2} and \ref{thmAiry} prove pointwise convergence.
However, our methods also give the following extension:
\begin{thm}
\label{thmAiry3}
Assume assumptions \ref{assWeakConv} and \ref{assIsol}.
Fix any corresponding points $t \in R_\mu^+ \cup R_{\l-\mu} \cup R_\mu^-$
and $(\chi,\eta) \in \EE_\mu^+ \cup \EE_{\l-\mu} \cup \EE_\mu^-$ with
$(\chi,\eta) = (\chi_\EE(t),\eta_\EE(t))$, where $t$ is a root of
$f_{(\chi,\eta)}'$ of multiplicity $2$. Assume assumption
\ref{asscases} if one of the relevant cases is satisfied.
Finally, fix compact sets $U,R,V,S \subset \R$. Then the convergence
in theorems \ref{thmAiry2} and \ref{thmAiry} holds uniformly for
$(u,r) \in U \times R$ and $(v,s) \in V \times S$.
\end{thm}
For clarity, we state that theorem \ref{thmAiry3} follows simply by taking
the appropriate uniform bounds at every step of our proof. 
We do not attempt to prove theorem \ref{thmAiry3} directly for the sake of
readability.

We end this section by stating that theorem \ref{thmAiry3} does not
truly prove {\em edge universality}: It does not yet prove that
the (rescaled) Gelfand-Tsetlin particle process at the edge
converges to the extended Airy kernel particle process.
Convergence of the respective Fredholm determinants of the
processes remains to be shown. To do this,
one could attempt to find an additional uniform bound of
$|K_n((u_n,r_n),(v_n,s_n))|$, similar to that seen in lemma 3.1
of \cite{Joh05a}. This step is often overlooked in the literature,
and we do not attempt to prove this here due to the
length and complexity of the paper. Nevertheless,
the results of this paper are an important step towards proving
edge universality for this broad class of models.

\subsection{Other asymptotic situations and conjectures}
\label{secOAS}

Recall the discussions given in sections \ref{secoverview} and
\ref{sectasoGTp}. Assume assumption \ref{assWeakConv}, fix
$(\chi,\eta)$ in the polygon on the right of figure \ref{figGTAsyShape},
and fix sequences of particle positions, $\{(u_n,r_n)\}_{n\ge1}$ and
$\{(v_n,s_n)\}_{n\ge1}$, which satisfy $\frac1n (u_n,r_n) \to (\chi,\eta)$
and $\frac1n (v_n,s_n) \to (\chi,\eta)$ as $n \to \infty$.
In equation (\ref{equnrnvnsn}), we fixed the corresponding points
$(\chi,\eta) \in \EE_\mu^+ \cup \EE_{\l-\mu} \cup \EE_\mu^-$,
and $t \in R_\mu^+ \cup R_{\l-\mu} \cup R_\mu^-$,
where $t$ is a root of $f_{(\chi,\eta)}'$ of multiplicity $2$ ($m=2$). In this section,
instead of equation (\ref{equnrnvnsn}), we briefly discuss other
natural asymptotic situations.

First, suppose that $(\chi,\eta) \in \LL$, and let $w_0 \in \mathbb{H}$,
denote the corresponding root of $f_{(\chi,\eta)}'$ of multiplicity $1$.
In \cite{Duse15a}, for general $\mu$, we mapped $\LL$ to $\mathbb{H}$ by
mapping to the unique root, showed that this is a homeomorphism (indeed,
a diffeomorphism), and found an expression for the inverse of this map.
In other words, we wrote $(\chi,\eta)$ as an explicit function of $w_0$.
Steepest descent analysis, combined with the above root behaviour
and equation (\ref{eqKnrnunsnvn1}), intuitively suggests that
$K_n((u_n,r_n),(v_n,s_n))$ should converge to the
{\em Sine} kernel as $n \to \infty$ (see, for example,
\cite{Joh01, Pas97}). When $r_n = s_n$,
convergence to the standard Sine kernel was shown in Metcalfe,
\cite{Met13}, for the analogous interlaced particle system where
the particles on each row of the Gelfand-Tsetlin patterns take
positions in $\R$. To our knowledge,
this situation has not yet been studied in this new setting.

Next suppose that $(\chi,\eta) \in \EE_\mu^+ \cup \EE_{\l-\mu} \cup \EE_\mu^-$,
and the corresponding root, $t \in R_\mu^+ \cup R_{\l-\mu} \cup R_\mu^-$,
of $f_{(\chi,\eta)}'$ is of multiplicity $3$ ($m=3$). Recall that
the set of all such points is discrete, and they define algebraic
cusps of first order in the edge curve. Examples are
shown in the top right of figure \ref{figexs}. Steepest descent
analysis, combined with the above root
behaviour and equation (\ref{eqKnrnunsnvn1}), intuitively suggests that
$K_n((u_n,r_n),(v_n,s_n))$ should converge to the
{\em Pearcey} kernel as $n \to \infty$ (see, for example,
Tracy and Widom, \cite{Tr06}). To our knowledge,
this situation has not yet been studied in this new setting.

Next suppose that $(\chi,\eta) \in \EE_0 \cup \EE_1 \cup \EE_2$.
Recall that the set of all such points is discrete, $m=1$ when
$(\chi,\eta) \in \EE_0$, $m=1$ or $m=2$ when
$(\chi,\eta) \in \EE_1 \cup \EE_2$, the edge curve
behaves like a parabola when $m=1$, and the edge curve
behaves like an algebraic cusp of first order when $m=2$. Examples are
shown in the lower left of figure \ref{figexs}. In
\cite{Duse15d}, we examined the local asymptotic behaviour when
$(\chi,\eta) \in \EE_1 \cup \EE_2$ and $m=2$. Scaling
$\{(u_n,r_n)\}_{n\ge1}$ and $\{(v_n,s_n)\}_{n\ge1}$ appropriately,
we showed that $K_n((u_n,r_n),(v_n,s_n))$ converges to a novel
kernel which we called the {\em Cusp-Airy} kernel. To our knowledge,
the situation when $(\chi,\eta) \in \EE_0 \cup \EE_1 \cup \EE_2$
and $m=1$ has not yet been studied. However, steepest descent analysis,
combined with the above root behaviour and equation (\ref{eqKnrnunsnvn1}),
intuitively suggests a universal asymptotic behaviour.

Finally suppose that $(\chi,\eta)$ lies on that grey part of $\partial \LL$
for the example depicted on the bottom right of figure \ref{figexs}.
In \cite{Duse15b}, we showed that the behaviour of the roots of
$f_{(\chi,\eta)}'$ is identical for every point on these sections.
Therefore these sections cannot be parameterised by defining unique roots
for each point, as we did for $\EE$. In \cite{Duse15b}, we instead made heavy
use of the theory of singular integrals to parameterise these sections. We
also examined similar, surprisingly subtle, situations. We now conjecture that
the local asymptotic behaviour of particles in neighbourhoods of $(\chi,\eta)$,
when $(\chi,\eta)$ lies on such sections of $\partial \LL$, are
non-universal.

\subsection{Notation and terminology}
\label{secNot}

We end the introduction by discussing the notation and terminology
that will be in use for the rest of the paper. The arguments are
quite technical, and we use these for simplicity, and to avoid needless repetition.
First recall (see equation (\ref{equnrnvnsn})) than we fix the corresponding
points $t \in R_\mu^+ \cup R_{\l-\mu} \cup R_\mu^-$
and $(\chi,\eta) \in \EE_\mu^+ \cup \EE_{\l-\mu} \cup \EE_\mu^-$ with
$(\chi,\eta) = (\chi_\EE(t),\eta_\EE(t))$. From now on, we simply label
$f_{(\chi,\eta)}$ with $t$ instead of $(\chi,\eta)$, i.e., we define,
\begin{equation}
\label{eqft}
f_t := f_{(\chi_\EE(t), \eta_\EE(t))} = f_{(\chi, \eta)}.
\end{equation}
Next, fix $a \in \C$, and a measure $\nu$ on $\R$. Also, for all $n \ge 1$,
fix $a_n \in \C$ and $b_n \in \C$, and a measure $\nu_n$ on $\R$. Then:
\begin{itemize}
\item
`$a_n \to a$' means `$a_n \to a$ as $n \to \infty$'.
\item
`$\nu_n \to \nu$ weakly' means `$\nu_n \to \nu$ as $n \to \infty$ in
the sense of weak convergence of measures'.
\item
`$a_n = b_n$' means `$a_n = b_n$ for all $n$ sufficiently large'.
In other words, if we do not explicitly state those $n \ge 1$ for which
$a_n = b_n$, then we implicitly understand that the equality holds for all
$n$ sufficiently large. Similarly for {\em any other expression or statement}
involving $n$.
\end{itemize}
Next, for all $n\ge1$, fix $c_n \in \R$ with $c_n > 0$. Then:
\begin{itemize}
\item
`$o(c_n)$' denotes a complex-number which is well-defined for all
$n$ sufficiently large. Moreover, $|o(c_n)|/c_n \to 0$ as $n \to \infty$.
\item
`$O(c_n)$' denotes a complex-number which  is well-defined for all $n$ sufficiently
large. Moreover, there exists a $C>0$ for which $|O(c_n)|/c_n \le C$
for all such $n$.
\end{itemize}
Next, fix $A \subset \C$, $f : \C \to \C$, and $m \ge 1$. Also,
for all $n\ge1$, fix $A_n \subset \C$, $f_n : \C \to \C$ and $g_n : \C \to \C$. Then:
\begin{itemize}
\item
`$f$ has $m$ roots in $A$' means
`$f$ has exactly $m$ roots in $A$ counting multiplicities'.
Similarly for `at least/most $m$ roots'.
\item
`$f_n$ has $m$ roots in $A_n$' means
`$f_n$ has exactly $m$ roots in $A_n$ for all $n$ sufficiently large
and counting multiplicities'.
Similarly for `at least/most $m$ roots'.
\item
`$f_n \to f$ uniformly in $A$' means
`$\sup_{w \in A} |f_n(w) - f(w)| \to 0$' as $n \to \infty$.
\item
`$|f_n(w)| > c_n$ uniformly for $w \in A$' means
`$\inf_{w \in A} |f_n(w)| > c_n$ for all $n$ sufficiently large'.
\item
`$f_n(w) = g_n(w) + o(c_n)$ uniformly for $w \in A$' means
`$\sup_{w \in A} |f_n(w) - g_n(w)|/c_n \to 0$' as $n \to \infty$.
\item
`$f_n(w) = g_n(w) + O(c_n)$ uniformly for $w \in A$' means
`there exists a $C>0$ for which
$\sup_{w \in A} |f_n(w) - g_n(w)|/c_n \le C$ for all $n$ sufficiently large.'
\end{itemize}
Finally, given $B \subset \C$, and $S \subset \R$ bounded:
\begin{itemize}
\item
$\text{cl}(B)$ is the closure of $B$.
\item
$\overline{S} := \sup S$ and $\underline{S} := \inf S$.
\end{itemize}

\section{The roots of the steepest descent functions}
\label{secabotrofn}

In this section, we assume assumptions \ref{assWeakConv} and \ref{assIsol},
and that equation (\ref{equnrnvnsn}) is satisfied for some fixed
$t \in R_\mu^+ \cup R_{\l-\mu} \cup R_\mu^-$, a root of
$f_t'$ of multiplicity $2$. We consider the behaviour of the roots of
$f_t'$, $f_n'$ and $\tilde{f}_n'$ (see equations (\ref{eqfn}, \ref{eqtildefn},
\ref{eqf}, \ref{eqft})). This results of this section enable us to perform
the steepest descent analysis of section \ref{secsdatak}.

\subsection{The roots of $f_t'$}
\label{secTROft}

In this section, we consider the behaviour of the roots of $f_t'$.
Recall the following basic facts which we make use of throughout this section:
Assumption \ref{assWeakConv} implies that $\mu$ is a probability measure on
$[a,b]$ with $\{a,b\} \in \supp(\mu)$ and $\mu \le \l$. Equation
(\ref{equnrnvnsn}) implies that $t \in R_\mu^+ \cup R_{\l-\mu} \cup R_\mu^-$
and $(\chi,\eta) \in \EE_\mu^+ \cup \EE_{\l-\mu} \cup \EE_\mu^-$ are the
corresponding points with $(\chi,\eta) = (\chi_\EE(t),\eta_\EE(t))$, and
$t$ is a root of $f_{(\chi,\eta)}'$ of multiplicity $2$. Also, since
$(\chi,\eta) = (\chi_\EE(t),\eta_\EE(t))$, definition
\ref{defEdge} and lemma \ref{lemEdge} imply that $t \in \R \setminus \{\chi,\chi+\eta-1\}$,
$(\chi,\eta) \in (a,b) \times (0,1)$, and $b > \chi > \chi + \eta - 1 > a$.
Finally, equations (\ref{eqf'}, \ref{eqft}) give,
\begin{equation}
\label{eqft'2}
f_t' (w)
= \int_{S_1} \frac{\mu[dx]}{w-x}
- \int_{S_2} \frac{(\l-\mu)[dx]}{w-x}
+ \int_{S_3} \frac{\mu[dx]}{w-x},
\end{equation}
for all $w \in \C \setminus S$, where $S = S_1 \cup S_2 \cup S_3$,
and $S_1,S_2,S_3$ are defined in equation (\ref{eqS1S2S3}).

First, note the following:
\begin{align}
\nonumber
&S_1 \neq \emptyset
\text{ and }
\mu[S_1] > 0,
\hspace{.3cm}
S_2 \neq \emptyset
\text{ and }
(\l-\mu)[S_2] > 0,
\hspace{.3cm}
S_3 \neq \emptyset
\text{ and }
\mu[S_3] > 0. \\
\label{eqS1S2S3In}
&\mu[S_1] - (\l-\mu)[S_2] + \mu[S_3]
= \mu[a,b] - \l[\chi+\eta-1,\chi] = \eta \in (0,1). \\
\nonumber
&b = \sup S_1 > \inf S_1 \ge \chi \ge \sup S_2
> \inf S_2 \ge \chi + \eta - 1 \ge \sup S_3 > \inf S_3 = a. \\
\nonumber
&[\sup S_i-\e,\sup S_i] \cup [\inf S_i, \inf S_i + \e] \subset S_i
\text{ for all } i \in \{1,2,3\} \text{ and some } \e>0.
\end{align}
The first part, above, follows from part (a) of corollary 3.2 of \cite{Duse15a}.
The second part follows since $\mu$ is a probability measure on $[a,b]$ and
$\eta \in (0,1)$. The third part follows from the first part and equation
(\ref{eqS1S2S3}) since $\mu \le \l$, $\{a,b\} \in \supp(\mu)$, and
$b > \chi > \chi + \eta - 1 > a$. Finally, the fourth part follows since
$\mu \le \l$ and $\{a,b\} \in \supp(\mu)$.

Next note, equation (\ref{eqS1S2S3In}) implies that $\C \setminus S$, the domain of $f_t'$,
can be partitioned as follows:
\begin{equation}
\label{eqf'domain2}
\C \setminus S = (\C \setminus \R) \cup J \cup K,
\end{equation}
where $J := \cup_{i=1}^4 J_i$, $K := \cup_{i=1}^3 K^{(i)}$, and
\begin{itemize}
\item
$J_1 := (\sup S_1,+\infty)$.
\item
$J_2 := (-\infty, \inf S_3)$.
\item
$J_3 := (\sup S_2,\inf S_1)$.
\item
$J_4 := (\sup S_3,\inf S_2)$.
\item
$K^{(i)} := [\inf S_i, \sup S_i] \setminus S_i$ for all $i \in \{1,2,3\}$.
\end{itemize}
This partition is depicted in figure \ref{figf'domain}.
Note that $J_1$ and $J_2$ are non-empty, but that
$J_3 = \emptyset$ when $\sup S_2 = \inf S_1$, $J_4 = \emptyset$
when $\sup S_3 = \inf S_2$, and $K^{(i)} = \emptyset$ when
$S_i = [\inf S_i, \sup S_i]$ for each $K^{(i)}$. Note also that $\chi$
is contained in the domain of $f_t'$ if and only
if $J_3$ is non-empty, in which case $\chi \in J_3$. Similarly
$\chi+\eta-1$ is contained in the domain of $f_t'$ if and only if
$J_4$ is non-empty, in which case $\chi+\eta-1 \in J_4$. Finally note that each
$K^{(i)}$ is open, and so can be partitioned as a set of pairwise disjoint open
intervals, which is unique up to order, and which is either empty or finite
or countable. We denote this partition of open intervals as
$\{K_1^{(i)}, K_2^{(i)}, \ldots\}$,
and we note that any $I \in \{K_1^{(i)}, K_2^{(i)}, \ldots\}$
must satisfy $\{\inf I, \sup I\} \subset S_i$.

Next note, since $t$ is real-valued root of $f_t'$, equation
(\ref{eqf'domain2}) implies that $t \in \R \setminus S = J \cup K$. We let
$L_t \subset \R \setminus S = J \cup K$ denote that open interval for which,
\begin{equation}
\label{eqIntervalt}
t \in L_t
\hspace{.5cm} \text{and} \hspace{.5cm}
L_t \in \{J_1,J_2,J_3,J_4\} \cup \cup_{i=1}^3 \{K_1^{(i)},K_2^{(i)},\ldots\}.
\end{equation}
In other words, $L_t$ is the largest open sub-interval of $\R \setminus S = J \cup K$
which contains $t$. These observations, and theorem 3.1 of \cite{Duse15a},
then immediately give the following, stated without proof:
\begin{lem}
\label{lemf'}
Counting multiplicities:
\begin{enumerate}
\item
$f_t'$ has a root of multiplicity $2$ at $t \in L_t$,
has $0$ roots in $L_t \setminus \{t\}$ when
$L_t \in \{J_1,J_2,J_3,J_4\}$, and has at most $1$ root in
$L_t \setminus \{t\}$ when $L_t \in \cup_{i=1}^3 \{K_1^{(i)},K_2^{(i)},\ldots\}$.
\item
$f_t'$ has $0$ roots in $\C \setminus \R$, and in each of
$\{J_1,J_2,J_3,J_4\} \setminus \{L_t\}$.
\item
$f_t'$ has at most $1$ root in each of
$\cup_{i=1}^3 \{K_1^{(i)},K_2^{(i)},\ldots\} \setminus \{L_t\}$.
\end{enumerate}
\end{lem}

\begin{figure}[t]
\centering
\begin{tikzpicture}

\draw [dashed] (-1,0) --++ (1,0);
\draw (0,0) --++ (2,0);
\draw [dashed] (2,0) --++ (3,0);
\draw (5,0) --++ (2,0);
\draw [dashed] (7,0) --++ (2,0);
\draw (9,0) --++ (2,0);
\draw [dashed] (11,0) --++ (1,0);
\draw (-1.5,0) node {$\R$};
\draw (5.5,1) node {$\mathbb{H}$};

\draw (0,.1) --++ (0,-.2);
\draw (0,-.4) node {\scriptsize $a = \underline{S_3}$};
\draw (1.2,-.4) node {\scriptsize $<$};
\draw (2,.1) --++ (0,-.2);
\draw (2,-.4) node {\scriptsize $\overline{S_3}$};
\draw (2.65,-.4) node {\scriptsize $\le$};
\draw (3.5,.1) --++ (0,-.2);
\draw (3.5,-.4) node {\scriptsize $\chi+\eta-1$};
\draw (4.4,-.4) node {\scriptsize $\le$};
\draw (5,.1) --++ (0,-.2);
\draw (5,-.4) node {\scriptsize $\underline{S_2}$};
\draw (6,-.4) node {\scriptsize $<$};
\draw (7,.1) --++ (0,-.2);
\draw (7,-.4) node {\scriptsize $\overline{S_2}$};
\draw (7.65,-.4) node {\scriptsize $\le$};
\draw (8,.1) --++ (0,-.2);
\draw (8,-.4) node {\scriptsize $\chi$};
\draw (8.4,-.4) node {\scriptsize $\le$};
\draw (9,.1) --++ (0,-.2);
\draw (9,-.4) node {\scriptsize $\underline{S_1}$};
\draw (9.8,-.4) node {\scriptsize $<$};
\draw (11,.1) --++ (0,-.2);
\draw (11,-.4) node {\scriptsize $\overline{S_1} = b$};

\draw [thick,decorate,decoration={brace,amplitude=10pt,mirror},xshift=0.2pt,yshift=-0.2pt]
(-1.5,-.7) -- (0,-.7) node[black,midway,yshift=-0.6cm] {\scriptsize $J_2$};
\draw [thick,decorate,decoration={brace,amplitude=10pt,mirror},xshift=0.4pt,yshift=-0.4pt]
(2.1,-.7) -- (4.9,-.7) node[black,midway,yshift=-0.6cm] {\scriptsize $J_4$};
\draw [thick,decorate,decoration={brace,amplitude=10pt,mirror},xshift=0.4pt,yshift=-0.4pt]
(7.1,-.7) -- (8.9,-.7) node[black,midway,yshift=-0.6cm] {\scriptsize $J_3$};
\draw [thick,decorate,decoration={brace,amplitude=10pt,mirror},xshift=0.2pt,yshift=-0.2pt]
(11,-.7) -- (12.5,-.7) node[black,midway,yshift=-0.6cm] {\scriptsize $J_1$};

\end{tikzpicture}
\caption{The sets of equation (\ref{eqf'domain2}), with
$K^{(i)} = [\inf S_i, \sup S_i] \setminus S_i$ for all $i \in \{1,2,3\}$.
Above, $\underline{S_1} := \inf S_1$, $\overline{S_1} := \sup S_1$,
etc (see section \ref{secNot}).}
\label{figf'domain}
\end{figure}

Next, we give the following useful result which describes the various situations
of theorem \ref{thmAiry} in explicit detail. A separate steepest descent analysis must be
performed for each case:
\begin{lem}
\label{lemCases}
The following 12 cases exhaust all possibilities:
\begin{enumerate}
\item
$t \in R_\mu^+$, $t > \chi$, $L_t = J_1$, $f_t'''(t) > 0$.
\item
$t \in R_\mu^+$, $t > \chi$, $L_t \in \{K_1^{(1)},K_2^{(1)},\ldots\}$, $f_t'''(t) > 0$.
\item
$t \in R_\mu^+$, $t > \chi$, $L_t \in \{K_1^{(1)},K_2^{(1)},\ldots\}$, $f_t'''(t) < 0$.
\item
$t \in R_\mu^+$, $t > \chi$, $\chi \in L_t$, $L_t = J_3$, $f_t'''(t) < 0$.
\item
$t \in R_{\l-\mu}$, $t \in (\chi+\eta-1,\chi)$, $\chi \in L_t$, $L_t = J_3$, $f_t'''(t) < 0$.
\item
$t \in R_{\l-\mu}$, $t \in (\chi+\eta-1,\chi)$, $L_t \in \{K_1^{(2)},K_2^{(2)},\ldots\}$, $f_t'''(t) < 0$.
\item
$t \in R_{\l-\mu}$, $t \in (\chi+\eta-1,\chi)$, $L_t \in \{K_1^{(2)},K_2^{(2)},\ldots\}$, $f_t'''(t) > 0$.
\item
$t \in R_{\l-\mu}$, $t \in (\chi+\eta-1,\chi)$, $\chi+\eta-1 \in L_t$, $L_t = J_4$, $f_t'''(t) > 0$.
\item
$t \in R_\mu^-$, $t < \chi+\eta-1$, $\chi+\eta-1 \in L_t$, $L_t = J_4$, $f_t'''(t) > 0$.
\item
$t \in R_\mu^-$, $t < \chi+\eta-1$, $L_t \in \{K_1^{(3)},K_2^{(3)},\ldots\}$, $f_t'''(t) > 0$.
\item
$t \in R_\mu^-$, $t < \chi+\eta-1$, $L_t \in \{K_1^{(3)},K_2^{(3)},\ldots\}$, $f_t'''(t) < 0$.
\item
$t \in R_\mu^-$, $t < \chi+\eta-1$, $L_t = J_2$, $f_t'''(t) < 0$.
\end{enumerate}
\end{lem}

\begin{proof}
First recall that $t \in \R \setminus \{\chi,\chi+\eta-1\}$.
Then, equations (\ref{eqS1S2S3In}, \ref{eqf'domain2}, \ref{eqIntervalt})
imply that the following exhaust all possibilities (see, also,
figure \ref{figf'domain}):
\begin{itemize}
\item
$t > \chi$, and $L_t = J_1$ or $L_t \in \{K_1^{(1)},K_2^{(1)},\ldots\}$
or $L_t = J_3$.
\item
$t \in (\chi+\eta-1,\chi)$, and $L_t = J_3$
or $L_t \in \{K_1^{(2)},K_2^{(2)},\ldots\}$ or $L_t = J_4$.
\item
$t < \chi+\eta-1$, and $L_t = J_4$
or $L_t \in \{K_1^{(3)},K_2^{(3)},\ldots\}$ or $L_t = J_2$.
\end{itemize}
We will show:
\begin{enumerate}
\item[(i)]
Case (1) is satisfied when $t > \chi$ and $L_t = J_1$.
\item[(ii)]
Either case (2) or (3) is satisfied when
$t > \chi$ and $L_t \in \{K_1^{(1)},K_2^{(1)},\ldots\}$.
\item[(iii)]
Case (4) is satisfied when $t > \chi$ and $L_t = J_3$.
\end{enumerate}
Similarly it can be shown that:
\begin{enumerate}
\item[(iv)]
Case (5) is satisfied when $t \in (\chi+\eta-1,\chi)$ and $L_t = J_3$.
\item[(v)]
Either case (6) or (7) is satisfied when
$t \in (\chi+\eta-1,\chi)$ and $L_t \in \{K_1^{(2)},K_2^{(2)},\ldots\}$.
\item[(vi)]
Case (8) is satisfied when $t \in (\chi+\eta-1,\chi)$ and $L_t = J_4$.
\item[(vii)]
Case (9) is satisfied when $t < \chi+\eta-1$ and $L_t = J_4$.
\item[(viii)]
Either case (10) or (11) is satisfied when $t < \chi+\eta-1$ and
$L_t \in \{K_1^{(3)},K_2^{(3)},\ldots\}$.
\item[(ix)]
Case (12) is satisfied when $t < \chi+\eta-1$ and $L_t = J_2$.
\end{enumerate}
The required result follows from (i-ix).

Consider (i). Recall that $t > \chi$ and $L_t = J_1$. Also,
since $t>\chi$ and $(\chi,\eta) = (\chi_\EE(t),\eta_\EE(t))$,
definition \ref{defEdge} and lemma \ref{lemEdge} imply that
$(\chi,\eta) \in \EE_\mu^+$ and $t \in R_\mu^+$. It thus
remains to show that $f_t'''(t) > 0$. To see this, first note,
equation (\ref{eqft'2}) implies that $(f_t') |_{J_1}$ is
continuous and real-valued. Moreover, since $L_t = J_1$,
part (1) of lemma \ref{lemf'} implies that
$f_t'(t) = f_t''(t) = 0$, $f_t'''(t) \neq 0$,
and $f_t'(s) \neq 0$ for all $s \in J_1 \setminus \{t\}
= (\sup S_1,t) \cup (t,+\infty)$. Therefore, it is sufficient
to show that there exists an $s_t \in (t,+\infty)$ with
$f_t'(s_t) > 0$. To see this, note equation (\ref{eqft'2}) gives
\begin{equation*}
f_t' (w)
= \bigg( \int_{S_1} \frac{\mu[dx]}w
- \int_{S_2} \frac{(\l-\mu)[dx]}w
+ \int_{S_3} \frac{\mu[dx]}w \bigg)
+ O \bigg( \frac1{|w|^2} \bigg),
\end{equation*}
for all $w \in \C$ with $|w|$ sufficiently large. Therefore
$f_t' (w) = ( \mu[S_1] - (\l-\mu)[S_2] + \mu[S_3] ) w^{-1}
+ O ( |w|^{-2} ) = \eta w^{-1} + O ( |w|^{-2} )$ for
all such $w$, where the last part follows from equation
(\ref{eqS1S2S3In}). Therefore, since $\eta > 0$,
there exists an $s_t \in (t,+\infty)$ with
$f_t'(s_t) > 0$. This proves (i).

Consider (ii). Recall that $t > \chi$ and
$L_t \in \{K_1^{(1)},K_2^{(1)},\ldots\}$. Also, since
$t>\chi$ and $(\chi,\eta) = (\chi_\EE(t),\eta_\EE(t))$,
definition \ref{defEdge} and lemma \ref{lemEdge} imply
that $(\chi,\eta) \in \EE_\mu^+$ and $t \in R_\mu^+$. Finally
recall that $t$ is a root of $f_t'$ of multiplicity $2$, and so
$f_t'''(t) \neq 0$. This proves (ii).

Consider (iii). Recall that $t > \chi$ and $L_t = J_3$. Also,
since $t>\chi$ and $(\chi,\eta) = (\chi_\EE(t),\eta_\EE(t))$,
definition \ref{defEdge} and lemma \ref{lemEdge} imply
that $(\chi,\eta) \in \EE_\mu^+$
and $t \in R_\mu^+$. Moreover, equations (\ref{eqS1S2S3In},
\ref{eqf'domain2}) imply that $\chi \in J_3 (= L_t)$
(see, also, figure \ref{figf'domain}). It thus
remains to show that $f_t'''(t) < 0$. To see this, first note,
equation (\ref{eqft'2}) implies that $(f_t') |_{J_3}$ is
continuous and real-valued. Moreover, since $L_t = J_3$,
part (1) of lemma \ref{lemf'} implies that
$f_t'(t) = f_t''(t) = 0$, $f_t'''(t) \neq 0$,
and $f_t'(s) \neq 0$ for all $s \in J_3 \setminus \{t\}
= (\sup S_2,t) \cup (t, \inf S_1)$. Therefore, it is sufficient
to show that there exists an $s_t \in (t, \inf S_t)$ with
$f_t'(s_t) < 0$.

To see the above, we use the notation and definitions and
results of next section. This is not a circular argument since
the current lemma is not used in the next section. First, fix $\xi>0$
as in lemma \ref{lemRootsNonAsy1}, and fix $s_t \in [t+\xi, \inf S_1)$.
We have shown above that $f_t'(s_t) \neq 0$, and so either
$f_t'(s_t) < 0$ or $f_t'(s_t) > 0$.
Next note, equation (\ref{eqfn'}) implies that $(f_n') |_{J_{3,n}}$ is
continuous and real-valued. Moreover, since $L_n = J_{3,n}$
(see equation (\ref{eqIntervaln})),
part (1) of lemma \ref{lemRootsNonAsy1} implies that
$f_n'(s) \neq 0$ for all $s \in J_{3,n} \setminus (t-\xi,t+\xi)
= (\max S_{2,n},t-\xi] \cup [t+\xi, \min S_{1,n})$. Also,
equation (\ref{eqfn'}) gives,
\begin{equation*}
\lim_{s \in \R, s \uparrow \min S_{1,n}} f_n'(s) = -\infty.
\end{equation*}
The above observations imply that $f_n'(s) < 0$ for all
$s \in [t+\xi, \min S_{1,n})$. In particular, since
$\min S_{1,n} = \inf S_1 + o(1)$ (see equation (\ref{eqIntervaln})
and note that $L_t = J_3$, $L_n = J_{3,n}$,
$\inf S_1 = \sup J_3$, $\min S_{1,n} = \sup J_{3,n}$),
and since $s_t \in [t+\xi, \inf S_1)$, this
gives $f_n'(s_t) < 0$. Finally note, equations
(\ref{eqft'2}, \ref{eqfn'}, \ref{eqS1nS1WeakConv}) give
$f_n'(s_t) \to f_t'(s_t)$. Therefore $f_t'(s_t) < 0$,
as required. This proves (iii).
\end{proof}

Next recall that $e^{C(\cdot)} : \C \setminus \R \to \C$
and $C'(\cdot) : \C \setminus R \to \C$ have those unique analytic
extensions to $R_\mu^+ \cup R_{\l-\mu} \cup R_\mu^-$ given in equations (\ref{eqC'Rmu},
\ref{eqC'Rl-mu}). We now consider some useful properties of these
extensions for cases, (1-12), of lemma \ref{lemCases}:
\begin{lem}
\label{lemAnalExt}
The following hold:
\begin{itemize}
\item
$e^{C(t)} > 1$ and $C'(t) < 0$ when $t \in R_\mu^+$, i.e., for cases (1-4).
\item
$e^{C(t)} < 0$ and $C'(t) > 0$ when $t \in R_{\l-\mu}$, i.e., for cases (5-8).
\item
$e^{C(t)} \in (0,1)$ and $C'(t) < 0$ when $t \in R_\mu^-$, i.e., for cases (9-12).
\end{itemize}
\end{lem}

\begin{proof}
First recall that $C(\cdot) : \C \setminus \R \to \C$ and
$C'(\cdot) : \C \setminus \R \to \C$ have those unique analytic
extensions to $\R \setminus \supp(\mu)$ given in equation
(\ref{eqC'Rmu}). Moreover, lemma \ref{lemEdge} gives
$R_\mu^+ = \{ s \in \R \setminus \supp(\mu) : C(s) > 0 \}$
and $R_\mu^- = \{ s \in \R \setminus \supp(\mu) : C(s) < 0 \}$.
These easily give $e^{C(t)} > 1$ when $t \in R_\mu^+$,
$e^{C(t)} \in (0,1)$ when $t \in R_\mu^-$, and $C'(t) < 0$
when $t \in R_\mu^+ \cup R_\mu^-$.
Next recall that $e^{C(\cdot)} : \C \setminus \R \to \C$ and
$C'(\cdot) : \C \setminus R \to \C$ have those unique analytic
extensions to $R_{\l-\mu}$ given in equation (\ref{eqC'Rl-mu}).
These easily give $e^{C(t)} < 0$ when $t \in R_{\l-\mu}$. Finally,
we showed in lemma 2.2 of \cite{Duse15a} that $C'(t) > 0$
when $t \in R_{\l-\mu}$.
\end{proof}

We end this section with the final technical assumption of theorem
\ref{thmAiry}. First note the following (see lemma \ref{lemCases}):
\begin{itemize}
\item
For case (4), $\chi < \inf S_1$ and
$t \in (\chi,\inf S_1) \subset \R \setminus \supp(\mu)$.
\item
For case (5), $\sup S_2 < \chi$ and
$t \in (\sup S_2,\chi) \subset \R \setminus \supp(\l-\mu)$.
\item
For case (8), $\chi+\eta-1 < \inf S_2$ and
$t \in (\chi+\eta-1,\inf S_2) \subset \R \setminus \supp(\l-\mu)$.
\item
For case (9), $\chi+\eta-1 > \sup S_3$ and
$t \in (\sup S_3,\chi+\eta-1) \subset \R \setminus \supp(\mu)$.
\end{itemize}
We assume:
\begin{ass}
\label{asscases}
If one of the cases (4,5,8,9) of lemma \ref{lemCases} is satisfied,
assume the following:
\begin{itemize}
\item
$\chi \not\in \supp(\mu)$ for case (4), i.e.,
$[\chi,\inf S_1) \subset \R \setminus \supp(\mu)$.
\item
$\chi \not\in \supp(\l-\mu)$ for case (5), i.e.,
$(\sup S_2,\chi] \subset \R \setminus \supp(\l-\mu)$.
\item
$\chi+\eta-1 \not\in \supp(\l-\mu)$ for case (8), i.e.,
$[\chi+\eta-1,\inf S_2) \subset \R \setminus \supp(\l-\mu)$.
\item
$\chi+\eta-1 \not\in \supp(\mu)$ for case (9), i.e.,
$(\sup S_3,\chi+\eta-1] \subset \R \setminus \supp(\mu)$.
\end{itemize}
\end{ass}
Note, with the above assumption, equation (\ref{eqS1S2S3}) implies
that $\chi = \sup S_2$ for case (4), $\chi = \inf S_1$ for case (5),
$\chi+\eta-1 = \sup S_3$ for case (8), and $\chi+\eta-1 = \inf S_2$
for case (9). Another effect of assumption \ref{asscases} can be seen,
for example, in part (ii) in the proof of lemma \ref{lemJn}. Part (ii)
is not necessarily true if assumption \ref{asscases} does not hold.

\subsection{The roots of $f_n'$ and $\tilde{f}_n'$}

In this section, we assume assumptions \ref{assWeakConv} and \ref{assIsol},
that equation (\ref{equnrnvnsn}) is satisfied for some fixed
$t \in R_\mu^+ \cup R_{\l-\mu} \cup R_\mu^-$ (a root of
$f_t'$ of multiplicity $2$), and assumption \ref{asscases}. We consider
the behaviour of the roots of $f_n'$ and $\tilde{f}_n'$ as $n \to \infty$.
We concentrate mainly on $f_n'$, since $\tilde{f}_n'$ has a similar behaviour.

First recall that $f_n$ and $\tilde{f}_n$ are defined in equations
(\ref{eqfn}, \ref{eqtildefn}). These equations give,
\begin{align}
\label{eqfn2}
f_n(w)
&= \frac1n \sum_{x \in S_{1,n}} \log(w-x)
- \frac1n \sum_{x \in S_{2,n}} \log(w-x)
+ \frac1n \sum_{x \in S_{3,n}} \log(w-x), \\
\label{eqtildefn2}
\tilde{f}_n(w)
&= \frac1n \sum_{x \in \tilde{S}_{1,n}} \log(w-x)
-  \frac1n \sum_{x \in \tilde{S}_{2,n}} \log(w-x)
+  \frac1n \sum_{x \in \tilde{S}_{3,n}} \log(w-x),
\end{align}
where the branch cuts are either $(-\infty,0]$ or $[0,+\infty)$, and where:
\begin{itemize}
\item
$S_{1,n} := \frac1n \{x_i: x_i > v_n\}$.
\item
$S_{2,n} := \frac1n \{v_n,v_n-1,\ldots,v_n+s_n-n\} \setminus \{x_1,x_2,\ldots,x_n\}$.
\item
$S_{3,n} := \frac1n \{x_i: x_i < v_n + s_n - n\}$.
\item
$\tilde{S}_{1,n} := \frac1n \{x_i: x_i \ge u_n\}$.
\item
$\tilde{S}_{2,n} := \frac1n \{u_n-1,u_n-2,\ldots,u_n+r_n-n+1\} \setminus \{x_1,x_2,\ldots,x_n\}$.
\item
$\tilde{S}_{3,n} := \frac1n \{x_i: x_i \le u_n + r_n - n\}$.
\end{itemize}

Consider $f_n$. First note, irrespective of the choices of the branches of
the logarithm in equation (\ref{eqfn2}), that
\begin{equation}
\label{eqfn'}
f_n'(w)
= \frac1n \sum_{x \in S_{1,n}} \frac1{w-x}
-  \frac1n \sum_{x \in S_{2,n}} \frac1{w-x}
+  \frac1n \sum_{x \in S_{3,n}} \frac1{w-x}.
\end{equation}
Also note, assumption \ref{assWeakConv} and equations (\ref{eqS1S2S3}, \ref{equnrnvnsn})
imply the following:
\begin{equation}
\label{eqS1nS1WeakConv}
\frac1n \sum_{x \in S_{1,n}} \delta_x \to \mu |_{S_1}
\hspace{.25cm} \text{and} \hspace{.25cm}
\frac1n \sum_{x \in S_{2,n}} \delta_x \to (\l-\mu) |_{S_2}
\hspace{.25cm} \text{and} \hspace{.25cm}
\frac1n \sum_{x \in S_{3,n}} \delta_x \to \mu |_{S_3}
\hspace{.25cm} \text{weakly}.
\end{equation}
Moreover, assumption \ref{assIsol} and equations (\ref{eqS1S2S3In}, \ref{eqS1nS1WeakConv})
imply the following:
\begin{align}
\nonumber
&S_{1,n} \neq \emptyset
\hspace{.5cm} \text{and} \hspace{.5cm}
\tfrac1n |S_{1,n}| = \mu[S_1] + o(1) = \mu[\chi,b] + o(1) > 0. \\
\nonumber
&S_{2,n} \neq \emptyset
\hspace{.5cm} \text{and} \hspace{.5cm}
\tfrac1n |S_{2,n}| = (\l-\mu)[S_2] + o(1) = (\l-\mu)[\chi+\eta-1,\chi] + o(1) > 0. \\
\label{eqS1nS2nS3nIn}
&S_{3,n} \neq \emptyset
\hspace{.5cm} \text{and} \hspace{.5cm}
\tfrac1n |S_{3,n}| = \mu[S_3] + o(1) = \mu[a,\chi+\eta-1] + o(1) > 0. \\
\nonumber
&\tfrac1n x_1^{(n)} = \max S_{1,n} > \min S_{1,n} >  \max S_{2,n}
> \min S_{2,n} > \max S_{3,n} > \min S_{3,n} = \tfrac1n x_n^{(n)}. \\
\nonumber
&\max S_{1,n} = \sup S_1 + o(1)
\hspace{.5cm} \text{and} \hspace{.5cm}
\min S_{3,n} = \inf S_3 + o(1).
\end{align}

Next note that $f_n'$ extends analytically to $\C \setminus S_n$,
where $S_n := S_{1,n} \cup S_{2,n} \cup S_{3,n}$.
Then, in analogy with equation (\ref{eqf'domain2}), partition the domain of
$f_n'$ as follows:
\begin{equation}
\label{eqfn'domain}
\C \setminus S_n = (\C \setminus \R) \cup J_n \cup K_n,
\end{equation}
where $J_n := \cup_{i=1}^4 J_{i,n}$, $K_n := \cup_{i=1}^3 K_n^{(i)}$, and
\begin{itemize}
\item
$J_{1,n} := (\max S_{1,n}, + \infty)$.
\item
$J_{2,n} := (-\infty, \min S_{3,n})$.
\item
$J_{3,n} := (\max S_{2,n},\min S_{1,n})$.
\item
$J_{4,n} := (\max S_{3,n},\min S_{2,n})$.
\item
$K_n^{(i)} := [\min S_{i,n}, \max S_{i,n}] \setminus S_{i,n}$ for all $i \in \{1,2,3\}$.
\end{itemize}
Note that each $K_n^{(i)}$ is open, and so can be partitioned as a set of
pairwise disjoint open intervals, which is unique up to order. We
denote this partition of open
intervals as $\{K_{1,n}^{(i)}, K_{2,n}^{(i)}, \ldots\}$,
and we note that $I_n \in \{K_{1,n}^{(i)}, K_{2,n}^{(i)}, \ldots\}$
if and only if $\inf I_n$ and $\sup I_n$ are two consecutive elements of $S_{i,n}$.

Next, recall (see equations (\ref{equnrnvnsn}, \ref{eqft'2}))
that $t \in \R \setminus S$ is a root of $f_t'$, and (see equation (\ref{eqIntervalt}))
that $L_t \subset \R \setminus S = J \cup K$ is that interval for which $t \in L_t$
and $L_t \in \{J_1,J_2,J_3,J_4\} \cup \cup_{i=1}^3 \{K_1^{(i)},K_2^{(i)},\ldots\}$.
Therefore, for all $\xi > 0$ sufficiently small,
\begin{equation}
\label{eqxi}
(t-4\xi,t+4\xi) \subset L_t,
\; B(t,4\xi) \subset \C \setminus S,
\text{ and }
t \text{ is the unique root of } f_t' \text{ in } B(t,4\xi).
\end{equation}
Fix such an $\xi>0$. Then:
\begin{lem}
\label{lemHassConv}
There exists $L_n \in \{J_{1,n},J_{2,n},J_{3,n},J_{4,n}\}
\cup \cup_{i=1}^3 \{K_{1,n}^{(i)},K_{2,n}^{(i)},\ldots\}$,
an open interval which satisfies,
\begin{align}
\nonumber
&\sup L_n = \sup L_t + o(1)
\hspace{.5cm} \text{and} \hspace{.5cm}
\inf L_n = \inf L_t + o(1), \\
\label{eqIntervaln}
&(t-2\xi,t+2\xi) \subset L_n
\hspace{.5cm} \text{and} \hspace{.5cm}
B(t,2\xi) \subset \C \setminus S_n, \\
\nonumber
&L_n = J_{i,n} \Leftrightarrow L_t = J_i \; \forall \; i, 
\hspace{.5cm} L_n \in \{K_{1,n}^{(i)},K_{2,n}^{(i)},\ldots\}
\Leftrightarrow L_t \in \{K_1^{(i)},K_2^{(i)},\ldots\} \; \forall \; i.
\end{align}
\end{lem}

\begin{proof}
Throughout this proof, it is helpful to refer to figure \ref{figf'domain} to visualise
the sets in question. Recall the exhaustive cases, (1-12), of lemma \ref{lemCases}. We
will prove the result for cases (1,2,4,7). The proof for case (12) is similar to case (1),
the proof for cases (3,10,11) are similar to case (2), the proof for case (6) is similar
to case (7), and the proof for cases (5,8,9) are similar to case (4).

For case (1), recall that $L_t = J_1 = (\sup S_1, + \infty)$,
$J_{1,n} = (\max S_{1,n}, + \infty)$, and
$\max S_{1,n} \to \sup S_1$ (see lemma \ref{lemCases} and equations
(\ref{eqS1nS2nS3nIn}, \ref{eqfn'domain})). The result for case (1)
then follows from equation (\ref{eqxi}).

For case (2), recall that $L_t \in \{K_1^{(1)}, K_2^{(1)}, \ldots \}$,
$L_t \in \R \setminus \supp(\mu)$, $\sup S_1 > \sup L_t$ and $\inf L_t > \chi$,
$\{\inf L_t, \sup L_t\} \subset S_1 = \supp(\mu |_{[\chi,b]})$,
and $S_1$ is entirely contained in $[\sup L_t, \sup S_1] \cup [\chi, \inf L_t]$
(see lemma \ref{lemCases}, and equations (\ref{eqS1S2S3}, \ref{eqS1S2S3In},
\ref{eqf'domain2})). Moreover, recall that $S_{1,n} = \frac1n \{x_i: x_i > v_n\}$,
$\max S_{1,n} \to \sup S_1$ and $\frac{v_n}n \to \chi$, and
$I_n \in \{K_{1,n}^{(1)}, K_{2,n}^{(1)}, \ldots\}$
if and only if $\inf I_n$ and $\sup I_n$ are two consecutive
elements of $S_{1,n}$ (see  equations (\ref{equnrnvnsn}, \ref{eqfn'},
\ref{eqfn'domain})). The result for case (2) then follows from
assumption \ref{assIsol} and equation (\ref{eqxi}).

For case (7), recall that $L_t \in \{K_1^{(2)}, K_2^{(2)}, \ldots \}$,
$L_t \in \R \setminus \supp(\l-\mu)$, $\chi > \sup L_t$ and $\inf L_t > \chi+\eta-1$,
$\{\inf L_t, \sup L_t\} \subset S_2 = \supp((\l-\mu) |_{[\chi+\eta-1, \chi]})$,
and $S_2$ is entirely contained in $[\sup L_t, \chi] \cup [\chi+\eta-1, \inf L_t]$
(see lemma \ref{lemCases}, and equations (\ref{eqS1S2S3}, \ref{eqS1S2S3In},
\ref{eqf'domain2})). Moreover, recall that
$S_{2,n} = \frac1n \{v_n,v_n-1,\ldots,v_n+s_n-n\} \setminus \{x_1,x_2,\ldots,x_n\}$,
$\frac{v_n}n \to \chi$ and $\frac{v_n+s_n-n}n \to \chi+\eta-1$, and
$I_n \in \{K_{1,n}^{(2)}, K_{2,n}^{(2)}, \ldots\}$
if and only if $\inf I_n$ and $\sup I_n$ are two consecutive
elements of $S_{2,n}$ (see  equations (\ref{equnrnvnsn}, \ref{eqfn'},
\ref{eqfn'domain})). The result for case (2) then follows from
assumption \ref{assIsol} and equation (\ref{eqxi}).

For case (4), recall that $L_t = J_3 = (\sup S_2,\inf S_1)$,
$\sup L_t \subset S_1 = \supp(\mu |_{[\chi,b]})$,
and $\inf L_t \subset S_2 = \supp((\l-\mu) |_{[\chi+\eta-1, \chi]})$
(see lemma \ref{lemCases}, and equations (\ref{eqS1S2S3}, \ref{eqS1S2S3In},
\ref{eqf'domain2})). Moreover, assumption \ref{asscases} implies
that $\sup S_2 = \chi$, i.e., that $\inf L_t = \chi$, and
$[\chi, \inf S_1) \subset \R \setminus \supp(\mu)$. Therefore
$(\chi-\delta, \inf S_1-\delta) \subset \R \setminus \supp(\mu)$
for all $\delta>0$ sufficiently small. Then,
for all such $\delta>0$, assumption \ref{assIsol} implies that
there exists an $n(\delta) \ge 1$ for which
$(\chi-\delta, \inf S_1-\delta) \cap (\frac1n \{x_1,x_2,\ldots,x_n\}) = \emptyset$
for all $n > n(\delta)$ .
Finally recall that $J_{3,n} = (\max S_{2,n},\min S_{1,n})$,
$S_{1,n} = \frac1n \{x_i: x_i > v_n\}$ and
$S_{2,n} = \frac1n \{v_n,v_n-1,\ldots,v_n+s_n-n\} \setminus \{x_1,x_2,\ldots,x_n\}$, 
and $\frac{v_n}n \to \chi$ (see  equations (\ref{equnrnvnsn}, \ref{eqfn'},
\ref{eqfn'domain})). The above observations imply that
$\sup J_{3,n} = \min S_{1,n} \ge \inf S_1 - \delta$ and
$\inf J_{3,n} = \max S_{2,n} = \frac{v_n}n$ for all $n>n(\delta)$.
Then, letting $\delta \downarrow 0$, the result
for case (4) follows from equation (\ref{eqxi}).
\end{proof}

Moreover:
\begin{lem}
\label{lemRootsNonAsy1}
The following hold:
\begin{enumerate}
\item
$f_n'$ has $2$ roots in $B(t,\xi)$, has $0$ roots in
$L_n \setminus (t-\xi,t+\xi)$ when $L_n \in \{J_{1,n},J_{2,n},J_{3,n},J_{4,n}\}$,
and has $1$ root in $L_n \setminus (t-\xi,t+\xi)$ when
$L_n \in \cup_{i=1}^3 \{K_{1,n}^{(i)},K_{2,n}^{(i)},\ldots\}$.
\item
$f_n'$ has $0$ roots in $\C \setminus (\R \cup B(t,\xi))$, and
in each of $\{J_{1,n},J_{2,n},J_{3,n},J_{4,n}\} \setminus \{L_n\}$.
\item
$f_n'$ has $1$ root in each of
$\cup_{i=1}^3 \{K_{1,n}^{(i)},K_{2,n}^{(i)},\ldots\} \setminus \{L_n\}$.
\end{enumerate}
Next, denote the roots of $f_n'$ in $B(t,\xi)$ by $\{t_{1,n},t_{2,n}\}$,
with the understanding that $t_{1,n} = t_{2,n}$ means that $t_{1,n}$ is
a root of multiplicity $2$. Then we can always choose the labelling such that
one of the following possibilities is satisfied:
\begin{enumerate}
\item[(a)]
$t_{1,n} \in (t-\xi,t+\xi)$ and $t_{1,n} = t_{2,n}$.
\item[(b)]
$\{t_{1,n},t_{2,n}\} \subset (t-\xi,t+\xi)$ and $t_{1,n} > t_{2,n}$.
\item[(c)]
$t_{1,n} \in B(t,\xi) \cap \mathbb{H}$ and
$t_{2,n}$ is the complex conjugate of $t_{1,n}$.
\end{enumerate}
\end{lem}

\begin{proof}
Consider (1-3). We will show that:
\begin{enumerate}
\item[(i)]
$f_n'$ has $\sum_{i=1}^3 |\{K_{1,n}^{(i)},K_{2,n}^{(i)},\ldots\}| + 2$ roots in
$\C \setminus S_n = (\C \setminus \R) \cup J_n \cup K_n$.
\item[(ii)]
$f_n'$ an odd number of roots in each of $\cup_{i=1}^3 \{K_{1,n}^{(i)},K_{2,n}^{(i)},\ldots\}$.
\item[(iii)]
$f_n'$ has $2$ roots in $B(t,\xi)$. Moreover, either both
are in $(t-\xi,t+\xi)$, or both are in $B(t,\xi) \setminus (t-\xi,t+\xi)$.
\end{enumerate}
Next recall (see equation (\ref{eqIntervaln})) that either
$L_n \in \{J_{1,n},J_{2,n},J_{3,n},J_{4,n}\}$ for all $n$ sufficiently
large, or $L_n \in \cup_{i=1}^3 \{K_{1,n}^{(i)},K_{2,n}^{(i)},\ldots\}$
for all $n$ sufficiently large. We consider both cases separately. First
suppose that $L_n \in \{J_{1,n},J_{2,n},J_{3,n},J_{4,n}\}$. Recall (see
equation (\ref{eqIntervaln})) that
$B(t,\xi) \subset (\C \setminus \R) \cup L_n$. Then parts (i,ii,iii)
and a simple counting argument imply that
the following is satisfied: $f_n'$ has $2$ roots in $(\C \setminus \R) \cup L_n$,
$0$ roots in $\{J_{1,n},J_{2,n},J_{3,n},J_{4,n}\} \setminus \{L_n\}$, and
$1$ root in each of $\cup_{i=1}^3 \{K_{1,n}^{(i)},K_{2,n}^{(i)},\ldots\}$.
Moreover, since $B(t,\xi) \subset (\C \setminus \R) \cup L_n$, part (iii)
implies that the $2$ roots in $(\C \setminus \R) \cup L_n$ must be in $B(t,\xi)$.
This proves (1-3) when $L_n \in \{J_{1,n},J_{2,n},J_{3,n},J_{4,n}\}$. Next suppose that
$L_n \in \cup_{i=1}^3 \{K_{1,n}^{(i)},K_{2,n}^{(i)},\ldots\}$. Recall that
$B(t,\xi) \subset (\C \setminus \R) \cup L_n$. 
Then parts (i,ii,iii) and a simple counting argument imply that the following
is satisfied:
$f_n'$ has $3$ roots in $(\C \setminus \R) \cup L_n$, $0$ roots in
$\{J_{1,n},J_{2,n},J_{3,n},J_{4,n}\}$, and $1$ root in each of
$\cup_{i=1}^3 \{K_{1,n}^{(i)},K_{2,n}^{(i)},\ldots\} \setminus \{L_n\}$.
Moreover, since $B(t,\xi) \subset (\C \setminus \R) \cup L_n$, parts (ii,iii)
imply that $2$ of the roots in $(\C \setminus \R) \cup L_n$ must be in $B(t,\xi)$,
and $1$ of the roots must be in $L_n \setminus (t-\xi,t+\xi)$. This proves
(1-3) when $L_n \in \cup_{i=1}^3 \{K_{1,n}^{(i)},K_{2,n}^{(i)},\ldots\}$.

Consider (i). First recall, (see
equation (\ref{eqS1nS2nS3nIn})) that $S_n = S_{1,n} \cup S_{2,n} \cup S_{3,n}$,
where $\{S_{1,n},S_{2,n},S_{3,n}\}$ are
mutually disjoint, each set is non-empty and composed of singletons,
and $\min S_{1,n} > \max S_{2,n}$ and $\min S_{2,n} > \max S_{3,n}$.
Note, for all $w \in \C \setminus S_n$, equation
(\ref{eqfn'}) gives $f_n'(w) = \frac1n ( \prod_{x \in S_n} \frac1{w-x} ) Q_n(w)$,
where $Q_n$ is the polynomial,
\begin{equation*}
Q_n(w) = \sum_{x \in S_{1,n} \cup S_{3,n}}
\bigg( \prod_{y \in S_n \setminus \{x\}} (w - y) \bigg)
- \sum_{x \in S_{2,n}}
\bigg( \prod_{y \in S_n \setminus \{x\}} (w - y) \bigg).
\end{equation*}
Note that $Q_n$ is a polynomial of degree $|S_n|-1$. Also note that $Q_n$ has no
roots in $S_n$, and so the roots of $Q_n$ and $f_n'$ coincide. Therefore $f_n'$
has $|S_n|-1$ roots in $\C \setminus S_n$. Finally
recall (see equation (\ref{eqfn'domain}) and the subsequent remarks) that
$\cup_{i=1}^3 \{K_{1,n}^{(i)}, K_{2,n}^{(i)},\ldots\}$ is a set of pairwise
disjoint open intervals, and that $I_n \in \{K_{1,n}^{(i)}, K_{2,n}^{(i)},\ldots\}$
if and only if $\inf I_n$ and $\sup I_n$ are two consecutive elements of $S_{i,n}$.
Therefore $|S_n| = \sum_{i=1}^3 |\{K_{1,n}^{(i)},K_{2,n}^{(i)},\ldots\}| + 3$, which
proves (i).

Consider (ii). Fix $i \in \{1,2,3\}$, and any interval
$I_n \in \{K_{1,n}^{(i)},K_{2,n}^{(i)},\ldots\}$,
and recall that $\inf I_n$ and $\sup I_n$ are consecutive elements
of $S_{i,n}$. Note, when $i=1$, equation (\ref{eqfn'}) gives,
\begin{equation*}
\lim_{w \in \R, w \uparrow \sup I_n} f_n'(w) = -\infty
\hspace{0.5cm} \text{and} \hspace{0.5cm}
\lim_{w \in \R, w \downarrow \inf I_n} f_n'(w) = +\infty.
\end{equation*}
Thus $f_n'$
has an odd number of roots in $I_n$. Similarly whenever $i \in \{2,3\}$.
This proves (ii).

Consider (iii). First note, part (1) of lemma \ref{lemf'} and equation
(\ref{eqxi}) imply that $f_t'$ has $2$ roots in $B(t,\xi)$.
Next note, equation (\ref{eqfn'}) implies that
non-real roots of $f_n'$ occur in complex conjugate pairs.
Part (iii) thus follows from Rouch\'{e}'s theorem if, we can show that
$|f_t'(w)| > |f_t'(w) - f_n'(w)|$ for all $w \in \partial B(t,\xi)$. We shall show:
\begin{equation*}
\inf_{w \in \partial B(t,\xi)} |f_t'(w)| > 0
\hspace{0.5cm} \mbox{and} \hspace{0.5cm}
\lim_{n \to \infty} \sup_{w \in \text{cl}(B(t,\xi))} |f_t'(w) - f_n'(w)| = 0.
\end{equation*}
The first part follows from the extreme value theorem, since $f_t'$ is
analytic in $B(t,2\xi)$ (see equation (\ref{eqxi})). We prove the second
part via contradiction: Assume that the second part does not hold. Then
there exists a $\d > 0$ for which, for all $n\ge1$, there exists
some $p_n \ge n$ and $z_n \in \text{cl}(B(t,\xi))$ with
$\d < | f_t'(z_n) - f_{p_n}'(z_n) |$. Choosing $\{z_n\}_{n\ge1}$ to be
convergent, and denoting the limit by $z$, the triangle inequality gives
\begin{equation}
\label{eqlemRootsNonAsy1}
\d < | f_t'(z_n) - f_t'(z) |
+ | f_t'(z) - f_{p_n}'(z) |
+ | f_{p_n}'(z) - f_{p_n}'(z_n) |.
\end{equation}
Note, $| f_t'(z_n) - f_t'(z) | \to 0$ since $z_n \to z$,
$\{z,z_1,z_2\ldots\} \subset \text{cl}(B(t,\xi))$, and $f_t'$ is
analytic in $B(t,2\xi)$. Also, since $z \in \text{cl}(B(t,\xi))$, and
$B(t,2\xi) \subset \C \setminus S$ and $B(t,2\xi) \subset \C \setminus S_n$
(see equations (\ref{eqxi}, \ref{eqIntervaln})),
equations (\ref{eqft'2}, \ref{eqfn'}, \ref{eqS1nS1WeakConv}) imply that
$| f_t'(z) - f_{p_n}'(z) | \to 0$. Finally, equation (\ref{eqfn'}) implies
that,
\begin{equation*}
| f_{p_n}'(z) - f_{p_n}'(z_n) |
\le \sum_{i=1}^3 \left( \frac1{p_n} |S_{i,p_n}| \right)
\sup_{x \in S_{i,p_n}} \left| \frac1{z-x} - \frac1{z_n-x} \right|.
\end{equation*}
This implies that $| f_{p_n}'(z) - f_{p_n}'(z_n) | \to 0$, since
$z_n \to z$, $\{z,z_1,z_2\ldots\} \subset \text{cl}(B(t,\xi))$,
$B(t,2\xi) \subset \C \setminus S_{p_n}$, and
$\frac1{p_n} |S_{i,p_n}| = O(1)$ for all $i \in \{1,2,3\}$ (see
equation (\ref{eqS1nS2nS3nIn})). The above observations contradict
equation (\ref{eqlemRootsNonAsy1}), and so our assumption is false. This proves (iii).

Finally, we show that one of the possibilities, (a,b,c), must be satisfied. First
recall, part (1) implies that $f_n'$ has $2$ roots in $B(t,\xi)$. Next note, equation
(\ref{eqfn'}) implies that non-real roots of $f_n'$ occur in complex
conjugate pairs. Possibilities (a,b,c) easily follow.
\end{proof}

\subsection{The rates of convergence}

In the previous sections, we assumed assumptions \ref{assWeakConv} and
\ref{assIsol}, and that equation (\ref{equnrnvnsn}) is satisfied for
some fixed $t \in R_\mu^+ \cup R_{\l-\mu} \cup R_\mu^-$, a root of
$f_t'$ of multiplicity $2$, and we considered the behaviour
of the roots of $f_t'$ and $f_n'$ and $\tilde{f}_n'$. We saw that $2$ roots
of $f_n'$ (and $\tilde{f}_n'$) converge to $t$ as $n \to \infty$.
We did not, however, discuss the rate of convergence.

In this section we assume the above, and additionally assume that the
sequences of particle positions, $\{(u_n,r_n)\}_{n\ge1} \subset \Z^2$
and $\{(v_n,s_n)\}_{n\ge1} \subset \Z^2$, are chosen as in equations
(\ref{equnrnvnsn2}, \ref{equnrnvnsn3}). Recall that these are defined
in terms of sequences, $\{m_n\}_{n\ge1} \subset \R$ and
$\{p_n\}_{n\ge1} \subset \R$, which we
have yet to define. In this section, we define these in a natural
way such that the rate of convergence of the roots is sufficiently
fast to allow a steepest descent analysis of the correlation
kernel. The sequences are defined in definition \ref{defmnpn},
and the rate of convergence of the roots is given in part (4)
of lemma \ref{lemNonAsyRoots}. In that lemma, we also examine the
asymptotic behaviour of $f_n'(t)$ and $\tilde{f}_n'(t)$,
and their higher order derivatives. We concentrate mainly on
$f_n'$, since $\tilde{f}_n'$ has a similar behaviour.

We begin by writing convenient expressions for $f_t'$ in neighbourhoods of $t$.
First recall that $\mu \le \l$ (see assumption \ref{assWeakConv}), and
$(\chi,\eta) = (\chi_\EE(t),\eta_\EE(t))$ where
$t \in R_\mu^+ \cup R_{\l-\mu} \cup R_\mu^-$ (see equation (\ref{equnrnvnsn})).
Next recall (see definition \ref{defEdge} and lemma \ref{lemEdge}) that
$t \in (\chi,+\infty) \setminus \supp(\mu)$ when $t \in R_\mu^+$,
$t \in (\chi+\eta-1,\chi) \setminus \supp(\l-\mu)$ when $t \in R_{\l-\mu}$,
and $t \in (-\infty,\chi+\eta-1) \setminus \supp(\mu)$ when $t \in R_\mu^-$.
Thus the following are satisfied for all $\xi>0$ sufficiently small:
\begin{align}
\nonumber
\mu |_{(t-4\xi,t+4\xi)} &= 0
\text{ and } (t-4\xi,t+4\xi) \subset (\chi,+\infty)
\text{ when } t \in R_\mu^+. \\
\label{eqxi1}
\mu |_{(t-4\xi,t+4\xi)} &= \l |_{(t-4\xi,t+4\xi)}
\text{ and } (t-4\xi,t+4\xi) \subset (\chi+\eta-1,\chi)
\text{ when } t \in R_{\l-\mu}. \\
\nonumber
\mu |_{(t-4\xi,t+4\xi)} &= 0
\text{ and } (t-4\xi,t+4\xi) \subset (-\infty,\chi+\eta-1)
\text{ when } t \in R_\mu^-.
\end{align}
Then, fixing such an $\xi>0$, equations (\ref{eqf}, \ref{eqft})
imply the following, which are well-defined and analytic for all $w \in B(t,4\xi)$:
\begin{align}
\label{eqft'Rmu}
f_t'(w)
&= \int_a^b \frac{\mu[dx]}{w-x} - \int_{\chi+\eta-1}^\chi \frac{dx}{w-x}
\hspace{.5cm} \text{when } t \in R_\mu^+ \cup R_\mu^-. \\
\label{eqft'Rl-mu}
f_t'(w)
&= \bigg[ \int_a^{t-4\xi} + \int_{t+4\xi}^b \bigg] \frac{\mu[dx]}{w-x}
- \bigg[ \int_{\chi+\eta-1}^{t-4\xi} + \int_{t+4\xi}^\chi \bigg] \frac{dx}{w-x}
\hspace{.5cm} \text{when } t \in R_{\l-\mu}.
\end{align}

Now, we use the above to inspire the definition of natural
non-asymptotic functions which have a similar root behaviour
in $B(t,4\xi)$. First, fix $\e>0$ sufficiently small such
that equation (\ref{eqtRmuRlmu}) is satisfied.
Next, fix the above $\xi>0$ sufficiently small
such that the following are also satisfied:
\begin{align}
\nonumber
&(t-4\xi,t+4\xi) \subset R_\mu^+(\e)
= \{s \in R_\mu^+ : (s-\e,s+\e) \subset R_\mu^+\}
\text{ when } t \in R_\mu^+. \\
\label{eqxi4}
&(t-4\xi,t+4\xi) \subset R_{\l-\mu}(\e)
= \{s \in R_{\l-\mu} : (s-\e,s+\e) \subset R_{\l-\mu}\}
\text{ when } t \in R_{\l-\mu}. \\
\nonumber
&(t-4\xi,t+4\xi) \subset R_\mu^-(\e)
= \{s \in R_\mu^- : (s-\e,s+\e) \subset R_\mu^-\}
\text{ when } t \in R_\mu^-.
\end{align}
Next, define $\mu_n$ as in equation (\ref{eqWeakmun}),
and recall that $\mu_n \to \mu$ weakly. Note that
$\supp(\mu_n) \subset P_n \cup \text{cl}(R_{\l-\mu}(\e))$,
where $P_n = \frac1n \{x_1,\ldots,x_n\}$. Therefore
$\supp(\mu_n) \subset (a-\e,b+\e)$, since $\supp(\mu) \subset [a,b]$
(see assumption \ref{assWeakConv}), since $d(P_n, \supp(\mu)) \to 0$
(see assumption \ref{assIsol}), and since $R_{\l-\mu}(\e) \subset (a,b)$
(indeed, $R_{\l-\mu}(\e) \subset R_{\l-\mu} = \R \setminus \supp(\l-\mu)
\subset \supp(\mu)^\circ \subset (a,b)$). Next, define
$(\chi_n,\eta_n)$ as in definition \ref{defEdgeNonAsy}, and recall that
$(\chi_n,\eta_n) \to (\chi,\eta)$ (see equation (\ref{eqNonAsyEdge})).
Finally, inspired by equations (\ref{eqft'Rmu}, \ref{eqft'Rl-mu}), define:
\begin{align}
\label{eqftn'Rmu}
f_{t,n}'(w)
&:= \int_{a-\e}^{b+\e} \frac{\mu_n[dx]}{w-x}
- \int_{\chi_n+\eta_n-1}^{\chi_n} \frac{dx}{w-x}
\hspace{.5cm} \text{when } t \in R_\mu^+ \cup R_\mu^-. \\
\label{eqftn'Rl-mu}
f_{t,n}'(w)
&:= \bigg[ \int_{a-\e}^{t-4\xi} + \int_{t+4\xi}^{b+\e} \bigg] \frac{\mu_n[dx]}{w-x}
- \bigg[ \int_{\chi_n+\eta_n-1}^{t-4\xi} + \int_{t+4\xi}^{\chi_n} \bigg] \frac{dx}{w-x}
\hspace{.5cm} \text{when } t \in R_{\l-\mu}.
\end{align}
The function $f_{t,n}$ is unimportant and left unspecified. Note that
these functions are well-defined and analytic in $B(t,4\xi)$. Indeed,
when $t \in R_\mu^+$, the second term on the RHS of equation
(\ref{eqftn'Rmu}) is well-defined and analytic since
$(\chi_n,\eta_n) \to (\chi,\eta)$ and $(t-4\xi,t+4\xi) \subset (\chi,+\infty)$
(see equation (\ref{eqxi1})). Moreover, when $t \in R_\mu^+$, note that
$(t-4\xi,t+4\xi) \subset \R \setminus \supp(\mu_n)$, since
$\supp(\mu_n) \subset P_n \cup \text{cl}(R_{\l-\mu}(\e))$, since
$(t-4\xi,t+4\xi) \subset R_\mu^+(\e) \subset \R \setminus P_n$ (see
equations (\ref{eqPnHnlarge}, \ref{eqxi4})), and since
$(t-4\xi,t+4\xi) \subset R_\mu^+(\e) \subset \R \setminus \text{cl}(R_{\l-\mu}(\e))$
(see equations (\ref{eqtRmuRlmu}, \ref{eqxi4})).
Thus the first term on the RHS of equation (\ref{eqftn'Rmu})
is also well-defined and analytic in $B(t,4\xi)$. Similarly the
terms on the RHS of equation (\ref{eqftn'Rmu})
are well-defined and analytic in $B(t,4\xi)$ when $t \in R_\mu^-$, and
the terms on the RHS of equation (\ref{eqftn'Rl-mu})
are well-defined and analytic in $B(t,4\xi)$ when $t \in R_{\l-\mu}$.
Moreover:
\begin{lem}
\label{lemftn'}
$f_{t,n}'(t) = f_{t,n}''(t) = 0$. Moreover, $f_{t,n}'''(t) \to f_t'''(t)$,
where $f_t'''(t) \neq 0$ (see equation (\ref{equnrnvnsn})).
Finally, $t$ is the unique root of $f_{t,n}'$ in $B(t,\xi)$.
\end{lem}

\begin{proof}
First suppose that $t \in R_{\mu}^+$. Recall, definition
\ref{defEdgeNonAsy} gives,
\begin{equation*}
\chi_n =  t + \frac{e^{C_n(t)}-1}{e^{C_n(t)} C_n'(t)}
\hspace{0.5cm} \text{and} \hspace{0.5cm}
\chi_n+\eta_n-1 =  t + \frac{e^{C_n(t)}-1}{C_n'(t)},
\end{equation*}
where $C_n$ is the Cauchy transform of $\mu_n$ given in equation
(\ref{eqCauTransn}). Recall that $(\chi_n,\eta_n) \to (\chi,\eta)$.
Equation (\ref{eqxi1}) thus gives $(t-4\xi,t+4\xi) \subset (\chi_n,+\infty)$.
Equations (\ref{eqCauTransn}, \ref{eqftn'Rmu}) then give
$f_{t,n}'(w) = C_n(w) + \log(w - \chi_n) - \log(w-\chi_n-\eta_n+1)$
for all $w \in B(t,4\xi)$, where the logarithms use $(-\infty,0)$ as
the branch cut. Finally, taking $w=t$, the above expressions give
$f_{t,n}'(t) = f_{t,n}''(t) = 0$. Moreover, recalling that
$\mu_n \to \mu$ weakly and $(\chi_n,\eta_n) \to (\chi,\eta)$, equations
(\ref{eqft'Rmu}, \ref{eqftn'Rmu}) give $f_{t,n}'''(t) \to f_t'''(t) \neq 0$.
Finally, using those equations, we can proceed as in part (iii) in the proof
of lemma \ref{lemRootsNonAsy1} to show that $f_{t,n}'$ has $2$ roots in
$B(t,\xi)$. Thus $t$ is the unique root of $f_{t,n}'$ in $B(t,\xi)$.
This proves the result when $t \in R_\mu^+$.

Next suppose that $t \in R_{\l-\mu}$. Note that the above expressions
for $\chi_n$ and $\chi_n+\eta_n-1$ also hold in this case, where now
$e^{C_n(t)}$ and $C_n'(t)$ are defined by analytic extensions. More
exactly, since $\mu_n = \l$ in $(t-4\xi,t+4\xi)$ (see equations
(\ref{eqWeakmun}, \ref{eqxi4})), lemma 2.2 of \cite{Duse15a} gives,
\begin{equation*}
e^{C_n(w)} = e^{B_n(w)} \left( \frac{w-t+4\xi}{w-t-4\xi} \right)
\hspace{0.25cm} \text{and} \hspace{0.25cm}
C_n'(w) = B_n'(w) - \frac1{w-t-4\xi} + \frac1{w-t+4\xi},
\end{equation*}
for all $w \in B(t,4\xi)$, where
$B_n(w) := \int_{[a,b] \setminus (t-4\xi,t+4\xi)} \frac{\mu_n[dx]}{w-x}$.
Recall that $(\chi_n,\eta_n) \to (\chi,\eta)$. Equation (\ref{eqxi1})
thus gives $(t-4\xi,t+4\xi) \subset (\chi_n+\eta_n-1,\chi_n)$. Equation
(\ref{eqftn'Rl-mu}) then gives $f_{t,n}'(w) = B_n(w) + \log(w-\chi_n)
- \log(w-t-2\xi) + \log(w-t+2\xi) - \log(w-\chi_n-\eta_n+1)$ for all
$w \in B(t,2\xi)$, where the first and second logarithms use $(0,+\infty)$
as the branch cut, and the third and fourth logarithms use $(-\infty,0)$ as
the branch cut. Finally, taking $w=t$, the above expressions give
$f_{t,n}'(t) = f_{t,n}''(t) = 0$. Moreover, recalling that $\mu_n \to \mu$
weakly and $(\chi_n,\eta_n) \to (\chi,\eta)$, equations
(\ref{eqft'Rl-mu}, \ref{eqftn'Rl-mu}) give $f_{t,n}'''(t) \to f_t'''(t) \neq 0$.
Finally, using those equations, we can proceed as in part (iii) in the proof
of lemma \ref{lemRootsNonAsy1} to show that $f_{t,n}'$ has $2$ roots in
$B(t,\xi)$. Thus $t$ is the unique root of $f_{t,n}'$ in $B(t,\xi)$. This
proves the result when $t \in R_{\l-\mu}$. Similarly when $t \in R_\mu^-$.
\end{proof}

Next, we write similar convenient expressions for $f_n'$. Recall the
definition for $f_n$ given in equation (\ref{eqfn}). Note, equations
(\ref{eqPnHn}, \ref{eqPnHnlarge}, \ref{eqxi4}) give
$P_n = \frac1n \{x_1,x_2,\ldots,x_n\}$,
$(t-4\xi,t+4\xi) \cap P_n = \emptyset$ when $t \in R_\mu^+ \cup R_\mu^-$,
and $(t-4\xi,t+4\xi) \cap \frac{\Z}n \subset P_n$ when $t \in R_{\l-\mu}$.
Next define $V_n := \frac1n \{v_n+s_n-n,v_n+s_n-n+1,\ldots,v_n\}$. Then,
since $\frac1n (v_n,s_n) \to (\chi,\eta)$ (see equation (\ref{equnrnvnsn})),
equation (\ref{eqxi1}) gives $(t-4\xi,t+4\xi) \cap V_n = \emptyset$ when
$t \in R_\mu^+ \cup R_\mu^-$, and $(t-4\xi,t+4\xi) \cap \frac{\Z}n \subset V_n$
when $t \in R_{\l-\mu}$. Finally, equation (\ref{eqfn}) and the above
observations imply the following, which are well-defined and analytic for all
$w \in B(t,4\xi)$:
\begin{align}
\label{eqfn'Rmu}
f_n'(w)
&= \frac1n \sum_{x \in P_n} \frac1{w - x} -
\frac1n \sum_{x \in V_n} \frac1{w - x}
\hspace{.5cm} \text{when } t \in R_\mu^+ \cup R_\mu^-. \\
\label{eqfn'Rl-mu}
f_n'(w)
&= \frac1n \sum_{x \in P_n ; x \not\in (t-4\xi,t+4\xi)} \frac1{w - x} -
\frac1n \sum_{x \in V_n ; x \not\in (t-4\xi,t+4\xi)} \frac1{w - x}
\hspace{.5cm} \text{when } t \in R_{\l-\mu}.
\end{align}

Finally recall that the sequences $\{(u_n,r_n)\}_{n\ge1} \subset \Z^2$ and
$\{(v_n,s_n)\}_{n\ge1} \subset \Z^2$ of equations
(\ref{equnrnvnsn2}, \ref{equnrnvnsn3}), depend on
sequences $\{m_n\}_{n\ge1} \subset \R^2$ and $\{p_n\}_{n\ge1} \subset \R^2$
which we have yet to define. We now define these, along with some other
useful sequences:
\begin{definition}
\label{defmnpn}
First, recall that $f_{t,n}'''(t) \to f_t'''(t) \neq 0$
(see lemma \ref{lemftn'}), and define $q_n := q_n(t)$ such that,
\begin{equation*}
\tfrac16 q_n^3 f_{t,n}'''(t) = \tfrac13.
\end{equation*}
Next, define $q_{1,n} := q_{1,n}(t)$ and $q_{2,n} := q_{2,n}(t)$ such that,
\begin{equation*}
q_n q_{1,n} = \tfrac12 q_n^2 q_{2,n} = 1.
\end{equation*}
Finally, recall that $e^{C_n(t)} \to e^{C(t)} \not\in \{0,1\}$
(see equation (\ref{eqNonAsyEdge}) and lemma \ref{lemAnalExt}),
and $\chi_n \to \chi \neq t$ (see equation (\ref{eqNonAsyEdge})),
and define $m_n := m_n(t)$ and $p_n := p_n(t)$ such that,
\begin{equation*}
q_{1,n}
= - \frac{p_n ((e^{C_n(t)}-1)^2 + 1)}{(t - \chi_n) e^{C_n(t)}}
\hspace{.5cm} \text{and} \hspace{.5cm}
q_{2,n}
= \frac{m_n (e^{C_n(t)} - 1)}{(t - \chi_n)^2 e^{C_n(t)}}.
\end{equation*}
\end{definition}
Note that each $\{m_n\}_{n\ge1}$, $\{p_n\}_{n\ge1}$,
$\{q_n\}_{n\ge1}$, $\{q_{1,n}\}_{n\ge1}$, $\{q_{2,n}\}_{n\ge1}$, are
convergent sequences of real-numbers with non-zero limits. Moreover:
\begin{lem}
\label{lemNonAsyRoots}
Assume that
$\{(u_n,r_n)\}_{n\ge1} \subset \Z^2$ and $\{(v_n,s_n)\}_{n\ge1} \subset \Z^2$
are chosen as in equations (\ref{equnrnvnsn2}, \ref{equnrnvnsn3}). Then:
\begin{enumerate}
\item
$f_n'(t) = n^{-\frac23} s \; q_{1,n} + O(n^{-1})$ and
$\tilde{f}_n'(t) = n^{-\frac23} r \; q_{1,n} + O(n^{-1})$.
\item
$f_n''(t) = n^{-\frac13} v \; q_{2,n} + O(n^{-\frac23})$ and
$\tilde{f}_n''(t) = n^{-\frac13} u \; q_{2,n} + O(n^{-\frac23})$.
\item
$f_n'''(t) = f_{t,n}'''(t) + O(n^{-\frac13})$ and
$\tilde{f}_n'''(t) = f_{t,n}'''(t) + O(n^{-\frac13})$.
\item
$t = t_{1,n} + O(n^{-\frac13}) = t_{2,n} + O(n^{-\frac13})$
and
$t = \tilde{t}_{1,n} + O(n^{-\frac13}) = \tilde{t}_{2,n} + O(n^{-\frac13})$.
\end{enumerate}
Above, $u, v, r, s$ are the parameters in equations (\ref{equnrnvnsn2}, \ref{equnrnvnsn3}),
$\{t_{1,n},t_{2,n}\}$ denotes the set of roots of $f_n'$ in
$B(t,\xi)$ (see lemma \ref{lemRootsNonAsy1}), and $\{\tilde{t}_{1,n},\tilde{t}_{2,n}\}$
denotes the analogous roots of $\tilde{f}_n'$.
\end{lem}

\begin{proof}
We prove the results for $f_n$, and state
that the results for $\tilde{f}_n$ follow similarly.
Consider (1,2,3) when $t \in R_\mu^+$. First note, equations
(\ref{eqWeakmun}, \ref{eqftn'Rmu}, \ref{eqfn'Rmu}) give,
\begin{align*}
f_{t,n}'(w) - f_n'(w)
&= \int_{R_{\l-\mu}(\e)} \frac{dx}{w-x}
- \frac1n \sum_{x \in P_n; x \in R_{\l-\mu}(\e)} \frac1{w - x} \\
&- \int_{\chi_n+\eta_n-1}^{\chi_n} \frac{dx}{w-x}
+ \frac1n \sum_{x \in V_n} \frac1{w - x},
\end{align*}
for all $w \in B(t,4\xi)$, where
$V_n = \frac1n \{v_n+s_n-n,v_n+s_n-n+1,\ldots,v_n\}$. Next recall (see
equation (\ref{eqxi4})) that $(t-4\xi,t+4\xi) \subset R_{\mu}^+(\e)
\subset \R \setminus R_{\l-\mu}(\e)$, and (see equation (\ref{eqPnHnlarge}))
that $\frac{\Z}n \cap R_{\l-\mu}(\e) \subset P_n$. Riemann approximations
thus give,
\begin{align}
\label{eqftn'fn'Rmu}
f_{t,n}'(w) - f_n'(w)
&= - \int_{\chi_n+\eta_n-1}^{\chi_n} \frac{dx}{w-x}
+ \int_{\frac{v_n}n+\frac{s_n}n-1}^{\frac{v_n}n} \frac{dx}{w-x}
+ O (n^{-1}), \\
\nonumber
f_{t,n}''(w) - f_n''(w)
&= \int_{\chi_n+\eta_n-1}^{\chi_n} \frac{dx}{(w-x)^2}
- \int_{\frac{v_n}n+\frac{s_n}n-1}^{\frac{v_n}n} \frac{dx}{(w-x)^2}
+ O (n^{-1}), \\
\nonumber
f_{t,n}'''(w) - f_n'''(w)
&= -\int_{\chi_n+\eta_n-1}^{\chi_n} \frac{2dx}{(w-x)^3}
+ \int_{\frac{v_n}n+\frac{s_n}n-1}^{\frac{v_n}n} \frac{2dx}{(w-x)^3}
+ O (n^{-1}),
\end{align}
uniformly for $w \in B(t,\xi)$. Recall that $\frac1n(v_n,s_n) \to (\chi,\eta)$
and $(\chi_n,\eta_n) \to (\chi,\eta)$, and $t \in (\chi,+\infty)$. Therefore,
\begin{align*}
f_{t,n}'(t) - f_n'(t)
&= - \log \bigg( \frac{t - \frac{v_n}n}{t - \chi_n} \bigg)
+ \log \bigg( \frac{t - \frac{v_n}n - \frac{s_n}n + 1}{t - \chi_n - \eta_n + 1} \bigg)
+ O (n^{-1}), \\
f_{t,n}''(t) - f_n''(t)
&= - \frac{\frac{v_n}n - \chi_n}{(t - \chi_n) (t - \frac{v_n}n)}
+ \frac{\frac{v_n}n + \frac{s_n}n - \chi_n - \eta_n}
{(t - \chi_n - \eta_n + 1) (t - \frac{v_n}n - \frac{s_n}n + 1)}
+ O (n^{-1}), \\
f_{t,n}'''(t) - f_n'''(t)
&= \frac{(\frac{v_n}n - \chi_n) ((t - \chi_n) + (t - \frac{v_n}n))}
{(t - \chi_n)^2 (t - \frac{v_n}n)^2} \\
&- \frac{(\frac{v_n}n + \frac{s_n}n - \chi_n - \eta_n)
((t - \chi_n - \eta_n + 1) + (t - \frac{v_n}n - \frac{s_n}n + 1))}
{(t - \chi_n - \eta_n + 1)^2 (t - \frac{v_n}n - \frac{s_n}n + 1)^2}
+ O (n^{-1}),
\end{align*}
where $\log$ now represents the natural logarithm. Equation
(\ref{equnrnvnsn3}) thus gives,
\begin{align*}
f_{t,n}'(t) - f_n'(t)
&= - \log \bigg( 1 - \frac{n^{-\frac13} m_n v + n^{-\frac23} p_n s
\; (e^{C_n(t)}-1) + O(n^{-1})}{t - \chi_n} \bigg) \\
&+ \log \bigg( 1 - \frac{n^{-\frac13} m_n v \; e^{C_n(t)}
+ n^{-\frac23} p_n s \; (e^{C_n(t)}-2) + O(n^{-1})}{t - \chi_n - \eta_n + 1} \bigg)
+ O (n^{-1}), \\
f_{t,n}''(t) - f_n''(t)
&= - \frac{n^{-\frac13} m_n v + O(n^{-\frac23})}
{(t - \chi_n) ((t - \chi_n) + O (n^{-\frac13}))} \\
&+ \frac{n^{-\frac13} m_n v \; e^{C_n(t)} + O(n^{-\frac23})}
{(t - \chi_n - \eta_n + 1) ((t - \chi_n - \eta_n + 1) + O (n^{-\frac13}))}
+ O (n^{-1}), \\
f_{t,n}'''(t) - f_n'''(t)
&= O (n^{-\frac13}),
\end{align*}
The third equation gives (3) when $t \in R_\mu^+$. Also note,
that $t - \chi_n - \eta_n + 1 = (t - \chi_n) e^{C_n(t)}$ (see definition
\ref{defEdgeNonAsy}). The first and second equations thus give,
\begin{align*}
f_{t,n}'(t) - f_n'(t)
&= - \log \bigg( 1 - \frac{n^{-\frac13} m_n v}{t - \chi_n} -
\frac{n^{-\frac23} p_n s \; (e^{C_n(t)}-1)}{t - \chi_n}  + O(n^{-1}) \bigg) \\
&+ \log \bigg( 1 - \frac{n^{-\frac13} m_n v}{t - \chi_n}
- \frac{n^{-\frac23} p_n s \; (e^{C_n(t)}-2)}{(t - \chi_n) e^{C_n(t)}} + O(n^{-1}) \bigg)
+ O (n^{-1}), \\
f_{t,n}''(t) - f_n''(t)
&= - \frac{n^{-\frac13} m_n v}{(t - \chi_n)^2}
+ \frac{n^{-\frac13} m_n v}{(t - \chi_n)^2 e^{C_n(t)}} + O (n^{-\frac23}).
\end{align*}
Then, since $f_{t,n}''(t) = 0$ (see lemma \ref{lemftn'})
the second equation and definition \ref{defmnpn} give (2) when $t \in R_\mu^+$.
Also, since $f_{t,n}'(t) = 0$
(see lemma \ref{lemftn'}) the first equation, definition \ref{defmnpn} and a
Taylor expansion give (1) when $t \in R_\mu^+$. We can similarly prove
(1,2,3) when $t \in R_\mu^-$.

Consider (1,2,3) when $t \in R_{\l-\mu}$. First recall
(see equations (\ref{eqPnHnlarge}, \ref{eqxi4})) that
$(t-4\xi,t+4\xi) \subset R_{\l-\mu}(\e)$
and $\frac{\Z}n \cap R_{\l-\mu}(\e) \subset P_n$. Also,
recall (see equation (\ref{eqxi1})) that
$(t-4\xi,t+4\xi) \subset (\chi+\eta-1,\chi)$, and that
$\frac1n(v_n,s_n) \to (\chi,\eta)$ and
$(\chi_n,\eta_n) \to (\chi,\eta)$. Next note, equations
(\ref{eqWeakmun}, \ref{eqftn'Rl-mu}, \ref{eqfn'Rl-mu}) give,
\begin{align*}
f_{t,n}'(w) - f_n'(w)
&= \int_{R_{\l-\mu}(\e) \setminus (t-4\xi,t+4\xi)} \frac{dx}{w-x}
- \frac1n \sum_{x \in P_n; x \in R_{\l-\mu}(\e) \setminus (t-4\xi,t+4\xi)} \frac1{w - x} \\
\nonumber
&- \int_{[\chi_n+\eta_n-1,\chi_n] \setminus (t-4\xi,t+4\xi)} \frac{dx}{w-x}
+ \frac1n \sum_{x \in V_n ; x \not\in (t-4\xi,t+4\xi)} \frac1{w - x},
\end{align*}
for all $w \in B(t,4\xi)$,
where $V_n = \frac1n \{v_n+s_n-n,v_n+s_n-n+1,\ldots,v_n\}$. Riemann
approximations thus give,
\begin{align}
\label{eqftn'fn'Rl-mu}
f_{t,n}'(w) - f_n'(w)
&= - \int_{[\chi_n+\eta_n-1,\chi_n] \setminus (t-4\xi,t+4\xi)} \frac{dx}{w-x} \\
\nonumber
&\;+ \int_{[\frac{v_n}n+\frac{s_n}n-1, \frac{v_n}n] \setminus (t-4\xi,t+4\xi)} \frac{dx}{w - x}
+ O (n^{-1}),
\end{align}
uniformly for $w \in B(t,\xi)$. We can then proceed similarly to
above to prove (1,2,3) when $t \in R_{\l-\mu}$.

Consider (4). First recall, that $f_n'$ and $f_{t,n}'$
are well-defined and analytic in
$B(t,4\xi)$. Also, lemma \ref{lemftn'} implies that $t$ is a root of
$f_{t,n}'$ of multiplicity $2$, and $t$ is the unique root of
$f_{t,n}'$ in $B(t,\xi)$. We will show that there exists constants
$c_1,c_2>0$ for which, given any $\{\xi_n\}_{n\ge1}$ with
$\xi_n \downarrow 0$,
\begin{equation}
\label{eqlemRoots1}
\inf_{w \in \partial B(t,\xi_n)} |f_{t,n}'(w)| > c_1 \xi_n^2
\hspace{0.3cm} \mbox{and} \hspace{0.2cm}
\sup_{w \in \text{cl}(B(t,\xi_n))} | f_{t,n}'(w) - f_n'(w) |
< c_2 (\xi_n n^{-\frac13} + n^{-\frac23}).
\end{equation}
For clarity we state that $c_1,c_2$
are independent of the choice of $\{\xi_n\}_{n\ge1}$. Thus
there exists a choice of $\{\xi_n\}_{n\ge1}$ with
$\xi_n \sim O(n^{-\frac13})$ for which, for all $n$ sufficiently
large, $|f_{t,n}'(w)| > |f_{t,n}'(w) - f_n'(w)|$ for all
$w \in \partial B(t,\xi_n)$. Part (4) then follows from Rouch\'{e}'s
theorem.

Fix $\{\xi_n\}_{n\ge1}$ with $\xi_n \downarrow 0$. Recall,
lemma \ref{lemftn'} implies that $f_{t,n}'(t) = f_{t,n}''(t) = 0$,
and $f_{t,n}'''(t) \to f_t'''(t) \neq 0$. Also note, equations
(\ref{eqftn'Rmu}, \ref{eqftn'Rl-mu}) give $|f_{t,n}^{(4)}(w)| = O(1)$
uniformly for $w \in B(t,\xi)$. Taylor's theorem thus gives
$f_{t,n}'(w) = \frac12 (w-t)^2 f_{t,n}'''(t) + O(\xi_n^3)$
uniformly for $w \in \text{cl}(B(w,\xi_n))$. Therefore,
\begin{equation*}
f_{t,n}'(w) =  \xi_n^2 \; \frac12 e^{i 2 \text{Arg} (w-t) } f_{t,n}'''(t) + O(\xi_n^3),
\end{equation*}
uniformly for $w \in \partial B(t,\xi_n)$. This proves the first part of equation
(\ref{eqlemRoots1}).

Consider the second part of equation (\ref{eqlemRoots1}) when $t \in R_\mu^+$.
Note, equation (\ref{eqftn'fn'Rmu}) and Taylor's theorem give,
\begin{align*}
f_{t,n}'(w) - f_n'(w)
&= - \log \bigg( \frac{t - \frac{v_n}n}{t - \chi_n} \bigg)
+ \log \bigg( \frac{t - \frac{v_n}n - \frac{s_n}n + 1}{t - \chi_n - \eta_n + 1} \bigg) \\
&+ O \bigg(|w-t| \bigg|\chi_n - \frac{v_n}n \bigg|
+ |w-t| \bigg|\eta_n - \frac{s_n}n\bigg| + n^{-1} \bigg),
\end{align*}
uniformly for $w \in B(t,\xi)$, where $\log$ now represents the
natural logarithm. Proceeding similarly to part (1) then gives,
\begin{equation*}
f_{t,n}'(w) - f_n'(w) = O \left( |w-t| n^{-\frac13} + n^{-\frac23} \right),
\end{equation*}
uniformly for $w \in B(t,\xi)$.
This proves the second part of equation (\ref{eqlemRoots1}) when $t \in R_\mu^+$.
Similarly when $t \in R_\mu^-$. Finally, the second part of equation (\ref{eqlemRoots1})
when $t \in R_{\l-\mu}$ can be shown using similar arguments and equation
(\ref{eqftn'fn'Rl-mu}).
\end{proof}

\subsection{The asymptotic behaviour of $f_n - \tilde{f}_n$}

In this section, we assume assumptions \ref{assWeakConv} and \ref{assIsol},
that equation (\ref{equnrnvnsn}) is satisfied for some fixed
$t \in R_\mu^+ \cup R_{\l-\mu} \cup R_\mu^-$ (a root of $f_t'$ of
multiplicity $2$), and that $\{(u_n,r_n)\}_{n\ge1} \subset \Z^2$ and
$\{(v_n,s_n)\}_{n\ge1} \subset \Z^2$ are chosen as in equations
(\ref{equnrnvnsn2}, \ref{equnrnvnsn3}). Define,
\begin{equation}
\label{eqFnw}
F_n := f_n - \tilde{f}_n.
\end{equation}
This function will be useful for the steepest descent analysis. In this
section, we examine the roots of $F_n'$, and the asymptotic behaviour of
$F_n$ as $n \to \infty$. Note, since much of the analysis of this section
is similar to that of the previous two sections, we do not go into as much
detail here.

First, it is useful to examine the following functions:
Define $G_t, G_{t,n} : \C \setminus \R \to \C$ as,
\begin{align}
\label{eqGt}
G_t(w)
&:= - \log(w - \chi) + e^{C(t)} \log(w - \chi - \eta + 1), \\
\nonumber
G_{t,n}(w)
&:= - \log(w - \chi_n) + e^{C_n(t)} \log(w - \chi_n - \eta_n + 1),
\end{align}
for all $w \in \C \setminus \R$, where $C$ and $C_n$ are defined
in equations (\ref{eqCauTrans}, \ref{eqCauTransn}) respectively,
$(\chi,\eta) = (\chi_\EE(t),\eta_\EE(t))$ and
$(\chi_n,\eta_n) = (\chi_n(t),\eta_n(t))$
(see equation (\ref{equnrnvnsn}) and definition \ref{defEdgeNonAsy}),
and the branch cuts are chosen as follows:
\begin{itemize}
\item
For cases (1-4) of lemma \ref{lemCases}, all branch cuts are $(-\infty,0]$.
\item
For cases (5-8) of lemma \ref{lemCases}, the branch cut in the 1st terms
on the RHS is $[0,+\infty)$, and branch cut in the 2nd terms is $(-\infty,0]$.
\item
For cases (9-12) of lemma \ref{lemCases}, all branch cuts are $[0,+\infty)$.
\end{itemize}
Recall that $e^{C_n(t)} \to e^{C(t)}$ and $(\chi_n,\eta_n) \to (\chi,\eta)$
(see equation (\ref{eqNonAsyEdge})). The above choices
for the branches, lemma \ref{lemCases} and equation (\ref{eqxi1}),
imply that $G_t$ and $G_{t,n}$ both extend analytically to
$(\C \setminus \R) \cup (t-2\xi,t+2\xi)$.

Next note, irrespective of the choices of the branches of the logarithms,
\begin{align}
\label{eqGt'}
G_t'(w)
&= - \frac1{w - \chi} + \frac{e^{C(t)}}{w - \chi - \eta + 1}, \\
\label{eqGtn'}
G_{t,n}'(w)
&= - \frac1{w - \chi_n} + \frac{e^{C_n(t)}}{w - \chi_n - \eta_n + 1},
\end{align}
for all $w \in (\C \setminus \R) \cup (t-2\xi,t+2\xi)$. Therefore, $G_t'$
and $G_{t,n}'$ extend analytically to
$\C \setminus \{\chi,\chi+\eta-1\}$ and $\C \setminus \{\chi_n,\chi_n+\eta_n-1\}$,
respectively. Moreover:
\begin{lem}
\label{lemGt'}
$G_t'$ has a root of multiplicity $1$ at $t$, and no other roots in
$\C \setminus \{\chi,\chi+\eta-1\}$. Moreover,
$G_{t,n}'$ has a root of multiplicity $1$ at $t$, and no other roots in
$\C \setminus \{\chi_n,\chi_n+\eta_n-1\}$. Finally, $G_{t,n}''(t) \to G_t''(t)$.
\end{lem}

\begin{proof}
Consider $G_t'$. Note, since $(\chi,\eta) =  (\chi_\EE(t),\eta_\EE(t))$,
equations (\ref{eqchiEEetaEE}, \ref{eqGt'}) imply that
$t$ is the only root of $G_t'$ in $\C \setminus \{\chi,\chi+\eta-1\}$.
These also give,
\begin{equation*}
G_t''(t) = \frac{e^{C(t)} C'(t)^2}{e^{C(t)} - 1}.
\end{equation*}
Lemma \ref{lemAnalExt} then implies that $G_t''(t) \neq 0$.

Consider $G_{t,n}'$. Note, since
$(\chi_n,\eta_n) =  (\chi_n(t),\eta_n(t))$, definition \ref{defEdgeNonAsy}
and equation (\ref{eqGtn'}) imply that $t$ is the only root of
$G_{t,n}'$ in $\C \setminus \{\chi_n,\chi_n+\eta_n-1\}$. These also give,
\begin{equation*}
G_{t,n}''(t) = \frac{e^{C_n(t)} C_n'(t)^2}{e^{C_n(t)} - 1}.
\end{equation*}
Equation (\ref{eqNonAsyEdge}) then gives $G_{t,n}''(t) \to G_t''(t) \neq 0$.
\end{proof}

Next consider $F_n$. Define, for convenience,
\begin{align}
\label{eqUnVn}
U_n &:= \tfrac1n \{u_n+r_n-n+1, u_n+r_n-n+2, \ldots, u_n-1\}, \\
\nonumber
V_n &:= \tfrac1n \{v_n+s_n-n, v_n+s_n-n+1, \ldots, v_n\}.
\end{align}
Thus, since $\eta \in (0,1)$ and $\frac1n (u_n,r_n) \to (\chi,\eta)$
and $\frac1n (v_n,s_n) \to (\chi,\eta)$ (see equation
(\ref{equnrnvnsn})), $\min \{u_n - 1, v_n\} > \max\{u_n + r_n - n + 1, v_n + s_n - n\}$.
Moreover,
\begin{equation}
\label{eqVn-Un}
V_n \setminus U_n = (VU^{(n)}) \cup (VU_{(n)})
\hspace{.5cm} \text{and} \hspace{.5cm}
U_n \setminus V_n = (UV^{(n)}) \cup (UV_{(n)}),
\end{equation}
where we define:
\begin{itemize}
\item
$VU^{(n)} :=
\frac1n \{u_n, u_n+1, \ldots, v_n\}$ when $v_n \ge u_n$.
\item
$UV^{(n)} := 
\frac1n \{v_n+1, v_n+2, \ldots, u_n-1\}$ when $v_n+1 \le u_n-1$.
\item
$VU_{(n)} := \frac1n \{v_n+s_n-n, v_n+s_n-n+1, \ldots, u_n+r_n-n\}$
when $v_n+s_n \le u_n+r_n$.
\item
$UV_{(n)} := \frac1n \{u_n+r_n-n+1, u_n+r_n-n+2, \ldots, v_n+s_n-n-1\}$
when $v_n+s_n-1 \ge u_n+r_n+1$.
\end{itemize}
Note, implicit in the above definitions is that $VU^{(n)} := \emptyset$ when
$v_n < u_n$, etc. Finally, fixing $\xi>0$ sufficiently
small such that equation (\ref{eqxi}, \ref{eqxi1}, \ref{eqxi4}) are satisfied,
and such that $\chi - 4\xi > \chi + \eta - 1 + 4 \xi$, equation
(\ref{equnrnvnsn}) gives:
\begin{equation}
\label{eqxi5}
VU^{(n)} \subset (\chi - 2\xi, \chi + 2\xi)
\hspace{.5cm} \text{and} \hspace{.5cm}
VU_{(n)} \subset (\chi+\eta-1 - 2\xi, \chi+\eta-1 + 2\xi).
\end{equation}
Similarly for $UV^{(n)}$
and $UV_{(n)}$. Finally note that
equations (\ref{eqfn}, \ref{eqtildefn}, \ref{eqFnw}, \ref{eqUnVn}, \ref{eqVn-Un}) give,
\begin{align}
\label{eqFnw2}
n F_n(w)
&= \bigg( 1_{(v_n+1 \le u_n-1)}
\sum_{x \in UV^{(n)}}
- 1_{(v_n \ge u_n)}
\sum_{x \in VU^{(n)}} \bigg) \log (w - x) \\
\nonumber
&+ \bigg( 1_{(v_n+s_n-1 \ge u_n+r_n+1)}
\sum_{x \in UV_{(n)}}
- 1_{(v_n+s_n \le u_n+r_n)}
\sum_{x \in VU_{(n)}} \bigg) \log (w - x),
\end{align}
for all $w \in \C \setminus \R$, where the branch cuts are chosen as follows:
\begin{itemize}
\item
For cases (1-4) of lemma \ref{lemCases}, all branch cuts are $(-\infty,0]$.
\item
For cases (5-8) of lemma \ref{lemCases}, the branch cuts in the 1st and 2nd
sums on the RHS are all $[0,+\infty)$, and the branch cuts in the 3rd
and 4th sums are all $(-\infty,0]$.
\item
For cases (9-12) of lemma \ref{lemCases}, all branch cuts are $[0,+\infty)$.
\end{itemize}
Note, the above branch cut choices are consistent with those made in
equations (\ref{eqfn}, \ref{eqtildefn}) (see the discussion
given before lemma \ref{lemfnftConv}). Also note, lemma \ref{lemCases}
and equations (\ref{eqxi}, \ref{eqxi1}, \ref{eqxi4}, \ref{eqxi5}), imply that $F_n$
extends analytically to $(\C \setminus \R) \cup (t-2\xi,t+2\xi)$.

Note, irrespective of the choices of the branches of the logarithms,
\begin{align}
\label{eqFn'w}
n F_n'(w)
&= \bigg( 1_{(v_n+1 \le u_n-1)} \sum_{x \in UV^{(n)}}
- 1_{(v_n \ge u_n)} \sum_{x \in VU^{(n)}} \bigg) \frac1{w - x} \\
\nonumber
&+ \bigg( 1_{(v_n+s_n-1 \ge u_n+r_n+1)} \sum_{x \in UV_{(n)}}
- 1_{(v_n+s_n \le u_n+r_n)} \sum_{x \in VU_{(n)}} \bigg) \frac1{w - x},
\end{align}
for all $(\C \setminus \R) \cup (t-2\xi,t+2\xi)$. Therefore $F_n'$ extends
analytically to $\C \setminus ((V_n \setminus U_n) \cup (U_n \setminus V_n))$.
Next we investigate the relationships between $G_t'$, $G_{t,n}'$ and
$F_n'$ in $B(t,\xi)$:
\begin{lem}
\label{lemGtGtnFn}
Fix $\xi>0$ as above. Then:
\begin{enumerate}
\item
$G_{t,n}'(w) = G_t'(w) + o(1)$ uniformly for
$w \in B(t,\xi)$. Similarly for $G_{t,n}''$, $G_t''$.
\end{enumerate}
Next, fix $\{(u_n,r_n)\}_{n\ge1}$, $\{(v_n,s_n)\}_{n\ge1}$, $\{m_n\}_{n\ge1}$,
$u$, $v$ as in equations (\ref{equnrnvnsn2}, \ref{equnrnvnsn3}). Additionally
assume that $u \neq v$. Then:
\begin{enumerate}
\setcounter{enumi}{1}
\item
$n^\frac13 F_n'(w) = m_n (v-u) G_{t,n}'(w) + O(n^{-\frac13})$ uniformly for
$w \in B(t,\xi)$. Similarly for $G_{t,n}''$, $F_n''$.
\end{enumerate}
\end{lem}

\begin{proof}
Consider (1). We will prove (1) only for $G_{t,n}', G_t'$. Part (1)
for $G_{t,n}'', G_t''$ follows similarly. First note,
equations (\ref{eqGt'}, \ref{eqGtn'}) give,
\begin{equation*}
G_{t,n}'(w)
= G_t'(w) - \frac{\chi_n - \chi}{(w - \chi) (w - \chi_n)}
+ \frac{e^{C_n(t)} - e^{C(t)}}{w - \chi_n - \eta_n + 1}
+ \frac{e^{C(t)} (\chi_n + \eta_n - \chi - \eta)}{(w - \chi - \eta + 1) (w - \chi_n - \eta_n + 1)},
\end{equation*}
for all $w \in B(t,\xi)$. Recall (see equation (\ref{eqNonAsyEdge})) that
$e^{C_n(t)} \to e^{C(t)}$ and $(\chi_n,\eta_n) \to (\chi,\eta)$.
Also recall (see equation (\ref{eqxi1})) that $|w - \chi| > 3 \xi$
and $|w - \chi - \eta + 1| > 3 \xi$ for all $w \in B(t,\xi)$. Combined, the above
prove (1) for $G_{t,n}', G_t'$.

Consider (2). Note, since $u \neq v$, the equations
(\ref{equnrnvnsn2}, \ref{equnrnvnsn3}, \ref{eqVn-Un})
imply that either $VU^{(n)} \neq \emptyset$ and $UV^{(n)} = \emptyset$
for all $n$ sufficiently large, or $VU^{(n)} = \emptyset$ and
$UV^{(n)} \neq \emptyset$ for all $n$ sufficiently large.
Similarly for $VU_{(n)}$ and $UV_{(n)}$. Moreover, these sets contain
at least $2$ distinct elements, whenever they are non-empty.

First suppose that $UV^{(n)} \neq \emptyset$ and $VU_{(n)} \neq \emptyset$ and
$UV_{(n)} = VU^{(n)} = \emptyset$. For this case, equation (\ref{eqFn'w}) gives,
\begin{equation*}
n F_n'(w)
= \sum_{x \in UV^{(n)}} \frac1{w - x}
- \sum_{x \in VU_{(n)}} \frac1{w - x},
\end{equation*}
for all $w \in B(t,\xi)$. We write this as,
\begin{align*}
n F_n'(w)
&= \sum_{x \in UV^{(n)}}
\bigg( \frac1{w - \chi_n} - \frac{\chi_n - x}{(w - \chi_n) (w - x)} \bigg) \\
&- \sum_{x \in VU_{(n)}} \bigg( \frac1{w - \chi_n - \eta_n + 1}
- \frac{\chi_n + \eta_n - 1 - x}{(w - \chi_n - \eta_n + 1) (w - x)} \bigg),
\end{align*}
for all $w \in B(t,\xi)$. Note, equations (\ref{eqNonAsyEdge}, \ref{eqxi1})
imply that $|w - \chi_n| > \xi$ and $|w - \chi_n - \eta_n + 1| > \xi$ uniformly
for $w \in B(t,\xi)$. Also note, equations (\ref{eqxi1}, \ref{eqxi5}) imply that
$|w - x| > \xi$ uniformly for $w \in B(t,\xi)$ and $x \in UV^{(n)}$, and
$|w - x| > \xi$ uniformly for $w \in B(t,\xi)$ and $x \in VU_{(n)}$.
Moreover, since $u \neq v$, equations (\ref{equnrnvnsn2}, \ref{equnrnvnsn3}, \ref{eqVn-Un}),
imply the following:
\begin{itemize}
\item
$\chi_n - x = O(n^{-\frac13})$ uniformly for $x \in UV^{(n)}$.
\item
$\chi_n+\eta_n-1 - x = O(n^{-\frac13})$ uniformly for $x \in VU_{(n)}$.
\item
$|UV^{(n)}| = O( n^\frac23 )$ and
$|VU_{(n)}| = O( n^\frac23 )$.
\end{itemize}
Combined, the above give,
\begin{equation*}
n F_n'(w)
= \sum_{x \in UV^{(n)}} \frac1{w - \chi_n}
- \sum_{x \in VU_{(n)}} \frac1{w - \chi_n - \eta_n + 1}
+ O(n^\frac13),
\end{equation*}
uniformly for $w \in B(t,\xi)$. Finally note, equations
(\ref{equnrnvnsn2}, \ref{equnrnvnsn3}, \ref{eqVn-Un}), give
$|UV^{(n)}| = n^\frac23 m_n (u-v) + O( n^\frac13 )$
and $|VU_{(n)}| = n^\frac23 m_n e^{C_n(t)} (u-v) + O( n^\frac13 )$.
Therefore,
\begin{equation*}
n F_n'(w) =  \frac{n^\frac23 m_n (u-v)}{w - \chi_n}
- \frac{n^\frac23 m_n e^{C_n(t)} (u-v)}{w - \chi_n - \eta_n + 1} + O(n^\frac13) ,
\end{equation*}
uniformly for $w \in B(t,\xi)$. Equation (\ref{eqGtn'}) finally proves (2) for
$F_n', G_{t,n}'$, when $UV^{(n)} \neq \emptyset$ and $VU_{(n)} \neq \emptyset$
and $UV_{(n)} = VU^{(n)} = \emptyset$. Part (2) for the other cases follows
similarly.
\end{proof}

Next we use the previous lemmas to investigate the roots of $F_n'$ when $u \neq v$.
\begin{lem}
\label{lemFn'}
Fix $\xi>0$ as above, and assume that $u \neq v$. Then:
\begin{enumerate}
\item
$F_n'$ has $1$ root in $(t-\xi,t+\xi)$.
\item
$F_n'$ has $0$ roots in $\C \setminus \R$.
\item
$F_n'$ has $1$ root in each interval
of the form $(x,y)$, when $x$ and $y$ are any two consecutive elements of
either $VU^{(n)}$ or $UV^{(n)}$
or $VU_{(n)}$ or $UV_{(n)}$.
\item
$F_n'$ has no other roots in $\R$ excluding those listed in parts (1,3).
\item
$w_n = t + O(n^{-\frac13})$, where $w_n$ denotes the root in part (1), above.
\end{enumerate}
\end{lem}

\begin{proof}
Consider (1). First, recall that $u \neq v$, and let $m_t \neq 0$
denote the non-zero limit of the sequence $\{m_n\}_{n\ge1}$ of real-numbers
(see definition \ref{defmnpn}). Next recall (see lemma \ref{lemGt'}) that $t$ is
the only root of $G_t'$ in $B(t,\xi)$. The extremal value theorem thus gives,
\begin{equation*}
\inf_{w \in \partial B(t,\xi)} |(v-u) m_t G_t'(w)| > 0.
\end{equation*}
Next note, parts (1,2) of lemma \ref{lemGtGtnFn} give
$n^\frac13 F_n'(w) = (v-u) m_t G_t'(w) + o(1)$ uniformly for
$w \in B(t,\xi)$. Combined, the above
imply that $|(v-u) m_t G_t'(w)| > |(v-u) m_t G_t'(w) - n^\frac13 F_n'(w)|$
for all $w \in \partial B(t,\xi)$. Rouch\'{e}'s theorem thus implies that
$G_t'$ and $F_n'$ have the same number of roots in $B(t,\xi)$.
Lemma \ref{lemGt'} thus implies that
$F_n'$ has $1$ root in $B(t,\xi)$.
This root is necessarily real-valued since non-real roots of $F_n'$
occur in complex conjugate pairs (see equation (\ref{eqFn'w})).
This proves (1).

Consider (2-4). As in the proof of part (2) of lemma \ref{lemGtGtnFn},
we will prove these only when $UV^{(n)} \neq \emptyset$ and
$VU_{(n)} \neq \emptyset$ and $UV_{(n)} = VU^{(n)} = \emptyset$. Part
(2-4) for the other cases follows similarly. Note, for the above case,
equation (\ref{eqFn'w}) gives,
\begin{equation}
\label{eqlemFn'2}
n F_n'(w)
= \sum_{x \in UV^{(n)}} \frac1{w - x}
- \sum_{x \in VU_{(n)}} \frac1{w - x},
\end{equation}
for all $w \in \C \setminus ((UV^{(n)}) \cup (VU_{(n)}))$. Recall, since
$u \neq v$, that $UV^{(n)}$ and $VU_{(n)}$ both contain at least $2$ elements.
Also recall (see equations (\ref{eqxi1}, \ref{eqxi5})) that $(t-\xi,t+\xi)$
and $UV^{(n)}$ and $VU_{(n)}$ are mutually disjoint. We will show:
\begin{enumerate}
\item[(i)]
$F_n'$ has $|UV^{(n)}| + |VU_{(n)}| - 1$
roots in $\C \setminus ((UV^{(n)}) \cup (VU_{(n)}))$.
\item[(ii)]
$F_n'$ has at least $1$ root in each interval of the form $(x,y)$,
where $x$ and $y$ are any two consecutive elements of either
$UV^{(n)}$ or $VU_{(n)}$.
\end{enumerate}
Finally recall (see part (1)) that $F_n'$ has $1$ root
in $(t-\xi,t+\xi)$. Combined, the above
imply that $F_n'$ has $1$ root in $(t-\xi,t+\xi)$, $1$
in each of the intervals listed in part (ii), and no other roots.
This proves parts (2-4) in this case.

Consider (i). First note equation (\ref{eqlemFn'2}) gives,
\begin{equation*}
F_n'(w) = \frac1n \bigg( \prod_{y \in (UV^{(n)})
\cup (VU_{(n)})} \frac1{w - y} \bigg) P_n(w),
\end{equation*}
for all $w \in \C \setminus ((UV^{(n)})
\cup (VU_{(n)}))$, where $P_n$ is the polynomial,
\begin{equation*}
P_n(w) = \sum_{x \in UV^{(n)}}
\bigg( \prod_{y \in ((UV^{(n)}) \setminus \{x\})
\cup VU_{(n)})} (w - y) \bigg)
- \sum_{x \in VU_{(n)}}
\bigg( \prod_{y \in (UV^{(n)})
\cup ((VU_{(n)}) \setminus \{x\})} (w - y) \bigg).
\end{equation*}
Note that $P_n$ has degree at most $|UV^{(n)}| +
|VU_{(n)}| - 1$. Moreover, since $UV^{(n)}$
and $VU_{(n)}$ are disjoint, $P_n$ has no roots in
$(UV^{(n)}) \cup (VU_{(n)})$.
Therefore the roots of $P_n$ and $F_n'$ coincide. This proves (i).

Consider (ii). Let $x$ and $y$ denote any two consecutive elements of
$UV^{(n)}$. Equation (\ref{eqlemFn'2}) then implies that
$F_n'$ is a real-valued continuous function on $(x,y)$, and
\begin{equation*}
\lim_{w \in (x,y), w \uparrow y} F_n'(w) = - \infty
\hspace{0.5cm} \text{and} \hspace{0.5cm}
\lim_{w \in (x,y), w \downarrow x} F_n'(w) = + \infty.
\end{equation*}
The intermediate value theorem thus implies that $F_n'$ has a
root in $(x,y)$. Similarly $F_n'$ has a
root in $(x,y)$, when $x$ and $y$ denote any two
consecutive elements of $VU_{(n)}$. This proves (ii).

Consider (5). First recall (see lemma \ref{lemGt'}) that $G_{t,n}'$
has a root of multiplicity $1$ at $t$,
and no other roots in $\C \setminus \{\chi_n,\chi_n+\eta_n-1\}$, and
$G_{t,n}''(t) \to G_t''(t) \neq 0$. Next, recall that
$v - u \neq 0$ (by assumption) and that $\{m_n\}_{n\ge1}$ is a convergent
sequence of real-numbers with a non-zero limit (see definition \ref{defmnpn}).
Then, using part (2) of lemma \ref{lemGtGtnFn}, we can
proceed similarly to the proof of part (4) of lemma \ref{lemNonAsyRoots}
to show the following: There exists a sequence $\{\xi_n\}_{n\ge1}$ of positive numbers for
which $\xi_n = O(n^{-\frac13})$ and
$|(v-u) m_n G_{t,n}'(w)| > |(v-u) m_n G_{t,n}'(w) - n^\frac13 F_n'(w)|$
for all $w \in \partial B(t,\xi_n)$. Rouch\'{e}'s theorem thus implies
that $F_n'$ and $G_{t,n}'$ have the
same number of roots in $B(t,\xi_n)$ for this choice of $\xi_n$, i.e.,
$1$ root. Thus the root
$w_n$ of $F_n'$, discovered in part (1), must satisfy $w_n \in B(t,\xi_n)$.
This proves (5).
\end{proof}

We end this section by examining the asymptotic behaviour, as $n \to \infty$,
of $F_n(t)$:
\begin{lem}
\label{lemFnt}
We have,
\begin{equation*}
\exp( n F_n(t) ) = \frac{A_{t,n} ((u_n,r_n),(v_n,s_n))}
{(t-\chi_n) (t-\chi_n - \eta_n + 1)} \exp(O(n^{-\frac13})),
\end{equation*}
where $A_{t,n} : (\Z^2)^2 \to \R \setminus \{0\}$ is defined by:
\begin{align*}
&A_{t,n}((U,R),(V,S)) := (t-\chi_n)^{-(V-U)} (t-\chi_n-\eta_n+1)^{V+S-U-R} \\
\nonumber
&\times \exp \bigg(
\frac{n}2 \frac{(\frac{V}n - \chi_n)^2 - (\frac{U}n - \chi_n)^2}{t - \chi_n}
- \frac{n}2 \frac{(\frac{V}n + \frac{S}n - \chi_n - \eta_n)^2
- (\frac{U}n + \frac{R}n - \chi_n - \eta_n)^2}{t - \chi_n - \eta_n + 1} \bigg) \\
\nonumber
&\times \exp \bigg(
\frac{n}6 \frac{(\frac{V}{n} - \chi_n)^3 - (\frac{U}{n} - \chi_n)^3}{(t - \chi_n)^2}
- \frac{n}6 \frac{(\frac{V}n + \frac{S}{n} - \chi_n - \eta_n)^3
- (\frac{U}n + \frac{R}n - \chi_n - \eta_n)^3}{(t - \chi_n - \eta_n + 1)^2} \bigg),
\end{align*}
for all $(U,R),(V,S) \in \Z^2$.
\end{lem}

\begin{proof}
We will prove this result when $VU^{(n)} \neq \emptyset$
and $VU_{(n)} \neq \emptyset$ and $UV^{(n)} = UV_{(n)} = \emptyset$,
and state that the result for the other cases follows from similar
considerations.

Assume the above case. Then, irrespective of the choices
of the branches of the logarithms, equations (\ref{eqVn-Un}, \ref{eqFnw2}) give,
\begin{equation}
\label{eqFnt}
\exp ( n F_n(t) )
= \bigg[ \prod_{j=u_n}^{v_n} (t-\tfrac{j}n) \bigg]^{-1}
\bigg[ \prod_{k=v_n+s_n-n}^{u_n+r_n-n} (t-\tfrac{k}n) \bigg]^{-1}.
\end{equation}
To examine this, first write,
\begin{equation*}
\prod_{j=u_n}^{v_n} (t-\tfrac{j}n)
= (t-\chi_n)^{v_n-u_n+1}
\prod_{j=u_n}^{v_n} \bigg( 1 - \frac{\tfrac{j}n-\chi_n}{t-\chi_n} \bigg).
\end{equation*}
Note that $\frac{j}n - \chi_n = O(n^{-\frac13})$ uniformly
for $j \in \{u_n, u_n+1, \ldots, v_n\}$ (see equations
(\ref{equnrnvnsn2}, \ref{equnrnvnsn3}))
and that $|t - \chi_n| > 2 \xi > 0$ (see equation (\ref{eqxi1}),
and recall that $\chi_n \to \chi$). Thus, we can write,
\begin{equation*}
\prod_{j=u_n}^{v_n} (t-\tfrac{j}n)
= (t-\chi_n)^{v_n-u_n+1} \exp \bigg[
\sum_{j=u_n}^{v_n} \log \bigg( 1 - \frac{\tfrac{j}n-\chi_n}{t-\chi_n} \bigg) \bigg],
\end{equation*}
where $\log$ denotes the natural logarithm. Moreover, Taylor's theorem gives,
\begin{equation*}
\log \bigg( 1 - \frac{\tfrac{j}n-\chi_n}{t-\chi_n} \bigg)
= - \frac{\tfrac{j}n - \chi_n}{t - \chi_n}
- \frac12 \frac{(\tfrac{j}n - \chi_n)^2}{(t - \chi_n)^2} + O(n^{-1}),
\end{equation*}
uniformly for $j \in \{u_n, u_n+1, \ldots, v_n\}$. Therefore,
since $v_n = n \chi_n + O(n^\frac23)$ and $u_n = n \chi_n + O(n^\frac23)$
(see equations (\ref{equnrnvnsn2}, \ref{equnrnvnsn3})),
\begin{align*}
\prod_{j=u_n}^{v_n} (t-\tfrac{j}n)
= (t-\chi_n)^{v_n-u_n+1} \exp \bigg[
&-\frac{n}2 \frac{(\frac{v_n}n - \chi_n)^2 - (\frac{u_n}n - \chi_n)^2}{t - \chi_n} \\
&- \frac{n}6 \frac{(\frac{v_n}n - \chi_n)^3 - (\frac{u_n}n - \chi_n)^3}{(t - \chi_n)^2}
+ O(n^{-\frac13}) \bigg].
\end{align*}
Similarly we can show that,
\begin{align*}
\prod_{k=v_n+s_n-n}^{u_n+r_n-n} (t-\tfrac{k}n)
&= (t-\chi_n-\eta_n+1)^{u_n+r_n-v_n-s_n+1} \exp \bigg[ \\
&-\frac{n}2 \frac{(\frac{u_n}n + \frac{r_n}n - \chi_n - \eta_n)^2
- (\frac{v_n}n + \frac{s_n}n - \chi_n - \eta_n)^2}{t - \chi_n - \eta_n + 1} \\
&- \frac{n}6 \frac{(\frac{u_n}n + \frac{r_n}n - \chi_n - \eta_n)^3
- (\frac{v_n}n + \frac{s_n}n - \chi_n - \eta_n)^3}{(t - \chi_n - \eta_n + 1)^2}
+ O(n^{-\frac13}) \bigg].
\end{align*}
Equation (\ref{eqFnt}) then gives the required result.
\end{proof}

\section{Steepest descent analysis}
\label{secsdatak}

In this section we prove theorem \ref{thmAiry} via steepest descent
analysis. Assume the conditions of that theorem:
Assume assumptions \ref{assWeakConv} and \ref{assIsol},
that equation (\ref{equnrnvnsn}) is satisfied for some fixed
$t \in R_\mu^+ \cup R_{\l-\mu} \cup R_\mu^-$ (a root of
$f_t'$ of multiplicity $2$), assumption \ref{asscases},
and that $\{(u_n,r_n)\}_{n\ge1} \subset \Z^2$ and
$\{(v_n,s_n)\}_{n\ge1} \subset \Z^2$ are chosen as in equations
(\ref{equnrnvnsn2}, \ref{equnrnvnsn3}).

\subsection{Local asymptotic behaviour}

In this section we examine the local behaviour of $f_t$, $f_n$ and
$\tilde{f}_n$ in neighbourhoods of $t$. We begin by using lemma \ref{lemCases},
which describes the various situations of theorem \ref{thmAiry} in explicit
detail, to choose the branches of the logarithms in equations (\ref{eqf2}, \ref{eqfn2})
so that both $f_t$ and $f_n$ are well-defined and analytic
in convenient open subsets of $\C$ which contain $t$. We similarly choose the
branches of the logarithms in equation (\ref{eqtildefn2}), for $\tilde{f}_n$.
Recall (see equation (\ref{eqf'domain2})) that
$\C \setminus S = (\C \setminus \R) \cup J \cup K$, is the domain of $f_t'$,
$S = S_1 \cup S_2 \cup S_3 \subset \R$ (see equation (\ref{eqS1S2S3})),
$J = \cup_{i=1}^4 J_i$, $K = \cup_{i=1}^3 K^{(i)}$, $K^{(i)}$ is partitioned
as $\{K_1^{(i)}, K_2^{(i)}, \ldots\}$ for all $i \in \{1,2,3\}$, and
$\{J_1,J_2,J_3,J_4\} \cup \cup_{i=1}^3 \{K_1^{(i)},K_2^{(i)},\ldots\}$ is a set of
pairwise disjoint open intervals. These sets are depicted in
figure \ref{figf'domain}, and properties of $S_1,S_2,S_3$ are discussed in
equation (\ref{eqS1S2S3In}). Also recall (see equation (\ref{eqIntervalt}))
that $L_t \in \{J_1,J_2,J_3,J_4\} \cup \cup_{i=1}^3 \{K_1^{(i)},K_2^{(i)},\ldots\}$
denotes that open interval with $t \in L_t$.
We write (see equations (\ref{eqf2}, \ref{eqft})),
\begin{align*}
f_t(w) 
&= \int_{S_1^+} \log (w-x) \mu[dx]
- \int_{S_2^+} \log (w-x) (\l-\mu)[dx]
+ \int_{S_3^+} \log (w-x) \mu[dx] \\
&+ \int_{S_1^-} \log (w-x) \mu[dx]
- \int_{S_2^-} \log (w-x) (\l-\mu)[dx]
+ \int_{S_3^-} \log (w-x) \mu[dx],
\end{align*}
for all $w \in \C \setminus \R$, where $S_i^+ = S_i \cap [\sup L_t,+\infty)$
and $S_i^- = S_i \cap (-\infty,\inf L_t]$. Thus $f_t$ is analytic in
$(\C \setminus \R) \cup L_t$ if we choose the branch cuts of all the
logarithms in the first three terms on the RHS to be $[0,+\infty)$,
and the branch cuts of all the logarithms in the last three terms
to be $(-\infty,0]$. Next, define
$L_n \subset \R \setminus S_n$ as in equation (\ref{eqIntervaln}). 
Then we can similarly choose the branches of the logarithms in equation
(\ref{eqfn2}) such that $f_n$ is well-defined and analytic in
$\C \setminus L_n$. Finally,
fix $\xi>0$ sufficiently small such that equation (\ref{eqxi}) is
satisfied, and define
\begin{equation}
\label{eqAnalSetftfn}
\C_\xi := \{w \in \C : \inf L_t + 2 \xi <  \text{Re}(w) < \sup L_t - 2 \xi 
\text{ or } |\text{Im}(w)| > \xi^4 \}.
\end{equation}
The use of $\xi^4$, above, is a choice of convenience which will
simplify some calculations later. Note, equations (\ref{eqxi}, \ref{eqAnalSetftfn})
imply that $B(t,2\xi) \subset \C_\xi$. Also
note that $\C_\xi \subset (\C \setminus \R) \cup L_t$, and so $f_t$
is well-defined and analytic in $\C_\xi$. Also, equation
(\ref{eqIntervaln}) implies that $\C_\xi \subset (\C \setminus \R) \cup L_n$,
and so $f_n$ is well-defined and analytic in
$\C_\xi$. Similarly for $\tilde{f}_n$. Moreover:
\begin{lem}
\label{lemfnftConv}
Fix $\xi>0$ as above, fix $r>0$, and fix an integer $k \ge 0$. Then:
\begin{enumerate}
\item
Then there exists a positive constant, $C = C(t,\xi,k,r)$, for which
\begin{equation*}
\sup_{w \in \C_\xi \cap B(0,r)} |f_t^{(k)}(w)| < C
\;\; \text{and} \;\;
\sup_{w \in \C_\xi \cap B(0,r)} |f_n^{(k)}(w)| < C.
\end{equation*}
\item
$f_n^{(k)} \to f_t^{(k)}$ uniformly in $\C_\xi \cap B(0,r)$.
\end{enumerate}
Similarly for $\tilde{f}_n$.
\end{lem}

\begin{proof}
Consider (1). First recall that $S_n = S_{1,n} \cup S_{2,n} \cup S_{3,n}$,
and $\tfrac1n |S_{i,n}| = O(1)$ for all $i \in \{1,2,3\}$ (see equations
(\ref{eqfn2}, \ref{eqS1nS2nS3nIn})). Equation \ref{eqfn2} then implies that
there exists a constant, $c>0$, for which
\begin{equation*}
\sup_{w \in \C_\xi \cap B(0,r)} |f_n(w)|
< c \sup_{(w,x) \in (\C_\xi \cap B(0,r)) \times S_n} |\log(w-x)|.
\end{equation*}
Next note, equation (\ref{eqS1nS2nS3nIn}) gives,
\begin{equation*}
\sup_{(w,x) \in (\C_\xi \cap B(0,r)) \times S_n} |w-x|
< r + 2 \max\{|\inf S_3|,|\sup S_1|\}.
\end{equation*}
Moreover,
\begin{equation*}
\inf_{(w,x) \in (\C_\xi \cap B(0,r)) \times S_n} |w-x|
> \min\{\xi,\xi^4\}.
\end{equation*}
Indeed, the above follows since either $|\text{Im} (w-x)| > \xi^4$
(see equation (\ref{eqAnalSetftfn})) or $|\text{Re} (w-x)|
> \min \{ (\inf L_t + 2\xi) - \inf L_n, \sup L_n - (\sup L_t - 2\xi) \}$
(see equations (\ref{eqIntervaln}, \ref{eqAnalSetftfn}) and note that
$L_n \subset \R \setminus S_n$ and $\{\sup L_n, \inf L_n\} \subset S_n$),
and since $\sup L_n = \sup L_t + o(1)$ and $\inf L_n = \inf L_t + o(1)$
(see equation (\ref{eqIntervaln})).
Combined, the above three inequalities prove part (1) for $f_n$.
Part (1) for $f_n^{(k)}$ for all $k\ge1$ follows similarly.
Also, part (1) for $f_t^{(k)}$ and $\tilde{f}_n^{(k)}$ for all $k\ge0$
follows similarly.

Consider (2). First note, for all $k \ge 0$, equations
(\ref{eqf2}, \ref{eqft}, \ref{eqfn2}, \ref{eqS1nS1WeakConv})
imply that $f_n^{(k)} \to f_t^{(k)}$ pointwise in $\C_\xi$.
Moreover, for all $k \ge 0$, part (1) implies that
$\{f_t^{(k)}, f_1^{(k)}, f_2^{(k)}, \ldots \}$
are equicontinuous in $\C_\xi \cap B(0,r)$. Part (2)
trivially follows.
\end{proof}

Next we examine the Taylor expansions of $f_t$ and $f_n$ and
$\tilde{f}_n$ in neighbourhoods of $t$:
\begin{lem}
\label{lemTay}
Fix $\xi > 0$ as above,
$\{q_n\}_{n\ge1} \subset \R$ as in definition \ref{defmnpn}, and
$\{\xi_n\}_{n\ge1} \subset \R$ such that $|\xi_n q_n| \le \xi$ for
all $n$. Recall that $f_{t,n}'''(t) \to f_t'''(t) \neq 0$ (see lemma
\ref{lemftn'}), and let $u,r,v,s$ be the parameters in equations
(\ref{equnrnvnsn2}, \ref{equnrnvnsn3}).
Then, uniformly for $\alpha \in (-\pi,\pi]$:
\begin{align*}
(1) \hspace{.25cm}
f_t(t + \xi_n q_n e^{i \alpha})
&= f_t(t) + \tfrac13 \xi_n^3 e^{3 i \alpha} + O(|\xi_n|^3 |f_{t,n}'''(t) - f_t'''(t)| + |\xi_n|^4), \\
(2) \hspace{.2cm}
f_n(t + \xi_n q_n e^{i \alpha})
&= f_n(t) + n^{-\frac23} \xi_n e^{i \alpha} s
+ n^{-\frac13}  \xi_n^2 e^{2 i \alpha} v
+ \tfrac13 \xi_n^3 e^{3 i \alpha} \\
&+ O(n^{-1} |\xi_n| + n^{-\frac23} |\xi_n|^2 + n^{-\frac13} |\xi_n|^3 + |\xi_n|^4), \\
(3) \hspace{.2cm}
\tilde{f}_n(t + \xi_n q_n e^{i \alpha})
&= \tilde{f}_n(t) + n^{-\frac23} \xi_n e^{i \alpha} r
+ n^{-\frac13} \xi_n^2 e^{2 i \alpha} u
+ \tfrac13 \xi_n^3 e^{3 i \alpha} \\
&+ O(n^{-1} |\xi_n| + n^{-\frac23} |\xi_n|^2 + n^{-\frac13} |\xi_n|^3 + |\xi_n|^4),
\end{align*}
\end{lem}

\begin{proof}
Consider (1). First recall that $f_t'(t) = f_t''(t) = 0$ and $f_t'''(t) \neq 0$.
Taylor's theorem and part (1) of lemma \ref{lemfnftConv} then give,
\begin{equation*}
f_t(t + \xi_n q_n e^{i \alpha})
= f_t(t) + \tfrac16 \xi_n^3 q_n^3 e^{3 i \alpha} f_t'''(t) + O(|\xi_n q_n|^4),
\end{equation*}
uniformly for $\alpha \in (-\pi,\pi]$. Therefore,
\begin{equation*}
f_t(t + \xi_n q_n e^{i \alpha})
= f_t(t) + \tfrac16 \xi_n^3 q_n^3  e^{3 i \alpha} f_{t,n}'''(t)
+ O(|\xi_n q_n|^3 |f_t'''(t) - f_{t,n}'''(t)| + |\xi_n q_n|^4),
\end{equation*}
uniformly for $\alpha \in (-\pi,\pi]$. Finally, recall (see definition
\ref{defmnpn}) that $\frac16 q_n^3 f_{t,n}'''(t) = \frac13$,
and that $\{q_n\}_{n\ge1}$ is a convergent sequence with a
non-zero limit. This proves (1).

Consider (2). First note, Taylor's theorem and part (1)
of lemma \ref{lemfnftConv} give,
\begin{equation*}
f_n(t + \xi_n q_n e^{i \alpha})
= f_n(t) + \xi_n q_n e^{i \alpha} f_n'(t)
+ \tfrac12 \xi_n^2 q_n^2 e^{2 i \alpha} f_n''(t)
+ \tfrac16 \xi_n^3 q_n^3 e^{3 i \alpha} f_n'''(t)
+ O(|\xi_n q_n|^4),
\end{equation*}
uniformly for $\alpha \in (-\pi,\pi]$. Parts (1-3) of lemma
\ref{lemNonAsyRoots} then give,
\begin{align*}
f_n(t + \xi_n q_n e^{i \alpha})
&= f_n(t) + n^{-\frac23} \xi_n q_n e^{i \alpha} s q_{1,n}
+ \tfrac12 n^{-\frac13} \xi_n^2 q_n^2 e^{2 i \alpha} v q_{2,n}
+ \tfrac16 \xi_n^3 q_n^3 e^{3 i \alpha} f_{t,n}'''(t) \\
&+ O(n^{-1} |\xi_n q_n| + n^{-\frac23} |\xi_n q_n|^2
+ n^{-\frac13} |\xi_n q_n|^3 + |\xi_n q_n|^4).
\end{align*}
Part (2) then follows from definition
\ref{defmnpn}. (3) follows similarly.
\end{proof}

A useful corollary is the following:
\begin{cor}
\label{corTay}
Fix $\{q_n\}_{n\ge1} \subset \R$
as in definition \ref{defmnpn}. Also fix $c>0$ and $\theta \in (\frac14,\frac13)$.
Then, uniformly in the appropriate sets:
\begin{enumerate}
\item
$f_n(w) = f_n(t) + O(n^{-1})$ for
$w \in B(t, c n^{-\frac13})$.
\item
$\tilde{f}_n(z) = \tilde{f}_n(t) + O(n^{-1})$ for
$z \in B(t, c n^{-\frac13})$.
\item
$n f_n(t + n^{-\frac13} q_n w)
= n f_n(t) +  w s  +  w^2 v + \tfrac13 w^3 + O(n^{1-4\theta})$
for $w \in \text{cl}(B(0,n^{\frac13-\theta}))$.
\item
$n \tilde{f}_n(t + n^{-\frac13} q_n z)
= n \tilde{f}_n(t) + z r  + z^2 u + \tfrac13 z^3 + O(n^{1-4\theta})$
for $z \in \text{cl}(B(0,n^{\frac13-\theta}))$.
\end{enumerate}
\end{cor}

\begin{proof}
First recall (see definition \ref{defmnpn}) that $\{q_n\}_{n\ge1} \subset \R$
is a convergent sequence with a non-zero limit.
Then, (1) and (2) follow from parts (2,3) of lemma \ref{lemTay}
by choosing $\xi_n := c n^{-\frac13} |q_n|^{-1}$. Next, choose $\xi_n := n^{-\theta}$.
Part (2) of lemma \ref{lemTay} then gives,
\begin{align*}
n f_n(t + n^{-\theta} q_n e^{i \alpha})
&= n f_n(t) + n^{\frac13-\theta} e^{i \alpha} s
+ n^{\frac23-2\theta} e^{2 i \alpha} v
+ \tfrac13 n^{1-3\theta} e^{3 i \alpha} \\
&+ O(n^{-\theta} + n^{\frac13-2\theta} + n^{\frac23-3\theta} + n^{1-4\theta}),
\end{align*}
uniformly for $\alpha \in (-\pi,\pi]$. Finally recall that $\theta \in (\frac14,\frac13)$.
Therefore $-\theta < \frac13-2\theta < \frac23-3\theta < 1-4\theta < 0$,
and so $O(n^{-\theta} + n^{\frac13-2\theta} +
n^{\frac23-3\theta} + n^{1-4\theta}) = O(n^{1-4\theta})$. This proves (3).
Similarly, (4) follows from part (3) of lemma \ref{lemTay}.
\end{proof}

\subsection{Contours of descent/ascent}

The main results of this section, lemmas \ref{lemDesAsc1-12} and
\ref{lemDesAsc1-12Rem}, prove the existence of appropriate contours
of descent/ascent. These proofs are the most difficult part of the
paper, and will be given in section \ref{secCont}.

First we consider contours of steepest descent/ascent for $f_t$, $f_n$
and $\tilde{f}_n$. We do not define these rigorously, and refer the
interested reader to \cite{Mur84}, for example, for more information. We consider these for
$f_t$ and for $f_n$ for some fixed $n$, and state that $\tilde{f}_n$ can be
treated similarly to $f_n$. Note, irrespective of the choices of the branches
of the logarithms in equations (\ref{eqf2}, \ref{eqfn2}), that the
real-parts of $f_t$ and $f_n$ have unique continuous extensions
to $\C \setminus S$ and $\C \setminus S_n$ respectively, denoted by $R_t$
and $R_n$, and given by,
\begin{align}
\label{eqRt}
R_t(w)
&:= \int_{S_1} \log |w-x| \mu[dx]
- \int_{S_2} \log |w-x| (\l-\mu)[dx]
+ \int_{S_3} \log |w-x| \mu[dx], \\
\label{eqRn}
R_n(w)
&:= \frac1n \sum_{x \in S_{1,n}} \log |w-x|
-  \frac1n \sum_{x \in S_{2,n}} \log |w-x|
+  \frac1n \sum_{x \in S_{3,n}} \log |w-x|,
\end{align}
where $\log$ now represents natural logarithm. Then:
\begin{lem}
\label{lemDesAsc}
Fix $n \ge 1$ and $z \in \C \setminus S_n$, and let $D_n,A_n \subset \C \setminus S_n$
denote contours of steepest descent and ascent (respectively) for $f_n$ which pass
through $z$. Also, let $m_n := m_n(z)$ denote the multiplicity of $z$ as a root of
$f_n'$ (with the understanding that $m_n=0$ means $f_n'(z) \neq 0$), and let
$\alpha_n := \alpha_n(z) \in (-\pi,\pi]$ denote the principal value of the
argument of $f_n^{(m_n(w)+1)}(z)$. Then:
\begin{enumerate}
\item
$R_n$ strictly decreases along $D_n$, and strictly increases along $A_n$.
\item
The imaginary-part of $f_n$ is constant along both $D_n$ and $A_n$.
\item
There are $m_n+1$ possible directions for both $D_n$ and $A_n$ at $z$, given by
$((2i+1) \pi - \alpha_n)/(m_n+1)$ and $(2i \pi - \alpha_n)/(m_n+1)$
respectively for each $i \in \{0,1,\ldots,m_n\}$.
\item
$D_n$ is bounded, i.e., there exists a $C > 0$
for which $D_n \subset B(0,C)$.
\item
For all $x \in S_n$, there exists a $c(x) > 0$ for which
$D_n$ does not intersect $B(x,c(x))$ when $x \in S_{2,n}$, and
$A_n$ does not intersect $B(x,c(x))$ when $x \in S_{1,n} \cup S_{3,n}$.
\end{enumerate}
The equivalent objects for $f_t$, denoted $D_t, A_t, m_t, \alpha_t$,
also satisfy parts (1-4).
\end{lem}

\begin{proof}
Parts (1-3) follow from general considerations about contours of
steepest descent/ascent. Consider (4) for $D_n$. First recall that
$|S_{1,n}| - |S_{2,n}| + |S_{3,n}| > 0$ (indeed, equations
(\ref{eqS1S2S3In}, \ref{eqS1nS2nS3nIn}) give
$\tfrac1n |S_{1,n}| + \tfrac1n |S_{3,n}| - \tfrac1n |S_{2,n}| \to \eta > 0$).
Equation (\ref{eqRn}) thus gives
$R_n(w) \sim \frac1n (|S_{1,n}| - |S_{2,n}| + |S_{3,n}| ) \log |w|$
for all $w \in \C \setminus S_n$ with $|w|$ sufficiently large.
Then, letting $t_n \in \C \setminus S_n$ be the initial point of $D_n$,
there exists a $C>0$ for which $R_n(w) > R_n(t_n)$
for all $w \in \C \setminus S_n$ with $|w| \ge C$. Also, part (1)
gives $R_n(w) \le R_n(t_n)$ for all $w$ on $D_n$. Part (4) for
$D_n$ easily follows. Part (4) for $D_t$ follows similarly.

Consider (5). Fix $x \in S_n$. Equation (\ref{eqRn}) then gives
$R_n(w) \sim s(x) \frac1n \log |w - x|$
for all $w \in \C \setminus S_n$ with $|w-x|$ sufficiently small,
where $s(x) = -1$ whenever $x \in S_{2,n}$ and $s(x) = 1$ whenever
$x \in S_{1,n} \cup S_{3,n}$. Thus, whenever $x \in S_{2,n}$,
letting $t_n \in \C \setminus S_n$ be the initial point of $D_n$,
there exists a $c(x)>0$ for which $R_n(w) > R_n(t_n)$ for all
$w \in \C \setminus S_n$ with $|w-x| < c(x)$. Also,
part (1) gives $R_n(w) \le R_n(t_n)$ for all $w$ on $D_n$.
Part (5) for $D_n$ easily follows. Part (5) for
$A_n$ follows similarly.
\end{proof}

We now discuss natural extensions of the real and imaginary parts of $f_t$
and $f_n$, from $\mathbb{H} = \{w \in \C : \text{Im}(w) > 0 \}$,
to $\R \setminus S$ and $\R \setminus S_n$ respectively. Our
motivation is the following: In section \ref{secCont} we will be examining
contours of steepest descent/ascent which are contained in $\mathbb{H}$
except (possibly) for the end-points. Part (2) of the previous lemma thus
show that these extensions are natural to examine.
First note, irrespective of the choices of the branches of the
logarithms in equations (\ref{eqft}, \ref{eqfn2}), that
\begin{align}
\label{eqImt}
&\text{Im} (f_t(w)) \\
\nonumber
&= \int_{S_1} \text{Arg} (w-x) \mu[dx]
- \int_{S_2} \text{Arg} (w-x) (\l-\mu)[dx]
+ \int_{S_3} \text{Arg} (w-x) \mu[dx], \\
\nonumber
&\text{Im} (f_n(w)) \\
\nonumber
&= \frac1n \sum_{x \in S_{1,n}} \text{Arg} (w-x)
- \frac1n \sum_{x \in S_{2,n}} \text{Arg} (w-x)
+ \frac1n \sum_{x \in S_{3,n}} \text{Arg} (w-x),
\end{align}
for all $w \in \mathbb{H}$, where $\text{Arg}$ represents
the principal value of the argument. Note that these has unique
extensions from $\mathbb{H}$ to $\R$, denoted by
$I_t$ and $I_n$ respectively and given by,
\begin{align}
\label{eqIt}
I_t(s)
&:= \pi \mu[ \{x \in S_1 : x > s\} ]
- \pi (\l-\mu) [ \{x \in S_2 : x > s\} ]
+ \pi \mu[ \{x \in S_3 : x > s\} ], \\
\label{eqIn}
I_n(s)
&:= \frac{\pi}n  | \{ x \in S_{1,n} : x > s\} |
-  \frac{\pi}n  | \{ x \in S_{2,n} : x > s\} |
+  \frac{\pi}n  | \{ x \in S_{3,n} : x > s\} |,
\end{align}
for all $s \in \R$. Note, since $\mu \le \l$ (see assumption
\ref{assWeakConv}), that $I_t: \R \to \R$ is continuous. Also,
equations (\ref{eqS1S2S3In}, \ref{eqf'domain2}) imply that
$I_t: \R \to \R$ is constant in sub-intervals of $\R \setminus S = J \cup K$,
is strictly decreasing in the interior of $S_1$ and $S_3$,
and is strictly increasing in the interior of $S_2$. Similarly,
equation (\ref{eqfn'domain}) implies that $I_n : \R \to \R$, is
constant in sub-intervals of $\R \setminus S_n = J_n \cup K_n$. Finally note that
each discrete element of $S_n$ acts as a point of discontinuity for
$I_n$: $I_n$ decreases by $\frac{\pi}n$ at each point of $S_{1,n}$
and $S_{3,n}$, and increases by $\frac{\pi}n$ at each point of
$S_{2,n}$. These sets, and the above  extensions, are depicted in
figure \ref{figImfnExt}.

\begin{figure}[t]
\centering
\begin{tikzpicture}[scale=0.6]

\draw [dotted] (-1.5,10) --++ (1,0);
\draw (-.5,10) --++(.5,0);
\draw plot [smooth, tension=1] coordinates { (0,10) (3,8) (4,6)};
\draw (4,6) --++(2,0);
\draw plot [smooth, tension=1] coordinates { (6,6) (7.5,8) (10,10)};
\draw (10,10) --++ (2,0);
\draw plot [smooth, tension=1] coordinates { (12,10) (14.5,8) (16,6)};
\draw (16,6) --++ (.5,0);
\draw [dotted] (16.5,6) --++ (1,0);

\draw [dashed] (0,10) --++ (0,-2.5);
\draw (0,7.2) node {\scriptsize $\inf S_3$};
\draw [dashed] (4,6) --++ (0,-1.5);
\draw (4,4.2) node {\scriptsize $\sup S_3$};
\draw [dashed] (6,6) --++ (0,-1.5);
\draw (6,4.2) node {\scriptsize $\inf S_2$};
\draw [dashed] (10,10) --++ (0,-2.5);
\draw (10,7.2) node {\scriptsize $\sup S_2$};
\draw [dashed] (12,10) --++ (0,-2.5);
\draw (12,7.2) node {\scriptsize $\inf S_1$};
\draw [dashed] (16,6) --++ (0,-1.5);
\draw (16,4.2) node {\scriptsize $\sup S_1$};

\draw [thick,decorate,decoration={brace,amplitude=10pt,mirror},xshift=0.2pt,yshift=-0.2pt]
(-3,3.8) -- (-.1,3.8) node[black,midway,yshift=-0.6cm] {\scriptsize $J_2$};
\draw [thick,decorate,decoration={brace,amplitude=10pt,mirror},xshift=0.2pt,yshift=-0.2pt]
(.1,3.8) -- (3.9,3.8) node[black,midway,yshift=-0.6cm] {\scriptsize $\;$ non-increasing};
\draw [thick,decorate,decoration={brace,amplitude=10pt,mirror},xshift=0.4pt,yshift=-0.4pt]
(4,3.8) -- (6,3.8) node[black,midway,yshift=-0.6cm] {\scriptsize $J_4$};
\draw [thick,decorate,decoration={brace,amplitude=10pt,mirror},xshift=0.2pt,yshift=-0.2pt]
(6.1,3.8) -- (9.9,3.8) node[black,midway,yshift=-0.6cm] {\scriptsize non-decreasing};
\draw [thick,decorate,decoration={brace,amplitude=10pt,mirror},xshift=0.4pt,yshift=-0.4pt]
(10,3.8) -- (12,3.8) node[black,midway,yshift=-0.6cm] {\scriptsize $J_3$};
\draw [thick,decorate,decoration={brace,amplitude=10pt,mirror},xshift=0.2pt,yshift=-0.2pt]
(12.1,3.8) -- (15.9,3.8) node[black,midway,yshift=-0.6cm] {\scriptsize non-increasing $\;$};
\draw [thick,decorate,decoration={brace,amplitude=10pt,mirror},xshift=0.2pt,yshift=-0.2pt]
(16.1,3.8) -- (18.9,3.8) node[black,midway,yshift=-0.6cm] {\scriptsize $J_1$};

\draw (1,10.4) node {\scriptsize $\pi \eta = \pi (\mu[S_1] - (\l-\mu)[S_2] + \mu[S_3])$};
\draw (5,9) node {\scriptsize $\pi (\mu[S_1] - (\l-\mu)[S_2])$};
\draw[arrows=->,line width=.5pt](5,8.8)--(5,6);
\draw (11,10.4) node {\scriptsize $\pi \mu[S_1]$};
\draw (16.75,6.3) node {\scriptsize $0$};

\draw [dotted] (-1.5,0) --++ (1,0);
\draw (-.5,0) --++(.5,0);
\draw [dashed] (0,0) --++ (0,-2.4);
\draw (0,-1) --++(1,0);
\draw [dashed] (1,-1) --++ (0,-2.4);
\draw (1,-2) --++(1,0);
\draw [dashed] (2,-2) --++ (0,-2.4);
\draw [dotted] (2,-2) --++ (1,-1);
\draw [dashed] (3,-3) --++ (0,-1.4);
\draw (3,-3) --++(1,0);
\draw [dashed] (4,-3) --++ (0,-2.4);
\draw (4,-4) --++(2,0);
\draw [dashed] (6,-3) --++ (0,-2.4);
\draw (6,-3) --++(1,0);
\draw [dashed] (7,-3) --++ (0,-1.4);
\draw [dotted] (7,-3) --++ (1,1);
\draw [dashed] (8,-2) --++ (0,-1.4);
\draw (8,-2) --++ (1,0);
\draw [dashed] (9,-1) --++ (0,-1.4);
\draw (9,-1) --++ (1,0);
\draw [dashed] (10,0) --++ (0,-2.4);
\draw (10,0) --++ (2,0);
\draw [dashed] (12,0) --++ (0,-2.4);
\draw (12,-1) --++ (1,0);
\draw [dashed] (13,-1) --++ (0,-1.4);
\draw [dotted] (13,-1) --++ (1,-1);
\draw [dashed] (14,-2) --++ (0,-1.4);
\draw (14,-2) --++ (1,0);
\draw [dashed] (15,-2) --++ (0,-2.4);
\draw (15,-3) --++ (1,0);
\draw [dashed] (16,-3) --++ (0,-2.4);
\draw (15,-3) --++ (1,0);
\draw (16,-4) --++ (.5,0);
\draw [dotted] (16.5,-4) --++ (1,0);

\draw (-1.2,-2.8) node {\scriptsize $\min S_{3,n} = \frac1n x_n^{(n)}$};
\draw (1,-3.8) node {\scriptsize $\frac1n x_{n-1}^{(n)}$};
\draw (2,-4.8) node {\scriptsize $\frac1n x_{n-2}^{(n)}$};
\draw (3.9,-5.8) node {\scriptsize $\max S_{3,n} \;$};
\draw (6.1,-5.8) node {\scriptsize $\; \min S_{2,n}$};
\draw (9.9,-2.8) node {\scriptsize $\max S_{2,n} \;$};
\draw (12.1,-2.8) node {\scriptsize $\; \min S_{1,n}$};
\draw (14,-3.8) node {\scriptsize $\frac1n x_3^{(n)}$};
\draw (15,-4.8) node {\scriptsize $\frac1n x_2^{(n)}$};
\draw (17.1,-5.8) node {\scriptsize $\frac1n x_1^{(n)} = \max S_{1,n}$};

\draw [thick,decorate,decoration={brace,amplitude=10pt,mirror},xshift=0.2pt,yshift=-0.2pt]
(-3,-6.2) -- (0,-6.2) node[black,midway,yshift=-0.6cm] {\scriptsize $J_{2,n}$};
\draw [thick,decorate,decoration={brace,amplitude=10pt,mirror},xshift=0.4pt,yshift=-0.4pt]
(4,-6.2) -- (6,-6.2) node[black,midway,yshift=-0.6cm] {\scriptsize $J_{4,n}$};
\draw [thick,decorate,decoration={brace,amplitude=10pt,mirror},xshift=0.4pt,yshift=-0.4pt]
(10,-6.2) -- (12,-6.2) node[black,midway,yshift=-0.6cm] {\scriptsize $J_{3,n}$};
\draw [thick,decorate,decoration={brace,amplitude=10pt,mirror},xshift=0.2pt,yshift=-0.2pt]
(16,-6.2) -- (19,-6.2) node[black,midway,yshift=-0.6cm] {\scriptsize $J_{1,n}$};

\draw (1,.4) node {\scriptsize $\frac{\pi}n (|S_{1,n}|-|S_{2,n}|+|S_{3,n}|)$};
\draw (5,-2) node {\scriptsize $\frac{\pi}n (|S_{1,n}|-|S_{2,n}|)$};
\draw[arrows=->,line width=.5pt](5,-2.2)--(5,-4);
\draw (11,.4) node {\scriptsize $\frac{\pi}n |S_{1,n}|$};
\draw (16.75,-3.7) node {\scriptsize $0$};

\end{tikzpicture}
\caption{The functions given in equations (\ref{eqIt}, \ref{eqIn}),
with $I_t$ on the top and $I_n$ on the bottom.
The identity $\eta = \mu[S_1] - (\l-\mu)[S_2] + \mu[S_3]$
is given in equation (\ref{eqS1S2S3In}). All jumps in $I_n$ are of
size $\frac{\pi}n$.}
\label{figImfnExt}
\end{figure}
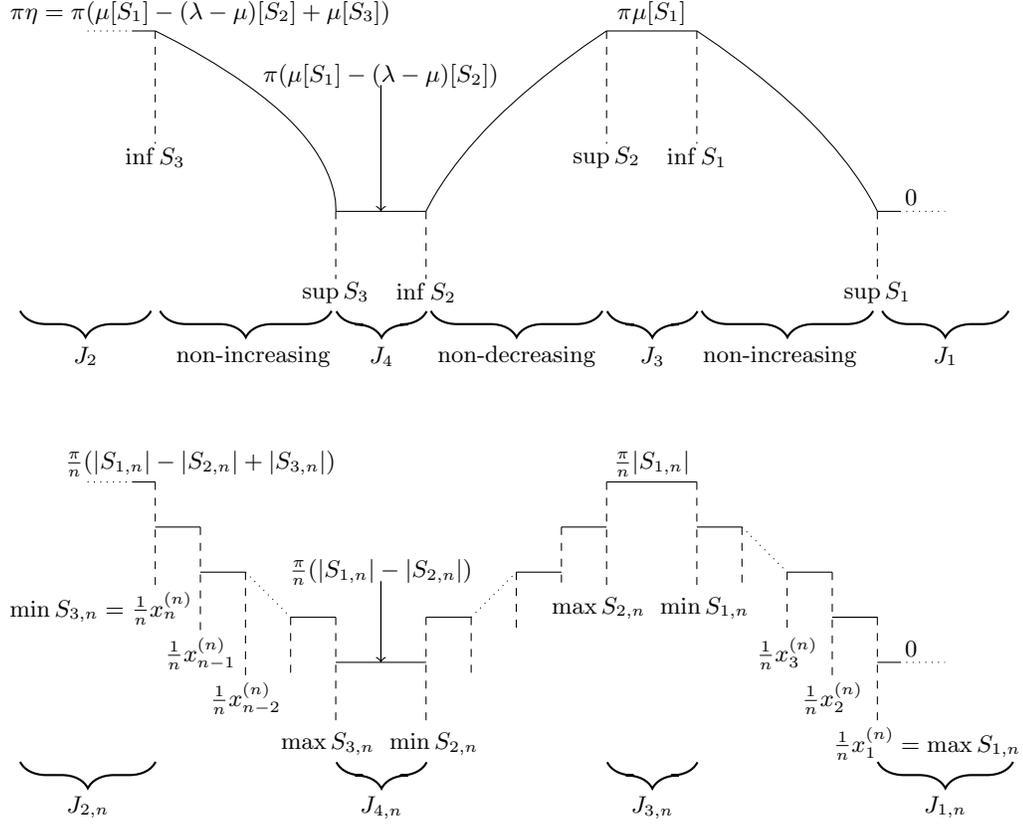

Next note, equations (\ref{eqRt}, \ref{eqRn}) give,
\begin{align*}
R_t(s)
&= \int_{S_1} \log |s-x| \mu[dx]
- \int_{S_2} \log |s-x| (\l-\mu)[dx]
+ \int_{S_3} \log |s-x| \mu[dx], \\
R_n(s)
&= \frac1n \sum_{x \in S_{1,n}} \log |s-x|
- \frac1n \sum_{x \in S_{2,n}} \log |s-x|
+ \frac1n \sum_{x \in S_{3,n}} \log |s-x|,
\end{align*}
for all $s \in \R \setminus S$ and  $s \in \R \setminus S_n$ respectively.
Thus the restrictions are real-valued, and we can regard them as functions from
$\R \setminus S = J \cup K$ and $\R \setminus S_n = J_n \cup K_n$ (respectively)
to $\R$. Moreover, note that,
\begin{equation}
\label{eqRnRes}
(R_t |_{\R \setminus S})' = (f_t') |_{\R \setminus S}
\hspace{0.5cm} \text{and} \hspace{0.5cm}
(R_n |_{\R \setminus S_n})' = (f_n') |_{\R \setminus S_n},
\end{equation}
and similarly for the higher order derivatives. Above, the functions on the
LHSs are the `real-derivative' of the real-valued restrictions,
and the functions on the RHSs are those given in equations
(\ref{eqft'2}, \ref{eqfn'}) (respectively) restricted to $\R \setminus S$
and to $\R \setminus S_n$.

We now state the main results of this section,
which will be proven in section \ref{secCont}. Recall that $L_t$ is
the largest open sub-interval of $\R \setminus S $ which contains $t$
(see equation (\ref{eqIntervalt})).
Also recall that lemma \ref{lemCases} splits the conditions of
theorem \ref{thmAiry} into 12 exhaustive cases. These lemmas prove
the existence of appropriate contours of descent/ascent for each case:
\begin{lem}
\label{lemDesAsc1-12}
Fix $\theta \in (\frac14,\frac13)$, and $\xi>0$ sufficiently small such
that equations (\ref{eqxi}, \ref{eqxi1}, \ref{eqxi4}, \ref{eqxi5}) are satisfied.
Then $(t-4\xi,t+4\xi) \subset L_t$, and $B(t,n^{-\theta} |q_n|) \subset B(t,\xi)$,
where $\{q_n\}_{n\ge1} \subset \R$ is given in definition \ref{defmnpn}.
Moreover, in each of the cases of lemma \ref{lemCases}, there exists simply
contours as shown in figure \ref{figDesAsc1-12} with the following properties:
\begin{enumerate}
\item
$\g_{1,n}^+$ and $\G_{1,n}^+$ both start at $t$, end in the interior of the intervals
shown in figure \ref{figDesAsc1-12}, are otherwise contained in $\mathbb{H}$, do not
intersect except at $t$, and are independent of $n$ outside $\text{cl}(B(t,\xi))$. $\g_{2,n}^+$
and $\G_{2,n}^+$ either start in $(t+\xi, \sup L_t)$ or in $(\inf L_t, t-\xi)$ as shown in
figure \ref{figDesAsc1-12}, end in the interior of the intervals shown, are otherwise
contained in $\mathbb{H}$, do not intersect $\g_{1,n}^+$ or $\G_{1,n}^+$ or
$\text{cl}(B(t,\xi))$, and are independent of $n$ everywhere.
\item
$\g_{1,n}^+ \cap B(t,n^{-\theta} |q_n|)$ and $\G_{1,n}^+ \cap B(t,n^{-\theta} |q_n|)$
are straight lines from $t$ to points $d_{1,n} \in \partial B(t,n^{-\theta} |q_n|)$ and
$\tilde{a}_{1,n} \in \partial B(t,n^{-\theta} |q_n|)$ (respectively) for which,
letting $\text{Arg}(\cdot)$ be the principal value of the argument,
$\text{Arg} (d_{1,n} - t) = \frac\pi3 + O(n^{-\frac13+\theta})$ and
$\text{Arg} (\tilde{a}_{1,n} - t) = \frac{2\pi}3 + O(n^{-\frac13+\theta})$
in cases (1,2,7,8,9,10), and
$\text{Arg} (d_{1,n} - t) = \frac{2\pi}3 + O(n^{-\frac13+\theta})$ and
$\text{Arg} (\tilde{a}_{1,n} - t) = \frac\pi3 + O(n^{-\frac13+\theta})$
in cases (3,4,5,6,11,12).
\item
$\text{Re}(f_n(w)) \le \text{Re}(f_n(d_{1,n}))$ for all
$w \in \g_{1,n}^+ \setminus B(t,n^{-\theta} |q_n|)$ and $w \in \g_{2,n}^+$.
\item
$\text{Re}(\tilde{f}_n(z)) \ge \text{Re}(\tilde{f}_n(\tilde{a}_{1,n}))$
for all $z \in \G_{1,n}^+ \setminus B(t,n^{-\theta} |q_n|)$ and
$z \in \G_{2,n}^+$.
\item
$|w-z|^{-1} = O(n^\theta)$ uniformly for
$w \in \g_{1,n}^+$ and $w \in \g_{2,n}^+$, and uniformly for
$z \in \G_{1,n}^+ \setminus B(t,n^{-\theta} |q_n|)$
and $z \in \G_{2,n}^+$. $|w-z|^{-1} = O(n^\theta)$ uniformly for
$w \in \g_{1,n}^+ \setminus B(t,n^{-\theta} |q_n|)$ and $w \in \g_{2,n}^+$,
and uniformly for $z \in \G_{1,n}^+$ and $z \in \G_{2,n}^+$.
\item
$|\g_{1,n}^+| = O(1)$ and $|\G_{1,n}^+| = O(1)$, where $|\cdot|$ represents length.
\end{enumerate}
\end{lem}

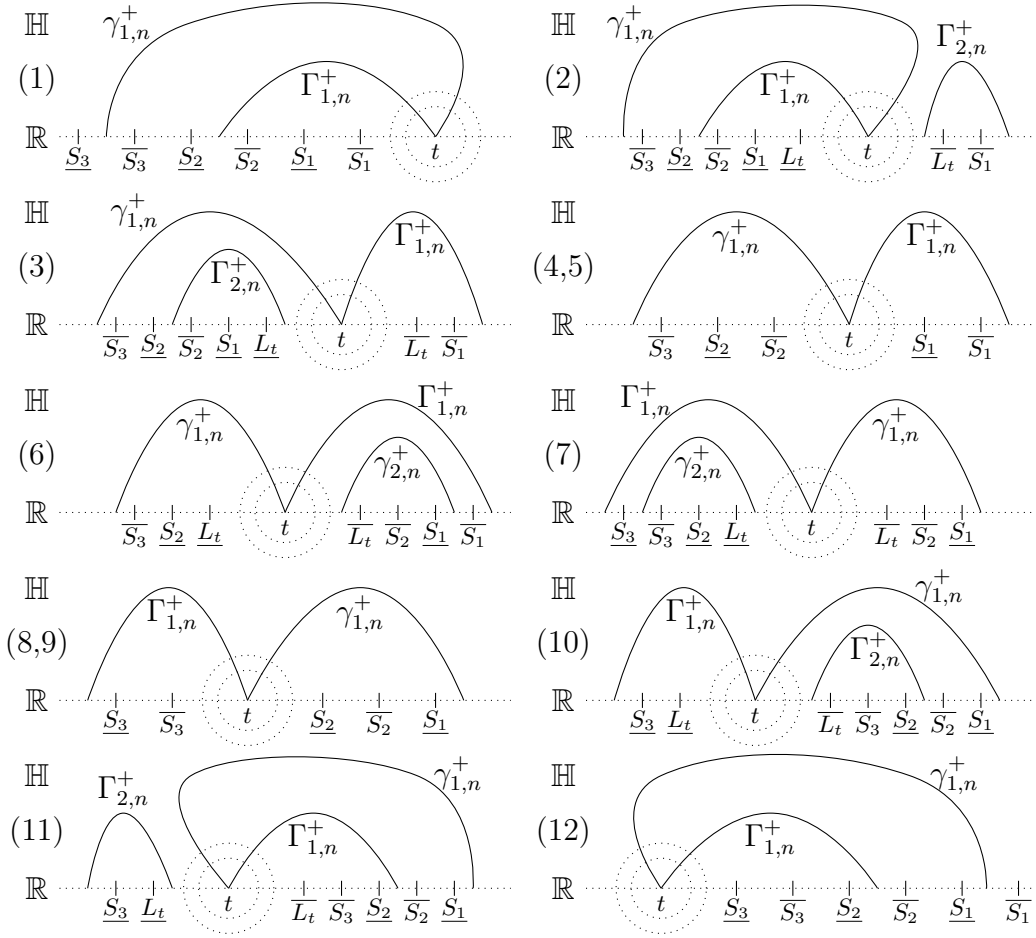
\begin{figure}[t]
\centering
\begin{tikzpicture};

\draw [dotted] (0,10) --++(6,0);
\draw (-.3,11.5) node {$\mathbb{H}$};
\draw (-.3,10.75) node {(1)};
\draw (-.3,10) node {$\R$};
\draw (.25,9.9) --++(0,.2);
\draw (.25,9.7) node {\scriptsize $\underline{S_3}$};
\draw (1,9.9) --++(0,.2);
\draw (1,9.7) node {\scriptsize $\overline{S_3}$};
\draw (1.75,9.9) --++(0,.2);
\draw (1.75,9.7) node {\scriptsize $\underline{S_2}$};
\draw (2.5,9.9) --++(0,.2);
\draw (2.5,9.7) node {\scriptsize $\overline{S_2}$};
\draw (3.25,9.9) --++(0,.2);
\draw (3.25,9.7) node {\scriptsize $\underline{S_1}$};
\draw (4,9.9) --++(0,.2);
\draw (4,9.7) node {\scriptsize $\overline{S_1}$};
\draw [dotted] (5,10) circle (.4cm);
\draw [dotted] (5,10) circle (.6cm);
\draw (5,9.8) node {\scriptsize $t$};

\draw plot [smooth, tension=.9] coordinates
{ (5,10) (5,11.5) (1.7,11.5) (.625,10) };
\draw (.9,11.5) node {$\g_{1,n}^+$};
\draw plot [smooth, tension=1] coordinates
{ (5,10) (3.55,11) (2.125,10) };
\draw (3.55,10.65) node {$\G_{1,n}^+$};

\draw [dotted] (7,10) --++(6,0);
\draw (6.7,11.5) node {$\mathbb{H}$};
\draw (6.7,10.75) node {(2)};
\draw (6.7,10) node {$\R$};
\draw (7.75,9.9) --++(0,.2);
\draw (7.75,9.7) node {\scriptsize $\overline{S_3}$};
\draw (8.25,9.9) --++(0,.2);
\draw (8.25,9.7) node {\scriptsize $\underline{S_2}$};
\draw (8.75,9.9) --++(0,.2);
\draw (8.75,9.7) node {\scriptsize $\overline{S_2}$};
\draw (9.25,9.9) --++(0,.2);
\draw (9.25,9.7) node {\scriptsize $\underline{S_1}$};
\draw (9.85,9.9) --++(0,.2);
\draw (9.75,9.7) node {\scriptsize $\underline{L_t}$};
\draw [dotted] (10.75,10) circle (.4cm);
\draw [dotted] (10.75,10) circle (.6cm);
\draw (10.75,9.8) node {\scriptsize $t$};
\draw (11.75,9.9) --++(0,.2);
\draw (11.75,9.7) node {\scriptsize $\overline{L_t}$};
\draw (12.25,9.9) --++(0,.2);
\draw (12.25,9.7) node {\scriptsize $\overline{S_1}$};

\draw plot [smooth, tension=.8] coordinates
{ (10.75,10) (11.25,11.5) (8.25,11.5) (7.5,10) };
\draw (7.6,11.5) node {$\g_{1,n}^+$};
\draw plot [smooth, tension=1] coordinates
{ (10.75,10) (9.65,11) (8.5,10) };
\draw (9.65,10.65) node {$\G_{1,n}^+$};
\draw plot [smooth, tension=1] coordinates
{ (11.5,10) (12,11) (12.625,10) };
\draw (12,11.3) node {$\G_{2,n}^+$};

\draw [dotted] (0,7.5) --++(6,0);
\draw (-.3,9) node {$\mathbb{H}$};
\draw (-.3,8.25) node {(3)};
\draw (-.3,7.5) node {$\R$};
\draw (.75,7.4) --++(0,.2);
\draw (.75,7.2) node {\scriptsize $\overline{S_3}$};
\draw (1.25,7.4) --++(0,.2);
\draw (1.25,7.2) node {\scriptsize $\underline{S_2}$};
\draw (1.75,7.4) --++(0,.2);
\draw (1.75,7.2) node {\scriptsize $\overline{S_2}$};
\draw (2.25,7.4) --++(0,.2);
\draw (2.25,7.2) node {\scriptsize $\underline{S_1}$};
\draw (2.75,7.4) --++(0,.2);
\draw (2.75,7.2) node {\scriptsize $\underline{L_t}$};
\draw [dotted] (3.75,7.5) circle (.4cm);
\draw [dotted] (3.75,7.5) circle (.6cm);
\draw (3.75,7.3) node {\scriptsize $t$};
\draw (4.75,7.4) --++(0,.2);
\draw (4.75,7.2) node {\scriptsize $\overline{L_t}$};
\draw (5.25,7.4) --++(0,.2);
\draw (5.25,7.2) node {\scriptsize $\overline{S_1}$};

\draw plot [smooth, tension=1] coordinates
{ (3.75,7.5) (2,9) (.5,7.5) };
\draw (1,9) node {$\g_{1,n}^+$};
\draw plot [smooth, tension=1] coordinates
{ (3.75,7.5) (4.7,9) (5.625,7.5) };
\draw (4.8,8.65) node {$\G_{1,n}^+$};
\draw plot [smooth, tension=1] coordinates
{ (3,7.5) (2.25,8.5) (1.5,7.5) };
\draw (2.35,8.15) node {$\G_{2,n}^+$};

\draw [dotted] (7,7.5) --++(6,0);
\draw (6.7,9) node {$\mathbb{H}$};
\draw (6.7,8.25) node {(4,5)};
\draw (6.7,7.5) node {$\R$};
\draw (8,7.4) --++(0,.2);
\draw (8,7.2) node {\scriptsize $\overline{S_3}$};
\draw (8.75,7.4) --++(0,.2);
\draw (8.75,7.2) node {\scriptsize $\underline{S_2}$};
\draw (9.5,7.4) --++(0,.2);
\draw (9.5,7.2) node {\scriptsize $\overline{S_2}$};
\draw [dotted] (10.5,7.5) circle (.4cm);
\draw [dotted] (10.5,7.5) circle (.6cm);
\draw (10.5,7.3) node {\scriptsize $t$};
\draw (11.5,7.4) --++(0,.2);
\draw (11.5,7.2) node {\scriptsize $\underline{S_1}$};
\draw (12.25,7.4) --++(0,.2);
\draw (12.25,7.2) node {\scriptsize $\overline{S_1}$};

\draw plot [smooth, tension=1] coordinates
{ (10.5,7.5) (9,9) (7.625,7.5) };
\draw (9,8.65) node {$\g_{1,n}^+$};
\draw plot [smooth, tension=1] coordinates
{ (10.5,7.5) (11.5,9) (12.625,7.5) };
\draw (11.6,8.65) node {$\G_{1,n}^+$};

\draw [dotted] (0,5) --++(6,0);
\draw (-.3,6.5) node {$\mathbb{H}$};
\draw (-.3,5.75) node {(6)};
\draw (-.3,5) node {$\R$};
\draw (1,4.9) --++(0,.2);
\draw (1,4.7) node {\scriptsize $\overline{S_3}$};
\draw (1.5,4.9) --++(0,.2);
\draw (1.5,4.7) node {\scriptsize $\underline{S_2}$};
\draw (2,4.9) --++(0,.2);
\draw (2,4.7) node {\scriptsize $\underline{L_t}$};
\draw [dotted] (3,5) circle (.4cm);
\draw [dotted] (3,5) circle (.6cm);
\draw (3,4.8) node {\scriptsize $t$};
\draw (4,4.9) --++(0,.2);
\draw (4,4.7) node {\scriptsize $\overline{L_t}$};
\draw (4.5,4.9) --++(0,.2);
\draw (4.5,4.7) node {\scriptsize $\overline{S_2}$};
\draw (5,4.9) --++(0,.2);
\draw (5,4.7) node {\scriptsize $\underline{S_1}$};
\draw (5.5,4.9) --++(0,.2);
\draw (5.5,4.7) node {\scriptsize $\overline{S_1}$};

\draw plot [smooth, tension=1] coordinates
{ (3,5) (1.875,6.5) (.75,5) };
\draw (1.875,6.15) node {$\g_{1,n}^+$};
\draw plot [smooth, tension=1] coordinates
{ (3,5) (4.375,6.5) (5.75,5) };
\draw (5.1,6.5) node {$\G_{1,n}^+$};
\draw plot [smooth, tension=1] coordinates
{ (3.75,5) (4.5,6) (5.25,5) };
\draw (4.5,5.65) node {$\g_{2,n}^+$};

\draw [dotted] (7,5) --++(6,0);
\draw (6.7,6.5) node {$\mathbb{H}$};
\draw (6.7,5.75) node {(7)};
\draw (6.7,5) node {$\R$};
\draw (7.5,4.9) --++(0,.2);
\draw (7.5,4.7) node {\scriptsize $\underline{S_3}$};
\draw (8,4.9) --++(0,.2);
\draw (8,4.7) node {\scriptsize $\overline{S_3}$};
\draw (8.5,4.9) --++(0,.2);
\draw (8.5,4.7) node {\scriptsize $\underline{S_2}$};
\draw (9,4.9) --++(0,.2);
\draw (9,4.7) node {\scriptsize $\underline{L_t}$};
\draw [dotted] (10,5) circle (.4cm);
\draw [dotted] (10,5) circle (.6cm);
\draw (10,4.8) node {\scriptsize $t$};
\draw (11,4.9) --++(0,.2);
\draw (11,4.7) node {\scriptsize $\overline{L_t}$};
\draw (11.5,4.9) --++(0,.2);
\draw (11.5,4.7) node {\scriptsize $\overline{S_2}$};
\draw (12,4.9) --++(0,.2);
\draw (12,4.7) node {\scriptsize $\underline{S_1}$};

\draw plot [smooth, tension=1] coordinates
{ (10,5) (8.625,6.5) (7.25,5) };
\draw (7.8,6.5) node {$\G_{1,n}^+$};
\draw plot [smooth, tension=1] coordinates
{ (10,5) (11.125,6.5) (12.25,5) };
\draw (11.125,6.15) node {$\g_{1,n}^+$};
\draw plot [smooth, tension=1] coordinates
{ (9.25,5) (8.5,6) (7.75,5) };
\draw  (8.5,5.65) node {$\g_{2,n}^+$};

\draw [dotted] (0,2.5) --++(6,0);
\draw (-.3,4) node {$\mathbb{H}$};
\draw (-.3,3.25) node {(8,9)};
\draw (-.3,2.5) node {$\R$};
\draw (.75,2.4) --++(0,.2);
\draw (.75,2.2) node {\scriptsize $\underline{S_3}$};
\draw (1.5,2.4) --++(0,.2);
\draw (1.5,2.2) node {\scriptsize $\overline{S_3}$};
\draw [dotted] (2.5,2.5) circle (.4cm);
\draw [dotted] (2.5,2.5) circle (.6cm);
\draw (2.5,2.3) node {\scriptsize $t$};
\draw (3.5,2.4) --++(0,.2);
\draw (3.5,2.2) node {\scriptsize $\underline{S_2}$};
\draw (4.25,2.4) --++(0,.2);
\draw (4.25,2.2) node {\scriptsize $\overline{S_2}$};
\draw (5,2.4) --++(0,.2);
\draw (5,2.2) node {\scriptsize $\underline{S_1}$};

\draw plot [smooth, tension=1] coordinates
{ (2.5,2.5) (1.45,4) (.375,2.5) };
\draw (1.5,3.65) node {$\G_{1,n}^+$};
\draw plot [smooth, tension=1] coordinates
{ (2.5,2.5) (4,4) (5.375,2.5) };
\draw (4,3.65) node {$\g_{1,n}^+$};

\draw [dotted] (7,2.5) --++(6,0);
\draw (6.7,4) node {$\mathbb{H}$};
\draw (6.7,3.25) node {(10)};
\draw (6.7,2.5) node {$\R$};
\draw (7.75,2.4) --++(0,.2);
\draw (7.75,2.2) node {\scriptsize $\underline{S_3}$};
\draw (8.25,2.4) --++(0,.2);
\draw (8.25,2.2) node {\scriptsize $\underline{L_t}$};
\draw [dotted] (9.25,2.5) circle (.4cm);
\draw [dotted] (9.25,2.5) circle (.6cm);
\draw (9.25,2.3) node {\scriptsize $t$};
\draw (10.25,2.4) --++(0,.2);
\draw (10.25,2.2) node {\scriptsize $\overline{L_t}$};
\draw (10.75,2.4) --++(0,.2);
\draw (10.75,2.2) node {\scriptsize $\overline{S_3}$};
\draw (11.25,2.4) --++(0,.2);
\draw (11.25,2.2) node {\scriptsize $\underline{S_2}$};
\draw (11.75,2.4) --++(0,.2);
\draw (11.75,2.2) node {\scriptsize $\overline{S_2}$};
\draw (12.25,2.4) --++(0,.2);
\draw (12.25,2.2) node {\scriptsize $\underline{S_1}$};

\draw plot [smooth, tension=1] coordinates
{ (9.25,2.5) (10.875,4) (12.5,2.5) };
\draw (11.7,4) node {$\g_{1,n}^+$};
\draw plot [smooth, tension=1] coordinates
{ (9.25,2.5) (8.3,4) (7.375,2.5) };
\draw (8.4,3.65) node {$\G_{1,n}^+$};
\draw plot [smooth, tension=1] coordinates
{ (10,2.5) (10.75,3.5) (11.5,2.5) };
\draw (10.85,3.15) node {$\G_{2,n}^+$};

\draw [dotted] (0,0) --++(6,0);
\draw (-.3,1.5) node {$\mathbb{H}$};
\draw (-.3,.75) node {(11)};
\draw (-.3,0) node {$\R$};
\draw (.75,-.1) --++(0,.2);
\draw (.75,-.3) node {\scriptsize $\underline{S_3}$};
\draw (1.25,-.1) --++(0,.2);
\draw (1.25,-.3) node {\scriptsize $\underline{L_t}$};
\draw [dotted] (2.25,0) circle (.4cm);
\draw [dotted] (2.25,0) circle (.6cm);
\draw (2.25,-.2) node {\scriptsize $t$};
\draw (3.25,-.1) --++(0,.2);
\draw (3.25,-.3) node {\scriptsize $\overline{L_t}$};
\draw (3.75,-.1) --++(0,.2);
\draw (3.75,-.3) node {\scriptsize $\overline{S_3}$};
\draw (4.25,-.1) --++(0,.2);
\draw (4.25,-.3) node {\scriptsize $\underline{S_2}$};
\draw (4.75,-.1) --++(0,.2);
\draw (4.75,-.3) node {\scriptsize $\overline{S_2}$};
\draw (5.25,-.1) --++(0,.2);
\draw (5.25,-.3) node {\scriptsize $\underline{S_1}$};

\draw plot [smooth, tension=.8] coordinates
{ (2.25,0) (1.75,1.5) (4.75,1.5) (5.5,0) };
\draw (5.3,1.5) node {$\g_{1,n}^+$};
\draw plot [smooth, tension=1] coordinates
{ (2.25,0) (3.375,1) (4.5,0) };
\draw (3.375,.65) node {$\G_{1,n}^+$};
\draw plot [smooth, tension=1] coordinates
{ (1.5,0) (.85,1) (.375,0) };
\draw (.85,1.3) node {$\G_{2,n}^+$};

\draw [dotted] (7,0) --++(6,0);
\draw (6.7,1.5) node {$\mathbb{H}$};
\draw (6.7,.75) node {(12)};
\draw (6.7,0) node {$\R$};
\draw [dotted] (8,0) circle (.4cm);
\draw [dotted] (8,0) circle (.6cm);
\draw (8,-.2) node {\scriptsize $t$};
\draw (9,-.1) --++(0,.2);
\draw (9,-.3) node {\scriptsize $\underline{S_3}$};
\draw (9.75,-.1) --++(0,.2);
\draw (9.75,-.3) node {\scriptsize $\overline{S_3}$};
\draw (10.5,-.1) --++(0,.2);
\draw (10.5,-.3) node {\scriptsize $\underline{S_2}$};
\draw (11.25,-.1) --++(0,.2);
\draw (11.25,-.3) node {\scriptsize $\overline{S_2}$};
\draw (12,-.1) --++(0,.2);
\draw (12,-.3) node {\scriptsize $\underline{S_1}$};
\draw (12.75,-.1) --++(0,.2);
\draw (12.75,-.3) node {\scriptsize $\overline{S_1}$};

\draw plot [smooth, tension=.9] coordinates
{ (8,0) (8,1.5) (11.3,1.5) (12.325,0) };
\draw (11.9,1.5) node {$\g_{1,n}^+$};
\draw plot [smooth, tension=1] coordinates
{ (8,0) (9.45,1) (10.875,0) };
\draw (9.45,.65) node {$\G_{1,n}^+$};

\end{tikzpicture}
\caption{The contours described in lemma
\ref{lemDesAsc1-12}, for the exhaustive cases (1-12) of
lemma \ref{lemCases}. Above, $\underline{S_1} := \inf S_1$,
$\overline{S_1} := \sup S_1$, etc.
The smaller circles represent $B(t,n^{-\theta} |q_n|)$, the larger
circles represent $B(t,\xi)$.}
\label{figDesAsc1-12}
\end{figure}

\begin{lem}
\label{lemDesAsc1-12Rem}
Define $F_n$ as in equation (\ref{eqFnw}). Assume that $v > u$. Fix
$\theta \in (\frac14,\frac13)$, and $\xi>0$ sufficiently small such
that equations (\ref{eqxi}, \ref{eqxi1}, \ref{eqxi4}, \ref{eqxi5}) are
satisfied. Then, in each of the cases of lemma \ref{lemCases}, there
exists a simple contour as shown in figure \ref{figDesAscRem} with the
following properties:
\begin{itemize}
\item
$\kappa_n^+$ starts at $t$, ends in the interior
of the intervals shown in figure \ref{figDesAscRem}, is otherwise contained
in $\mathbb{H}$, and is independent of $n$ outside $\text{cl}(B(t,\xi))$.
\item
$\kappa_n^+ \cap B(t,n^{-\theta} |q_n|)$ is a straight
line from $t$ to a point $D_{1,n} \in \partial B(t,n^{-\theta} |q_n|)$
for which $\text{Arg} (D_{1,n} - t) = \frac\pi2 + O(n^{-\frac13+\theta})$.
\item
$\text{Re}(F_n(w)) \le \text{Re}(F_n(D_{1,n}))$ for all
$w \in \kappa_n^+ \setminus B(t,n^{-\theta} |q_n|)$.
\item
$|\kappa_n^+| = O(1)$, where $|\cdot|$ represents length.
\end{itemize}
\end{lem}

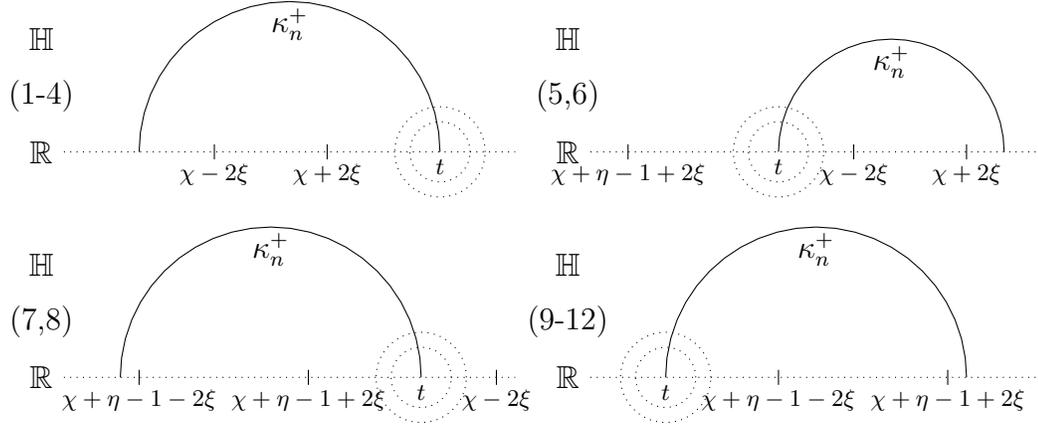
\begin{figure}[t]
\centering
\begin{tikzpicture};

\draw [dotted] (0,3) --++(6,0);
\draw (-.3,4.5) node {$\mathbb{H}$};
\draw (-.3,3.75) node {(1-4)};
\draw (-.3,3) node {$\R$};
\draw (2,2.9) --++(0,.2);
\draw (2,2.7) node {\scriptsize $\chi-2\xi$};
\draw (3.5,2.9) --++(0,.2);
\draw (3.5,2.7) node {\scriptsize $\chi+2\xi$};
\draw [dotted] (5,3) circle (.4cm);
\draw [dotted] (5,3) circle (.6cm);
\draw (5,2.8) node {\scriptsize $t$};

\draw [domain=0:180] plot ({3+2*(cos(\x))}, {3+2*(sin(\x))});
\draw (3,4.7) node {$\kappa_n^+$};

\draw [dotted] (7,3) --++(6,0);
\draw (6.7,4.5) node {$\mathbb{H}$};
\draw (6.7,3.75) node {(5,6)};
\draw (6.7,3) node {$\R$};
\draw (7.5,2.9) --++(0,.2);
\draw (7.5,2.7) node {\scriptsize $\chi+\eta-1 + 2\xi$};
\draw [dotted] (9.5,3) circle (.4cm);
\draw [dotted] (9.5,3) circle (.6cm);
\draw (9.5,2.8) node {\scriptsize $t$};
\draw (10.5,2.9) --++(0,.2);
\draw (10.5,2.7) node {\scriptsize $\chi-2\xi$};
\draw (12,2.9) --++(0,.2);
\draw (12,2.7) node {\scriptsize $\chi+2\xi$};

\draw [domain=0:180] plot ({11+(1.5)*(cos(\x))}, {3+(1.5)*(sin(\x))});
\draw (11,4.2) node {$\kappa_n^+$};

\draw [dotted] (0,0) --++(6,0);
\draw (-.3,1.5) node {$\mathbb{H}$};
\draw (-.3,.75) node {(7,8)};
\draw (-.3,0) node {$\R$};
\draw (5.75,-.1) --++(0,.2);
\draw (5.75,-.3) node {\scriptsize $\chi-2\xi$};
\draw [dotted] (4.75,0) circle (.4cm);
\draw [dotted] (4.75,0) circle (.6cm);
\draw (4.75,-.2) node {\scriptsize $t$};
\draw (1,-.1) --++(0,.2);
\draw (1,-.3) node {\scriptsize $\chi+\eta-1-2\xi$};
\draw (3.25,-.1) --++(0,.2);
\draw (3.25,-.3) node {\scriptsize $\chi+\eta-1+2\xi$};

\draw [domain=0:180] plot ({2.75-2*(cos(\x))}, {2*(sin(\x))});
\draw (2.75,1.7) node {$\kappa_n^+$};

\draw [dotted] (7,0) --++(6,0);
\draw (6.7,1.5) node {$\mathbb{H}$};
\draw (6.7,.75) node {(9-12)};
\draw (6.7,0) node {$\R$};
\draw (11.75,-.1) --++(0,.2);
\draw (11.75,-.3) node {\scriptsize $\chi+\eta-1+2\xi$};
\draw (9.5,-.1) --++(0,.2);
\draw (9.5,-.3) node {\scriptsize $\chi+\eta-1-2\xi$};
\draw [dotted] (8,0) circle (.4cm);
\draw [dotted] (8,0) circle (.6cm);
\draw (8,-.2) node {\scriptsize $t$};

\draw [domain=0:180] plot ({10+2*(cos(\x))}, {2*(sin(\x))});
\draw (10,1.7) node {$\kappa_n^+$};

\end{tikzpicture}
\caption{The contours described in lemma \ref{lemDesAsc1-12Rem} for
the exhaustive cases (1-12) of lemma \ref{lemCases}, when $v>u$. The
smaller circles represent $B(t,n^{-\theta} |q_n|)$, the larger
circles represent $B(t,\xi)$.}
\label{figDesAscRem}
\end{figure}

\subsection{Alternative contour integral expressions}

The main results of the previous section, lemmas \ref{lemDesAsc1-12}
and \ref{lemDesAsc1-12Rem}, prove the existence of appropriate contours
of descent/ascent for the cases, (1-12), of lemma \ref{lemCases}.
In this section, we will use these contours to find alternative contour integral
expressions for the correlation kernel than that given in equation (\ref{eqKnrnunsnvn1}).
These new expressions will allow us to perform a steepest descent analysis
for each case, (1-12). First, using lemmas \ref{lemDesAsc1-12} and
\ref{lemDesAsc1-12Rem},
we define:
\begin{definition}
\label{defConCases}
For cases (1-12) of lemma \ref{lemCases}, define
$\g_{1,n}$ to be the following simple closed contour with
counter-clockwise orientation:
$\g_{1,n} := \g_{1,n}^+ + \g_{1,n}^-$,
where $\g_{1,n}^-$ is the reflection of $\g_{1,n}^+$
in $\R$. Similarly define $\g_{2,n}, \G_{1,n}, \G_{2,n}$.
Similarly define $\kappa_n$ when $v>u$. 
Finally, define $\g_n := \g_{1,n}$ when $\g_{2,n}$ does not exist, and
$\g_n := \g_{1,n} + \g_{2,n}$ when $\g_{2,n}$ exists.
Similarly define $\Gamma_n$.
\end{definition}
The main result of this section is then:
\begin{lem}
\label{lemKnJnPhin}
Assume the conditions of theorem \ref{thmAiry}. Recall that one of the
cases, (1-12), of lemma \ref{lemCases} must be satisfied. Then,
\begin{equation}
\tag{1}
\beta_n \; K_n((u_n,r_n),(v_n,s_n))
= \left\{
\begin{array}{rcl}
J_n - \Phi_n & ; & \text{ for cases (1-4),} \\
J_n + \Phi_n & ; & \text{ for cases (5,6),} \\  
-J_n - \Phi_n & ; & \text{ for cases (7,8),} \\
-J_n + \Phi_n & ; & \text{ for cases (9-12),} 
\end{array}
\right. 
\end{equation}
where we define:
\begin{equation*}
\beta_n := \frac{(n-r_n-1)!}{(n-s_n)!} n^{r_n+1-s_n},
\end{equation*}
and
\begin{equation*}
J_n := \frac1{(2\pi i)^2} \int_{\g_n} dw \int_{\G_n} dz \;
\frac{\prod_{j=u_n+r_n-n+1}^{u_n-1}
(z - \frac{j}n)}{\prod_{j=v_n+s_n-n}^{v_n} (w - \frac{j}n)} \;
\frac1{w-z} \; \prod_{i=1}^n \left( \frac{w - \frac{x_i}n}{z - \frac{x_i}n} \right),
\end{equation*}
and
\begin{align*}
\Phi_n
&:= \left\{ \begin{array}{lll}
1_{(v_n \ge u_n, s_n > r_n)} \;
\frac{(v_n+s_n-u_n-r_n-1)!}{(s_n-r_n-1)! (v_n-u_n)!} \; \beta_n
& ; & \text{for cases (1-4)}, \\
1_{(v_n \ge u_n,  v_n+s_n \le u_n+r_n, s_n \le r_n)} \;
\frac{(-1)^{v_n-u_n} (r_n-s_n)!}{(v_n-u_n)! (u_n+r_n-v_n-s_n)!} \; \beta_n
& ; & \text{for cases (5,6)}, \\
1_{(v_n \ge u_n,  v_n+s_n \le u_n+r_n, s_n \le r_n)} \;
\frac{(-1)^{v_n-u_n-1} (r_n-s_n)!}{(v_n-u_n)! (u_n+r_n-v_n-s_n)!} \; \beta_n
& ; & \text{for cases (7,8)}, \\
1_{(v_n+s_n \le u_n+r_n, s_n > r_n)} \;
\frac{(-1)^{s_n-r_n-1} (u_n-v_n-1)!}{(s_n-r_n-1)! (u_n+r_n-v_n-s_n)!} \; \beta_n
& ; & \text{for cases (9-12)}.
\end{array} \right.
\end{align*}
Moreover, when $(u,r) \neq (v,s)$,
\begin{equation}
\tag{2}
\Phi_n = 1_{(v>u)} \frac1{2\pi i} \int_{\kappa_n} dw \;
\frac{\prod_{j=u_n+r_n-n+1}^{u_n-1} (w - \frac{j}n)}
{\prod_{j=v_n+s_n-n}^{v_n} (w - \frac{j}n)}.
\end{equation}
\end{lem}

We will prove the above using a number of sub-results.
First, we examine $J_n$ using the Residue theorem.
Note that $J_n$ can be written as follows:
\begin{equation}
\label{eqJ_n2}
J_n = \frac1{(2\pi i)^2} \int_{\g_n} dw \int_{\G_n} dz \;
\frac{\prod_{y \in U_n} (z - y)}{\prod_{x \in V_n} (w - x)} \;
\frac1{w-z} \; \frac{\prod_{x \in P_n} (w - x)}{\prod_{y \in P_n} (z - y)},
\end{equation}
where:
\begin{itemize}
\item
$P_n := \frac1n \{x_1^{(n)}, x_2^{(2)}, \ldots, x_n^{(n)}\}$ (as in equation (\ref{eqPnHn})).
\item
$U_n := \frac1n \{u_n+r_n-n+1, u_n+r_n-n+2, \ldots, u_n-1\}$ (as in equation (\ref{eqUnVn})).
\item
$V_n := \frac1n \{v_n+s_n-n, v_n+s_n-n+1, \ldots, v_n\}$ (as in equation (\ref{eqUnVn})).
\end{itemize}
Also note, the above definitions, and equations (\ref{eqfn2}, \ref{eqtildefn2}), give
\begin{align}
\label{eqPn-Un1}
\tilde{S}_{1,n} = \{y \in P_n : y \ge \tfrac{u_n}n \}
\hspace{.5cm} &\text{and} \hspace{.5cm}
\tilde{S}_{3,n} = \{y \in P_n : y \le \tfrac{u_n+r_n-n}n \}, \\
\label{eqPn-Un2}
P_n \setminus U_n = \tilde{S}_{1,n} \cup \tilde{S}_{3,n}
\hspace{.5cm} &\text{and} \hspace{.5cm}
V_n \setminus P_n = S_{2,n}.
\end{align}
Finally recall the decomposition,
$V_n \setminus U_n = (VU^{(n)}) \cup (VU_{(n)})$,
given in equation (\ref{eqVn-Un}). Then:
\begin{lem}
\label{lemJn}
Assume the conditions of theorem \ref{thmAiry}. Define $U_n$, $V_n$,
$P_n$, $VU^{(n)}$, $VU_{(n)}$ as above.
Let $A \in \{1,2,\ldots,12\}$ denote that case of lemma \ref{lemCases} which is
satisfied. Then, for cases (1-6),
\begin{align*}
J_n
&= \sum_{y' \in \tilde{S}_{1,n}} \sum_{x' \in S_{2,n}} \;
\frac{\prod_{y \in U_n} (y' - y)}{\prod_{x \in V_n \setminus x'} (x' - x)} \;
\frac{\prod_{x \in P_n \setminus y'} (x' - x)}{\prod_{y \in P_n \setminus y'} (y' - y)} \\
&+ 1_{(v_n \ge u_n)} \bigg( 1_{(A \in \{1,2,3\})}
\sum_{y' \in (VU^{(n)}) \cap P_n}
- 1_{(A=6)} \sum_{y' \in (VU^{(n)}) \setminus P_n} \bigg)
\frac{\prod_{y \in U_n} (y' - y)}{\prod_{x \in V_n \setminus y'} (y' - x)}.
\end{align*}
Moreover, for cases (7-12),
\begin{align*}
J_n
&= \sum_{y' \in \tilde{S}_{3,n}} \sum_{x' \in S_{2,n}} \;
\frac{\prod_{y \in U_n} (y' - y)}{\prod_{x \in V_n \setminus x'} (x' - x)} \;
\frac{\prod_{x \in P_n \setminus y'} (x' - x)}{\prod_{y \in P_n \setminus y'} (y' - y)} \\
&+ 1_{(v_n+s_n \le u_n+r_n)} \bigg( 1_{(A \in \{10,11,12\})}
\sum_{y' \in (VU_{(n)}) \cap P_n}
- 1_{(A=7)} \sum_{y' \in (VU_{(n)}) \setminus P_n} \bigg)
\frac{\prod_{y \in U_n} (y' - y)}{\prod_{x \in V_n \setminus y'} (y' - x)}.
\end{align*}
\end{lem}

\begin{proof}
In this lemma, we let $\g_n^\circ$ and $\G_n^\circ$ respectively denote the
interiors of $\g_n$ and $\G_n$. Note, in equation (\ref{eqJ_n2}), we perform
the $\G_n$ integral first, and so we consider the $w \in \g_n$ in the
integrand to be fixed. Also note that $w \not\in P_n$ since $\g_n$ can always
be chosen so that it does not intersect $P_n$ (see remark \ref{remgnPn}), and
that each element of $P_n$ is distinct. The integrand of equation
(\ref{eqJ_n2}) thus has a simple pole at each distinct element of
$P_n \setminus U_n = \tilde{S}_{1,n} \cup \tilde{S}_{3,n}$
(see equation (\ref{eqPn-Un2})), a simple pole or removable singularity at $w$,
and no other singularities. Therefore, the
Residue theorem gives,
\begin{align}
\nonumber
J_n
&= \frac1{2\pi i} \int_{\g_n} dw \sum_{y' \in \tilde{S}_{1,n} \cup \tilde{S}_{3,n}}
\; 1_{(y' \in \G_n^\circ)} \;
\frac{\prod_{y \in U_n} (y' - y)}{\prod_{x \in V_n} (w - x)} \;
\frac1{w-y'} \; \frac{\prod_{x \in P_n} (w - x)}{\prod_{y \in P_n \setminus y'} (y' - y)} \\
\nonumber
&- \frac1{2\pi i} \int_{\g_n} dw \; 1_{(w \in \G_n^\circ)} \;
\frac{\prod_{y \in U_n} (w - y)}{\prod_{x \in V_n} (w - x)} \;
\frac{\prod_{x \in P_n} (w - x)}{\prod_{y \in P_n} (w - y)} \\
\label{eqJ_n3}
&= \sum_{y' \in \tilde{S}_{1,n} \cup \tilde{S}_{3,n}}
\; 1_{(y' \in \G_n^\circ)} \; \frac1{2\pi i} \int_{\g_n} dw \;
\frac{\prod_{y \in U_n} (y' - y)}{\prod_{x \in V_n} (w - x)} \;
\frac{\prod_{x \in P_n \setminus y'} (w - x)}{\prod_{y \in P_n \setminus y'} (y' - y)} \\
\nonumber
&- \frac1{2\pi i} \int_{\g_n} dw \; 1_{(w \in \G_n^\circ)} \;
\frac{\prod_{y \in U_n} (w - y)}{\prod_{x \in V_n} (w - x)}.
\end{align}

Consider the first term on the RHS, above. Recall that $y' \in P_n$
is fixed, and (see equation (\ref{eqPn-Un2})) $V_n \setminus P_n = S_{2,n}$. Therefore
$V_n = S_{2,n} \cup (V_n \cap \{y'\}) \cup (V_n \cap (P_n \setminus \{y'\})$, a
disjoint union. The integrand thus has a simple pole at each
distinct element of $S_{2,n} \cup (V_n \cap \{y'\})$, and no other singularities.
The Residue theorem thus implies that the first term equals,
\begin{align*}
&\sum_{y' \in \tilde{S}_{1,n} \cup \tilde{S}_{3,n}} 1_{(y' \in \G_n^\circ)} \;
\bigg( \sum_{x' \in S_{2,n}} 1_{(x' \in \g_n^\circ)} \;
\frac{\prod_{y \in U_n} (y' - y)}{\prod_{x \in V_n \setminus x'} (x' - x)} \;
\frac{\prod_{x \in P_n \setminus y'} (x' - x)}{\prod_{y \in P_n \setminus y'} (y' - y)} \\
&\hspace{3.6cm} + 1_{(y' \in V_n, \; y' \in \g_n^\circ)} \;
\frac{\prod_{y \in U_n} (y' - y)}{\prod_{x \in V_n \setminus y'} (y' - x)} \;
\frac{\prod_{x \in P_n \setminus y'} (y' - x)}{\prod_{y \in P_n \setminus y'} (y' - y)} \bigg) \\
&= \sum_{y' \in \tilde{S}_{1,n} \cup \tilde{S}_{3,n}} 1_{(y' \in \G_n^\circ)}
\sum_{x' \in S_{2,n}} 1_{(x' \in \g_n^\circ)} \;
\frac{\prod_{y \in U_n} (y' - y)}{\prod_{x \in V_n \setminus x'} (x' - x)} \;
\frac{\prod_{x \in P_n \setminus y'} (x' - x)}{\prod_{y \in P_n \setminus y'} (y' - y)} \\
&+ \sum_{y' \in (\tilde{S}_{1,n} \cup \tilde{S}_{3,n}) \cap V_n}
\; 1_{(y' \in \g_n^\circ \cap \G_n^\circ)} \;
\frac{\prod_{y \in U_n} (y' - y)}{\prod_{x \in V_n \setminus y'} (y' - x)}.
\end{align*}
Next consider the second term on the RHS of equation (\ref{eqJ_n3}).
Note, definition \ref{defConCases} and figure \ref{figDesAsc1-12}
imply that $\G_n$ contains $\g_{2,n}$ and none of $\g_{1,n}$ for cases (6,7),
and $\G_n$ contains no parts of $\g_n$ for all other cases. 
Therefore, the second term equals,
\begin{equation*}
- 1_{(A \in \{6,7\})} \; \int_{\g_{2,n}} dw \;
\frac{\prod_{y \in U_n} (w - y)}{\prod_{x \in V_n} (w - x)}
= - 1_{(A \in \{6,7\})} \; \sum_{y' \in V_n \setminus U_n} \; 1_{(y' \in \g_{2,n}^\circ)}
\frac{\prod_{y \in U_n} (y' - y)}{\prod_{x \in V_n \setminus y'} (y' - x)}.
\end{equation*}
where $A \in \{1,2,\ldots,12\}$ denotes that particular case of lemma
\ref{lemCases} which is satisfied. Combined the above give,
\begin{align*}
J_n
&= \sum_{y' \in \tilde{S}_{1,n} \cup \tilde{S}_{3,n}} 1_{(y' \in \G_n^\circ)}
\sum_{x' \in S_{2,n}} 1_{(x' \in \g_n^\circ)} \;
\frac{\prod_{y \in U_n} (y' - y)}{\prod_{x \in V_n \setminus x'} (x' - x)} \;
\frac{\prod_{x \in P_n \setminus y'} (x' - x)}{\prod_{y \in P_n \setminus y'} (y' - y)} \\
&+ \left( \sum_{y' \in (\tilde{S}_{1,n} \cup \tilde{S}_{3,n}) \cap V_n}
1_{(y' \in \g_n^\circ \cap \G_n^\circ)}
- 1_{(A \in \{6,7\})} \sum_{y' \in V_n \setminus U_n} 1_{(y' \in \g_{2,n}^\circ)} \right)
\frac{\prod_{y \in U_n} (y' - y)}{\prod_{x \in V_n \setminus y'} (y' - x)}.
\end{align*}
Finally recall that $V_n \setminus U_n = (VU^{(n)}) \cup (VU_{(n)})$ (see
equation (\ref{eqVn-Un})), and $\tilde{S}_{1,n} \cap V_n = (VU^{(n)}) \cap P_n$
and $\tilde{S}_{3,n} \cap V_n = (VU_{(n)}) \cap P_n$ (see equations
(\ref{eqUnVn}, \ref{eqVn-Un}, \ref{eqPn-Un1})). The required result thus follows if
we can show that:
\begin{enumerate}
\item[(i)]
$\{x' \in S_{2,n} : x' \in \g_n^\circ\}$ equals $S_{2,n}$ for all cases, (1-12),
of lemma \ref{lemCases}.
\item[(ii)]
$\{y' \in \tilde{S}_{1,n} \cup \tilde{S}_{3,n} : y' \in \G_n^\circ\}$ equals
$\tilde{S}_{1,n}$ for (1-6), and equals $\tilde{S}_{3,n}$ for (7-12).
\item[(iii)]
$\{y' \in (\tilde{S}_{1,n} \cup \tilde{S}_{3,n}) \cap V_n : y' \in \g_n^\circ \cap \G_n^\circ\}$
equals $(VU^{(n)}) \cap P_n$ for (1-3) and (6)
when $v_n \ge u_n$, equals $(VU_{(n)}) \cap P_n$
for (7) and (10-12) when $v_n+s_n \le u_n+r_n$, and is empty
otherwise.
\item[(iv)]
$\{y' \in V_n \setminus U_n : y' \in \g_{2,n}^\circ\}$ equals
$VU^{(n)}$ for (6) when $v_n \ge u_n$, equals $VU_{(n)}$ for
(7) when $v_n+s_n \le u_n+r_n$, is empty for (6) when
$u_n > v_n$, and is empty for (7) when $v_n+s_n > u_n+r_n$.
\end{enumerate}
We will prove (i) only for cases (1-6). The proof of (i) for case (12) is
similar to case (1), the proof of (i) for case (11) is similar to case (2),
etc. We will prove (ii) only for case (4). The proof of (ii)
for all other cases follows from similar considerations. We will not prove
(iii) and (iv), but we state that their proofs also follow from
similar considerations.

Consider (i) for cases (1-4). Recall (see equation (\ref{eqfn2}))
that $S_{2,n} \subset \frac1n \{v_n+s_n-n,v_n+s_n-n+1,\ldots, v_n\}$. Thus,
fixing $\d>0$, equation (\ref{equnrnvnsn}) implies that
$S_{2,n} \subset (\chi+\eta-1-\d,\chi+\d)$.
Next recall (see lemma \ref{lemCases} and equation
(\ref{eqS1S2S3In})) that $t > \chi > \chi+\eta-1 \ge \overline{S_3}$.
Definition \ref{defConCases} and figure \ref{figDesAsc1-12} thus imply
that we can choose the above $\d>0$ such that $\g_n$ contains $(\chi+\eta-1-\d,\chi+\d)$.
This proves (i) for cases (1-4).

Consider (i) for case (6). First recall (see lemma \ref{lemCases}) that
$t \in L_t \subset \R \setminus \supp (\l-\mu)$. Thus, fixing $\d>0$ such that
$\overline{L_t} - \d > t > \underline{L_t} + \d$, equation (\ref{eqPnHnlarge})
gives $\frac{\Z}n \cap (\underline{L_t} + \d, \overline{L_t} - \d) \subset P_n$.
Next recall that $S_{2,n} = \frac1n \{v_n+s_n-n,v_n+s_n-n+1,\ldots, v_n\} \setminus P_n$
(see equation (\ref{eqfn2})), that
$\frac1n \{v_n+s_n-n,v_n+s_n-n+1,\ldots, v_n\} \subset (\chi+\eta-1-\d,\chi+\d)$
(see equation (\ref{equnrnvnsn})). Therefore
$S_{2,n} \subset (\chi+\eta-1-\d, \underline{L_t} + \d] \cup [\overline{L_t} - \d, \chi+\d)$.
Next recall (see lemma \ref{lemCases} and equation
(\ref{eqS1S2S3In})) that $\underline{S_1} \ge \chi > \overline{L_t} > t >
\underline{L_t} > \chi+\eta-1 \ge \overline{S_3}$. Definition \ref{defConCases}
and figure \ref{figDesAsc1-12} thus imply that we can choose the above $\d>0$
such that $\g_{1,n}$ contains $(\chi+\eta-1-\d, \underline{L_t} + \d]$ and
$\g_{2,n}$ contains $[\overline{L_t}-\d,\chi+\d)$.
This proves (i) for case (6).

Consider (i) for case (5). First recall (see lemma \ref{lemCases}) that
$t \in (\overline{S_2},\chi) \subset \R \setminus \supp (\l-\mu)$. Also recall
(see assumption \ref{asscases}) that $\chi \in \R \setminus \supp (\l-\mu)$.
Thus we can fix a $\d>0$ such that $t > \overline{S_2} + \d$ and
$(\overline{S_2} + \d,\chi+\d) \subset \R \setminus \supp (\l-\mu)$.
Then we can proceed similarly to case (6), above, to show that
$S_{2,n} \subset (\chi+\eta-1-\d, \overline{S_2} + \d]$. Moreover,
definition \ref{defConCases} and figure \ref{figDesAsc1-12} imply
that we can choose the $\d>0$ such that $\g_n$ contains
$(\chi+\eta-1-\d, \overline{S_2} + \d]$. This proves (i) for case (5).

Consider (ii) for case (4). We must show that $\G_n$ contains all
of $\tilde{S}_{1,n}$ and none of
$\tilde{S}_{3,n}$. First recall (see lemma \ref{lemCases}) that
$t \in (\chi,\underline{S_1}) \subset \R \setminus \supp (\mu)$.
Also recall (see assumption \ref{asscases}) that
$\chi \in \R \setminus \supp (\mu)$. Thus we can fix a $\d>0$ such
that $t < \underline{S_1} - \d$ and
$(\chi-\d,\underline{S_1} - \d) \subset \R \setminus \supp (\mu)$.
Equation (\ref{eqPnHnlarge}) thus gives
$P_n \cap (\chi-\d,\underline{S_1} - \d) = \emptyset$.
Equations (\ref{equnrnvnsn}, \ref{eqPn-Un1}) then give
$\tilde{S}_{1,n} = \{y \in P_n : y \ge \underline{S_1} - \d \}$.
Equation (\ref{eqS1nS2nS3nIn}) then gives
$\tilde{S}_{1,n} \subset[\underline{S_1} - \d, \overline{S_1} + \delta)$.
Next recall (see equation (\ref{eqPn-Un1})) that
$\tilde{S}_{3,n} = \{y \in P_n : y \le \frac{u_n+r_n-n}n \}$.
Equations (\ref{equnrnvnsn}, \ref{eqS1nS2nS3nIn}) then give
$\tilde{S}_{3,n} \subset (\underline{S}_3-\d,\chi+\eta-1+\d)$.
Finally recall (see lemma \ref{lemCases} and equation (\ref{eqS1S2S3In}))
that $\underline{S_1} > t > \chi > \chi+\eta-1$. Definition \ref{defConCases}
and figure \ref{figDesAsc1-12} thus imply that we can choose the above
$\d>0$ such that $\G_n$ contains all of $[\underline{S_1} - \d, \overline{S_1} + \d)$
and none of $(\underline{S}_3-\d,\chi+\eta-1+\d)$. This proves (ii) for case (4).
\end{proof}

Next we prove the following technical result:
\begin{lem}
\label{lemphi}
Fix $k,i \ge 1$. Then, for all $x \in \Z$,
\begin{equation}
\tag{1}
\sum_{l=0}^k \frac{\prod_{j=1}^i (x+l-j)}{\prod_{j=0, j \neq l}^k (l-j)} 
= \begin{dcases}
\frac{i!}{k! (i-k)!} \prod_{j=1}^{i-k} (x-j) & ; \text{ when } k < i, \\
1 & ; \text{ when } k = i, \\
0 & ; \text{ when } k > i.
\end{dcases}
\end{equation}
Next, fix and $a,b \in \Z$. Then, for all $x \in \Z$ with $x+b,x+a \ge 0$,
\begin{equation}
\tag{2}
\sum_{l=0}^k \frac{(x+b+l)!}{(x+a+l)!} \; \frac1{\prod_{j=0, j \neq l}^k (l-j)} 
= (b-a)(b-a-1) \cdots (b-a-k+1) \; \frac1{k!} \; \frac{(x+b)!}{(x+a+k)!}.
\end{equation}
\end{lem}

\begin{proof}
First note, letting $g : \Z \to \R$ be any function,
and letting $\Delta$ denote the finite difference
operator (i.e., $(\Delta g)(x) := g(x+1) - g(x)$,
$(\Delta^2 g) (x) := (\Delta (\Delta g)) (x)
= g(x+2) - 2 g(x+1) + g(x)$, etc), then
\begin{equation}
\label{eqlemphi}
(\Delta^k g)(x)
= \sum_{l=0}^k (-1)^{k-l} \binom{k}{l} g(x+l)
= k! \sum_{l=0}^k \frac{g(x+l)}{\prod_{j=0,j\ne l}^k (l-j)},
\end{equation}
for all $x \in \Z$.
Above, the first step follows by induction, and the second step
follows since $(-1)^{k-l} l! (k-l)! = \prod_{j=0,j\ne l}^k (l-j)$
for all $l \in \{0,\ldots,k\}$. We will use the above identity to prove
(1) and (2).

Consider (1). In this case, take
\begin{equation*}
g(x) := \prod_{j=1}^i (x-j),
\end{equation*}
for all $x \in \Z$. Then, induction gives,
\begin{equation*}
(\Delta^k g)(x)
= \begin{dcases}
\frac{i!}{(i-k)!} \prod_{j=1}^{i-k} (x-j) & ; \text{ when } k < i, \\
k! & ; \text{ when } k = i, \\
0 & ; \text{ when } k > i,
\end{dcases}
\end{equation*}
Moreover, equation (\ref{eqlemphi}) gives,
\begin{equation*}
(\Delta^k g)(x)
= k! \sum_{l=0}^k \frac{\prod_{j=1}^i (x+l-j)}{\prod_{j=0,j\ne l}^k (l-j)},
\end{equation*}
for all $x \in \Z$. Combined, the above prove (1).

Consider (2). In this case take,
\begin{equation*}
g(x) := \frac{(x+b)!}{(x+a)!},
\end{equation*}
for all $x \in \Z$ with $x+b,x+a \ge 0$. Then, induction gives,
\begin{equation*}
(\Delta^k g)(x)
= (b-a)(b-a-1) \cdots (b-a-k+1) \; \frac{(x+b)!}{(x+a+k)!},
\end{equation*}
for all $x \in \Z$ with $x+b,x+a \ge 0$. Moreover, equation (\ref{eqlemphi}) gives,
\begin{equation*}
(\Delta^k g)(x)
= k! \sum_{l=0}^k \frac{(x+b+l)!}{(x+a+l)!} \; \frac1{\prod_{j=0,j\ne l}^k (l-j)},
\end{equation*}
for all $x \in \Z$ with $x+b,x+a \ge 0$.  Combined, the above prove (2).
\end{proof}

Next we examine $K_n((u_n,r_n),(v_n,s_n))$:
\begin{lem}
\label{lemKnSums}
Assume the conditions of theorem \ref{thmAiry}. Define the sets
$U_n$, $V_n$, $P_n$, $VU^{(n)}$, $VU_{(n)}$ as above. Also define
$\beta_n$ as in lemma \ref{lemKnJnPhin}. Then,
\begin{align}
\tag{1}
&\beta_n \; K_n((u_n,r_n),(v_n,s_n))
= \sum_{y' \in \tilde{S}_{1,n}} \sum_{x' \in S_{2,n}} \;
\frac{\prod_{y \in U_n} (y' - y)}{\prod_{x \in V_n \setminus x'} (x' - x)} \;
\frac{\prod_{x \in P_n \setminus y'} (x' - x)}{\prod_{y \in P_n \setminus y'} (y' - y)} \\
\nonumber
&+ 1_{(v_n \ge u_n)} \bigg( 1_{(s_n \le r_n)} \sum_{y' \in (VU^{(n)}) \cap P_n}
- 1_{(s_n > r_n)} \sum_{y' \in (VU^{(n)}) \setminus P_n} \bigg)
\frac{\prod_{y \in U_n} (y' - y)}{\prod_{x \in V_n \setminus y'} (y' - x)}.
\end{align}
Moreover,
\begin{align}
\tag{2}
&\beta_n \; K_n((u_n,r_n),(v_n,s_n))
= -\sum_{y' \in \tilde{S}_{3,n}} \sum_{x' \in S_{2,n}} \;
\frac{\prod_{y \in U_n} (y' - y)}{\prod_{x \in V_n \setminus x'} (x' - x)} \;
\frac{\prod_{x \in P_n \setminus y'} (x' - x)}{\prod_{y \in P_n \setminus y'} (y' - y)} \\
\nonumber
&+ 1_{(v_n+s_n \le u_n+r_n)} \bigg( - 1_{(s_n \le r_n)} \sum_{y' \in (VU_{(n)}) \cap P_n}
+ 1_{(s_n > r_n)} \sum_{y' \in (VU_{(n)}) \setminus P_n} \bigg)
\frac{\prod_{y \in U_n} (y' - y)}{\prod_{x \in V_n \setminus y'} (y' - x)}.
\end{align}
\end{lem}

\begin{proof}
First note, equation (\ref{eqKnrusvFixTopLine}) and the expression for $\beta_n$
(see statement of lemma \ref{lemKnJnPhin}) give,
\begin{align*}
&\beta_n \; K_n((u_n,r_n),(v_n,s_n)) =
- \beta_n \; \phi_{r_n,s_n}(u_n,v_n) \\
&+ \sum_{k=1}^n 1_{(x_k^{(n)} \ge  u_n)} \sum_{l=v_n+s_n-n}^{v_n}
\frac{\prod_{j=u_n+r_n-n+1}^{u_n-1} (\frac{x_k^{(n)}}n - \frac{j}n)}
{\prod_{j=v_n+s_n-n, \; j \neq l}^{v_n} (\frac{l}n - \frac{j}n)} \;
\prod_{i=1, \; i \neq k}^n
\left( \frac{\frac{l}n - \frac{x_i^{(n)}}n}{\frac{x_k^{(n)}}n - \frac{x_i^{(n)}}n} \right).
\end{align*}
Equations (\ref{eqJ_n2}, \ref{eqPn-Un1}) then give,
\begin{align}
\label{eqlemKnJn}
&\beta_n \; K_n((u_n,r_n),(v_n,s_n)) =
- \beta_n \; \phi_{r_n,s_n}(u_n,v_n) \\
\nonumber
&+ \sum_{y' \in \tilde{S}_{1,n}} \sum_{x' \in V_n}
\frac{\prod_{y \in U_n} (y' - y)}{\prod_{x \in V_n \setminus x'} (x' - x)} \;
\frac{\prod_{x \in P_n \setminus y'} (x' - x)}{\prod_{y \in P_n \setminus y'} (y' - y)}.
\end{align}
First, we will show:
\begin{equation}
\tag{i}
\sum_{y' \in V_n} \frac{\prod_{y \in U_n} (y'-y)}{\prod_{x \in V_n \setminus y'} (y'-x)}
= \begin{dcases}
\frac{\beta_n \prod_{j=1}^{s_n-r_n-1} (v_n-u_n+s_n-r_n-j)}{(s_n-r_n-1)!} & ; \; s_n > r_n+1, \\
1 & \; ; \; s_n = r_n+1, \\
0 & \; ; \; s_n \le r_n.
\end{dcases}
\end{equation}
Then, we will use this to show:
\begin{equation}
\tag{ii}
\beta_n \; \phi_{r_n,s_n}(u_n,v_n)
= 1_{(v_n \ge u_n, s_n > r_n)} \sum_{y' \in VU^{(n)}}
\frac{\prod_{y \in U_n} (y' - y)}{\prod_{x \in V_n \setminus y'} (y' - x)}.
\end{equation}
Next, recalling that $S_{2,n} \subset V_n$ (see equation (\ref{eqPn-Un2})), we will show:
\begin{align}
\tag{iii}
&\sum_{y' \in \tilde{S}_{1,n}} \sum_{x' \in V_n \setminus S_{2,n}}
\frac{\prod_{y \in U_n} (y' - y)}{\prod_{x \in V_n \setminus x'} (x' - x)} \;
\frac{\prod_{x \in P_n \setminus y'} (x' - x)}{\prod_{y \in P_n \setminus y'} (y' - y)} \\
\nonumber
&= 1_{(v_n \ge u_n)} \sum_{y' \in (VU^{(n)}) \cap P_n}
\frac{\prod_{y \in U_n} (y' - y)}{\prod_{x \in V_n \setminus y'} (y' - x)}.
\end{align}
Thus, since $S_{2,n} \subset V_n$,
equation (\ref{eqlemKnJn}) and parts (ii,iii) prove (1).
Next we will show:
\begin{align}
\tag{iv}
&\sum_{y' \in \tilde{S}_{1,n} \cup \tilde{S}_{3,n}} \sum_{x' \in V_n}
\frac{\prod_{y \in U_n} (y' - y)}{\prod_{x \in V_n \setminus x'} (x' - x)} \;
\frac{\prod_{x \in P_n \setminus y'} (x' - x)}{\prod_{y \in P_n \setminus y'} (y' - y)} \\
\nonumber
&= \bigg( 1_{(v_n \ge u_n, s_n > r_n)} \sum_{y' \in VU^{(n)}}
+ 1_{(v_n+s_n \le u_n+r_n, s_n > r_n)} \sum_{y' \in VU_{(n)}} \bigg)
\frac{\prod_{y \in U_n} (y' - y)}{\prod_{x \in V_n \setminus y'} (y' - x)}.
\end{align}
Finally we will show:
\begin{align}
\tag{v}
&\sum_{y' \in \tilde{S}_{3,n}} \sum_{x' \in V_n \setminus S_{2,n}}
\frac{\prod_{y \in U_n} (y' - y)}{\prod_{x \in V_n \setminus x'} (x' - x)} \;
\frac{\prod_{x \in P_n \setminus y'} (x' - x)}{\prod_{y \in P_n \setminus y'} (y' - y)} \\
\nonumber
&= 1_{(v_n+s_n \le u_n+r_n)} \sum_{y' \in (VU_{(n)}) \cap P_n}
\frac{\prod_{y \in U_n} (y' - y)}{\prod_{x \in V_n \setminus y'} (y' - x)}.
\end{align}
Thus, since $S_{2,n} \subset V_n$,
equation (\ref{eqlemKnJn}) and parts (ii,iv,v) prove (2).

Consider (i). Note, taking $k:=n-s_n$ and $i := n-r_n-1$ and $x := v_n-u_n+s_n-r_n$,
part (1) of lemma \ref{lemphi} gives,
\begin{align*}
&\sum_{l=0}^{n-s_n} \frac{\prod_{j=1}^{n-r_n-1} (v_n-u_n+s_n-r_n+l-j)}
{\prod_{j=0, j \neq l}^{n-s_n} (l-j)} \\
&= \begin{dcases}
\frac{(n-r_n-1)!}{(n-s_n)! (s_n-r_n-1)!} \prod_{j=1}^{s_n-r_n-1} (v_n-u_n+s_n-r_n-j)
& ; \text{ when } s_n > r_n+1, \\
1 & ; \text{ when } s_n = r_n+1, \\
0 & ; \text{ when } s_n < r_n+1.
\end{dcases}
\end{align*}
Shifting the dummy variables implies that the above LHS equals,
\begin{equation*}
\sum_{l=v_n+s_n-n}^{v_n} \frac{\prod_{j=u_n+r_n-n+1}^{u_n-1} (l-j)}
{\prod_{j=v_n+s_n-n, j \neq l}^{v_n} (l-j)}.
\end{equation*}
Equation (\ref{eqJ_n2}) and the expression for $\beta_n$
(see statement of lemma \ref{lemKnJnPhin})
then prove (i).

Consider (ii). First note, part (i), equation (\ref{eqphirsuv}),
and the expression for $\beta_n$ (see statement of lemma \ref{lemKnJnPhin})
prove the following:
\begin{equation*}
\beta_n \; \phi_{r_n,s_n}(u_n,v_n)
= 1_{(v_n \ge u_n, s_n > r_n)} \sum_{y' \in V_n}
\frac{\prod_{y \in U_n} (y' - y)}{\prod_{x \in V_n \setminus y'} (y' - x)}.
\end{equation*}
Therefore,
\begin{equation*}
\beta_n \; \phi_{r_n,s_n}(u_n,v_n)
= 1_{(v_n \ge u_n, s_n > r_n)} \sum_{y' \in V_n \setminus U_n}
\frac{\prod_{y \in U_n} (y' - y)}{\prod_{x \in V_n \setminus y'} (y' - x)}.
\end{equation*}
Next note, equation (\ref{eqVn-Un}) implies that
$V_n \setminus U_n = (VU^{(n)}) \cup (VU_{(n)})$.
Finally note, equation (\ref{eqVn-Un}) also implies that
$VU^{(n)} \neq \emptyset$ and $VU_{(n)} = \emptyset$ when
$v_n \ge u_n$ and $s_n > r_n$. This proves (ii).

Consider (iii). Recall (see equations (\ref{eqPn-Un1}, \ref{eqPn-Un2})) that
$\tilde{S}_{1,n} \subset P_n$ and $S_{2,n} = V_n \setminus P_n$. Thus, for all
$y' \in \tilde{S}_{1,n}$, $V_n$ equals the disjoint union
$S_{2,n} \cup (V_n \cap \{y'\}) \cup (V_n \cap (P_n \setminus y'))$.
Therefore,
\begin{align*}
&\sum_{y' \in \tilde{S}_{1,n}} \sum_{x' \in V_n \setminus S_{2,n}}
\frac{\prod_{y \in U_n} (y' - y)}{\prod_{x \in V_n \setminus x'} (x' - x)} \;
\frac{\prod_{x \in P_n \setminus y'} (x' - x)}{\prod_{y \in P_n \setminus y'} (y' - y)} \\
&= \sum_{y' \in \tilde{S}_{1,n} \cap V_n}
\frac{\prod_{y \in U_n} (y' - y)}{\prod_{x \in V_n \setminus y'} (y' - x)} \;
\frac{\prod_{x \in P_n \setminus y'} (y' - x)}{\prod_{y \in P_n \setminus y'} (y' - y)}
= \sum_{y' \in \tilde{S}_{1,n} \cap (V_n \setminus U_n)}
\frac{\prod_{y \in U_n} (y' - y)}{\prod_{x \in V_n \setminus y'} (y' - x)}.
\end{align*}
Note that equations (\ref{eqVn-Un}, \ref{eqJ_n2}) give
$\tilde{S}_{1,n} \cap (V_n \setminus U_n) = (VU^{(n)}) \cap P_n$
when $v_n \ge u_n$, and $\tilde{S}_{1,n} \cap (V_n \setminus U_n) = \emptyset$ when
$v_n < u_n$. This proves (iii). Part (v) follows similarly.

Consider (iv). First recall (see equation (\ref{eqPn-Un2})) that
$P_n \setminus U_n = \tilde{S}_{1,n} \cup \tilde{S}_{3,n}$, a disjoint
union. Therefore,
\begin{align*}
&\sum_{y' \in \tilde{S}_{1,n} \cup \tilde{S}_{3,n}} \sum_{x' \in V_n}
\frac{\prod_{y \in U_n} (y' - y)}{\prod_{x \in V_n \setminus x'} (x' - x)} \;
\frac{\prod_{x \in P_n \setminus y'} (x' - x)}{\prod_{y \in P_n \setminus y'} (y' - y)} \\
&= \sum_{y' \in P_n} \sum_{x' \in V_n}
\frac{\prod_{y \in U_n} (y' - y)}{\prod_{x \in V_n \setminus x'} (x' - x)} \;
\frac{\prod_{x \in P_n \setminus y'} (x' - x)}{\prod_{y \in P_n \setminus y'} (y' - y)}.
\end{align*}
Thus, since $P_n$ and $U_n$ are sets of distinct points, and since $|U_n| < |P_n|$
(see equation (\ref{eqJ_n2})), Lagrange interpolation
gives,
\begin{equation*}
\sum_{y' \in \tilde{S}_{1,n} \cup \tilde{S}_{3,n}} \sum_{x' \in V_n}
\frac{\prod_{y \in U_n} (y' - y)}{\prod_{x \in V_n \setminus x'} (x' - x)} \;
\frac{\prod_{x \in P_n \setminus y'} (x' - x)}{\prod_{y \in P_n \setminus y'} (y' - y)}
= \sum_{x' \in V_n}
\frac{\prod_{y \in U_n} (x' - y)}{\prod_{x \in V_n \setminus x'} (x' - x)}.
\end{equation*}
Part (i) then implies part (iv) when $s_n \le r_n$. To
see part (iv) when $s_n > r_n$, first note the above equation gives,
\begin{equation*}
\sum_{y' \in \tilde{S}_{1,n} \cup \tilde{S}_{3,n}} \sum_{x' \in V_n}
\frac{\prod_{y \in U_n} (y' - y)}{\prod_{x \in V_n \setminus x'} (x' - x)} \;
\frac{\prod_{x \in P_n \setminus y'} (x' - x)}{\prod_{y \in P_n \setminus y'} (y' - y)}
= \sum_{x' \in V_n \setminus U_n}
\frac{\prod_{y \in U_n} (x' - y)}{\prod_{x \in V_n \setminus x'} (x' - x)}.
\end{equation*}
Next note, equation (\ref{eqVn-Un}) implies that
$V_n \setminus U_n = (VU^{(n)}) \cup (VU_{(n)}$).
Finally note, equation (\ref{eqVn-Un}) also implies the following:
\begin{itemize}
\item
$VU^{(n)} \neq \emptyset$ and $VU_{(n)} = \emptyset$ when
$s_n > r_n$ and $v_n \ge u_n$.
\item
$VU^{(n)} = \emptyset$ and $VU_{(n)} = \emptyset$ when
$s_n > r_n$ and $v_n < u_n$ and $v_n+s_n > u_n+r_n$.
\item
$VU^{(n)} = \emptyset$ and $VU_{(n)} \neq \emptyset$ when
$s_n > r_n$ and $v_n+s_n \le u_n+r_n$.
\end{itemize}
The above exhaust all possibilities when $s_n>r_n$. This proves (iv).
\end{proof}

Next we examine $\Phi_n$:
\begin{lem}
\label{lemPhin}
Assume the conditions of theorem \ref{thmAiry}. Define the sets
$U_n$, $V_n$, $P_n$, $VU^{(n)}$, $VU_{(n)}$ as above. Also define
$\Phi_n$ and $\beta_n$ as in lemma \ref{lemKnJnPhin}. Recall that
one of the cases, (1-12), of lemma \ref{lemCases} must be satisfied.
Then:
\begin{align*}
\Phi_n
&= 1_{(v_n \ge u_n, s_n>r_n)} \sum_{y' \in VU^{(n)}}
\frac{\prod_{y \in U_n} (y' - y)}{\prod_{x \in V_n \setminus y'} (y' - x)}
\text{ for cases (1-4),} \\
\Phi_n
&= 1_{(v_n \ge u_n, s_n \le r_n)} \sum_{y' \in VU^{(n)}}
\frac{\prod_{y \in U_n} (y' - y)}{\prod_{x \in V_n \setminus y'} (y' - x)}
\text{ for cases (5,6),} \\
\Phi_n
&= 1_{(v_n+s_n \le u_n+r_n, s_n \le r_n)} \sum_{y' \in VU_{(n)}}
\frac{\prod_{y \in U_n} (y' - y)}{\prod_{x \in V_n \setminus y'} (y' - x)}
\text{ for cases (7,8),} \\
\Phi_n
&= 1_{(v_n+s_n \le u_n+r_n, s_n > r_n)} \sum_{y' \in VU_{(n)}}
\frac{\prod_{y \in U_n} (y' - y)}{\prod_{x \in V_n \setminus y'} (y' - x)}
\text{ for cases (9-12).}
\end{align*}
\end{lem}

\begin{proof}
We will prove the result for cases (5,6) of lemma \ref{lemCases}. The other cases follow
from similar considerations.

Assume that one of (5,6) is satisfied. First note, the definition of
$\Phi_n$ given in the statement of lemma \ref{lemKnJnPhin} gives
$\Phi_n = 0$ when $v_n < u_n$ or $s_n > r_n$. This proves
the result for these cases. Next note, the definition also gives
$\Phi_n = 0$ when $v_n \ge u_n$ and $s_n \le r_n$ and $v_n+s_n > u_n+r_n$.
Thus, in this case, it is necessary to show that,
\begin{equation*}
\sum_{y' \in VU^{(n)}}
\frac{\prod_{y \in U_n} (y' - y)}{\prod_{x \in V_n \setminus y'} (y' - x)}
= 0.
\end{equation*}
To see this, first note that equations (\ref{eqUnVn}, \ref{eqVn-Un}) give,
\begin{equation*}
\sum_{y' \in VU^{(n)}}
\frac{\prod_{y \in U_n} (y' - y)}{\prod_{x \in V_n \setminus y'} (y' - x)}
= n^{r_n+1-s_n} \; \sum_{l=u_n}^{v_n}
\frac{\prod_{j=u_n+r_n-n+1}^{u_n-1} (l - j)}
{\prod_{j=v_n+s_n-n, j \neq l}^{v_n} (l - j)}.
\end{equation*}
Therefore, since $v_n+s_n > u_n+r_n$,
\begin{equation*}
\sum_{y' \in VU^{(n)}}
\frac{\prod_{y \in U_n} (y' - y)}{\prod_{x \in V_n \setminus y'} (y' - x)}
= n^{r_n+1-s_n} \; \sum_{l=v_n+s_n-n}^{v_n}
\frac{\prod_{j=u_n+r_n-n+1}^{u_n-1} (l - j)}
{\prod_{j=v_n+s_n-n, j \neq l}^{v_n} (l - j)}.
\end{equation*}
Finally, since $s_n \le r_n$, we can proceed as in part (i) in the proof of lemma
\ref{lemKnSums} to show that the above RHS equals $0$. This
proves the result when $v_n \ge u_n$ and $s_n \le r_n$ and
$v_n+s_n > u_n+r_n$. It thus remains to prove the result
when $v_n \ge u_n$ and $s_n \le r_n$ and $v_n+s_n \le u_n+r_n$.

Suppose first that $v_n = u_n$ and $s_n \le r_n$ and $v_n+s_n \le u_n+r_n$.
Then, equations (\ref{eqUnVn}, \ref{eqVn-Un}) give,
\begin{align*}
\sum_{y' \in VU^{(n)}}
\frac{\prod_{y \in U_n} (y' - y)}{\prod_{x \in V_n \setminus y'} (y' - x)}
&= n^{r_n+1-s_n} \; \sum_{l=u_n}^{u_n}
\frac{\prod_{j=u_n+r_n-n+1}^{u_n-1} (l - j)}
{\prod_{j=u_n+s_n-n, j \neq l}^{u_n} (l - j)} \\
&= n^{r_n+1-s_n} \;
\frac{\prod_{j=u_n+r_n-n+1}^{u_n-1} (u_n - j)}
{\prod_{j=u_n+s_n-n}^{u_n-1} (u_n - j)}.
\end{align*}
Thus, since $s_n \le r_n$,
\begin{equation*}
\sum_{y' \in VU^{(n)}}
\frac{\prod_{y \in U_n} (y' - y)}{\prod_{x \in V_n \setminus y'} (y' - x)}
= n^{r_n+1-s_n} \; \frac1{\prod_{j=u_n+s_n-n}^{u_n+r_n-n} (u_n - j)}
= n^{r_n+1-s_n} \; \frac{(n-r_n-1)!}{(n-s_n)!}.
\end{equation*}
The definitions of $\Phi_n$ and $\beta_n$ in the
statement of lemma \ref{lemKnJnPhin} then prove the result
when $v_n = u_n$ and $s_n \le r_n$ and $v_n+s_n \le u_n+r_n$.

Finally suppose that $v_n > u_n$ and $s_n \le r_n$ and
$v_n+s_n \le u_n+r_n$. Then, equations (\ref{eqUnVn}, \ref{eqVn-Un}) give,
\begin{align*}
\frac{\prod_{y \in U_n} (y' - y)}{\prod_{x \in V_n \setminus y'} (y' - x)}
&= n^{r_n+1-s_n} \; \sum_{l=u_n}^{v_n}
\frac{\prod_{j=u_n+r_n-n+1}^{u_n-1} (l - j)}
{\prod_{j=v_n+s_n-n, j \neq l}^{v_n} (l - j)} \\
&= n^{r_n+1-s_n} \; \sum_{l=u_n}^{v_n}
\frac1{\prod_{j=v_n+s_n-n}^{u_n+r_n-n} (l - j)} \;
\frac1{\prod_{j=u_n, j \neq l}^{v_n} (l - j)}.
\end{align*}
Note that $l - j > 0$ for all $l \in \{u_n,u_n+1,\ldots,v_n\}$
and $j \in \{v_n+s_n-n, v_n+s_n-n+1,\ldots,u_n+r_n-n\}$
(indeed, equation (\ref{equnrnvnsn}) gives
$l-j = n(1-\eta) + o(n)$ uniformly for $l,j$). We can thus write,
\begin{equation*}
\frac{\prod_{y \in U_n} (y' - y)}{\prod_{x \in V_n \setminus y'} (y' - x)}
= n^{r_n+1-s_n} \; \sum_{l=u_n}^{v_n}
\frac{(l-(u_n+r_n-n+1))!}{(l-(v_n+s_n-n))!} \;
\frac1{\prod_{j=u_n, j \neq l}^{v_n} (l - j)}.
\end{equation*}
Shifting the dummy variables on the RHS then gives,
\begin{equation*}
\frac{\prod_{y \in U_n} (y' - y)}{\prod_{x \in V_n \setminus y'} (y' - x)}
= n^{r_n+1-s_n} \; \sum_{l=0}^{v_n-u_n}
\frac{(l-(r_n-n+1))!}{(l-(v_n-u_n+s_n-n))!} \;
\frac1{\prod_{j=0, j \neq l}^{v_n-u_n} (l - j)}.
\end{equation*}
Then, part (2) of lemma \ref{lemphi} (take
$k := v_n - u_n$ and $x := n$ and $b := -r_n-1$ and
$a := u_n - v_n - s_n$), and the definition of $\beta_n$
in the statement of lemma \ref{lemKnJnPhin} gives,
\begin{equation*}
\frac{\prod_{y \in U_n} (y' - y)}{\prod_{x \in V_n \setminus y'} (y' - x)}
= \frac{(s_n-r_n) (s_n-r_n+1) \cdots ((v_n+s_n)-(u_n+r_n)-1)}{(v_n-u_n)!}
\beta_n.
\end{equation*}
The definition of $\Phi_n$ in the statement of lemma \ref{lemKnJnPhin}
then proves the result when $v_n > u_n$ and $s_n \le r_n$ and $v_n+s_n \le u_n+r_n$.
\end{proof}

Next we prove the following technical result:
\begin{lem}
\label{lemvnunIneq}
Fix $(u,r) \in \R^2$ and $(v,s) \in \R^2$, and define $\{(u_n,r_n)\}_{n\ge1}$ and
$\{(v_n,s_n)\}_{n\ge1}$ as in equations (\ref{equnrnvnsn2}, \ref{equnrnvnsn3}).
Recall that one of the cases, (1-12), of
lemma \ref{lemCases} must be satisfied. Fix $\xi>0$ sufficiently small
such that equations (\ref{eqxi}, \ref{eqxi1}, \ref{eqxi4},
\ref{eqxi5}) are satisfied. Assume that $(u,r) \neq (v,s)$,
i.e., that either $v > u$, or $v < u$, or $v = u$ and $s \neq r$.

When $v>u$, the following are satisfied:
\begin{itemize}
\item
For cases (1-4),
$v_n > u_n$ (and so $VU^{(n)} \neq \emptyset$)
and $v_n + s_n > u_n + r_n$ (and so $VU_{(n)} = \emptyset$)
and $s_n > r_n+1$. Moreover,
$t - 2\xi > \chi+2\xi > \max(VU^{(n)})
> \min(VU^{(n)}) > \chi-2\xi > \chi+\eta-1+2\xi$.
\item
For cases (5-8), $v_n > u_n$ (and so $VU^{(n)} \neq \emptyset$) and
$v_n + s_n < u_n + r_n$ (and so $VU_{(n)} \neq \emptyset$)
and $s_n < r_n+1$. Moreover, $\min(VU^{(n)}) > \chi-2\xi > t+2\xi >
t - 2\xi > \chi+\eta-1+2\xi > \max(VU_{(n)})
> \min(VU_{(n)}) > \chi+\eta-1-2\xi$.
\item
For cases (9-12), $v_n < u_n$ (and so $VU^{(n)} = \emptyset$) and
$v_n + s_n < u_n + r_n$ (and so $VU_{(n)} \neq \emptyset$) and
$s_n > r_n+1$. Moreover,
$\chi-2\xi > \chi+\eta-1+2\xi > \max(VU_{(n)})
> \min(VU_{(n)}) > \chi+\eta-1-2\xi > t + 2\xi$.
\end{itemize}
Moreover, when $v < u$:
\begin{itemize}
\item
For cases (1-4), $v_n < u_n$ and $v_n + s_n < u_n + r_n$ and $s_n < r_n+1$.
\item
For cases (5-8), $v_n < u_n$ and $v_n + s_n > u_n + r_n$ and $s_n > r_n+1$.
\item
For cases (9-12), $v_n > u_n$ and $v_n+s_n > u_n+r_n$ and $s_n < r_n+1$.
\end{itemize}
Finally, when $v = u$ and $s \neq r$:
\begin{itemize}
\item
For cases (1-4), $v_n - u_n$ and $s_n - (r_n+1)$ have opposite signs.
\item
For cases (5-12), $v_n - u_n$ and $s_n - (r_n+1)$ and
$(v_n + s_n) - (u_n + r_n)$ have the same sign.
\end{itemize}
\end{lem}

\begin{proof}
First note, equations (\ref{equnrnvnsn2}, \ref{equnrnvnsn3}) give,
\begin{align*}
v_n - u_n
&= n^\frac23 m_n (v-u) + n^\frac13 p_n (e^{C_n(t)}-1) (s-r) + O(1), \\
s_n - (r_n+1)
&= n^\frac23 m_n (e^{C_n(t)}-1) (v-u) + n^\frac13 p_n (-1) (s-r) + O(1), \\
(v_n+s_n) - (u_n+r_n)
&= n^\frac23 m_n (e^{C_n(t)}) (v-u) + n^\frac13 p_n (e^{C_n(t)}-2) (s-r) + O(1),
\end{align*}
where $\{m_n\}_{n\ge1}$ and $\{p_n\}_{n\ge1}$ are those convergent sequences of
real numbers with non-zero limits given in definition \ref{defmnpn}. Note that
this definition also gives,
$m_n (e^{C_n(t)} - 1)/e^{C_n(t)} > 0$, and so $m_n$ and $(e^{C_n(t)} - 1)/e^{C_n(t)}$
have the same sign. Next recall (see lemma \ref{lemCases} and \ref{lemAnalExt})
that $e^{C(t)} > 1$ for cases (1-4), $e^{C(t)} < 0$ for
cases (5-8), and $e^{C(t)} \in (0,1)$ for cases (9-12). Also recall
that $e^{C_n(t)} \to e^{C(t)}$ (see equation (\ref{eqNonAsyEdge})).

Consider cases (1-4) with $v>u$. The above observations then imply
that $\{m_n\}_{n\ge1}$ is a convergent sequence of real numbers with a
positive limit, and that $v_n > u_n$ and $v_n + s_n > u_n + r_n$ and
$s_n > r_n+1$. Equation (\ref{eqVn-Un})
then implies that $VU^{(n)} \neq \emptyset$ and
$VU_{(n)} = \emptyset$.
Next note, since $t > \chi$ (see cases (1-4) of lemma \ref{lemCases}), equations
(\ref{eqxi1}, \ref{eqxi5}) give $t - 2\xi > \chi+2\xi > \max(VU^{(n)})
> \min(VU^{(n)}) > \chi-2\xi > \chi+\eta-1+2\xi$. We have thus shown the required result for
cases (1-4) when $v>u$. The other cases follow similarly.
\end{proof}

Finally, we prove lemma \ref{lemKnJnPhin}:
\begin{proof}[Proof of lemma \ref{lemKnJnPhin}]
We will prove part (1) for cases (1-4) of lemma \ref{lemCases},
and part (2) for cases (1-4) and (7,8). Parts (1,2) for
the remaining cases follow from similar considerations.
 
Part (1) for cases (1-3) of lemma \ref{lemCases} easily follows
from lemmas \ref{lemJn} and \ref{lemKnSums} and \ref{lemPhin}.
Consider part (1) for case (4) of lemma \ref{lemCases}.
Note, lemmas \ref{lemJn} and \ref{lemKnSums} give,
\begin{align*}
&\beta_n \; K_n((u_n,r_n),(v_n,s_n)) = J_n \\
&+ 1_{(v_n \ge u_n)} \bigg( 1_{(s_n \le r_n)} \sum_{y' \in (VU^{(n)}) \cap P_n}
- 1_{(s_n > r_n)} \sum_{y' \in (VU^{(n)}) \setminus P_n} \bigg)
\frac{\prod_{y \in U_n} (y' - y)}{\prod_{x \in V_n \setminus y'} (y' - x)}.
\end{align*}
Lemma \ref{lemPhin} thus implies that the result for case (4) follows if
$(VU^{(n)}) \cap P_n = \emptyset$. To see this, fix $\d>0$
such that $(\chi-\d,\chi+\d) \subset \R \setminus \supp(\mu)$ (see assumption
\ref{asscases}). Assumption \ref{assIsol} and equation (\ref{eqPnHnlarge}) then
give $(\chi-\d,\chi+\d) \cap P_n = \emptyset$.
Finally recall (see equations (\ref{equnrnvnsn}, \ref{eqVn-Un}))
that $VU^{(n)} \subset (\chi-\d,\chi+\d)$. Therefore $(VU^{(n)}) \cap P_n = \emptyset$,
as required.

Consider (2) for cases (1-4) of lemma \ref{lemCases}. First recall
(see lemma \ref{lemPhin}) that,
\begin{equation*}
\Phi_n
= 1_{(v_n \ge u_n, s_n > r_n)} \sum_{y' \in VU^{(n)}}
\frac{\prod_{y \in U_n} (y' - y)}{\prod_{x \in V_n \setminus y'} (y' - x)}.
\end{equation*}
Next recall (see statement of lemma \ref{lemKnJnPhin}) that
$(u,r) \neq (v,s)$, i.e., that either $v > u$, or $v < u$,
or $v = u$ and $s \neq r$. Moreover, lemma \ref{lemvnunIneq}
implies that $1_{(v_n \ge u_n, s_n>r_n)} = 1$ when $v > u$,
$1_{(v_n \ge u_n, s_n>r_n)} = 0$ when $v < u$, and
$1_{(v_n \ge u_n, s_n>r_n)} = 0$ when $v = u$ and $s \neq r$.
Therefore,
\begin{equation*}
\Phi_n = 1_{(v>u)} \sum_{y' \in VU^{(n)}}
\frac{\prod_{y \in U_n} (y' - y)}{\prod_{x \in V_n \setminus y'} (y' - x)}.
\end{equation*}
It thus remains to show that,
\begin{equation*}
\sum_{y' \in VU^{(n)}}
\frac{\prod_{y \in U_n} (y' - y)}{\prod_{x \in V_n \setminus y'} (y' - x)}
= \frac1{2\pi i} \int_{\kappa_n} dw \;
\frac{\prod_{y \in U_n} (w - y)}{\prod_{x \in V_n} (w - x)},
\end{equation*}
when $v>u$.
To see the above, first note that the integrand on the RHS
has a simple pole at each distinct element of $V_n \setminus U_n$. Next note, since
$v>u$ and one of cases (1-4) is satisfied, equation (\ref{eqVn-Un}) and lemma
\ref{lemvnunIneq} imply that
$V_n \setminus U_n = VU^{(n)} \subset (\chi-2\xi, \chi+2\xi)$.
Finally, lemma \ref{lemDesAsc1-12Rem} and
definition \ref{defConCases} and figure \ref{figDesAscRem} clearly imply
that $\kappa_n$ contains $(\chi-2\xi, \chi+2\xi)$.
The above equation thus follows from the Residue theorem. This proves
part (2) for cases (1-4).

Consider (2) for cases (7,8) of lemma \ref{lemCases}.
First recall (see lemma \ref{lemPhin}) that,
\begin{equation*}
\Phi_n
= 1_{(v_n+s_n \le u_n+r_n, s_n \le r_n)} \sum_{y' \in VU_{(n)}}
\frac{\prod_{y \in U_n} (y' - y)}{\prod_{x \in V_n \setminus y'} (y' - x)}.
\end{equation*}
Next recall (see statement of lemma \ref{lemKnJnPhin}) that
$(u,r) \neq (v,s)$, i.e., that either $v > u$, or $v < u$,
or $v = u$ and $s \neq r$. Moreover, lemma \ref{lemvnunIneq}
implies that $1_{(v_n+s_n \le u_n+r_n, s_n \le r_n)} = 1$ when $v > u$, and
$1_{(v_n+s_n \le u_n+r_n, s_n \le r_n)} = 0$ when $v < u$.
Finally, when $v = u$ and $s \neq r$, lemma \ref{lemvnunIneq}
implies that either of the following is satisfied:
\begin{itemize}
\item
$v_n+s_n > u_n+r_n$ and $s_n > r_n+1$ and $v_n > u_n$ for all $n$
sufficiently large. In this case $1_{(v_n+s_n \le u_n+r_n, s_n \le r_n)} = 0$.
\item
$v_n+s_n < u_n+r_n$ and $s_n < r_n+1$ and $v_n < u_n$ for all $n$
sufficiently large. In this case equation (\ref{eqVn-Un})
gives $V_n \setminus U_n = VU_{(n)}$, and part (i) in the proof
of lemma \ref{lemKnSums} gives,
\begin{equation*}
\sum_{y' \in VU_{(n)}}
\frac{\prod_{y \in U_n} (y' - y)}{\prod_{x \in V_n \setminus y'} (y' - x)} = 0.
\end{equation*}
\end{itemize}
Combined, the above observations give,
\begin{equation*}
\Phi_n = 1_{(v>u)} \sum_{y' \in VU_{(n)}}
\frac{\prod_{y \in U_n} (y' - y)}{\prod_{x \in V_n \setminus y'} (y' - x)}.
\end{equation*}
We can then proceed similar to above to prove part (2) for cases (7,8).
\end{proof}

\subsection{Proof of theorem \ref{thmAiry}}

In this section we finally prove theorem \ref{thmAiry} using the results
of the previous sections. We will prove the result only when $t \in R_\mu^+$,
i.e, when one of cases (1-4) of lemma \ref{lemCases} is satisfied. The results
when $t \in R_{\l-\mu}$ (cases (5-8)), and when $t \in R_\mu^-$ (cases (9-12)),
follow from similar considerations.

Assume the conditions of theorem \ref{thmAiry}. Additionally assume
that one of cases, (1-4), of lemma \ref{lemCases} is satisfied.
Lemma \ref{lemKnJnPhin} thus gives,
\begin{equation}
\label{eqKnIntCase1}
\beta_n \; K_n((u_n,r_n),(v_n,s_n))
= J_n - \Phi_n.
\end{equation}
We begin by using a steepest descent argument to examine
the asymptotic behaviour of $J_n$. First, fix $\theta \in (\frac14,\frac13)$,
and $\{q_n\}_{n\ge1} \subset \R$ as in definition \ref{defmnpn}.
Next, using lemma \ref{lemDesAsc1-12}
and definition \ref{defConCases}, we partition $\g_n$ as follows:
\begin{equation}
\label{eqConLocRem}
\g_n = \g_n^{(l)} + \g_n^{(r)}
\hspace{.5cm} \text{and} \hspace{.5cm}
\G_n = \G_n^{(l)} + \G_n^{(r)},
\end{equation}
where $\g_n^{(l)}$ and $\G_n^{(l)}$ are (respectively) those {\em local
sections} of $\g_n$ and $\G_n$ inside $B(t, n^{-\theta} |q_n|)$, and
$\g_n^{(r)}$ and $\G_n^{(r)}$ are (respectively) the {\em remaining sections}
of $\g_n$ and $\G_n$ outside $B(t, n^{-\theta} |q_n|)$. Then, the definition
of $J_n$ in the statement of lemma \ref{lemKnJnPhin} gives,
\begin{equation}
\label{eqJn11Jn12}
J_n = J_n^{(l,l)} + J_n^{(l,r)} + J_n^{(r,l)} + J_n^{(r,r)},
\end{equation}
where,
\begin{equation*}
J_n^{(l,l)} := \frac1{(2\pi i)^2} \int_{\g_n^{(l)}} dw \int_{\G_n^{(l)}} dz \;
\frac{\prod_{j=u_n+r_n-n+1}^{u_n-1}
(z - \frac{j}n)}{\prod_{j=v_n+s_n-n}^{v_n} (w - \frac{j}n)} \;
\frac1{w-z} \; \prod_{i=1}^n \left( \frac{w - \frac{x_i}n}{z - \frac{x_i}n} \right).
\end{equation*}
The other three terms on the RHS of equation (\ref{eqJn11Jn12}) are defined
analogously. As we shall see in the following lemmas, the asymptotic
behaviour of $J_n^{(l,l)}$ dominates the other three terms:
\begin{lem}
\label{lemJn11Case1}
Assume the conditions of theorem \ref{thmAiry}. Additionally assume
that one of cases, (1-4), of lemma \ref{lemCases} is satisfied.
Fix $\theta \in (\frac14,\frac13)$, and $\{q_n\}_{n\ge1} \subset \R$
as in definition \ref{defmnpn}. Define $J_n^{(l,l)}$ as in equation
(\ref{eqJn11Jn12}), $\widetilde{K}_\text{Ai} : (\R^2)^2 \to \R$ as in
equation (\ref{eqAitilde}), and $A_{t,n} : (\Z^2)^2 \to \R \setminus \{0\}$
as in lemma \ref{lemFnt}. Then,
\begin{equation*}
\frac{n^\frac13 |q_n|^{-1} (t - \chi_n) (t - \chi_n - \eta_n +1)}
{A_{t,n}((u_n,r_n),(v_n,s_n))}
\; J_n^{(l,l)} \to \widetilde{K}_\text{Ai}((v,s),(u,r)).
\end{equation*}
\end{lem}

\begin{proof}
First note, equations (\ref{eqfn}, \ref{eqtildefn}, \ref{eqJn11Jn12}) give,
\begin{equation*}
J_n^{(l,l)} = \frac1{(2\pi i)^2} \int_{\g_n^{(l)}} dw \int_{\G_n^{(l)}} dz \;
\frac{\exp(n f_n(w) - n \tilde{f}_n(z))}{w-z}.
\end{equation*}
Define $d_{1,n}$ and $\tilde{a}_{1,n}$
as in lemma \ref{lemDesAsc1-12}. Also define,
\begin{align}
\label{eqArgd_1n}
\alpha_n
&:= \left\{ \begin{array}{rcl}
\text{Arg}(d_{1,n}-t)
& ; & \text{ for cases (1,2) of lemma \ref{lemCases}}, \\
\text{Arg}(\tilde{a}_{1,n}-t)
& ; & \text{ for cases (3,4) of lemma \ref{lemCases}},
\end{array} \right. \\
\nonumber
\zeta_n
&:= \left\{ \begin{array}{rcl}
\text{Arg}(\tilde{a}_{1,n}-t)
& ; & \text{ for cases (1,2) of lemma \ref{lemCases}}, \\
\text{Arg}(d_{1,n}-t)
& ; & \text{ for cases (3,4) of lemma \ref{lemCases}},
\end{array} \right.
\end{align}
where $\text{Arg}$ represents the principal value of the
argument, and note that part (2) of lemma \ref{lemDesAsc1-12} gives
$\alpha_n = \frac\pi3 + O(n^{-\frac13+\theta})$
and $\zeta_n = \frac{2\pi}3 + O(n^{-\frac13+\theta})$.
Recall that $\g_n^{(l)}$ and $\G_n^{(l)}$ are (respectively) those
sections of $\g_n$ and $\G_n$ inside $B(t, n^{-\theta} |q_n|)$
(see equation (\ref{eqConLocRem})), and $\g_n$ and $\G_n$
are counter-clockwise (see definition \ref{defConCases}). Lemma
\ref{lemDesAsc1-12} and figure \ref{figDesAsc1-12} then imply,
for cases (1,2) of lemma \ref{lemCases}, that:
\begin{itemize}
\item
$\g_n^{(l)}$ is the lines from
$t + n^{-\theta} |q_n| e^{-i \alpha_n}$ to $t$,
and from $t$ to $t + n^{-\theta} |q_n| e^{i \alpha_n}$.
\item
$\G_n^{(l)}$ is the lines from
$t + n^{-\theta} |q_n| e^{-i \zeta_n}$ to $t$,
and from $t$ to $t + n^{-\theta} |q_n| e^{i \zeta_n}$.
\end{itemize}
Moreover, for cases (3,4):
\begin{itemize}
\item
$\g_n^{(l)}$ is the lines from
$t + n^{-\theta} |q_n| e^{-i \zeta_n}$ to $t$,
and from $t$ to $t + n^{-\theta} |q_n| e^{i \zeta_n}$.
\item
$\G_n^{(l)}$ is the lines from
$t + n^{-\theta} |q_n| e^{i \alpha_n}$ to $t$,
and from $t$ to $t + n^{-\theta} |q_n| e^{-i \alpha_n}$.
\end{itemize}
A change of variables thus gives,
\begin{equation*}
J_n^{(l,l)} = \begin{dcases}
\frac{n^{-\frac13} |q_n|}{(2\pi i)^2} \int_{h_n} dw \int_{H_n} dz \;
\frac{\exp(n f_n(t + n^{-\frac13} |q_n| w) - n \tilde{f}_n(t + n^{-\frac13} |q_n| z))}{w-z}
& ; \text{ for (1,2)}, \\
\frac{n^{-\frac13} |q_n|}{(2\pi i)^2} \int_{h_n} dw \int_{H_n} dz \;
\frac{\exp(n f_n(t - n^{-\frac13} |q_n| w) - n \tilde{f}_n(t - n^{-\frac13} |q_n| z))}{w-z}
& ; \text{ for (3,4).}
\end{dcases}
\end{equation*}
Above, for cases (1-4):
\begin{itemize}
\item
$h_n$ is the lines from $n^{\frac13-\theta} e^{-i \alpha_n}$ to $0$,
and from $0$ to $n^{\frac13-\theta} e^{i \alpha_n}$.
\item
$H_n$ is the lines from $n^{\frac13-\theta} e^{-i \zeta_n}$ to $0$,
and from $0$ to $n^{\frac13-\theta} e^{i \zeta_n}$.
\end{itemize}
These contours are shown on the left of figure \ref{figConReScaleCase1}.

Next recall that $q_n > 0$ for cases (1,2) of lemma \ref{lemCases},
and $q_n<0$ for cases (3,4) (see definition \ref{defmnpn} and lemma
\ref{lemCases}). Therefore, for cases (1-4),
\begin{equation*}
J_n^{(l,l)} =
\frac{n^{-\frac13} |q_n|}{(2\pi i)^2} \int_{h_n} dw \int_{H_n} dz \;
\frac{\exp(n f_n(t + n^{-\frac13} q_n w) - n \tilde{f}_n(t + n^{-\frac13} q_n z))}{w-z}.
\end{equation*}
Parts (3,4) of corollary \ref{corTay} then give,
\begin{equation*}
J_n^{(l,l)}
= n^{-\frac13} |q_n| \exp(n f_n(t) - n \tilde{f}_n(t) + O(n^{1-4\theta})) \; I_n,
\end{equation*}
for cases (1-4), where
\begin{equation*}
I_n := \frac1{(2\pi i)^2} \int_{h_n} dw \int_{H_n} dz \; \frac1{w-z}
\frac{\exp(w s  + w^2 v + \frac13 w^3)}{\exp(z r  + z^2 u + \frac13 z^3)}.
\end{equation*}
Equation (\ref{eqFnw}) and lemma \ref{lemFnt} then give,
\begin{equation*}
J_n^{(l,l)}
= \frac{n^{-\frac13} |q_n| A_{t,n} ((u_n,r_n),(v_n,s_n))}
{(t-\chi_n) (t-\chi_n - \eta_n + 1)}
\exp(O(n^{-\frac13} + n^{1-4\theta})) \; I_n.
\end{equation*}
Next, recall that $\alpha_n = \frac\pi3 + O(n^{-\frac13+\theta})$
and $\zeta_n = \frac{2\pi}3 + O(n^{-\frac13+\theta})$ for cases (1-4),
and define:
\begin{itemize}
\item
$l_n$ is the lines from $n^{\frac13-\theta} e^{-i \frac\pi3}$ to $0$,
and from $0$ to $n^{\frac13-\theta} e^{i \frac\pi3}$.
$c_n$ is the smallest arcs of $\partial B(0, n^{\frac13-\theta})$
from $n^{\frac13-\theta} e^{-i \alpha_n}$ to $n^{\frac13-\theta} e^{-i \frac\pi3}$,
and from $n^{\frac13-\theta} e^{i \frac\pi3}$ to $n^{\frac13-\theta} e^{i \alpha_n}$.
\item
$L_n$ is the lines from $n^{\frac13-\theta} e^{-i \frac{2\pi}3}$ to $0$,
and from $0$ to $n^{\frac13-\theta} e^{i \frac{2\pi}3}$.
$C_n$ is the smallest arcs of $\partial B(0, n^{\frac13-\theta})$
from $n^{\frac13-\theta} e^{-i \zeta_n}$ to $n^{\frac13-\theta} e^{-i \frac{2\pi}3}$,
and from $n^{\frac13-\theta} e^{i \frac{2\pi}3}$ to $n^{\frac13-\theta} e^{i \zeta_n}$.
\end{itemize}
These contours are shown on the right of figure \ref{figConReScaleCase1}.
\begin{figure}[t]
\centering
\begin{tikzpicture};

\draw (-3.5,1.5) node {$\mathbb{H}$};
\draw [dotted] (-3,0) --++(6,0);
\draw (-3.5,0) node {$\R$};

\draw [dotted] (0,0) circle (2cm);
\draw [fill] (0,0) circle (.05cm);
\draw (0,-.3) node {\scriptsize $0$};
\draw (0,3.2) node {$B(0,n^{\frac13-\theta})$};
\draw[arrows=->,line width=0.5pt](0,2.9)--(0,2.1);
\draw (2,2.4) node {\scriptsize $n^{\frac13-\theta} e^{i \alpha_n}$};
\draw[arrows=->,line width=0.5pt](2,2.2)--(1.1,1.8);
\draw (-2,2.4) node {\scriptsize $n^{\frac13-\theta} e^{i \zeta_n}$};
\draw[arrows=->,line width=0.5pt](-2,2.2)--(-1.1,1.8);
\draw (2,-2.4) node {\scriptsize $n^{\frac13-\theta} e^{-i \alpha_n}$};
\draw[arrows=->,line width=0.5pt](2,-2.2)--(1.1,-1.8);
\draw (-2,-2.4) node {\scriptsize $n^{\frac13-\theta} e^{-i \zeta_n}$};
\draw[arrows=->,line width=0.5pt](-2,-2.2)--(-1.1,-1.8);

\draw (0,0) --++ (1,1.732);
\draw[arrows=->,line width=1pt](.5,0.866)--(.505,.875);
\draw (1,0.866) node {$h_n$};
\draw (0,0) --++ (1,-1.732);
\draw[arrows=->,line width=1pt](.505,-.875)--(.5,-0.866);
\draw (1,-0.866) node {$h_n$};

\draw (0,0) --++ (-1,1.732);
\draw[arrows=->,line width=1pt](-.5,0.866)--(-.505,.875);
\draw (-1,0.866) node {$H_n$};
\draw (0,0) --++ (-1,-1.732);
\draw[arrows=->,line width=1pt](-.505,-.875)--(-.5,-0.866);
\draw (-1,-0.866) node {$H_n$};

\draw [dotted] (4,0) --++(6,0);

\draw [dotted] (7,0) circle (2cm);
\draw [fill] (7,0) circle (.05cm);
\draw (7,-.3) node {\scriptsize $0$};
\draw (7,3.2) node {$B(0,n^{\frac13-\theta})$};
\draw[arrows=->,line width=0.5pt](7,2.9)--(7,2.1);
\draw (9.2,2.4) node {\scriptsize $n^{\frac13-\theta} e^{i \frac\pi3}$};
\draw[arrows=->,line width=0.5pt](9,2.2)--(8.5,1.5);
\draw (5,2.4) node {\scriptsize $n^{\frac13-\theta} e^{i \frac{2\pi}3}$};
\draw[arrows=->,line width=0.5pt](5,2.2)--(5.5,1.5);
\draw (9,-2.4) node {\scriptsize $n^{\frac13-\theta} e^{-i \frac\pi3}$};
\draw[arrows=->,line width=0.5pt](9,-2.2)--(8.5,-1.5);
\draw (5,-2.4) node {\scriptsize $n^{\frac13-\theta} e^{-i \frac{2\pi}3}$};
\draw[arrows=->,line width=0.5pt](5,-2.2)--(5.5,-1.5);

\draw (7,0) --++ (1,1.732);
\draw[arrows=->,line width=1pt](7.5,0.866)--(7.505,.875);
\draw (7.5,1.5) node {$h_n$};
\draw (7,0) --++ (1,-1.732);
\draw[arrows=->,line width=1pt](7.505,-.875)--(7.5,-0.866);
\draw (7.5,-1.5) node {$h_n$};

\draw (7,0) --++ (-1,1.732);
\draw[arrows=->,line width=1pt](6.5,0.866)--(6.495,.875);
\draw (6.6,1.5) node {$H_n$};
\draw (7,0) --++ (-1,-1.732);
\draw[arrows=->,line width=1pt](6.495,-.875)--(6.5,-0.866);
\draw (6.6,-1.5) node {$H_n$};

\draw (7,0) --++ (1.414,1.414);
\draw[arrows=->,line width=1pt](7.707,0.707)--(7.712,0.712);
\draw (8.1,0.707) node {$l_n$};
\draw (7,0) --++ (1.414,-1.414);
\draw[arrows=->,line width=1pt](7.712,-0.712)--(7.707,-0.707);
\draw (8.1,-0.707) node {$l_n$};

\draw (7,0) --++ (-1.414,1.414);
\draw[arrows=->,line width=1pt](6.293,0.707)--(6.288,0.712);
\draw (5.8,0.707) node {$L_n$};
\draw (7,0) --++ (-1.414,-1.414);
\draw[arrows=->,line width=1pt](6.288,-0.712)--(6.293,-0.707);
\draw (5.8,-0.707) node {$L_n$};

\draw [domain=45:60] plot ({7+2*cos(\x)}, {2*sin(\x)});
\draw[arrows=->,line width=1pt](8.227,1.579)--(8.220,1.585);
\draw (8.3,1.9) node {$c_n$};
\draw [domain=-60:-45] plot ({7+2*cos(\x)}, {2*sin(\x)});
\draw[arrows=->,line width=1pt](8.220,-1.585)--(8.227,-1.579);
\draw (8.3,-1.9) node {$c_n$};

\draw [domain=120:135] plot ({7+2*cos(\x)}, {2*sin(\x)});
\draw[arrows=->,line width=1pt](5.773,1.579)--(5.780,1.585);
\draw (5.7,1.9) node {$C_n$};
\draw [domain=-135:-120] plot ({7+2*cos(\x)}, {2*sin(\x)});
\draw[arrows=->,line width=1pt](5.780,-1.585)--(5.773,-1.579);
\draw (5.7,-1.9) node {$C_n$};

\end{tikzpicture}
\caption{Left: The contours $h_n$ and $H_n$. Recall
$\theta \in (\frac14, \frac13)$,
$\alpha_n = \frac\pi3 + O(n^{-\frac13+\theta})$ and
$\zeta_n = \frac{2\pi}3 + O(n^{-\frac13+\theta})$ (see equation
(\ref{eqArgd_1n})).
Right: The contours $h_n$, $H_n$, $l_n$, $L_n$, $c_n, C_n$.}
\label{figConReScaleCase1}
\end{figure}
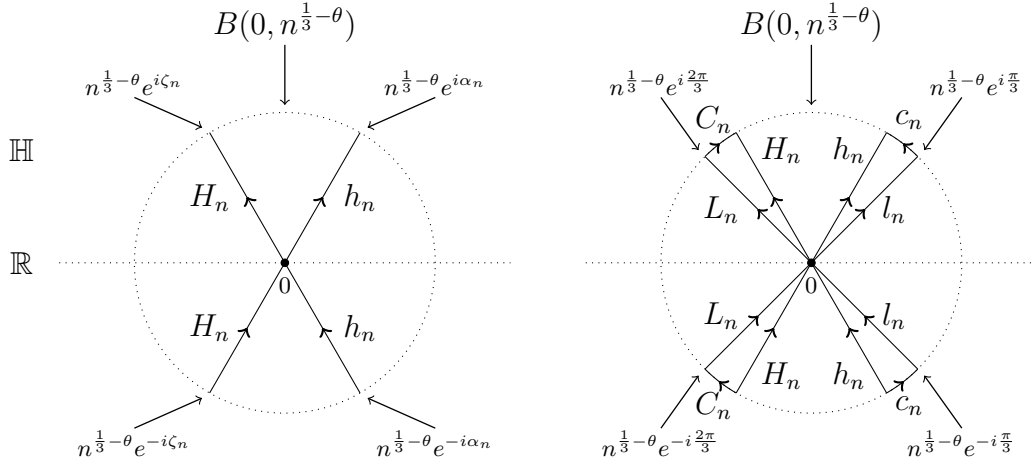
Then, noting that $h_n$ and $c_n + l_n$ have the same initial and final
points, and similarly for $H_n$ and $C_n + L_n$,
\begin{equation*}
I_n = I_{1,n} + I_{2,n} + I_{3,n} + I_{4,n},
\end{equation*}
where,
\begin{align}
\nonumber
I_{1,n}
&:= \frac1{(2\pi i)^2} \int_{l_n} dw \int_{L_n} dz \; \frac1{w-z}
\frac{\exp(w s  + w^2 v + \frac13 w^3)}{\exp(z r  + z^2 u + \frac13 z^3)}, \\
\label{eqCase1I2n1}
I_{2,n}
&:= \frac1{(2\pi i)^2} \int_{l_n} dw \int_{C_n} dz \; \frac1{w-z}
\frac{\exp(w s  + w^2 v + \frac13 w^3)}{\exp(z r  + z^2 u + \frac13 z^3)}, \\
\nonumber
I_{3,n}
&:= \frac1{(2\pi i)^2} \int_{c_n} dw \int_{L_n} dz \; \frac1{w-z}
\frac{\exp(w s  + w^2 v + \frac13 w^3)}{\exp(z r  + z^2 u + \frac13 z^3)}, \\
\nonumber
I_{4,n}
&:= \frac1{(2\pi i)^2} \int_{c_n} dw \int_{C_n} dz \; \frac1{w-z}
\frac{\exp(w s  + w^2 v + \frac13 w^3)}{\exp(z r  + z^2 u + \frac13 z^3)}.
\end{align}
Finally, we will show that,
\begin{enumerate}
\item[(i)]
$I_{1,n} \to \widetilde{K}_\text{Ai}((v,s),(u,r))$.
\item[(ii)]
$I_{2,n} \to 0$ and $I_{3,n} \to 0$ and $I_{4,n} \to 0$.
\end{enumerate}
Thus, since $\theta \in (\frac14,\frac13)$, the required result follows
from parts (i,ii) and the above expression of $J_n^{(l,l)}$.

Consider (i). First define:
\begin{itemize}
\item
$r_n$ is the lines from $\infty e^{-i \frac\pi3}$ to
$n^{\frac13-\theta} e^{-i \frac\pi3}$, and from
$n^{\frac13-\theta} e^{i \frac\pi3}$ to $\infty e^{i \frac\pi3}$.
\item
$R_n$ is the lines from $\infty e^{-i \frac{2\pi}3}$ to
$n^{\frac13-\theta} e^{-i \frac{2\pi}3}$, and from
$n^{\frac13-\theta} e^{i \frac{2\pi}3}$ to $\infty e^{i \frac{2\pi}3}$.
\end{itemize}
It thus follows from figures
\ref{figAirtCont} and \ref{figConReScaleCase1} that $l = l_n + r_n$ and
$L = L_n + R_n$. Equation (\ref{eqAitilde}) and
the definition of $I_{1,n}$, above, thus give,
\begin{equation*}
\widetilde{K}_\text{Ai}((v,s),(u,r)) = I_{1,n} + I_{1,n}' + I_{1,n}'' + I_{1,n}''',
\end{equation*}
where,
\begin{align}
\label{eqCase1I1n'1}
I_{1,n}'
&:= \frac1{(2\pi i)^2} \int_{l_n} dw \int_{R_n} dz \; \frac1{w-z}
\frac{\exp(w s  + w^2 v + \frac13 w^3)}{\exp(z r  + z^2 u + \frac13 z^3)}, \\
\nonumber
I_{1,n}''
&:= \frac1{(2\pi i)^2} \int_{r_n} dw \int_{L_n} dz \; \frac1{w-z}
\frac{\exp(w s  + w^2 v + \frac13 w^3)}{\exp(z r  + z^2 u + \frac13 z^3)}, \\
\nonumber
I_{1,n}'''
&:= \frac1{(2\pi i)^2} \int_{r_n} dw \int_{R_n} dz \; \frac1{w-z}
\frac{\exp(w s  + w^2 v + \frac13 w^3)}{\exp(z r  + z^2 u + \frac13 z^3)}.
\end{align}
Part (i) thus follows if we can show that $I_{1,n}' = o(1)$, $I_{1,n}'' = o(1)$,
and $I_{1,n}''' = o(1)$.

Consider $I_{1,n}'$, defined in equation (\ref{eqCase1I1n'1}).
The contours in this expression are given in figure \ref{figConI1n'}.
\begin{figure}[t]
\centering
\begin{tikzpicture}

\draw (-3.5,1.5) node {$\mathbb{H}$};
\draw [dotted] (-3,0) --++(6,0);
\draw (-3.5,0) node {$\R$};

\draw [dotted] (0,0) circle (2cm);
\draw [fill] (0,0) circle (.05cm);
\draw (0,-.3) node {\scriptsize $0$};
\draw (0,3.2) node {$B(0,n^{\frac13-\theta})$};
\draw[arrows=->,line width=0.5pt](0,2.9)--(0,2.1);

\draw (0,0) --++ (1,1.732);
\draw[arrows=->,line width=1pt](.75,1.299)--(.755,1.308);
\draw (1.1,1.2) node {$l_n$};
\draw (0,0) --++ (1,-1.732);
\draw[arrows=->,line width=1pt](.755,-1.308)--(.75,-1.299);
\draw (1.1,-1.2) node {$l_n$};

\draw [dotted] (0,0) --++ (-1,1.732);
\draw (-1,1.732) --++ (-.5,0.866);
\draw [dashed] (-1.5,2.598) --++ (-.5,0.866);
\draw[arrows=->,line width=1pt](-1.375,2.382)--(-1.38,2.39);
\draw (-2,2.598) node {$R_n$};
\draw [dotted] (0,0) --++ (-1,-1.732);
\draw (-1,-1.732) --++ (-.5,-0.866);
\draw [dashed] (-1.5,-2.598) --++ (-.5,-0.866);
\draw[arrows=->,line width=1pt](-1.38,-2.39)--(-1.375,-2.382);
\draw (-2,-2.598) node {$R_n$};

\draw [dotted,domain=-60:60] plot ({cos(\x)}, {sin(\x)});
\draw [dotted,domain=120:240] plot ({cos(\x)}, {sin(\x)});
\draw (.55,.35) node {$\frac\pi3$};
\draw (-.55,.35) node {$\frac\pi3$};
\draw (-.55,-.35) node {$\frac\pi3$};
\draw (.55,-.35) node {$\frac\pi3$};

\end{tikzpicture}
\caption{The contours $l_n$ and $R_n$.
$l_n$ is the lines from $n^{\frac13-\theta} e^{-i \frac\pi3}$ to $0$,
and from $0$ to $n^{\frac13-\theta} e^{i \frac\pi3}$.
$R_n$ is the lines from $\infty e^{-i \frac{2\pi}3}$ to
$n^{\frac13-\theta} e^{-i \frac{2\pi}3}$, and from
$n^{\frac13-\theta} e^{i \frac{2\pi}3}$ to $\infty e^{i \frac{2\pi}3}$.}
\label{figConI1n'}
\end{figure}
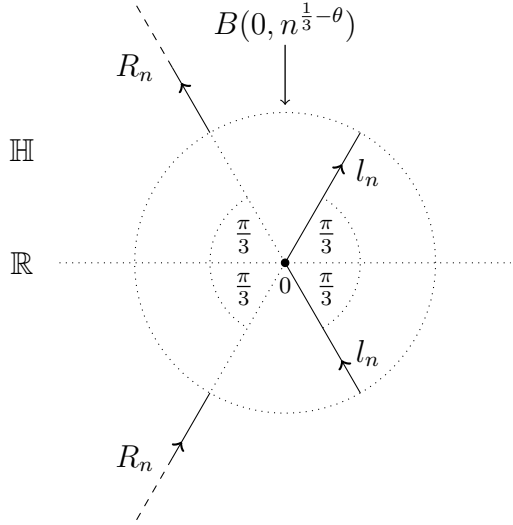
Note that,
\begin{equation}
\label{eqCase1I1n'2}
\frac1{|w-z|}
\le \frac1{ n^{\frac13-\theta} \cos (\frac{\pi}3)}
= \frac2{n^{\frac13-\theta}},
\end{equation}
uniformly for $w$ on $l_n$ and $z$ on $R_n$. Also,
\begin{equation*}
\left| \frac{\exp(w s  + w^2 v + \frac13 w^3)}
{\exp(z r  + z^2 u + \frac13 z^3)} \right|
= \frac{\exp(\text{Re}(w s  + w^2 v + \frac13 w^3))}
{\exp(\text{Re}(z r  + z^2 u + \frac13 z^3))},
\end{equation*}
for all $w$ on $l_n$ and $z$ on $R_n$. Moreover,
$|\text{Arg}(w)| = \frac{\pi}3$ and $|\text{Arg}(z)| = \frac{2\pi}3$
for all $w$ on $l_n$ and $z$ on $R_n$, and so
$\text{Re} (w^3) = - |w|^3$ and $\text{Re} (z^3) = |z|^3$. Therefore,
\begin{equation}
\label{eqCase1I1n'3}
\left| \frac{\exp(w s  + w^2 v + \frac13 w^3)}
{\exp(z r  + z^2 u + \frac13 z^3)} \right|
\le  \frac{\exp(|w| |s|  + |w|^2 |v| - \frac13 |w|^3)}
{\exp(-|z| |r| - |z|^2 |u| + \frac13 |z|^3)},
\end{equation}
for all $w$ on $l_n$ and $z$ on $R_n$.
Equations (\ref{eqCase1I1n'1}, \ref{eqCase1I1n'2}, \ref{eqCase1I1n'3})
then give,
\begin{equation*}
|I_{1,n}'|
\le 4 \frac1{(2 \pi)^2} \int_0^\infty dy_1 \int_0^\infty dy_2
\; \frac2{n^{\frac13-\theta}}
\frac{\exp(y_1 |s| + y_1^2 |v| - \frac13 y_1^3)}
{\exp(-y_2 |r| - y_2^2 |u| + \frac13 y_2^3)}.
\end{equation*}
The above integral converges, and so $I_{1,n}' = O( n^{-(\frac13-\theta)} )$.
Similarly it can be shown that $I_{1,n}'' = O( n^{-(\frac13-\theta)} )$ and
$I_{1,n}''' = O( n^{-(\frac13-\theta)} )$. Finally, since $\theta \in (\frac14,\frac13)$,
it follows that $I_{1,n}' = o(1)$ and $I_{1,n}'' = o(1)$ and $I_{1,n}''' = o(1)$.
This proves (i).

Consider (ii). Recall that $I_{2,n}$ is defined in equation (\ref{eqCase1I2n1}).
The contours in this expression are given in figure \ref{figConI2n}.
\begin{figure}[t]
\centering
\begin{tikzpicture};

\draw (-4,1.5) node {$\mathbb{H}$};
\draw [dotted] (-3.5,0) --++(7,0);
\draw (-4,0) node {$\R$};

\draw [dotted] (0,0) circle (2cm);
\draw [fill] (0,0) circle (.05cm);
\draw (0,-.3) node {\scriptsize $0$};
\draw (3.5,.5) node {$B(0,n^{\frac13-\theta})$};
\draw[arrows=->,line width=0.5pt](2.5,.5)--(2,.5);

\draw (0,0) --++ (1,1.732);
\draw[arrows=->,line width=1pt](.5,0.866)--(.505,.875);
\draw (.9,0.866) node {$l_n$};
\draw (0,0) --++ (1,-1.732);
\draw[arrows=->,line width=1pt](.505,-.875)--(.5,-0.866);
\draw (.9,-0.866) node {$l_n$};

\draw [domain=120:135] plot ({2*cos(\x)}, {2*sin(\x)});
\draw[arrows=->,line width=1pt](-1.227,1.579)--(-1.22,1.585);
\draw (-1,1.3) node {$C_n$};
\draw [domain=-135:-120] plot ({2*cos(\x)}, {2*sin(\x)});
\draw[arrows=->,line width=1pt](-1.22,-1.585)--(-1.227,-1.579);
\draw (-1,-1.3) node {$C_n$};

\draw (2,2.4) node {\scriptsize $n^{\frac13-\theta} e^{i \frac\pi3}$};
\draw[arrows=->,line width=0.5pt](2,2.2)--(1.1,1.8);
\draw (-1,2.7) node {\scriptsize $n^{\frac13-\theta} e^{i \zeta_n}$};
\draw[arrows=->,line width=0.5pt](-1,2.5)--(-1,1.8);
\draw (-2.7,1.414) node {\scriptsize $n^{\frac13-\theta} e^{i \frac{2\pi}3}$};
\draw[arrows=->,line width=0.5pt](-1.9,1.414)--(-1.5,1.414);
\draw (2,-2.4) node {\scriptsize $n^{\frac13-\theta} e^{-i \frac\pi3}$};
\draw[arrows=->,line width=0.5pt](2,-2.2)--(1.1,-1.8);
\draw (-1,-2.7) node {\scriptsize $n^{\frac13-\theta} e^{-i \zeta_n}$};
\draw[arrows=->,line width=0.5pt](-1,-2.5)--(-1,-1.8);
\draw (-2.7,-1.414) node {\scriptsize $n^{\frac13-\theta} e^{-i \frac{2\pi}3}$};
\draw[arrows=->,line width=0.5pt](-1.9,-1.414)--(-1.5,-1.414);

\end{tikzpicture}
\caption{The contours $l_n$ and $C_n$.
$l_n$ is the lines from $n^{\frac13-\theta} e^{-i \frac\pi3}$ to $0$,
and from $0$ to $n^{\frac13-\theta} e^{i \frac\pi3}$.
$C_n$ is the smallest arcs of $\partial B(0, n^{\frac13-\theta})$
from $n^{\frac13-\theta} e^{-i \zeta_n}$ to $n^{\frac13-\theta} e^{-i \frac{2\pi}3}$,
and from $n^{\frac13-\theta} e^{i \frac{2\pi}3}$ to $n^{\frac13-\theta} e^{i \zeta_n}$.}
\label{figConI2n}
\end{figure}
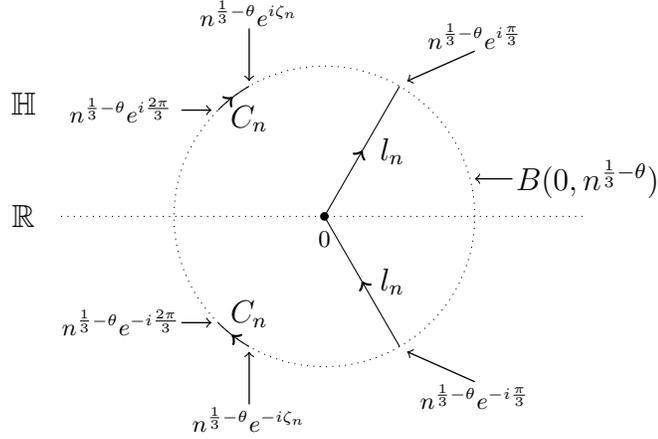
Also recall that $\zeta_n = \frac{2\pi}3 + O(n^{-\frac13+\theta})$.
It thus follows that $|\text{Arg}(w)| = \frac{\pi}3$ for all $w$ on $l_n$,
and $|\text{Arg}(z)| = \frac{2\pi}3 + O(n^{-\frac13+\theta})$
uniformly for $z$ on $C_n$. Moreover,
\begin{equation}
\label{eqCase1I2n2}
\frac1{|w-z|}
\le \frac1{ n^{\frac13-\theta} \cos (\frac{\pi}3 + o(1))}
< \frac4{n^{\frac13-\theta}},
\end{equation}
uniformly for $w$ on $l_n$ and $z$ on $C_n$. Also,
\begin{equation*}
\left| \frac{\exp(w s  + w^2 v + \frac13 w^3)}{\exp(z r  + z^2 u + \frac13 z^3)} \right|
= \frac{\exp(\text{Re}(w s  + w^2 v + \frac13 w^3))}
{\exp(\text{Re}(z r  + z^2 u + \frac13 z^3))},
\end{equation*}
for all $w$ on $l_n$ and $z$ on $C_n$.
Moreover, $|\text{Arg}(w)| = \frac{\pi}3$ and
$|\text{Arg}(z)| = \frac{2\pi}3 + O(n^{-\frac13+\theta})$ uniformly for
$w$ on $l_n$ and $z$ on $C_n$, and
so $\text{Re} (w^3) = - |w|^3$ and $\text{Re} (z^3) > \frac12 |z|^3$
for all such $w,z$. Therefore,
\begin{equation*}
\left| \frac{\exp(w s  + w^2 v + \frac13 w^3)}
{\exp(z r  + z^2 u + \frac13 z^3)} \right|
< \frac{\exp(|w| |s|  + |w|^2 |v| - \frac13 |w|^3)}
{\exp(-|z| |r| - |z|^2 |u| + \frac16 |z|^3)}.
\end{equation*}
Finally, note that $|z| = n^{\frac13-\theta}$ for all $z$ on $C_n$,
and that $n^{\frac13-\theta} \to \infty$ since
$\theta \in (\frac14,\frac13)$. Therefore,
\begin{equation}
\label{eqCase1I2n3}
\left| \frac{\exp(w s  + w^2 v + \frac13 w^3)}
{\exp(z r  + z^2 u + \frac13 z^3)} \right|
< \frac{\exp(|w| |s|  + |w|^2 |v| - \frac13 |w|^3)}
{\exp(\frac1{12} n^{1-3\theta})},
\end{equation}
uniformly for $w$ on $l_n$ and $z$ on $C_n$. Equations
(\ref{eqCase1I2n1}, \ref{eqCase1I2n2}, \ref{eqCase1I2n3}) then give,
\begin{equation*}
|I_{2,n}| < 4 \frac1{(2 \pi)^2}
\int_0^\infty dy_1 \int_0^{2\pi} n^{\frac13-\theta} dy_2 \; 
\frac4{n^{\frac13-\theta}} \frac{\exp(y_1 |s| + y_1^2 |v| - \frac13 y_1^3)}
{\exp(\frac1{12} n^{1-3\theta})}.
\end{equation*}
The above integral converges, and so $I_{2,n} = O( \exp(-\frac1{12} n^{1-3\theta}) )$.
Therefore $I_{2,n} = o(1)$ since $\theta \in (\frac14,\frac13)$.
Similarly, we can show that $I_{3,n} = o(1)$ and $I_{4,n} = o(1)$. This proves (ii).
\end{proof}

Next we examine the asymptotic behaviour of the remaining terms
of equation (\ref{eqJn11Jn12}):
\begin{lem}
\label{lemJn12Case1}
Assume the conditions of theorem \ref{thmAiry}. Additionally assume
that one of cases, (1-4), of lemma \ref{lemCases} is satisfied.
Fix $\theta \in (\frac14,\frac13)$, and $\{q_n\}_{n\ge1} \subset \R$
as in definition \ref{defmnpn}. Define
$J_n^{(l,r)}, J_n^{(r,l)}, J_n^{(r,r)}$ as in equation
(\ref{eqJn11Jn12}), and $A_{t,n} : (\Z^2)^2 \to \R \setminus \{0\}$ as
in lemma \ref{lemFnt}. Then,
\begin{equation*}
\frac{n^\frac13 |q_n|^{-1} |t - \chi_n| \; |t - \chi_n - \eta_n +1|}
{|A_{t,n}((u_n,r_n),(v_n,s_n))|} \; |J_n^{(l,r)}|
= O \bigg( n^\frac13  \exp \bigg( - \frac1{12} n^{1-3\theta} \bigg) \bigg).
\end{equation*}
Similarly for $J_n^{(r,l)}$ and $J_n^{(r,r)}$.
\end{lem}

\begin{proof}
We prove the result for $J_n^{(l,r)}$ for cases (1,2) of lemma
\ref{lemCases}, and state that the remaining results follow similarly.

Assume that one of cases, (1,2), of lemma \ref{lemCases} is satisfied.
Lemma \ref{lemCases} and definition \ref{defmnpn} then imply that
$q_n$ converges to a positive constant as $n \to \infty$.
Next note, equations (\ref{eqfn}, \ref{eqtildefn},
\ref{eqJn11Jn12}) give,
\begin{equation*}
|J_n^{(l,r)}| \le \frac1{(2\pi)^2} |\g_n^{(l)}| |\G_n^{(r)}|
\sup_{(w,z) \in \g_n^{(l)} \times \G_n^{(r)}}
\left| \frac{\exp(n f_n(w) - n \tilde{f}_n(z))}{w-z} \right|.
\end{equation*}
Recall (see equation (\ref{eqConLocRem})) that $\g_n^{(l)}$ is that 
section of $\g_n$ inside $B(t, n^{-\theta} q_n)$.
Therefore $|\g_n^{(l)}| = 2 n^{-\theta} q_n$ (see part (2) of lemma
\ref{lemDesAsc1-12} and definition \ref{defConCases}).
Thus, since $q_n$ converges to a positive constant as $n \to \infty$,
$|\g_n^{(l)}| = O(n^{-\theta})$. Next recall (see equation
(\ref{eqConLocRem})) that $\G_n^{(r)}$ is that part of $\G_n$ outside
$B(t, n^{-\theta} q_n)$. Definition \ref{defConCases} and parts
(4,5,6) of lemma \ref{lemDesAsc1-12} thus give,
\begin{equation*}
|J_n^{(l,r)}| \le C \sup_{w \in \g_n^{(l)}}
|\exp(n f_n(w) - n \tilde{f}_n(\tilde{a}_{1,n}))|,
\end{equation*}
where $\tilde{a}_{1,n} \subset \partial B(t, n^{-\theta} q_n)$
is defined in part (2) of lemma \ref{lemDesAsc1-12},
and where $C>0$ is some fixed constant. Recall that
$q_n$ converges to a positive constant as $n \to \infty$,
$\g_n^{(l)} \subset B(t, n^{-\theta} q_n)$ and
$\tilde{a}_{1,n} \in \partial B(t, n^{-\theta} q_n)$. Also
recall (see previous lemma), since one of cases (1,2) of lemma
\ref{lemCases} is satisfied,
$|\text{Arg}(w-t)| = \frac\pi3 + O(n^{-\frac13+\theta})$
uniformly for $w$ on $\g_n^{(l)}$, and
$\text{Arg}(\tilde{a}_{1,n}-t) = \frac{2\pi}3 + O(n^{-\frac13+\theta})$.
Therefore,
\begin{equation*}
|J_n^{(l,r)}| \le C \sup_{w \in h_n}
\left| \exp(n f_n(t + n^{-\frac13} q_n w)
- n \tilde{f}_n(t + n^{-\frac13} q_n z_n)) \right|,
\end{equation*}
where $h_n \subset B(0,n^{\frac13-\theta})$
is defined as on the left of figure \ref{figConReScaleCase1}, and
$z_n \subset \partial B(0,n^{\frac13-\theta})$ is defined
by $z_n := n^{\frac13} q_n^{-1} (\tilde{a}_{1,n} - t)$.
Note, $|\text{Arg}(w)| = \frac\pi3 + O(n^{-\frac13+\theta})$
uniformly for $w$ on $h_n$, and
$\text{Arg}(z_n) = \frac{2\pi}3 + O(n^{-\frac13+\theta})$. Also note, parts
(3,4) of corollary \ref{corTay} give,
\begin{equation*}
|J_n^{(l,r)}|
\le C | \exp(n f_n(t) - n \tilde{f}_n(t) + O(n^{1-4\theta}) ) |
\sup_{w \in h_n} \left| \frac{\exp(w s  + w^2 v + \frac13 w^3)}
{\exp(z_n r  + (z_n)^2 u + \frac13 (z_n)^3)} \right|.
\end{equation*}
Lemma \ref{lemFnt} then gives,
\begin{align*}
|J_n^{(l,r)}|
&\le C \frac{ |A_{t,n} ((u_n,r_n),(v_n,s_n))|}{|t-\chi_n| \; |t-\chi_n - \eta_n + 1|}
\exp(O(n^{-\frac13} + n^{1-4\theta})) \\
&\times \sup_{w \in h_n} \left| \frac{\exp(w s  + w^2 v + \frac13 w^3)}
{\exp(z_n r + (z_n)^2 u + \frac13 (z_n)^3)} \right|.
\end{align*}
Finally recall that $\theta \in (\frac14,\frac13)$. Therefore we can choose
the constant $C>0$ such that,
\begin{equation}
\label{eqJn12Case1}
|J_n^{(l,r)}|
\le C \frac{ |A_{t,n} ((u_n,r_n),(v_n,s_n))|}{|t-\chi_n| \; |t-\chi_n - \eta_n + 1|}
\sup_{w \in h_n} \left| \frac{\exp(w s  + w^2 v + \frac13 w^3)}
{\exp(z_n r  + (z_n)^2 u + \frac13 (z_n)^3)} \right|.
\end{equation}

Next note that,
\begin{equation*}
\left| \frac{\exp(w s  + w^2 v + \frac13 w^3)}
{\exp(z_n r  + (z_n)^2 u + \frac13 (z_n)^3)} \right|
= \frac{\exp(\text{Re}(w s  + w^2 v + \frac13 w^3))}
{\exp(\text{Re}(z_n r  + (z_n)^2 u + \frac13 (z_n)^3))},
\end{equation*}
for all $w$ on $h_n$. Recall that
$|\text{Arg}(w)| = \frac\pi3 + O(n^{-\frac13+\theta})$
uniformly for $w$ on $h_n$. Therefore $\text{Re}(w^3) \le 0$
for all $w$ on $h_n$. Moreover, recall that
$\text{Arg}(z_n) = \frac{2\pi}3 + O(n^{-\frac13+\theta})$. Therefore
$\text{Re} ((z_n)^3) > \frac12 |z_n|^3$. Combine the above to get,
\begin{equation*}
\left| \frac{\exp(w s  + w^2 v + \frac13 w^3)}
{\exp(z_n r  + (z_n)^2 u + \frac13 (z_n)^3)} \right|
\le \frac{\exp(|w| |s|  + |w|^2 |v| + 0)}
{\exp(-|z_n| |r| - |z_n|^2 |u| + \frac16 |z_n|^3)},
\end{equation*}
for all $w$ on $h_n$. Finally
recall that $|w| \le n^{\frac13-\theta}$ for all $w$ on $h_n$
(since $h_n \subset B(0, n^{\frac13-\theta})$),
$|z_n| = n^{\frac13-\theta}$ (since
$z_n \in \partial B(0, n^{\frac13-\theta})$), and
$n^{\frac13-\theta} \to \infty$ (since
$\theta \in (\frac14,\frac13)$). Therefore,
\begin{equation*}
\sup_{w \in h_n} \left| \frac{\exp(w s  + w^2 v + \frac13 w^3)}
{\exp(z_n r  + (z_n)^2 u + \frac13 (z_n)^3)} \right|
\le \frac{\exp(O(n^{\frac23 - 2\theta}))}
{\exp(O(n^{\frac23 - 2\theta}) + \frac16 n^{1 - 3\theta})}
\le \frac1{\exp( \frac1{12} n^{1 - 3\theta} )}.
\end{equation*}
Substitute into equation (\ref{eqJn12Case1})
to get the required result.
\end{proof}

Equation (\ref{eqJn11Jn12}), and the above two lemmas, give the asymptotic
behaviour of $J_n$. It remains to consider the asymptotic behaviour
of $\Phi_n$ (see equation (\ref{eqKnIntCase1})). Since
many of the arguments of the following lemma are similar to those used in
the proofs of lemmas \ref{lemJn11Case1} and \ref{lemJn12Case1}, we do not
go into as much detail here:
\begin{lem}
\label{lemRemv>uCase1}
Assume the conditions of theorem \ref{thmAiry}. Additionally assume
that one of cases, (1-4), of lemma \ref{lemCases} is satisfied.
Fix $\theta \in (\frac14,\frac13)$, and $\{q_n\}_{n\ge1} \subset \R$
as in definition \ref{defmnpn}. Define $\Phi_n$ as in lemma
\ref{lemKnJnPhin}, $\Phi :  (\R^2)^2 \to \R$ as in equation
(\ref{eqPhi}), and $A_{t,n} : (\Z^2)^2 \to \R \setminus \{0\}$ as
in lemma \ref{lemFnt}. Then, when the parameters of
equations (\ref{equnrnvnsn2}, \ref{equnrnvnsn3}) satisfy $(u,r) \neq (v,s)$,
\begin{equation*}
\frac{n^\frac13 |q_n|^{-1} (t - \chi_n) (t - \chi_n - \eta_n +1)}
{A_{t,n}((u_n,r_n),(v_n,s_n))} \; \Phi_n \to 1_{(v>u)} \; \Phi((v,s),(u,r)).
\end{equation*}
\end{lem}

\begin{proof}
First note, since $(u,r) \neq (v,s)$, part (2) of lemma \ref{lemKnJnPhin} gives,
\begin{equation*}
\Phi_n = 1_{(v>u)} \frac1{2\pi i} \int_{\kappa_n} dw \;
\frac{\prod_{j=u_n+r_n-n+1}^{u_n-1} (w - \frac{j}n)}
{\prod_{j=v_n+s_n-n}^{v_n} (w - \frac{j}n)}.
\end{equation*}
This proves the result when $v < u$, and when
$v = u$ and $s \neq r$. It thus remains to show the
result when $v > u$. Assuming this, partition $\kappa_n$ (see definition
\ref{defConCases} and lemma \ref{lemDesAsc1-12Rem}) as 
$\kappa_n = \kappa_n^{(l)} + \kappa_n^{(r)}$,
where $\kappa_n^{(l)}$ is that {\em local section} of $\kappa_n$ inside
$B(t, n^{-\theta} |q_n| )$, and $\kappa_n^{(l)}$ is the {\em remaining section}
of $\kappa_n$ outside $B(t, n^{-\theta} |q_n| )$. Therefore
$\Phi_n = \Phi_n^{(l)} + \Phi_n^{(r)}$, where,
\begin{equation*}
\Phi_n^{(l)} := \frac1{2\pi i} \int_{\kappa_n^{(l)}} dw \;
\frac{\prod_{j=u_n+r_n-n+1}^{u_n-1} (w - \frac{j}n)}
{\prod_{j=v_n+s_n-n}^{v_n} (w - \frac{j}n)},
\end{equation*}
and $\Phi_n^{(r)}$ is defined analogously. We will show that:
\begin{align}
\tag{i}
\frac{n^\frac13 |q_n|^{-1} (t - \chi_n) (t - \chi_n - \eta_n +1)}
{A_{t,n}((u_n,r_n),(v_n,s_n))} \; \Phi_n^{(l)}
&\to \frac1{2 \sqrt{\pi(v-u)}}
\exp \left( -\frac14 \frac{(s-r)^2}{v-u} \right), \\
\tag{ii}
\frac{n^\frac13 |q_n|^{-1} (t - \chi_n) (t - \chi_n - \eta_n +1)}
{A_{t,n}((u_n,r_n),(v_n,s_n))} \; \Phi_n^{(r)}
&\to 0.
\end{align}
Parts (i,ii), and equation (\ref{eqPhi}), then give the required result.

Consider (i). First note, equations
(\ref{eqfn}, \ref{eqtildefn}, \ref{eqFnw}) give,
\begin{equation*}
\Phi_n^{(l)}
= \frac1{2\pi i}  \int_{\kappa_n^{(l)}} dw \; \exp (n F_n(w) ).
\end{equation*}
Define $D_{1,n}$ as in lemma \ref{lemDesAsc1-12Rem}. Also define,
$\psi_n := \text{Arg}(D_{1,n}-t)$, and note that lemma
\ref{lemDesAsc1-12Rem} gives $\psi_n = \frac\pi2 + O(n^{-\frac13+\theta})$.
Recall that $\kappa_n^{(l)}$ is that section of $\kappa_n$ inside
$B(t, n^{-\theta} |q_n|)$, and $\kappa_n$ is counter-clockwise
(see definition \ref{defConCases}). Lemma \ref{lemDesAsc1-12Rem} and
figure \ref{figDesAscRem} then imply that $\kappa_n^{(l)}$ is the lines
from $t + n^{-\theta} |q_n| e^{-i \psi_n}$ to $t$, and from $t$ to
$t + n^{-\theta} |q_n| e^{i \psi_n}$. A change of variables thus gives,
\begin{equation*}
\Phi_n^{(l)} = \frac{n^{-\frac13} |q_n|}{2\pi i}
\int_{k_n} dw \; \exp (n F_n(t + n^{-\frac13} |q_n| w) ),
\end{equation*}
where $k_n$ is the lines from $n^{\frac13-\theta} e^{-i \psi_n}$ to $0$,
and from $0$ to $n^{\frac13-\theta} e^{i \psi_n}$.
Thus, since $q_n > 0$ for cases (1,2) of lemma \ref{lemCases},
and $q_n<0$ for cases (3,4),
\begin{equation*}
\Phi_n^{(l)}
= \begin{dcases}
\frac{n^{-\frac13} |q_n|}{2\pi i}
\int_{k_n} dw \; \exp (n F_n(t + n^{-\frac13} q_n w) )
& ; \text{ for cases (1,2)}, \\
\frac{n^{-\frac13} |q_n|}{2\pi i}
\int_{k_n} dw \; \exp (n F_n(t - n^{-\frac13} q_n w) )
& ; \text{ for cases (3,4)}.
\end{dcases}
\end{equation*}
Equation (\ref{eqFnw}), and parts (3,4) of corollary \ref{corTay}, then give,
\begin{equation*}
\Phi_n^{(l)} = n^{-\frac13} |q_n| \exp(n F_n(t) + O(n^{1-4\theta})) \; \phi_n,
\end{equation*}
where,
\begin{equation*}
\phi_n
:= \begin{dcases}
\frac1{2\pi i} \int_{k_n} dw \; \exp(w (s-r)  + w^2 (v-u))
& ; \text{ for cases (1,2)}, \\
\frac1{2\pi i} \int_{k_n} dw \; \exp((-w) (s-r)  + (-w)^2 (v-u))
& ; \text{ for cases (3,4)}.
\end{dcases}
\end{equation*}
Lemma \ref{lemFnt} then gives,
\begin{equation*}
\Phi_n^{(l)} = \frac{n^{-\frac13} |q_n| A_{t,n} ((u_n,r_n),(v_n,s_n))}
{(t-\chi_n) (t-\chi_n - \eta_n + 1)}
\exp(O(n^{-\frac13} + n^{1-4\theta})) \; \phi_n.
\end{equation*}
Note that $\exp(O(n^{-\frac13} + n^{1-4\theta})) \to 1$ since $\theta \in (\frac14,\frac13)$.
Also, since $v > u$ and $\theta \in (\frac14,\frac13)$, we can proceed similarly to
lemma \ref{lemJn11Case1} to show that,
\begin{equation*}
\phi_n \to \frac1{2\pi i}
\int_{-\infty}^{+\infty} d (i y) \exp ( i y (s-r) + (i y)^2 (v-u) ),
\end{equation*}
for cases (1,2). The RHS is integrable since $v > u$, and
integrating gives,
\begin{equation*}
\phi_n \to \frac1{2 \sqrt{\pi(v-u)}}
\exp \left( -\frac14 \frac{(s-r)^2}{v-u} \right).
\end{equation*}
for cases (1,2). This proves part (i) for
cases (1,2). Similarly for cases (3,4).

Consider (ii). First note, equations
(\ref{eqfn}, \ref{eqtildefn}, \ref{eqFnw}) give,
\begin{equation*}
\Phi_n^{(r)}
= \frac1{2\pi i} \int_{\kappa_n^{(r)}} dw \; \exp (n F_n(w) ),
\end{equation*}
where $\kappa_n^{(r)}$ is that section of $\kappa_n$ outside
$B(t, n^{-\theta} |q_n|)$. Therefore,
\begin{equation*}
| \Phi_n^{(r)} |
\le \frac1{2\pi} |\kappa_n^{(r)}|
\; \sup_{w \in \kappa_n^{(r)}} | \exp (n F_n(w) ) |
\le C \;| \exp (n F_n(D_{1,n}) ) |,
\end{equation*}
where $C>0$ is some fixed constant, and the second part above
follows from lemma \ref{lemDesAsc1-12Rem}. We can then proceed
similarly to the proof of lemma \ref{lemJn12Case1} to prove part (ii).
\end{proof}

We are finally ready to prove the main result of this section:
\begin{proof}[Proof of theorem \ref{thmAiry} when $t \in R_\mu^+$:]
In this proof, for brevity, denote
$a_n := (t - \chi_n) (t - \chi_n - \eta_n +1)$.
First recall that $t \in R_\mu^+$ if and only if one of cases (1-4) of
lemma \ref{lemCases} is satisfied. Then, equations (\ref{eqConjFactB})
and (\ref{eqKnIntCase1}), and the expression for $\beta_n$ in the
statement of lemma \ref{lemKnJnPhin} give,
\begin{equation*}
\frac{n}{n-r_n} \; \frac{K_n((u_n,r_n),(v_n,s_n))}{B_n(r_n,s_n)}
= J_n - \Phi_n.
\end{equation*}
Equation (\ref{eqKntilde}) then gives,
\begin{equation*}
\mathcal{K}_n((v_n,s_n),(u_n,r_n))
= \frac{n-r_n}n \; \frac{J_n - \Phi_n}{A_{t,n} ((u_n,r_n),(v_n,s_n))}.
\end{equation*}
Then, equation (\ref{eqJn11Jn12}), and lemmas \ref{lemJn11Case1} and
\ref{lemJn12Case1} give,
\begin{align}
\label{eqthmAiry1}
\mathcal{K}_n((v_n,s_n),(u_n,r_n))
&= n^{-\frac13} |q_n| a_n^{-1} \tfrac{n-r_n}n
(\widetilde{K}_\text{Ai}((v,s),(u,r)) + o(1)) \\
\nonumber
&- \frac{n-r_n}n \frac{\Phi_n}{A_{t,n} ((u_n,r_n),(v_n,s_n))}.
\end{align}
Moreover, since one of cases (1-4) is satisfied, the expression
for $\Phi_n$ in part (1) of lemma \ref{lemKnJnPhin} gives,
\begin{align}
\label{eqthmAiry2}
&\frac{n-r_n}n \; \frac{\Phi_n}{A_{t,n} ((u_n,r_n),(v_n,s_n))} \\
\nonumber
&= \frac{1_{(v_n \ge u_n, s_n > r_n)}}{A_{t,n} ((u_n,r_n),(v_n,s_n))} \;
\frac{(v_n+s_n-u_n-r_n-1)!}{(s_n-r_n-1)! (v_n-u_n)!} \;
\frac{(n-r_n)!}{(n-s_n)!} n^{r_n-s_n}.
\end{align}
Alternatively, when $(u,r) \neq (v,s)$, lemma \ref{lemRemv>uCase1} gives,
\begin{equation}
\label{eqthmAiry3}
\frac{n-r_n}n \; \frac{\Phi_n}{A_{t,n} ((u_n,r_n),(v_n,s_n))}
= n^{-\frac13} |q_n| a_n^{-1} \frac{n-r_n}n
(1_{(v>u)} \Phi((v,s),(u,r)) + o(1)).
\end{equation}

Recall, for cases (1-4), that $f_t'''(t) \neq 0$, $e^{C(t)} > 1$,
and $C'(t) < 0$ (see lemma \ref{lemAnalExt}). Moreover, defining $f_{t,n}'$ as in
equation (\ref{eqftn'Rmu}), recall (see lemma \ref{lemftn'}) that
$f_{t,n}'''(t) \to f_t'''(t)$ as $n \to \infty$.
Definition \ref{defmnpn} then gives,
\begin{equation*}
|q_n|
= 2^{\frac13} |f_{t,n}'''(t)|^{-\frac13}
\to 2^{\frac13} |f_t'''(t)|^{-\frac13},
\end{equation*}
a non-zero constant. Also, recalling that
$a_n = (t - \chi_n) (t - \chi_n - \eta_n +1)$,
and $(\chi_n,\eta_n) = (\chi_n(t),\eta_n(t))$,
definition \ref{defEdgeNonAsy} gives,
\begin{equation*}
a_n = \frac{(e^{C_n(t)}-1)^2}{e^{C_n(t)} C_n'(t)^2}.
\end{equation*}
Moreover, equation (\ref{equnrnvnsn2}) and definition \ref{defEdgeNonAsy} give,
\begin{equation*}
\frac{n-r_n}n
= 1 - \eta_n + O(n^{-\frac13})
= - \frac{(e^{C_n(t)}-1)^2}{e^{C_n(t)} C_n'(t)} + O(n^{-\frac13}).
\end{equation*}
Thus, since $e^{C_n(t)} \to e^{C(t)} \not\in \{0,1\}$, and
$C_n'(t) \to C'(t) \neq 0$ (see equation (\ref{eqNonAsyEdge})),
\begin{equation*}
a_n \to \frac{(e^{C(t)}-1)^2}{e^{C(t)} C'(t)^2}
\hspace{.5cm} \text{and} \hspace{.5cm}
\frac{n-r_n}n
\to - \frac{(e^{C(t)}-1)^2}{e^{C(t)} C'(t)},
\end{equation*}
non-zero constants. Combined, the above give,
\begin{equation}
\label{eqthmAiry4}
|q_n| a_n^{-1} \tfrac{n-r_n}n
\to - 2^{\frac13} |f_t'''(t)|^{-\frac13} C'(t)
= 2^{\frac13} |f_t'''(t)|^{-\frac13} |C'(t)|
= \beta(t),
\end{equation}
where the second part follows since $C'(t) < 0$,
the last part follows from the definition of $\beta(t)$
in the statement of theorem \ref{thmAiry2}.

Note, equations (\ref{eqthmAiry1}, \ref{eqthmAiry4}) give,
\begin{align*}
\mathcal{K}_n((v_n,s_n),(u_n,r_n))
&= n^{-\frac13} (\beta(t) + o(1))
(\widetilde{K}_\text{Ai}((v,s),(u,r)) + o(1)) \\
&- \frac{n-r_n}n \frac{\Phi_n}{A_{t,n} ((u_n,r_n),(v_n,s_n))}.
\end{align*}
Recall that $t > \chi_n > \chi_n + \eta_n - 1$ for cases (1-4)
of lemma \ref{lemCases}, and so $A_{t,n} ((u_n,r_n),(v_n,s_n)) > 0$ (see lemma
\ref{lemFnt}). Equations (\ref{eqConjFactB}, \ref{eqthmAiry2}) then give,
\begin{equation*}
\frac{n-r_n}n \; \frac{\Phi_n}{A_{t,n} ((u_n,r_n),(v_n,s_n))}
= \frac{1_{(v_n \ge u_n, s_n > r_n)}}{|A_{t,n} ((u_n,r_n),(v_n,s_n))| B(r_n,s_n)} \;
\frac{(v_n+s_n-u_n-r_n-1)!}{(s_n-r_n-1)! (v_n-u_n)!}.
\end{equation*}
Moreover, when $(u,r) \neq (v,s)$, equations (\ref{eqthmAiry3},
\ref{eqthmAiry4}) alternatively give,
\begin{equation*}
\frac{n-r_n}n \; \frac{\Phi_n}{A_{t,n} ((u_n,r_n),(v_n,s_n))}
= n^{-\frac13} (\beta(t) + o(1)) (1_{(v>u)} \Phi((v,s),(u,r)) + o(1)).
\end{equation*}
Swap $(u_n,r_n)$ and $(v_n,s_n)$ in the
above expressions to get the required result for cases (1-4) of
lemma \ref{lemCases}, i.e., when $t \in R_\mu^+$.
\end{proof}

\section{Existence of appropriate contours of descent/ascent}
\label{secCont}

We finally prove lemmas \ref{lemDesAsc1-12} and \ref{lemDesAsc1-12Rem}.

\subsection{Lemma \ref{lemDesAsc1-12} for case (1) of lemma \ref{lemCases}}
\label{secContCase1}

Assume the conditions of lemma \ref{lemDesAsc1-12}. Additionally assume that case
(1) of lemma \ref{lemCases} is satisfied. Fix $\xi>0$ sufficiently
small such that equations (\ref{eqxi}, \ref{eqxi1}, \ref{eqxi4}, \ref{eqxi5})
are satisfied. 

We begin by considering the roots of the functions $f_t'$, $f_n'$ and
$\tilde{f}_n'$ in this case. We consider $f_n'$ and state that $\tilde{f}_n'$
can be treated similarly. Recall the definitions given in equations
(\ref{eqS1S2S3}, \ref{eqf'domain2}, \ref{eqIntervalt},
\ref{eqfn2}, \ref{eqfn'domain}, \ref{eqIntervaln}), and
the properties discussed in equations (\ref{eqS1S2S3In}, \ref{eqS1nS2nS3nIn}).
Recall that $\underline{S_1} := \inf S_1$, $\overline{S_1} := \sup S_1$, etc
(see section \ref{secNot}). Then:
\begin{lem}
\label{lemfn'Case1}
Assume the above conditions. Then, $t \in J_1 = (\overline{S_1}, +\infty)$.
Moreover:
\begin{enumerate}
\item
$f_t'(s) > 0$ for all $s \in (\overline{S_1},t)$,
$f_t'(t) = f_t''(t) = 0$ and $f_t'''(t) > 0$, and $f_t'(s) > 0$
for all $s \in (t,+\infty)$.
\item
$f_t'$ has $0$ roots in each of $\{\C \setminus \R,J_2,J_3,J_4\}$.
\item
$f_t'$ has at most $1$ root in each of
$\cup_{i=1}^3 \{K_1^{(i)},K_2^{(i)},\ldots\}$.
\end{enumerate}

Next note that $t \in J_{1,n} = (\overline{S_{1,n}}, +\infty)$.
Indeed, fixing $\xi > 0$ as above, we have $t - 4\xi > \overline{S_1}$ and
$t - 2\xi > \overline{S_{1,n}}$. Also
$f_n'$ has $2$ roots in $B(t,\xi)$. We denote these by $\{t_{1,n},t_{2,n}\}$ as
in lemma \ref{lemRootsNonAsy1}, and we recall that one of the possibilities,
(a,b,c), discussed in that lemma must be satisfied. Then,
whenever possibility (a) is satisfied:
\begin{enumerate}
\item[(a1)]
$t_{1,n} \in (t-\xi,t+\xi)$ and $t_{1,n} = t_{2,n}$.
Moreover $f_n'(s) > 0$ for all $s \in (\overline{S_{1,n}},t_{1,n})$,
$f_n'(t_{1,n}) = f_n''(t_{1,n}) = 0$ and $f_n'''(t_{1,n}) > 0$,
and $f_n'(s) > 0$ for all $s \in (t_{1,n},+\infty)$.
\item[(a2)]
$f_n'$ has $0$ roots in each of
$\{\C \setminus \R, J_{2,n},J_{3,n},J_{4,n}\}$.
\item[(a3)]
$f_n'$ has $1$ root in each of
$\cup_{i=1}^3 \{K_{1,n}^{(i)},K_{2,n}^{(i)},\ldots\}$.
\end{enumerate}
Moreover, whenever possibility (b) is satisfied:
\begin{enumerate}
\item[(b1)]
$\{t_{1,n},t_{2,n}\} \subset (t-\xi,t+\xi)$ and $t_{1,n} > t_{2,n}$.
Moreover $f_n'(s) > 0$ for all $s \in (\overline{S_{1,n}},t_{2,n})$,
$f_n'(t_{2,n}) = 0$ and $f_n''(t_{2,n}) < 0$,
$f_n'(s) < 0$ for all $s \in (t_{2,n},t_{1,n})$,
$f_n'(t_{1,n}) = 0$ and $f_n''(t_{1,n}) > 0$, and 
$f_n'(s) > 0$ for all $s \in (t_{1,n},+\infty)$.
\item[(b2)]
$f_n'$ has $0$ roots in each of
$\{\C \setminus \R, J_{2,n},J_{3,n},J_{4,n}\}$.
\item[(b3)]
$f_n'$ has $1$ root in each of
$\cup_{i=1}^3 \{K_{1,n}^{(i)},K_{2,n}^{(i)},\ldots\}$.
\end{enumerate}
Finally, whenever possibility (c) is satisfied:
\begin{enumerate}
\item[(c1)]
$t_{1,n} \in B(t,\xi) \cap \mathbb{H}$ and $t_{2,n}$ is the complex conjugate of $t_{1,n}$.
Moreover, $f_n'(s) > 0$ for all $s \in (\overline{S_{1,n}},+\infty)$.
\item[(c2)]
$f_n'$ has $0$ roots in each of
$\{\C \setminus (\R \cup \{t_{1,n},t_{2,n}\}), J_{2,n},J_{3,n},J_{4,n}\}$
\item[(c3)]
$f_n'$ has $1$ root in each of
$\cup_{i=1}^3 \{K_{1,n}^{(i)},K_{2,n}^{(i)},\ldots\}$.
\end{enumerate}
\end{lem}

\begin{proof}
Consider $f_t'$. Note, equations (\ref{eqf'domain2}, \ref{eqIntervalt}), and
case (1) of lemma \ref{lemCases}, give $t \in L_t = J_1 = (\overline{S_1}, +\infty)$. Also,
equation (\ref{equnrnvnsn}) gives $f_t'(t) = f_t''(t) = 0$, and case (1) of
lemma \ref{lemCases} gives $f_t'''(t) > 0$. Parts (1,2,3) then follow from
lemma \ref{lemf'}.

Consider $f_n'$. First note, since $L_t = J_1 = (\overline{S_1}, +\infty)$,
equation (\ref{eqxi}) gives $t - 4 \xi > \overline{S_1}$. Also,
equation (\ref{eqIntervaln}) gives
$t \in L_n = J_{1,n} = (\overline{S_{1,n}}, +\infty)$, and $t - 2\xi > \overline{S_{1,n}}$.
Also, part (1) of lemma \ref{lemRootsNonAsy1} implies that $f_n'$ has
$2$ roots in $B(t,\xi)$.

Consider part (b1). First note, since possibility (b) of lemma
\ref{lemRootsNonAsy1}, is satisfied, that $\{t_{1,n},t_{2,n}\}
\subset (t-\xi,t+\xi) \subset L_n = J_{1,n} = (\overline{S_{1,n}}, +\infty)$,
$t_{1,n} > t_{2,n}$, and $t_{1,n}$ and $t_{2,n}$ are roots of
$f_n'$ of multiplicity $1$. Next note, equation
(\ref{eqfn'}) implies that $(f_n') |_{J_{1,n}}$ is real-valued and continuous,
and
\begin{equation*}
\lim_{w \in \R, w \downarrow \overline{S_{1,n}}} f_n'(w) = + \infty.
\end{equation*}
Finally note that part (1) of lemma \ref{lemRootsNonAsy1} implies that
$f_n'$ has $0$ roots in $(t-\xi,t+\xi) \setminus \{t_{1,n},t_{2,n}\}$,
and $0$ roots in
$L_n \setminus (t-\xi,t+\xi) = J_{1,n} \setminus (t-\xi,t+\xi)
= (\overline{S_{1,n}}, t-\xi] \cup [t+\xi,+\infty)$. (b1) follows
from the above observations. Parts (b2,b3) similarly follow
from lemma \ref{lemRootsNonAsy1}, as do parts (a1,a2,a3) and (c1,c2,c3).
\end{proof}

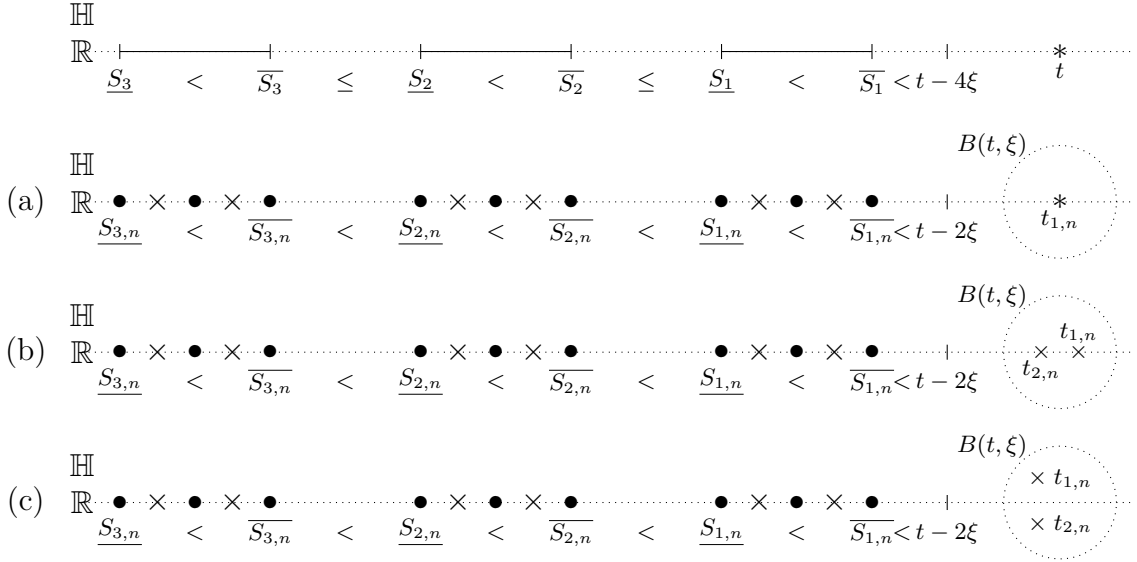
\begin{figure}[t]
\centering
\begin{tikzpicture}

\draw [dotted] (-.5,6) --++(14,0);
\draw (-.5,6) node {$\R$};
\draw (-.5,6.5) node {$\mathbb{H}$};

\draw (0,6.1) --++(0,-.2);
\draw (0,6) --++(2,0);
\draw (2,6.1) --++(0,-.2);
\draw (4,6.1) --++(0,-.2);
\draw (4,6) --++(2,0);
\draw (6,6.1) --++(0,-.2);
\draw (8,6.1) --++(0,-.2);
\draw (8,6) --++(2,0);
\draw (10,6.1) --++(0,-.2);
\draw (11,6.1) --++(0,-.2);

\draw (0,5.6) node {\scriptsize $\underline{S_3}$};
\draw (1,5.6) node {\scriptsize $<$};
\draw (2,5.6) node {\scriptsize $\overline{S_3}$};
\draw (3,5.6) node {\scriptsize $\le$};
\draw (4,5.6) node {\scriptsize $\underline{S_2}$};
\draw (5,5.6) node {\scriptsize $<$};
\draw (6,5.6) node {\scriptsize $\overline{S_2}$};
\draw (7,5.6) node {\scriptsize $\le$};
\draw (8,5.6) node {\scriptsize $\underline{S_1}$};
\draw (9,5.6) node {\scriptsize $<$};
\draw (10,5.6) node {\scriptsize $\overline{S_1}$};
\draw (10.4,5.6) node {\scriptsize $<$};
\draw (11,5.6) node {\scriptsize $t-4\xi$};

\draw (12.5,6) node {$\ast$};
\draw (12.5,5.75) node {\scriptsize $t$};

\draw [dotted] (-.5,4) --++(14,0);
\draw (-.5,4) node {$\R$};
\draw (-.5,4.5) node {$\mathbb{H}$};
\draw (-1.25,4) node {(a)};

\draw (0,4) node {$\bullet$};
\draw (.5,4) node {$\times$};
\draw (1,4) node {$\bullet$};
\draw (1.5,4) node {$\times$};
\draw (2,4) node {$\bullet$};
\draw (4,4) node {$\bullet$};
\draw (4.5,4) node {$\times$};
\draw (5,4) node {$\bullet$};
\draw (5.5,4) node {$\times$};
\draw (6,4) node {$\bullet$};
\draw (8,4) node {$\bullet$};
\draw (8.5,4) node {$\times$};
\draw (9,4) node {$\bullet$};
\draw (9.5,4) node {$\times$};
\draw (10,4) node {$\bullet$};
\draw (11,4.1) --++(0,-.2);

\draw (0,3.6) node {\scriptsize $\underline{S_{3,n}}$};
\draw (1,3.6) node {\scriptsize $<$};
\draw (2,3.6) node {\scriptsize $\overline{S_{3,n}}$};
\draw (3,3.6) node {\scriptsize $<$};
\draw (4,3.6) node {\scriptsize $\underline{S_{2,n}}$};
\draw (5,3.6) node {\scriptsize $<$};
\draw (6,3.6) node {\scriptsize $\overline{S_{2,n}}$};
\draw (7,3.6) node {\scriptsize $<$};
\draw (8,3.6) node {\scriptsize $\underline{S_{1,n}}$};
\draw (9,3.6) node {\scriptsize $<$};
\draw (10,3.6) node {\scriptsize $\overline{S_{1,n}}$};
\draw (10.4,3.6) node {\scriptsize $<$};
\draw (11,3.6) node {\scriptsize $t-2\xi$};

\draw [dotted] (12.5,4) circle (.75cm);
\draw (11.6,4.75) node {\scriptsize $B(t,\xi)$};
\draw (12.5,4) node {$\ast$};
\draw (12.5,3.75) node {\scriptsize $t_{1,n}$};

\draw [dotted] (-.5,2) --++(14,0);
\draw (-.5,2) node {$\R$};
\draw (-.5,2.5) node {$\mathbb{H}$};
\draw (-1.25,2) node {(b)};

\draw (0,2) node {$\bullet$};
\draw (.5,2) node {$\times$};
\draw (1,2) node {$\bullet$};
\draw (1.5,2) node {$\times$};
\draw (2,2) node {$\bullet$};
\draw (4,2) node {$\bullet$};
\draw (4.5,2) node {$\times$};
\draw (5,2) node {$\bullet$};
\draw (5.5,2) node {$\times$};
\draw (6,2) node {$\bullet$};
\draw (8,2) node {$\bullet$};
\draw (8.5,2) node {$\times$};
\draw (9,2) node {$\bullet$};
\draw (9.5,2) node {$\times$};
\draw (10,2) node {$\bullet$};
\draw (11,2.1) --++(0,-.2);

\draw (0,1.6) node {\scriptsize $\underline{S_{3,n}}$};
\draw (1,1.6) node {\scriptsize $<$};
\draw (2,1.6) node {\scriptsize $\overline{S_{3,n}}$};
\draw (3,1.6) node {\scriptsize $<$};
\draw (4,1.6) node {\scriptsize $\underline{S_{2,n}}$};
\draw (5,1.6) node {\scriptsize $<$};
\draw (6,1.6) node {\scriptsize $\overline{S_{2,n}}$};
\draw (7,1.6) node {\scriptsize $<$};
\draw (8,1.6) node {\scriptsize $\underline{S_{1,n}}$};
\draw (9,1.6) node {\scriptsize $<$};
\draw (10,1.6) node {\scriptsize $\overline{S_{1,n}}$};
\draw (10.4,1.6) node {\scriptsize $<$};
\draw (11,1.6) node {\scriptsize $t-2\xi$};

\draw [dotted] (12.5,2) circle (.75cm);
\draw (11.6,2.75) node {\scriptsize $B(t,\xi)$};
\draw (12.75,2) node {\scriptsize $\times$};
\draw (12.75,2.25) node {\scriptsize $t_{1,n}$};
\draw (12.25,2) node {\scriptsize $\times$};
\draw (12.25,1.75) node {\scriptsize $t_{2,n}$};

\draw [dotted] (-.5,0) --++(14,0);
\draw (-.5,0) node {$\R$};
\draw (-.5,.5) node {$\mathbb{H}$};
\draw (-1.25,0) node {(c)};

\draw (0,0) node {$\bullet$};
\draw (.5,0) node {$\times$};
\draw (1,0) node {$\bullet$};
\draw (1.5,0) node {$\times$};
\draw (2,0) node {$\bullet$};
\draw (4,0) node {$\bullet$};
\draw (4.5,0) node {$\times$};
\draw (5,0) node {$\bullet$};
\draw (5.5,0) node {$\times$};
\draw (6,0) node {$\bullet$};
\draw (8,0) node {$\bullet$};
\draw (8.5,0) node {$\times$};
\draw (9,0) node {$\bullet$};
\draw (9.5,0) node {$\times$};
\draw (10,0) node {$\bullet$};
\draw (11,.1) --++(0,-.2);

\draw (0,-.4) node {\scriptsize $\underline{S_{3,n}}$};
\draw (1,-.4) node {\scriptsize $<$};
\draw (2,-.4) node {\scriptsize $\overline{S_{3,n}}$};
\draw (3,-.4) node {\scriptsize $<$};
\draw (4,-.4) node {\scriptsize $\underline{S_{2,n}}$};
\draw (5,-.4) node {\scriptsize $<$};
\draw (6,-.4) node {\scriptsize $\overline{S_{2,n}}$};
\draw (7,-.4) node {\scriptsize $<$};
\draw (8,-.4) node {\scriptsize $\underline{S_{1,n}}$};
\draw (9,-.4) node {\scriptsize $<$};
\draw (10,-.4) node {\scriptsize $\overline{S_{1,n}}$};
\draw (10.4,-.4) node {\scriptsize $<$};
\draw (11,-.4) node {\scriptsize $t-2\xi$};

\draw [dotted] (12.5,0) circle (.75cm);
\draw (11.6,.75) node {\scriptsize $B(t,\xi)$};
\draw (12.5,.3) node {\scriptsize $\times \; t_{1,n}$};
\draw (12.5,-.3) node {\scriptsize $\times \; t_{2,n}$};

\end{tikzpicture}
\caption{Top: $T_t$, the set of roots of $f_t'$, as described by lemma
\ref{lemfn'Case1}. Note, for each $i \in \{1,2,3\}$, $S_i$ does not
necessarily equal $[\underline{S_i}, \overline{S_i}]$, and subintervals
of $[\underline{S_i}, \overline{S_i}] \setminus S_i$ contain at most
$1$ root. (a,b,c): $T_n$, the set of roots of
$f_n'$, as described by lemma \ref{lemfn'Case1} for possibilities (a,b,c).
Roots of multiplicity $1$ and $2$ are represented by $\times$ and $\ast$
respectively, and elements of $S_n = S_{1,n} \cup S_{2,n} \cup S_{3,n}$
are represented by $\bullet$. }
\label{figSnC_xi1}
\end{figure}

\begin{figure}[t]
\centering
\begin{tikzpicture}[scale=0.8];

\draw [dashed] (-.5,13) --++(0,4);
\draw (-.5,12.7) node {\scriptsize $\overline{S_1}$};

\draw plot [smooth, tension=1] coordinates
{ (0,13.5) (1,14.55) (2.25,15) (3.4,16.3) (4.5,17)};
\draw [dotted] (0,13.5) --++(-.3,-.5);
\draw [dotted] (4.5,17) --++(.6,.25);

\draw [dashed] (1.6,13) --++(0,3);
\draw (1.6,12.7) node {\scriptsize $t$};
\draw (1.6,16.3) node {\scriptsize $R_t'''(t) > 0$};
\draw (1.6,16.9) node {\scriptsize $R_t'(t) = R_t''(t) = 0$};

\draw [dashed] (-.5,8) --++(0,4);
\draw (-.5,7.7) node {\scriptsize $\overline{S_{1,n}}$};
\draw (-1.5,10) node {$(a)$};

\draw plot [smooth, tension=1] coordinates
{ (0,8.5) (1,9.55) (2.25,10) (3.4,11.3) (4.5,12)};
\draw [dotted] (0,8.5) --++(-.3,-.5);
\draw [dotted] (4.5,12) --++(.6,.25);

\draw [dashed] (1.6,8) --++(0,3);
\draw (1.6,7.7) node {\scriptsize $t_{1,n}$};
\draw (1.6,11.3) node {\scriptsize $R_n'''(t_{1,n}) > 0$};
\draw (1.7,11.9) node {\scriptsize $R_n'(t_{1,n}) = R_n''(t_{1,n}) = 0$};

\draw [dashed] (-.5,4) --++(0,3);
\draw (-.5,3.7) node {\scriptsize $\overline{S_{1,n}}$};
\draw (-1.5,5.5) node {$(b)$};

\draw plot [smooth, tension=1] coordinates
{ (0,4) (1.5,6) (3,4.5) (4.5,6.5)};
\draw [dotted] (0,4) --++(-.2,-.5);
\draw [dotted] (4.5,6.5) --++(.2,.5);

\draw [dashed] (1.6,4) --++(0,2.4);
\draw (1.6,3.7) node {\scriptsize $t_{2,n}$};
\draw (1,7.1) node {\scriptsize $R_n''(t_{2,n}) < 0$};
\draw (1,6.6) node {\scriptsize $R_n'(t_{2,n}) = 0$};

\draw [dashed] (2.9,4) --++(0,2.4);
\draw (2.9,3.7) node {\scriptsize $t_{1,n}$};
\draw (3.4,7.1) node {\scriptsize $R_n''(t_{1,n}) > 0$};
\draw (3.4,6.6) node {\scriptsize $R_n'(t_{1,n}) = 0$};

\draw [dashed] (-.5,0) --++(0,3);
\draw (-.5,-.4) node {\scriptsize $\overline{S_{1,n}}$};
\draw (-1.5,1.5) node {$(c)$};

\draw plot [smooth, tension=1] coordinates
{ (0,0) (2,1.75) (4.5,2.5)};
\draw [dotted] (0,0) --++(-.3,-.4);
\draw [dotted] (4.5,2.5) --++(.5,.05);

\draw (1.5,2.5) node {\scriptsize $R_n'(s) > 0 \text{ for all } s$};

\draw (13.2,16.5) node {$\mathbb{H}$};
\draw [dotted] (5,15) --++(8,0);
\draw (13.2,15) node {$\R$};

\draw[arrows=->,line width=.75pt](9,15)--(11,15);
\draw (11.3,15) node {a};
\draw[arrows=->,line width=.75pt](9,15)--(10,{15+sqrt(3)});
\draw (10.1,16.95) node {d};
\draw[arrows=->,line width=.75pt](9,15)--(8,{15+sqrt(3)});
\draw (7.9,16.95) node {a};
\draw[arrows=->,line width=.75pt](9,15)--(7,15);
\draw (6.7,15) node {d};
\draw[arrows=->,line width=.75pt](9,15)--(8,{15-sqrt(3)});
\draw (7.9,13.05) node {a};
\draw[arrows=->,line width=.75pt](9,15)--(10,{15-sqrt(3)});
\draw (10.1,13.05) node {d};

\fill[white] (9,15) circle (.35cm);
\draw (9,15) node {\scriptsize $t$};

\draw [dotted] (9,15) circle (1cm);
\draw (9.55,15.35) node {\scriptsize $\frac\pi3$};
\draw (9,15.6) node {\scriptsize $\frac\pi3$};
\draw (8.45,15.35) node {\scriptsize $\frac\pi3$};
\draw (9,14.4) node {\scriptsize $\frac\pi3$};
\draw (8.45,14.65) node {\scriptsize $\frac\pi3$};
\draw (9.55,14.65) node {\scriptsize $\frac\pi3$};

\draw (13.2,11.5) node {$\mathbb{H}$};
\draw [dotted] (5,10) --++(8,0);
\draw (13.2,10) node {$\R$};

\draw[arrows=->,line width=.75pt](9,10)--(11,10);
\draw (11.3,10) node {a};
\draw[arrows=->,line width=.75pt](9,10)--(10,{10+sqrt(3)});
\draw (10.1,11.95) node {d};
\draw[arrows=->,line width=.75pt](9,10)--(8,{10+sqrt(3)});
\draw (7.9,11.95) node {a};
\draw[arrows=->,line width=.75pt](9,10)--(7,10);
\draw (6.7,10) node {d};
\draw[arrows=->,line width=.75pt](9,10)--(8,{10-sqrt(3)});
\draw (7.9,8.05) node {a};
\draw[arrows=->,line width=.75pt](9,10)--(10,{10-sqrt(3)});
\draw (10.1,8.05) node {d};

\fill[white] (9,10) circle (.35cm);
\draw (9,10) node {\scriptsize $t_{1,n}$};

\draw [dotted] (9,10) circle (1cm);
\draw (9.55,10.35) node {\scriptsize $\frac\pi3$};
\draw (9,10.6) node {\scriptsize $\frac\pi3$};
\draw (8.45,10.35) node {\scriptsize $\frac\pi3$};
\draw (9,9.4) node {\scriptsize $\frac\pi3$};
\draw (8.45,9.65) node {\scriptsize $\frac\pi3$};
\draw (9.55,9.65) node {\scriptsize $\frac\pi3$};

\draw (13.2,6.5) node {$\mathbb{H}$};
\draw [dotted] (5,5) --++(8,0);
\draw (13.2,5) node {$\R$};

\draw[arrows=->,line width=.75pt](7,5)--(8.5,5);
\draw (8.7,5) node {d};
\draw[arrows=->,line width=.75pt](7,5)--(7,6.5);
\draw (7,6.7) node {a};
\draw[arrows=->,line width=.75pt](7,5)--(5.5,5);
\draw (5.3,5) node {d};
\draw[arrows=->,line width=.75pt](7,5)--(7,3.5);
\draw (7,3.3) node {a};

\fill[white] (7,5) circle (.35cm);
\draw (7,5) node {\scriptsize $t_{2,n}$};

\draw [dotted] (7,5) circle (1cm);
\draw (7.5,5.4) node {\scriptsize $\frac\pi2$};
\draw (6.5,5.4) node {\scriptsize $\frac\pi2$};
\draw (6.5,4.6) node {\scriptsize $\frac\pi2$};
\draw (7.5,4.6) node {\scriptsize $\frac\pi2$};

\draw[arrows=->,line width=.75pt](11,5)--(12.5,5);
\draw (12.7,5) node {a};
\draw[arrows=->,line width=.75pt](11,5)--(11,6.5);
\draw (11,6.7) node {d};
\draw[arrows=->,line width=.75pt](11,5)--(9.5,5);
\draw (9.3,5) node {a};
\draw[arrows=->,line width=.75pt](11,5)--(11,3.5);
\draw (11,3.3) node {d};

\fill[white] (11,5) circle (.35cm);
\draw (11,5) node {\scriptsize $t_{1,n}$};

\draw [dotted] (11,5) circle (1cm);
\draw (11.5,5.4) node {\scriptsize $\frac\pi2$};
\draw (10.5,5.4) node {\scriptsize $\frac\pi2$};
\draw (10.5,4.6) node {\scriptsize $\frac\pi2$};
\draw (11.5,4.6) node {\scriptsize $\frac\pi2$};

\draw (13.2,1) node {$\mathbb{H}$};
\draw [dotted] (5,-.6) --++(8,0);
\draw (13.2,-.6) node {$\R$};

\draw[arrows=->,line width=.75pt](9,1)--(9+sqrt{9/2},{1+sqrt{9/2}});
\draw (10.65,2.65) node {d};
\draw[arrows=->,line width=.75pt](9,1)--(9-sqrt{9/2},{1+sqrt{9/2}});
\draw (7.35,2.65) node {a};
\draw[arrows=->,line width=.75pt](9,1)--(9-sqrt{9/2},{1-sqrt{9/2}});
\draw (7.35,-.65) node {d};
\draw[arrows=->,line width=.75pt](9,1)--(9+sqrt{9/2},{1-sqrt{9/2}});
\draw (10.65,-.65) node {a};

\fill[white] (9,1) circle (.35cm);
\draw (9,1) node {\scriptsize $t_{1,n}$};

\draw [dotted] (9,1) circle (1cm);
\draw (9.6,1) node {\scriptsize $\frac\pi2$};
\draw (9,1.6) node {\scriptsize $\frac\pi2$};
\draw (8.4,1) node {\scriptsize $\frac\pi2$};
\draw (9,.4) node {\scriptsize $\frac\pi2$};

\end{tikzpicture}
\caption{
Top left: $R_t |_{J_1}$. Top right: The directions of steepest decent/ascent
for $f_t$ at $t$. `a' represents ascent and `d' represents descent. (a,b,c)
left: $R_n |_{J_{1,n}}$ for possibilities (a,b,c) of lemma \ref{lemfn'Case1}.
(a,b,c) right: The directions of steepest decent/ascent for $f_n$ at
$t_{1,n}/t_{2,n}$.}
\label{figConDirCase1}
\end{figure}
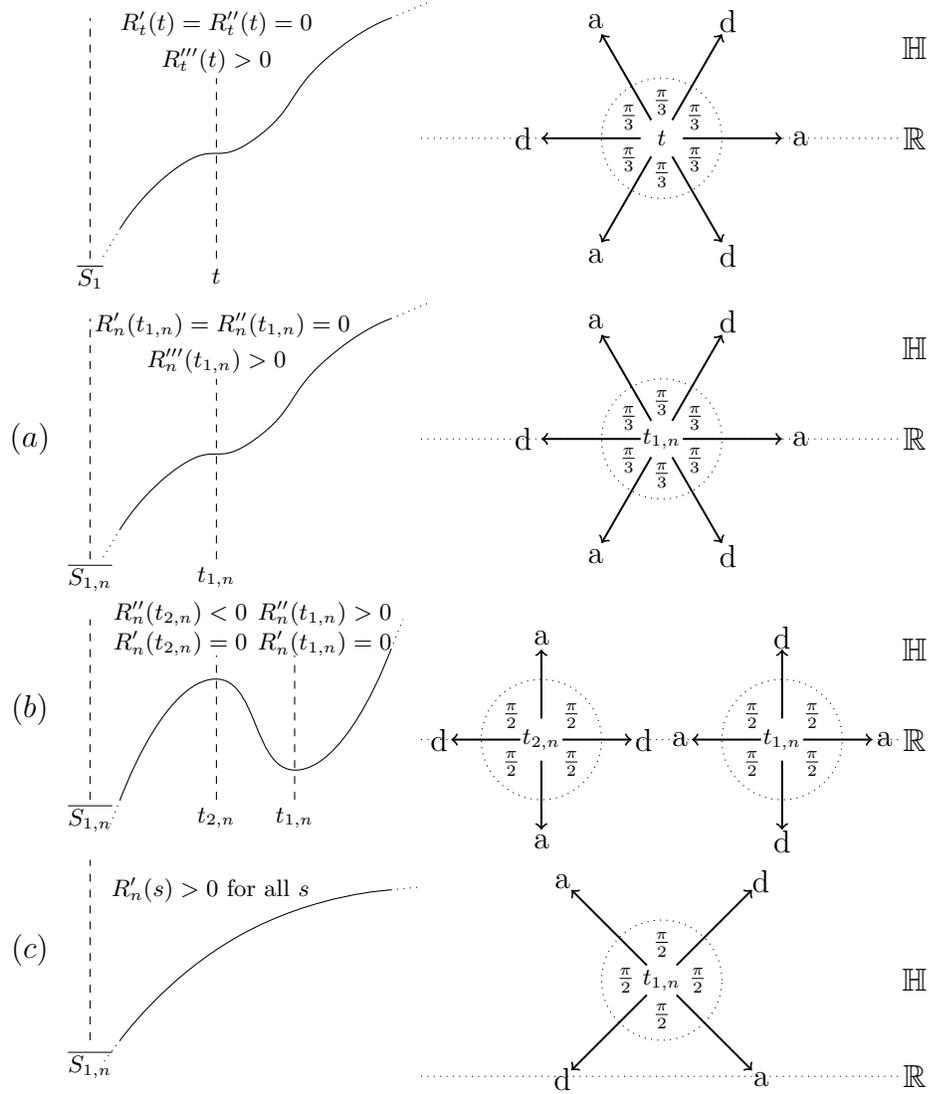

Next, for convenience, define,
\begin{equation}
\label{eqRootsTn}
T_t := \text{ Set of roots of } f_t'
\hspace{.5cm} \text{and} \hspace{.5cm}
T_n : = \text{ Set of roots of } f_n'.
\end{equation}
Then, recalling that $f_t' : \C \setminus S \to \C$ and $f_n' : \C \setminus S_n \to \C$
are analytic, $T_t$ and $T_n$ are discrete subsets of $\C \setminus S$ and
$\C \setminus S_n$ respectively. The previous lemma discusses the locations
of the elements of these sets, and this is depicted in figure \ref{figSnC_xi1}. 
Note that,
\begin{equation}
\label{eqRootsT1nT2n}
T_t = T_{1,t} \cup T_{2,t} \cup T_{3,t} \cup \{t\}
\hspace{.5cm} \text{and} \hspace{.5cm}
T_n = T_{1,n} \cup T_{2,n} \cup T_{3,n} \cup \{t_{1,n},t_{2,n}\},
\end{equation}
where $T_{i,t}$ is the set of roots of $f_t'$ in
$[\underline{S_i}, \overline{S_i}] \setminus S_i$,
and similarly for $T_{i,n}$.

Next recall equation (\ref{eqRnRes}) gives,
\begin{equation*}
(R_t |_{J_1})' = (f_t') |_{J_1}
\hspace{.5cm} \text{and} \hspace{.5cm}
(R_n |_{J_{1,n}})' = (f_n') |_{J_{1,n}},
\end{equation*}
and similarly for the higher order derivatives. Above, $R_t$ is the
real-part of $f_t$ and $R_n$ is the real-part of $f_n$ (see equations
(\ref{eqRt}, \ref{eqRn})). The functions on the LHSs
are the `real-derivatives' of $R_t |_{J_1}$ and $R_n |_{J_{1,n}}$,
and the functions on the RHSs are those given in equations
(\ref{eqft'2}, \ref{eqfn'}) restricted to $J_1$ and $J_{1,n}$
respectively. Part (1) of lemma \ref{lemfn'Case1} then implies that $R_t |_{J_1}$
has the behaviour shown on the top left of figure \ref{figConDirCase1}.
Also, parts (a1,b1,c1) of lemma \ref{lemfn'Case1} imply that
$R_n |_{J_{1,n}}$ have those behaviours shown on the left of figure
\ref{figConDirCase1} for the possibilities (a,b,c). Part (3) of lemma
\ref{lemDesAsc} also shows that $f_t$ and $f_n$ have those
directions of steepest descent/ascent shown on the right of figure
\ref{figConDirCase1}. The following lemma examines some of the resulting
contours of steepest descent/ascent:
\begin{lem}
\label{lemConAscDesCase1}
There exists simple contours, $D_t$ and $A_t$, as shown on
the top of figure \ref{figConDesAscCase1} with the following properties:
\begin{enumerate}
\item
$D_t$ and $A_t$ both start at $t$, enter $\mathbb{H}$
in the directions $\frac\pi3$ and $\frac{2\pi}3$ respectively,
end in the intervals shown, and are otherwise
contained in $\mathbb{H}$.
\item
$D_t$ and $A_t$ are contours of steepest descent and ascent
(respectively) for $f_t$.
\item
$D_t$ and $A_t$ do not intersect except at $t$.
\end{enumerate}

Also, whenever possibility (a) is
satisfied, there exists simple contours, $D_n$ and $A_n$, as shown
in figure \ref{figConDesAscCase1} with the following properties:
\begin{enumerate}
\item[(a1)]
$D_n$ and $A_n$ both start at $t_{1,n}$, enter $\mathbb{H}$
in the directions $\frac\pi3$ and $\frac{2\pi}3$ respectively,
end in the intervals shown, and are otherwise
contained in $\mathbb{H}$.
\item[(a2)]
$D_n$ and $A_n$ are contours of steepest descent and ascent
(respectively) for $f_n$.
\item[(a3)]
$D_n$ and $A_n$ do not intersect except at $t_{1,n}$.
\end{enumerate}
Next, whenever possibility (b) is
satisfied, there exists simple contours, $D_n$ and $A_n$, as shown
in figure \ref{figConDesAscCase1} with the following properties:
\begin{enumerate}
\item[(b1)]
$D_n$ and $A_n$ start at $t_{1,n}$ and $t_{2,n}$ respectively, both
enter $\mathbb{H}$ in the direction $\frac\pi2$, end in the
intervals shown, and are otherwise contained in $\mathbb{H}$.
\item[(b2)]
$D_n$ and $A_n$ are contours of steepest descent and ascent
(respectively) for $f_n$.
\item[(b3)]
$D_n$ and $A_n$ do not intersect.
\end{enumerate}
Next, whenever possibility (c) is
satisfied, there exists simple contours, $D_n,A_n,D_n',A_n'$, as shown
in figure \ref{figConDesAscCase1} with the following properties:
\begin{enumerate}
\item[(c1)]
$D_n,A_n,D_n',A_n'$ all start at $t_{1,n} \in \mathbb{H}$, leave $t_{1,n}$
in orthogonal directions in the counter-clockwise order $D_n,A_n,D_n',A_n'$,
end in the intervals shown or are unbounded, and are
otherwise contained in $\mathbb{H}$.
\item[(c2)]
$D_n,D_n'$ are contours of steepest descent for $f_n$,
and $A_n,A_n'$ are contours of steepest ascent for $f_n$.
\item[(c3)]
$D_n,D_n',A_n,A_n'$ do not intersect except at $t_{1,n}$.
\end{enumerate}
\end{lem}

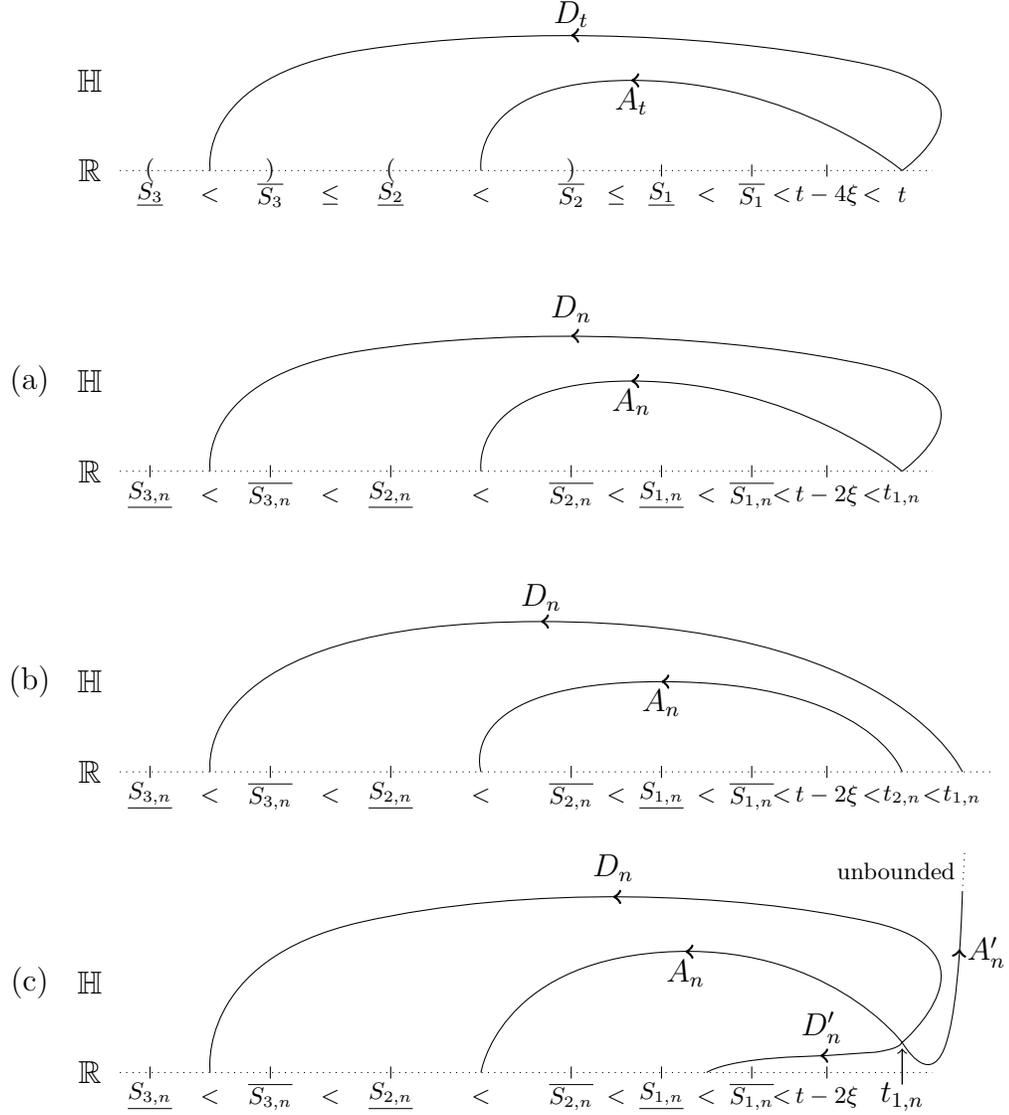
\begin{figure}[t]
\centering
\begin{tikzpicture}[scale=0.8];

\draw (-1.5,16.5) node {$\mathbb{H}$};
\draw [dotted] (-1,15) --++(13.5,0);
\draw (-1.5,15) node {$\R$};
\draw (-.5,15) node {\scriptsize $($};
\draw (1.5,15) node {\scriptsize $)$};
\draw (3.5,15) node {\scriptsize $($};
\draw (6.5,15) node {\scriptsize $)$};
\draw (8,14.9) --++(0,.2);
\draw (9.5,14.9) --++(0,.2);
\draw (10.75,14.9) --++(0,.2);

\draw (-.5,14.6) node {\scriptsize $\underline{S_3}$};
\draw (.5,14.6) node {\scriptsize $<$};
\draw (1.5,14.6) node {\scriptsize $\overline{S_3}$};
\draw (2.5,14.6) node {\scriptsize $\le$};
\draw (3.5,14.6) node {\scriptsize $\underline{S_2}$};
\draw (5,14.6) node {\scriptsize $<$};
\draw (6.5,14.6) node {\scriptsize $\overline{S_2}$};
\draw (7.25,14.6) node {\scriptsize $\le$};
\draw (8,14.6) node {\scriptsize $\underline{S_1}$};
\draw (8.75,14.6) node {\scriptsize $<$};
\draw (9.5,14.6) node {\scriptsize $\overline{S_1}$};
\draw (10,14.6) node {\scriptsize $<$};
\draw (10.75,14.6) node {\scriptsize $t-4\xi$};
\draw (11.5,14.6) node {\scriptsize $<$};
\draw (12,14.6) node {\scriptsize $t$};

\draw plot [smooth, tension=.9] coordinates
{ (12,15) (11.5,16.75) (3,17) (.5,15) };
\draw[arrows=->,line width=1pt](6.51,17.25)--(6.5,17.25);
\draw (6.5,17.6) node {$D_t$};
\draw plot [smooth, tension=1.4] coordinates
{ (12,15) (7.5,16.5) (5,15) };
\draw[arrows=->,line width=1pt](7.51,16.5)--(7.5,16.5);
\draw (7.5,16.15) node {$A_t$};

\draw (-2.5,11.5) node {(a)};
\draw (-1.5,11.5) node {$\mathbb{H}$};
\draw [dotted] (-1,10) --++(13.5,0);
\draw (-1.5,10) node {$\R$};
\draw (-.5,9.9) --++(0,.2);
\draw (1.5,9.9) --++(0,.2);
\draw (3.5,9.9) --++(0,.2);
\draw (6.5,9.9) --++(0,.2);
\draw (8,9.9) --++(0,.2);
\draw (9.5,9.9) --++(0,.2);
\draw (10.75,9.9) --++(0,.2);

\draw (-.5,9.6) node {\scriptsize $\underline{S_{3,n}}$};
\draw (.5,9.6) node {\scriptsize $<$};
\draw (1.5,9.6) node {\scriptsize $\overline{S_{3,n}}$};
\draw (2.5,9.6) node {\scriptsize $<$};
\draw (3.5,9.6) node {\scriptsize $\underline{S_{2,n}}$};
\draw (5,9.6) node {\scriptsize $<$};
\draw (6.5,9.6) node {\scriptsize $\overline{S_{2,n}}$};
\draw (7.25,9.6) node {\scriptsize $<$};
\draw (8,9.6) node {\scriptsize $\underline{S_{1,n}}$};
\draw (8.75,9.6) node {\scriptsize $<$};
\draw (9.5,9.6) node {\scriptsize $\overline{S_{1,n}}$};
\draw (10,9.6) node {\scriptsize $<$};
\draw (10.75,9.6) node {\scriptsize $t-2\xi$};
\draw (11.5,9.6) node {\scriptsize $<$};
\draw (12,9.6) node {\scriptsize $t_{1,n}$};

\draw plot [smooth, tension=.9] coordinates
{ (12,10) (11.5,11.75) (3,12) (.5,10) };
\draw[arrows=->,line width=1pt](6.51,12.25)--(6.5,12.25);
\draw (6.5,12.7) node {$D_n$};
\draw plot [smooth, tension=1.4] coordinates
{ (12,10) (7.5,11.5) (5,10) };
\draw[arrows=->,line width=1pt](7.51,11.5)--(7.5,11.5);
\draw (7.5,11.15) node {$A_n$};

\draw (-2.5,6.5) node {(b)};
\draw (-1.5,6.5) node {$\mathbb{H}$};
\draw [dotted] (-1,5) --++(14.5,0);
\draw (-1.5,5) node {$\R$};
\draw (-.5,4.9) --++(0,.2);
\draw (1.5,4.9) --++(0,.2);
\draw (3.5,4.9) --++(0,.2);
\draw (6.5,4.9) --++(0,.2);
\draw (8,4.9) --++(0,.2);
\draw (9.5,4.9) --++(0,.2);
\draw (10.75,4.9) --++(0,.2);

\draw (-.5,4.6) node {\scriptsize $\underline{S_{3,n}}$};
\draw (.5,4.6) node {\scriptsize $<$};
\draw (1.5,4.6) node {\scriptsize $\overline{S_{3,n}}$};
\draw (2.5,4.6) node {\scriptsize $<$};
\draw (3.5,4.6) node {\scriptsize $\underline{S_{2,n}}$};
\draw (5,4.6) node {\scriptsize $<$};
\draw (6.5,4.6) node {\scriptsize $\overline{S_{2,n}}$};
\draw (7.25,4.6) node {\scriptsize $<$};
\draw (8,4.6) node {\scriptsize $\underline{S_{1,n}}$};
\draw (8.75,4.6) node {\scriptsize $<$};
\draw (9.5,4.6) node {\scriptsize $\overline{S_{1,n}}$};
\draw (10,4.6) node {\scriptsize $<$};
\draw (10.75,4.6) node {\scriptsize $t-2\xi$};
\draw (11.5,4.6) node {\scriptsize $<$};
\draw (12,4.6) node {\scriptsize $t_{2,n}$};
\draw (12.5,4.6) node {\scriptsize $<$};
\draw (13,4.6) node {\scriptsize $t_{1,n}$};

\draw plot [smooth, tension=1.7] coordinates
{ (13,5) (6,7.5) (.5,5) };
\draw[arrows=->,line width=1pt](6.01,7.5)--(6,7.5);
\draw (6,7.9) node {$D_n$};
\draw plot [smooth, tension=1.8] coordinates
{ (12,5) (8,6.5) (5,5) };
\draw[arrows=->,line width=1pt](8.01,6.5)--(8,6.5);
\draw (8,6.15) node {$A_n$};

\draw (-2.5,1.5) node {(c)};
\draw (-1.5,1.5) node {$\mathbb{H}$};
\draw [dotted] (-1,0) --++(13.5,0);
\draw (-1.5,0) node {$\R$};
\draw (-.5,-.1) --++(0,.2);
\draw (1.5,-.1) --++(0,.2);
\draw (3.5,-.1) --++(0,.2);
\draw (6.5,-.1) --++(0,.2);
\draw (8,-.1) --++(0,.2);
\draw (9.5,-.1) --++(0,.2);
\draw (10.75,-.1) --++(0,.2);

\draw (-.5,-.4) node {\scriptsize $\underline{S_{3,n}}$};
\draw (.5,-.4) node {\scriptsize $<$};
\draw (1.5,-.4) node {\scriptsize $\overline{S_{3,n}}$};
\draw (2.5,-.4) node {\scriptsize $<$};
\draw (3.5,-.4) node {\scriptsize $\underline{S_{2,n}}$};
\draw (5,-.4) node {\scriptsize $<$};
\draw (6.5,-.4) node {\scriptsize $\overline{S_{2,n}}$};
\draw (7.25,-.4) node {\scriptsize $<$};
\draw (8,-.4) node {\scriptsize $\underline{S_{1,n}}$};
\draw (8.75,-.4) node {\scriptsize $<$};
\draw (9.5,-.4) node {\scriptsize $\overline{S_{1,n}}$};
\draw (10,-.4) node {\scriptsize $<$};
\draw (10.75,-.4) node {\scriptsize $t-2\xi$};
\draw (12,-.4) node {$t_{1,n}$};
\draw[arrows=->,line width=.5pt](12,-.2)--(12,.4);

\draw plot [smooth, tension=.9] coordinates
{ (12,0.5) (11.5,2.5) (3,2.5) (.5,0) };
\draw[arrows=->,line width=1pt](7.21,2.92)--(7.2,2.92);
\draw (7.2,3.4) node {$D_n$};
\draw plot [smooth, tension=1.4] coordinates
{ (12,0.5) (8,2) (5,0) };
\draw[arrows=->,line width=1pt](8.41,2.01)--(8.4,2.01);
\draw (8.4,1.65) node {$A_n$};
\draw plot [smooth, tension=.95] coordinates
{ (12,0.52) (11.75,.375) (11,.3) (9.5,.2) (8.75,0)};
\draw[arrows=->,line width=1pt](10.66,.28)--(10.64,.28);
\draw (10.66,.75) node {$D_n'$};
\draw plot [smooth, tension=.95] coordinates
{ (12,0.5) (12.7,.4) (13,3)};
\draw [dotted] (13,3) --++(.04,.7);
\draw (11.9,3.35) node {\scriptsize unbounded};
\draw[arrows=->,line width=1pt](12.95,2)--(12.954,2.05);
\draw (13.4,2) node {$A_n'$};

\end{tikzpicture}
\caption{The contours of lemma \ref{lemConAscDesCase1}.}
\label{figConDesAscCase1}
\end{figure}

\begin{proof}
Consider $f_n$. Recall (see equation (\ref{eqRootsTn})) that $T_n$ is the set of
roots of $f_n'$, and the behaviour of $T_n$ is described
in lemma \ref{lemfn'Case1} and displayed in figure \ref{figSnC_xi1}. Also,
the directions of steepest decent/ascent for $f_n$ at $t_{1,n}/t_{2,n}$ are
shown on the right of figure \ref{figConDirCase1}. Define:
\begin{enumerate}
\item[(i)]
For possibility (a), let $D_n$ and $A_n$ denote the contours of steepest
descent and ascent (respectively) for $f_n$, which start at
$t_{1,n} \in (\overline{S_{1,n}}, +\infty)$, and which enter $\mathbb{H}$ in
the directions $\frac\pi3$ and $\frac{2\pi}3$ respectively. For possibility
(b), let $D_n$ and $A_n$ denote the contours of steepest descent and ascent
(respectively), which start at $t_{1,n} \in (\overline{S_{1,n}}, +\infty)$ and
$t_{2,n} \in (\overline{S_{1,n}}, +\infty)$ respectively,
and which enter $\mathbb{H}$ in the direction $\frac\pi2$. For possibility (c),
let $\{D_n,D_n'\}$ and $\{A_n,A_n'\}$ denote the contours of steepest
descent and ascent (respectively) which start at $t_{1,n} \in \mathbb{H}$.
All contours are defined to follow the unique directions of steepest
descent and ascent (as appropriate) at each $w \in \C \setminus (S_n \cup T_n)$
(see part (3) of lemma \ref{lemDesAsc}), and are defined to end whenever/if
they first intersect a point in $S_n \cup T_n$. Finally, if no such
point of intersection exists, the contours are unbounded.
\end{enumerate}
Then, using the above definition, we will first show:
\begin{enumerate}
\item[(ii)]
For possibility (a), $D_n$ and $A_n$ do not intersect except at
$t_{1,n}$. For possibility (b), $D_n$ and $A_n$ never intersect.
For possibility (c), $D_n,A_n,D_n',A_n'$ do not intersect except
at $t_{1,n}$. For all possibilities, the contours are simple.
\end{enumerate}

Next we investigate the possible end-points of the contours in
$S_n \cup T_n$. Recall that these are depicted in figure \ref{figSnC_xi1},
and that $T_n = T_{1,n} \cup T_{2,n} \cup T_{3,n} \cup \{t_{1,n},t_{2,n}\}$
(see equation (\ref{eqRootsT1nT2n})). Also, since $t_{1,n}$ and $t_{2,n}$
are the start-points of $D_n,A_n,D_n',A_n'$, part (ii) implies that
the contours do not intersect $\{t_{1,n},t_{2,n}\}$ again. Part (i) thus
implies that they end whenever/if they intersect a point in
$S_{1,n} \cup T_{1,n} \cup S_{2,n} \cup T_{2,n} \cup S_{3,n} \cup T_{3,n} \subset \R$.
Next, we will show:
\begin{enumerate}
\item[(iii)]
Each of $D_n,A_n,D_n',A_n'$ either eventually intersect and end in
$S_{1,n} \cup T_{1,n} \cup S_{2,n} \cup T_{2,n} \cup S_{3,n} \cup T_{3,n}
\subset \R$, or they do not intersect $\R$ and are unbounded. Moreover,
if we exclude the start-points and end-points, the contours are
contained in $\mathbb{H}$.
\end{enumerate}
Finally, we use the above observations to show:
\begin{enumerate}
\item[(iv)]
Each of $D_n,D_n'$ eventually intersect and end in $S_{1,n} \cup T_{1,n}
\cup S_{3,n} \cup T_{3,n} \subset \R$.
\item[(v)]
Each of $A_n,A_n'$ either eventually intersect and end in $S_{2,n} \cup T_{2,n} \subset \R$,
or they do not intersect $\R$ and are unbounded.
\item[(vi)]
For possibilities (a) and (b), $D_n$ ends in $S_{3,n} \cup T_{3,n}$
and $A_n$ ends in $S_{2,n} \cup T_{2,n}$. For possibility (c), one of
$\{D_n, D_n'\}$ ends in $S_{3,n} \cup T_{3,n}$, one of $\{A_n, A_n'\}$
ends in $S_{2,n} \cup T_{2,n}$, one of $\{D_n, D_n'\}$ ends in
$S_{1,n} \cup T_{1,n}$, one of $\{A_n, A_n'\}$ does not
intersect $\R$ and is unbounded. Moreover, if we label
so that $D_n$ ends in $S_{3,n} \cup T_{3,n}$ and $A_n$ ends
in $S_{2,n} \cup T_{2,n}$, then they leave $t_{1,n}$
in the counter-clockwise order $D_n,A_n,D_n',A_n'$.
\end{enumerate}
The required results follow from parts (i,ii,iii,vi) since
$S_{i,n} \cup T_{i,n} \subset [\underline{S_{i,n}}, \overline{S_{i,n}}]$
for all $i \in \{1,2,3\}$ (see figure \ref{figSnC_xi1} and
equation (\ref{eqRootsT1nT2n})).

Consider (ii) for possibility (a). Recall (see part (i))
that $D_n,A_n$ both start at $t_{1,n}$. Also, recall (see
part (1) of lemma \ref{lemDesAsc}) that $R_n$ strictly decreases
along $D_n$, where $R_n$ is the real-part of $f_n$, and $R_n$
strictly increases along $A_n$. A contradiction argument then proves
part (ii) for possibility (a). Consider (ii) for possibility
(b). Recall (see part (i)) that $D_n$ starts at $t_{1,n} \in \R$ and
$A_n$ starts at $t_{2,n} \in \R$, and $t_{1,n} > t_{2,n}$. Also,
recall (see left of figure \ref{figConDirCase1}) that $R_n$
strictly decreases as we move from $t_{2,n}$ to $t_{1,n}$ along $\R$.
A similar contradiction argument then proves part (ii) for possibility (b).
Part (ii) for possibility (c) also follows similarly.

Consider (iii) for possibility (a). Recall (see part (i)) that $D_n$ and
$A_n$ both start at $t_{1,n} \in \R$ and immediately enter $\mathbb{H}$.
Also, recall that the contours either end in $S_n \cup T_n$ or they are
unbounded. Finally, recall that $S_n \cup T_n \subset \R$. Thus, to prove
part (iii) for possibility (a), it is sufficient to show that contours of
steepest descent and ascent cannot intersect $\R \setminus (S_n \cup T_n)$
from $\mathbb{H}$. To show this, fix $s \in \R \setminus (S_n \cup T_n)$.
Recall that $f_n'(s) \in \R$ (see equation (\ref{eqfn'})), and note that
$f_n'(s) \neq 0$ since $T_n$ is the set of roots of $f_n'$. Therefore $f_n'(s) > 0$
or $f_n'(s) < 0$. Part (3) of lemma \ref{lemDesAsc} then implies that
$f_n$ has $1$ direction of steepest ascent at $s$. Moreover, the direction
of steepest ascent is along the positive real axis whenever $f_n'(s) > 0$,
and along the negative real axis whenever $f_n'(s) < 0$. Thus a contour of
steepest descent cannot intersect $s$ from $\mathbb{H}$. Similarly for
contours of steepest ascent. This proves part (iii) for possibility (a).
We can similarly prove part (iii) for possibilities (b) and (c). 

Consider (iv). First note, part (4) of lemma \ref{lemDesAsc} implies
that $D_n,D_n'$ are bounded. Part (iii) thus implies that each of
$D_n,D_n'$ eventually intersects and ends in $S_{1,n} \cup T_{1,n}
\cup S_{2,n} \cup T_{2,n} \cup S_{3,n} \cup T_{3,n} \subset \R$.
Next note, part (5) of lemma \ref{lemDesAsc} implies
that $D_n,D_n'$ do not intersect $S_{2,n}$. Thus, to prove
part (iv), it remains to show that contours of steepest
descent cannot intersect $T_{2,n}$ from $\mathbb{H}$. To show this,
fix $s \in T_{2,n}$. Recall (see figure \ref{figSnC_xi1} and
equation (\ref{eqRootsTn})) that $s \in (x,y)$ where $x$ and $y$ are
consecutive elements of $S_{2,n}$, $s$ is a root of $f_n'$ of
multiplicity $1$, and $s$ is the unique root of $f_n'$ in $(x,y)$
(see parts (a3,b3,c3) of lemma \ref{lemfn'Case1}).
Also, equation (\ref{eqfn'}) gives $f_n'(w) \in \R$ for all $w \in (x,y)$, and
\begin{equation*}
\lim_{w \in \R, w \uparrow y} f_n'(w) = +\infty
\hspace{0.5cm} \text{and} \hspace{0.5cm}
\lim_{w \in \R, w \downarrow x} f_n'(w) = -\infty.
\end{equation*}
It thus follows that $f_n'(s) = 0$ and $f_n''(s) > 0$. Part (3) of lemma
\ref{lemDesAsc} thus shows that there are $2$ directions of steepest
descent and $2$ directions of steepest ascent for $f_n$ at $s$. Moreover,
the directions of steepest descent are given by $-\frac\pi2$ and $\frac\pi2$,
and the directions of steepest ascent are given by $0$ and $\pi$. Thus a
contour of steepest descent cannot intersect $s$ from $\mathbb{H}$. This
proves (iv).

Consider (v). First, recall (see part (iii)) that each of
$A_n,A_n'$ eventually intersects and ends in $S_{1,n} \cup T_{1,n}
\cup S_{2,n} \cup T_{2,n} \cup S_{3,n} \cup T_{3,n} \subset \R$,
or they do not intersect $\R$ and are unbounded.
Next note, part (5) of lemma \ref{lemDesAsc} implies
that $A_n,A_n'$ do not intersect $S_{1,n} \cup S_{3,n}$. Thus, to prove
part (v), it remains to show that contours of
steepest ascent cannot intersect $T_{1,n} \cup T_{3,n}$ from $\mathbb{H}$.
This follows from similar arguments to those used in the
proof of part (iv).

Consider (vi) for possibility (a). Recall (see part (iv)) that $D_n$
ends either in $S_{1,n} \cup T_{1,n}$ or in $S_{3,n} \cup T_{3,n}$.
We argue by contradiction: Assume that $D_n$ ends in $S_{1,n} \cup T_{1,n}$.
Recall (see part (i)) that $D_n$ and $A_n$ leave $t_{1,n}$ in the
directions $\frac\pi3$ and $\frac{2\pi}3$ respectively. Next recall
(see part (v)) that $A_n$ ends either in $S_{2,n} \cup T_{2,n}$, or
$A_n$ does not intersect $\R$ and is unbounded. Thus,
since $S_{1,n} \cup T_{1,n} \subset [\underline{S_{1,n}}, \overline{S_{1,n}}]$
and $S_{2,n} \cup T_{2,n} \subset [\underline{S_{2,n}}, \overline{S_{2,n}}]$,
figure \ref{figSnC_xi1} clearly implies that $A_n$ must intersect $D_n$.
This contradicts part (ii). Thus $D_n$ cannot end in $S_{1,n} \cup T_{1,n}$,
and must end in $S_{3,n} \cup T_{3,n}$. Moreover, a similar argument by
contradiction shows that $A_n$ cannot be unbounded, and so $A_n$ must end in
$S_{2,n} \cup T_{2,n}$. This proves (vi) for possibility (a).
(vi) for possibility (b) follows similarly.

Consider (vi) for possibility (c). Recall (see part (iv)) each of
$\{D_n,D_n'\}$ end either in $S_{1,n} \cup T_{1,n}$ or in
$S_{3,n} \cup T_{3,n}$. We argue by contradiction: Assume that both
of $\{D_n,D_n'\}$ end in $S_{1,n} \cup T_{1,n}$. Recall (see part
(i) and right of figure \ref{figConDirCase1}) that $D_n,D_n',A_n,A_n'$
leave from $t_{1,n} \in \mathbb{H}$ in orthogonal directions, and we
alternately encounter elements from $\{D_n,D_n'\}$ and $\{A_n,A_n'\}$
when proceeding counter-clockwise in a neighbourhood of $t_{1,n}$.
Next recall (see part (v)) that each of $\{A_n,A_n'\}$ end either in
$S_{2,n} \cup T_{2,n}$, or they do not intersect $\R$ and are unbounded.
Thus, since $S_{1,n} \cup T_{1,n} \subset [\underline{S_{1,n}}, \overline{S_{1,n}}]$
and $S_{2,n} \cup T_{2,n} \subset [\underline{S_{2,n}}, \overline{S_{2,n}}]$,
figure \ref{figSnC_xi1} clearly implies that one of $\{A_n,A_n'\}$
must intersect at least one of $\{D_n,D_n'\}$. This contradicts part (ii).
Thus both of $\{D_n,D_n'\}$ cannot end in $S_{1,n} \cup T_{1,n}$.
Similarly, both of $\{D_n,D_n'\}$ cannot end in $S_{3,n} \cup T_{3,n}$.
Thus one of $\{D_n,D_n'\}$ must end in $S_{1,n} \cup T_{1,n}$, and
the other must end $S_{3,n} \cup T_{3,n}$. Also, we can similarly
argue by contradiction that one of $\{A_n,A_n'\}$ must end in
$S_{2,n} \cup T_{2,n}$, and the other is unbounded. Finally, if we choose
the labelling as described in part (vi), it remains to show that
the contours do not leave $t_{1,n}$ in the counter-clockwise order
$D_n,A_n',D_n',A_n$. This again follows from a similar argument by
contradiction. This proves (vi).

Next consider $f_t$. Recall (see equation (\ref{eqRootsTn})) that $T_t$ is
the set of roots of $f_t'$, and the behaviour of $T_t$ is described
in lemma \ref{lemfn'Case1} and displayed in figure \ref{figSnC_xi1}. Also,
the directions of steepest decent/ascent for $f_t$ at $t$ are
shown on the right of figure \ref{figConDirCase1}. We define:
\begin{enumerate}
\item[(vii)]
Let $D_t$ and $A_t$ denote the contours of steepest descent and ascent
(respectively) for $f_t$ which start $t \in (\overline{S_1},+\infty)$, and which
enter $\mathbb{H}$ in the directions $\frac\pi3$ and $\frac{2\pi}3$
respectively. The contours are defined to follow the unique directions
of steepest descent and ascent (as appropriate) at each $w \in \C \setminus (S \cup T)$
(see part (3) of lemma \ref{lemDesAsc}), and are defined to end whenever/if they
first intersect a point in $S \cup T$. Finally, if no such point of intersection
exists, the contours are unbounded.
\end{enumerate}
This definition, and similar arguments to those used to prove
part (ii), above, then give:
\begin{enumerate}
\item[(viii)]
$D_t$ and $A_t$ do not intersect except at $t$, and they are simple.
\end{enumerate}
Next we investigate the possible end-points of the contours in
$S \cup T_t$. Recall that these are depicted in figure \ref{figSnC_xi1},
and $T_t = T_{1,t} \cup T_{2,t} \cup T_{3,t} \cup \{t\}$
(see equation (\ref{eqRootsT1nT2n})). Also, since $t$ is the
start-point of $D_t$ and $A_t$, part (viii) implies that
the contours do not intersect $\{t\}$ again. Part (vii) thus
implies that they end whenever/if they intersect a point in
$S_1 \cup T_{1,t} \cup S_2 \cup T_{2,t} \cup S_3 \cup T_{3,t} \subset \R$.
Similar arguments to those used to prove part (iii), above, then give:
\begin{enumerate}
\item[(ix)]
$D_t$ and $A_t$ either eventually intersect and end in
$S_1 \cup T_{1,t} \cup S_2 \cup T_{2,t} \cup S_3 \cup T_{3,t} \subset \R$,
or they do not intersect $\R$ and are unbounded. Moreover, if we exclude
the start-points and end-points, the contours are contained in $\mathbb{H}$.
\end{enumerate}
Next, we use the above observations to show:
\begin{enumerate}
\item[(x)]
$D_t$ ends in
$(\underline{S_3}, \overline{S_3}] \cup [\underline{S_2}, \overline{S_2})$.
\item[(xi)]
$A_t$ ends in
$(\underline{S_3}, \overline{S_3}] \cup [\underline{S_2}, \overline{S_2})$,
or it does not intersect $\R$ and is unbounded.
\end{enumerate}
Finally, we will show:
\begin{enumerate}
\item[(xii)]
$D_t$ ends in $(\underline{S_3}, \overline{S_3})$.
\item[(xiii)]
$A_t$ ends in $(\underline{S_2}, \overline{S_2})$.
\end{enumerate}
The required results then follows from parts (vii,viii,ix,xii,xiii).

Consider (x). First note, part (4) of lemma \ref{lemDesAsc} implies
that $D_t$ is bounded. Part (ix) then implies that $D_t$ ends in
$S_1 \cup T_{1,t} \cup S_2 \cup T_{2,t} \cup S_3 \cup T_{3,t} \subset \R$.
Thus, since $S_i \cup T_{i,t} \subset [\underline{S_i}, \overline{S_i}]$
for all $i \in \{1,2,3\}$ (see equation (\ref{eqRootsT1nT2n})), $D_t$ ends
in $[\underline{S_3}, \overline{S_3}] \cup [\underline{S_2}, \overline{S_2}]
\cup [\underline{S_1}, \overline{S_1}]$. Thus, to show part (x), it is
sufficiently to show that $D_t$ does not end in
$\{\underline{S_3}, \overline{S_2}\} \cup [\underline{S_1}, \overline{S_1}]$.
We argue by contradiction: First assume that $D_t$ ends at $\overline{S_1}$.
Then (see parts (vii,ix)) $D_t$ starts at $t > \overline{S_1}$ in the
direction $\frac\pi3$, ends at $\overline{S_1}$, and is otherwise
contained in $\mathbb{H}$. Consider the simple closed contour consisting
of $D_t$ and the interval $[\overline{S_1},t]$. Note, figure \ref{figImfnExt}
and part (2) of lemma \ref{lemDesAsc} implies that the function
$w \mapsto \text{Im}(f_t(w))$ is constant on this contour
(indeed, the constant valve is $0$). Therefore, since
this function is harmonic, it is also constant in the domain
bounded by the simple closed contour. Equation (\ref{eqImt}) easily
shows that this is not true. Thus we have a contradiction, and so
$D_t$ does not end at $\overline{S_1}$. Next assume that $D_t$ ends at
a point $d_t \in [\underline{S_1},\overline{S_1})$. Note, as above,
$w \mapsto \text{Im}(f_t(w))$ is constant on $D_t$. In particular,
this gives $\text{Im}(f_t(d_t)) = \text{Im}(f_t(t))$. Recall (see
equation (\ref{eqS1S2S3In})) that 
$(\overline{S_1}-\e,\overline{S_1}] \subset S_1$ for all $\e>0$
sufficiently small. Equation (\ref{eqIt}) and figure \ref{figImfnExt}
then give $\text{Im}(f_t(s)) > 0$ for all $s \in [\underline{S_1}, \overline{S_1})$.
However, these also give $\text{Im}(f_t(t)) = 0$. This contradicts
$\text{Im}(f_t(d_t)) = \text{Im}(f_t(t))$, and so $D_t$ does not end in
$[\underline{S_1},\overline{S_1})$. Next assume that $D_t$ ends at $\overline{S_2}$
or at $\underline{S_3}$. Note, equations (\ref{eqS1S2S3In}, \ref{eqIt}) and figure
\ref{figImfnExt} give $\text{Im}(f_t(\overline{S_2})) = \pi \mu[S_1] > 0$
and $\text{Im}(f_t(\underline{S_3})) = \pi \eta > 0$. We then proceed as in
the previous case to get a contradiction, and so $D_t$ does not end
at $\overline{S_2}$ or at $\underline{S_3}$. This proves (x). Part (xi) follows similarly.

Consider (xii). Part (x) implies that it is sufficient to show that $D_t$
does not end in $\{\overline{S_3}\} \cup [\underline{S_2},\overline{S_2})$. We
argue by contradiction: First assume that $D_t$ ends in
$[\underline{S_2},\overline{S_2})$. Then (see parts (vii,ix)) $D_t$
starts at $t$ in the direction $\frac\pi3$, ends in $[\underline{S_2},\overline{S_2})$,
and is otherwise contained in $\mathbb{H}$. Denote the end-point by $d_t$.
Parts (vii,viii,xi) then imply that $A_t$ stars at $t$ in the direction
$\frac{2\pi}3$, $A_t$ and $D_t$ do not intersect except at $t$, $A_t$
ends in $(d_t,\overline{S_2})$, and $A_t$ is otherwise contained in $\mathbb{H}$.
Denote the end-point of $A_t$ by $a_t \in (d_t,\overline{S_2})$. Note, part (2)
of lemma \ref{lemDesAsc}
implies that the function $w \mapsto \text{Im}(f_t(w))$ is constant on
$D_t$ and $A_t$ (indeed, the constant valve is $0$). In particular this gives
$\text{Im}(f_t(d_t)) = \text{Im}(f_t(a_t))$. Equation (\ref{eqIt}) and
figure \ref{figImfnExt} imply that this can only happen when
$(d_t,a_t)$ is entirely contained in a sub-interval of
$[\underline{S_2},\overline{S_2}] \setminus S_2$, in which case
$w \mapsto \text{Im}(f_t(w))$ is constant on the interval $[d_t,a_t]$
also. Therefore $w \mapsto \text{Im}(f_t(w))$ is an harmonic
function which is constant on the simple closed contour
consisting of $D_t$ and $A_t$ and interval $[d_t,a_t]$, and so it also constant
in the domain bounded by the closed contour. Equation (\ref{eqImt})
easily shows that this is not true. Thus we have a contradiction, and
so $D_t$ does not end in $[\underline{S_2},\overline{S_2})$. Similarly, we can
argue by contradiction that $D_t$ does not end at $\overline{S_3}$. This
proves (xii).

Consider (xiii). Part (xi) implies that it is sufficient to show that $A_t$
does not end in $(\underline{S_3},\overline{S_3}] \cup \{\underline{S_2}\}$,
and that $A_t$ is not unbounded. Using, parts (vii,xii), this follows
from arguments by contradiction similar to those used in the proof of (xii).
\end{proof}

\begin{rem}
Note, the nature of $D_t$ and $A_t$ for case (1) of lemma \ref{lemCases},
as described above, proves an interesting inequality. Recall that
$D_t$ and $A_t$ both start at $t \in (\overline{S_1}, +\infty)$, $D_t$
ends at a point $d_t \in (\underline{S_3}, \overline{S_3})$, and $A_t$
ends at a point $a_t \in (\underline{S_2}, \overline{S_2})$. Also recall that
$w \mapsto \text{Im}(f_t(w))$ is constant along $D_t$ and $A_t$.
Figure \ref{figImfnExt} implies that $\text{Im}(f_t(t)) = 0$,
and so $0 = \text{Im}(f_t(t)) = \text{Im}(f_t(d_t)) = \text{Im}(f_t(a_t))$.
Figure \ref{figImfnExt} also implies that this can only occur if
$\mu[S_1] - (\l-\mu)[S_2] < 0$. This inequality must be satisfied
whenever case (1) is satisfied. The authors are not aware of a direct
proof of this inequality, or of its significance. Though we will not discuss
them, analogous inequalities exist for the other cases of lemma \ref{lemCases}.
\end{rem}

We are now in a position to define the contours, $\gamma_{1,n}^+$ and $\Gamma_{1,n}^+$,
that satisfy lemma \ref{lemDesAsc1-12} for case (1) of lemma \ref{lemCases}. 
As above, fix $\xi>0$ sufficiently small such that equations
(\ref{eqxi}, \ref{eqxi1}, \ref{eqxi4}, \ref{eqxi5}) are
satisfied. Next, fix $\theta \in (\frac14,\frac13)$ as in lemma \ref{lemDesAsc1-12},
and $\{q_n\}_{n\ge1} \subset \R$ as in definition \ref{defmnpn}.
Note, since $f_t'''(t) > 0$ (see case (1) of lemma \ref{lemCases}),
lemma \ref{lemftn'} and definition \ref{defmnpn} imply that $\{q_n\}_{n\ge1}$
converges to a positive constant as $n \to \infty$. Also, part (4) of
lemma \ref{lemNonAsyRoots} gives $t_{1,n} = t + O(n^{-\frac13})$ and
$t_{2,n} = t + O(n^{-\frac13})$, and so
\begin{equation*}
\{t_{1,n},t_{2,n}\} \subset B(t, n^{-\theta} q_n) \subset B(t,\xi).
\end{equation*}
Thus, $D_n$ and $A_n$ (see previous lemma) both start inside
$B(t, n^{-\theta} q_n)$ and $B(t,\xi)$. Moreover, assuming for simplicity
that $[0,1]$ is the domain of definition of these contours, we define:
\begin{align}
\nonumber
d_{1,n}
&:= D_n(\sup \{ y \in [0,1] : D_n(y) \in \text{cl}(B(t,n^{-\theta} q_n)) \}), \\
\label{eqdn}
d_{2,n}
&:= D_n(\sup \{ y \in [0,1] : D_n(y) \in \text{cl}(B(t,\xi)) \}), \\
\nonumber
a_{1,n}
&:= A_n(\sup \{ y \in [0,1] : A_n(y) \in \text{cl}(B(t,n^{-\theta} q_n)) \}), \\
\nonumber
a_{2,n}
&:= A_n(\sup \{ y \in [0,1] : A_n(y) \in \text{cl}(B(t,\xi)) \}).
\end{align}
In words, $d_{1,n}$ and $d_{2,n}$
denote the points where $D_n$ `exits' $B(t, n^{-\theta} q_n)$ and
$B(t, \xi)$ respectively. Similarly for $a_{1,n}$, $a_{2,n}$ and
$A_n$.

Next denote the equivalent quantities for $D_t$ and $A_t$ by
$d_{1,t}, d_{2,t}, a_{1,t}, a_{2,t}$. Also, fixing $\xi>0$
sufficiently small as above, note that case (1) of lemma
\ref{lemCases} and equation (\ref{eqAnalSetftfn}) give,
\begin{equation*}
\C_\xi = \{w \in \C : \text{Re}(w) > \overline{S_1} + 2 \xi 
\text{ or } |\text{Im}(w)| > \xi^4 \}.
\end{equation*}
Recall that $D_t$ ends in $(\underline{S_3}, \overline{S_3}) \subset \C \setminus \C_\xi$,
and $A_t$ ends in $(\underline{S_2}, \overline{S_2}) \subset \C \setminus \C_\xi$,
and let $d_{3,t}$ and $a_{3,t}$ denote the points where $D_t$ and $A_t$ `exit' $\C_\xi$
respectively. Note, it is always possible to choose the $\xi>0$ sufficiently small
that $d_{3,t}$ and $a_{3,t}$ are on the upper boundary of $\C_\xi$:
\begin{equation}
\label{eqd3ta3t1}
\text{Im}(d_{3,t}) = \text{Im}(a_{3,t}) = \xi^4.
\end{equation}
Also note, since $D_t$ and $A_t$ end in $(\underline{S_3}, \overline{S_3})$
and $(\underline{S_2}, \overline{S_2})$ respectively, it is always
possible to choose the $\xi>0$ sufficiently small such that,
\begin{equation}
\label{eqd3ta3t2}
\text{Re}(d_{3,t}) \in (\underline{S_3} + c, \overline{S_3} - c)
\hspace{0.5cm} \text{and} \hspace{0.5cm}
\text{Re}(a_{3,t}) \in (\underline{S_2} + c, \overline{S_2} - c),
\end{equation}
where $c = c(t) > 0$ is some constant which is independent of $\xi$.
Finally define the following contours, which are depicted in figure
\ref{figCongnGnCase1}:
\begin{definition}
\label{defgnGnCase1}
Define $\g_{1,n}^+$ to be the simple contour which:
\begin{itemize}
\item
starts at $t \in (\overline{S_1}, +\infty)$,
\item
then traverses the straight line from $t$ to
$d_{1,n} \in \partial B(t, n^{-\theta} q_n)$,
\item
then traverses that section of $D_n$ from $d_{1,n} \in \partial B(t, n^{-\theta} q_n)$
to $d_{2,n} \in \partial B(t, \xi)$,
\item
then traverses the shortest arc of $\partial B(t, \xi)$ from
$d_{2,n} \in \partial B(t, \xi)$ to $d_{2,t} \in \partial B(t, \xi)$,
\item
then traverses that section of $D_t$ from $d_{2,t}$ to $d_{3,t}$,
\item
then traverses the straight line from $d_{3,t}$ to $\text{Re}(d_{3,t})$,
\item
then ends at $\text{Re}(d_{3,t}) \in (\underline{S_3} + c, \overline{S_3} - c)$.
\end{itemize}

Next, let $\tilde{D}_n$ and $\tilde{A}_n$ denote the contours of steepest
descent and ascent (respectively) for $\tilde{f}_n$ which
are analogous to $D_n$ and $A_n$. Moreover, let
$\tilde{d}_{1,n}, \tilde{d}_{2,n}, \tilde{a}_{1,n}, \tilde{a}_{2,n}$ denote
analogous quantities to those defined in equation
(\ref{eqdn}). Finally, define $\G_{1,n}^+$ to be the simple contour which:
\begin{itemize}
\item
starts at $t \in (\overline{S_1}, +\infty)$,
\item
then traverses the straight line from $t$ to
$\tilde{a}_{1,n} \in \partial B(t, n^{-\theta} q_n)$,
\item
then traverses that section of $\tilde{A}_n$ from
$\tilde{a}_{1,n} \in \partial B(t, n^{-\theta} q_n)$
to $\tilde{a}_{2,n} \in \partial B(t, \xi)$,
\item
then traverses the shortest arc of $\partial B(t, \xi)$ from
$\tilde{a}_{2,n} \in \partial B(t, \xi)$ to $a_{2,t} \in \partial B(t, \xi)$,
\item
then traverses that section of $A_t$ from $a_{2,t}$ to $a_{3,t}$,
\item
then traverses the straight line from $a_{3,t}$ to $\text{Re}(a_{3,t})$,
\item
then ends at $\text{Re}(a_{3,t}) \in (\underline{S_2} + c, \overline{S_2} - c)$.
\end{itemize}
\end{definition}

\begin{figure}[t]
\centering
\begin{tikzpicture};

\draw (12,0) --++(.375,.65);
\draw[arrows=->,line width=1pt](12.2,.34641)--(12.205,.35507);
\fill [black] (12.375,.65) circle (.05cm);
\draw (12.8,.65) node {\scriptsize $d_{1,n}$};
\draw (12.375,.65) --++(.557,.525);
\fill [black] (12.932,1.175) circle (.05cm);
\draw (13.3,1.2) node {\scriptsize $d_{2,n}$};
\draw [domain=51.567:68.755] plot ({12+(1.5)*(cos(\x))}, {(1.5)*(sin(\x))});
\fill [black] (12.543,1.4) circle (.05cm);
\draw (12.35,1.65) node {\scriptsize $d_{2,t}$};
\draw plot [smooth, tension=1] coordinates
{ (12.543,1.398) (11,2.25) (4,2.25) (1.4,0)};
\draw[arrows=->,line width=1pt](7,2.59)--(6.99,2.59);
\draw (7,3) node {$\g_{1,n}^+$};
\fill (1.5,0) [white] circle (.5cm);
\draw (1.5,.5) --++(0,-.5);
\fill [black] (1.5,.5) circle (.05cm);
\draw (1.2,.7) node {\scriptsize $d_{3,t}$};

\draw (12,0) --++(-.375,.65);
\draw[arrows=->,line width=1pt](11.8,.34641)--(11.795,.35507);
\fill [black] (11.625,.65) circle (.05cm);
\draw (11.25,.7) node {\scriptsize $\tilde{a}_{1,n}$};
\draw (11.625,.65) --++(-.557,.525);
\fill [black] (11.068,1.175) circle (.05cm);
\draw (10.7,1.25) node {\scriptsize $\tilde{a}_{2,n}$};
\draw [domain=111.567:128.755] plot ({12+(1.5)*(cos(\x))}, {(1.5)*(sin(\x))});
\fill [black] (11.456,1.4) circle (.05cm);
\draw (11.65,1.6) node {\scriptsize $a_{2,t}$};
\draw plot [smooth, tension=1] coordinates
{ (11.456,1.398) (10,2) (6.5,1.5) (5.45,0)};
\draw[arrows=->,line width=1pt](9,2.01)--(8.99,2.01);
\draw (9,1.6) node {$\G_{1,n}^+$};
\fill (5.4,0) [white] circle (.5cm);
\draw (5.5,.5) --++(0,-.5);
\fill [black] (5.5,.5) circle (.05cm);
\draw (5.2,.7) node {\scriptsize $a_{3,t}$};

\fill [black] (12,0) circle (.05cm);
\draw (12,-.25) node {\scriptsize $t$};
\draw [domain=0:180, dotted] plot ({12+(.75)*(cos(\x))}, {(.75)*(sin(\x))});
\draw [domain=0:180, dotted] plot ({12+(1.5)*(cos(\x))}, {(1.5)*(sin(\x))});

\draw [dashed] (-.75,.5) --++(10.75,0) --++ (0,-.5);
\draw[arrows=<->,line width=1pt](2.5,0)--(2.5,.5);
\draw (2.75,0.25) node {\scriptsize $\xi^4$};

\draw (10,-.1) --++(0,.2);
\draw (9,-.1) --++(0,.2);
\draw (8,-.1) --++(0,.2);
\draw (7,-.1) --++(0,.2);
\draw (6,.1) --++(0,-.6);
\draw (5,.1) --++(0,-.6);
\draw (4,-.1) --++(0,.2);
\draw (3,-.1) --++(0,.2);
\draw (2,.1) --++(0,-.6);
\draw (1,.1) --++(0,-.6);
\draw (0,-.1) --++(0,.2);

\draw (10,-.4) node {\scriptsize $\overline{S_1}+2\xi$};
\draw (9,-.4) node {\scriptsize $\overline{S_1}$};
\draw (8.5,-.4) node {\scriptsize $<$};
\draw (8,-.4) node {\scriptsize $\underline{S_1}$};
\draw (7.5,-.4) node {\scriptsize $\le$};
\draw (7,-.4) node {\scriptsize $\overline{S_2}$};
\draw (6.1,-.7) node {\scriptsize $\overline{S_2}-c$};
\draw (4.9,-.7) node {\scriptsize $\underline{S_2}+c$};
\draw (4,-.4) node {\scriptsize $\underline{S_2}$};
\draw (3.5,-.4) node {\scriptsize $\le$};
\draw (3,-.4) node {\scriptsize $\overline{S_3}$};
\draw (2.1,-.7) node {\scriptsize $\overline{S_3}-c$};
\draw (.9,-.7) node {\scriptsize $\underline{S_3}+c$};
\draw (0,-.4) node {\scriptsize $\underline{S_3}$};

\draw (-.75,1.5) node {$\mathbb{H}$};
\draw [dotted] (-.55,0) --++(14.15,0);
\draw (-.75,0) node {$\R$};
\draw[arrows=<->,line width=1pt](12,-.5)--(13.5,-.5);
\draw[arrows=<->,line width=1pt](12,-.65)--(12.75,-.65);
\draw (13.125,-.75) node {\scriptsize $\xi$};
\draw (12.375,-.9) node {\scriptsize $n^{-\theta} q_n$};

\end{tikzpicture}
\caption{The contours defined in definition \ref{defgnGnCase1}.
The dashed lines represent boundaries of $\C_\xi$.
$c = c(t) > 0$ is independent of $\xi$.}
\label{figCongnGnCase1}
\end{figure}
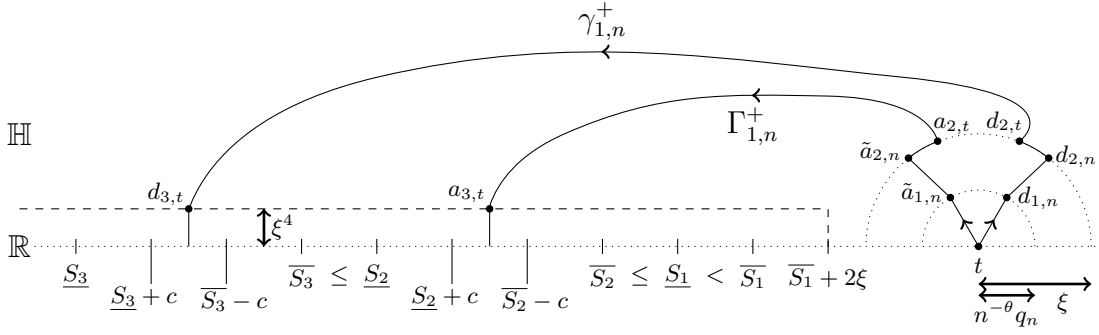

We finally show that the above contours satisfy the requirements of lemma
\ref{lemDesAsc1-12}:
\begin{proof}[Proof of lemma \ref{lemDesAsc1-12} for case (1) of lemma \ref{lemCases}:]
Part (1) of lemma \ref{lemDesAsc1-12} follows easily from lemma \ref{lemConAscDesCase1}
and definition \ref{defgnGnCase1} (see also figures \ref{figDesAsc1-12} and
\ref{figCongnGnCase1}). Consider (2). We will show that:
\begin{enumerate}
\item[(i)]
$\text{Arg}(d_{1,n}-t) = \frac\pi3 + O(n^{-\frac13+\theta})$
and
$\text{Arg}(a_{1,n}-t) = \frac{2\pi}3 + O(n^{-\frac13+\theta})$.
\end{enumerate}
Similarly, we can show that
$\text{Arg}(\tilde{d}_{1,n}-t) = \frac\pi3 + O(n^{-\frac13+\theta})$
and $\text{Arg}(\tilde{a}_{1,n}-t) = \frac{2\pi}3 + O(n^{-\frac13+\theta})$.
This proves (2).

Consider (3). We will show, for all $\xi>0$ sufficiently small
as in the statement of this lemma, that there
exists an integer $n(\xi) > 0$ such that the following are
satisfied for all $n>n(\xi)$:
\begin{enumerate}
\item[(ii)]
$\text{Arg} (d_{2,n} - t) = \frac\pi3 + O(\xi)$
and $\text{Arg} (a_{2,n} - t) = \frac{2\pi}3 + O(\xi)$.
\end{enumerate}
Then, letting $R_n$ denote the real-part of $f_n$ (see equation
(\ref{eqRn})), we will use this to show that there exists a choice
of the above $\xi$ such that the following are satisfied:
\begin{enumerate}
\item[(iii)]
$R_n(w) \le R_n(d_{1,n})$ for all $w$ on that section of $D_n$ from
$d_{1,n}$ to $d_{2,n}$.
\item[(iv)]
$R_n(w) \le R_n(d_{1,n})$ for all $w$ on the shortest arc of
$\partial B(t, \xi)$ from $d_{2,n}$ to $d_{2,t}$.
\item[(v)]
$R_n(w) \le R_n(d_{1,n})$ for all $w$ on that section of $D_t$ from
$d_{2,t}$ to $d_{3,t}$.
\item[(vi)]
$R_n(w) \le R_n(d_{1,n})$ for all $w$ on the straight line from
$d_{3,t}$ to $\text{Re}(d_{3,t})$.
\end{enumerate}
Definition \ref{defgnGnCase1} then implies that $R_n(w) \le R_n(d_{1,n})$
for all $w \in \g_{1,n}^+ \setminus B(t, n^{-\theta} q_n)$. This proves (3).
(4) follows similarly.

Consider (5). First note, we can proceed similarly to the proofs of
parts (i,ii) to show the following: For all $\xi>0$ sufficiently
small, there exists an integer $n(\xi) > 0$ such that
$\text{Arg} (w - t) = \frac\pi3 + O(\xi)$ for all $n>n(\xi)$ and
uniformly for $w \in \g_{1,n}^+ \cap \text{cl}(B(t,\xi))$, and
$\text{Arg} (z - t) = \frac{2\pi}3 + O(\xi)$ for all $n>n(\xi)$ and
uniformly for $z \in \G_{1,n}^+ \cap \text{cl}(B(t,\xi))$. These
contour sections are thus contained in the cones shown in figure
\ref{figCongnGnCase1Cones}. Next note, definition \ref{defgnGnCase1}
implies that $\g_{1,n}^+$ and $\G_{1,n}^+$ depend on $n$ inside the
cones, are independent of $n$ outside the cones, and the parts outside
the cones never intersect. Then, for $\xi>0$ fixed sufficiently small,
figure \ref{figCongnGnCase1Cones} and a simple geometric argument
proves part (5).

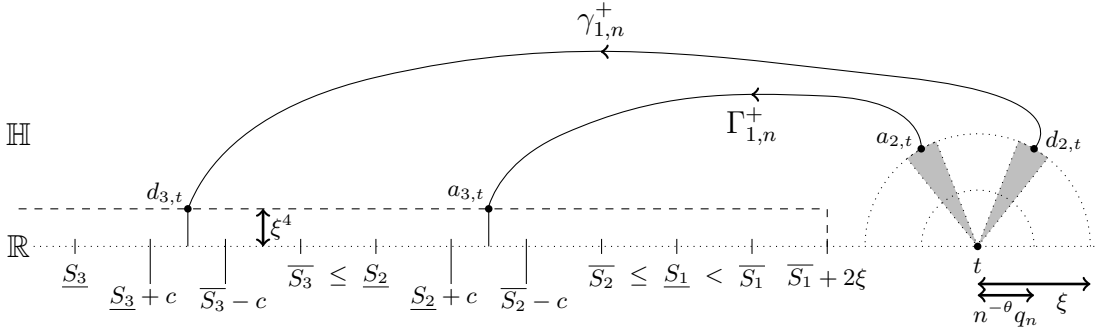
\begin{figure}[t]
\centering
\begin{tikzpicture};
 
\filldraw[fill=gray!50!white, dotted]
(12,0) -- (12.932,1.175) arc (51.567:68.755:1.5) -- (12,0);
\draw (13.15,1.4) node {\scriptsize $d_{2,t}$};
\fill [black] (12.75,1.299) circle (.05cm);
\draw plot [smooth, tension=1] coordinates
{ (12.75,1.299) (11,2.25) (4,2.25) (1.4,0)};
\draw[arrows=->,line width=1pt](7,2.59)--(6.99,2.59);
\draw (7,3) node {$\g_{1,n}^+$};
\fill (1.5,0) [white] circle (.5cm);
\draw (1.5,.5) --++(0,-.5);
\fill [black] (1.5,.5) circle (.05cm);
\draw (1.2,.7) node {\scriptsize $d_{3,t}$};

\filldraw[fill=gray!50!white, dotted]
(12,0) -- (11.449,1.395) arc (111.567:128.755:1.5) -- (12,0);
\draw (10.9,1.4) node {\scriptsize $a_{2,t}$};
\fill [black] (11.25,1.299) circle (.05cm);
\draw plot [smooth, tension=1] coordinates
{ (11.25,1.299) (10,2) (6.5,1.5) (5.45,0)};
\draw[arrows=->,line width=1pt](9,2.02)--(8.99,2.02);
\draw (9,1.6) node {$\G_{1,n}^+$};
\fill (5.4,0) [white] circle (.5cm);
\draw (5.5,.5) --++(0,-.5);
\fill [black] (5.5,.5) circle (.05cm);
\draw (5.2,.7) node {\scriptsize $a_{3,t}$};

\fill [black] (12,0) circle (.05cm);
\draw (12,-.25) node {\scriptsize $t$};
\draw [domain=0:180, dotted] plot ({12+(.75)*(cos(\x))}, {(.75)*(sin(\x))});
\draw [domain=0:180, dotted] plot ({12+(1.5)*(cos(\x))}, {(1.5)*(sin(\x))});

\draw [dashed] (-.75,.5) --++(10.75,0) --++ (0,-.5);
\draw[arrows=<->,line width=1pt](2.5,0)--(2.5,.5);
\draw (2.75,0.25) node {\scriptsize $\xi^4$};

\draw (10,-.1) --++(0,.2);
\draw (9,-.1) --++(0,.2);
\draw (8,-.1) --++(0,.2);
\draw (7,-.1) --++(0,.2);
\draw (6,.1) --++(0,-.6);
\draw (5,.1) --++(0,-.6);
\draw (4,-.1) --++(0,.2);
\draw (3,-.1) --++(0,.2);
\draw (2,.1) --++(0,-.6);
\draw (1,.1) --++(0,-.6);
\draw (0,-.1) --++(0,.2);

\draw (10,-.4) node {\scriptsize $\overline{S_1}+2\xi$};
\draw (9,-.4) node {\scriptsize $\overline{S_1}$};
\draw (8.5,-.4) node {\scriptsize $<$};
\draw (8,-.4) node {\scriptsize $\underline{S_1}$};
\draw (7.5,-.4) node {\scriptsize $\le$};
\draw (7,-.4) node {\scriptsize $\overline{S_2}$};
\draw (6.1,-.7) node {\scriptsize $\overline{S_2}-c$};
\draw (4.9,-.7) node {\scriptsize $\underline{S_2}+c$};
\draw (4,-.4) node {\scriptsize $\underline{S_2}$};
\draw (3.5,-.4) node {\scriptsize $\le$};
\draw (3,-.4) node {\scriptsize $\overline{S_3}$};
\draw (2.1,-.7) node {\scriptsize $\overline{S_3}-c$};
\draw (.9,-.7) node {\scriptsize $\underline{S_3}+c$};
\draw (0,-.4) node {\scriptsize $\underline{S_3}$};

\draw (-.75,1.5) node {$\mathbb{H}$};
\draw [dotted] (-.55,0) --++(14.15,0);
\draw (-.75,0) node {$\R$};
\draw[arrows=<->,line width=1pt](12,-.5)--(13.5,-.5);
\draw[arrows=<->,line width=1pt](12,-.65)--(12.75,-.65);
\draw (13.125,-.75) node {\scriptsize $\xi$};
\draw (12.375,-.9) node {\scriptsize $n^{-\theta} q_n$};

\end{tikzpicture}
\caption{The contours defined in definition \ref{defgnGnCase1} and
depicted in figure \ref{figCongnGnCase1}. Here, we do not depict
those sections of the contours in $\text{cl}(B(t, \xi))$. Instead, we depict
two cones (the shaded areas). For all $\xi>0$ sufficiently small
and $n>n(\xi)$, $\text{Arg}(w-t) = \frac{\pi}3 + O(\xi)$ and
$\text{Arg}(z-t) = \frac{2\pi}3 + O(\xi)$ uniformly for
$w$ and $z$ in the right and left cones respectively,
$\g_{1,n}^+ \cap B (t, \xi)$ and $\G_{1,n}^+ \cap B (t, \xi)$ are
contained in the right and left cones respectively, and
$\g_{1,n}^+ \setminus B (t, \xi)$ and $\G_{1,n}^+ \setminus B (t, \xi)$
are independent of $n$.}
\label{figCongnGnCase1Cones}
\end{figure}

Consider (6). Recall (see the proof of part (5)), that for all $\xi>0$
sufficiently small, there exists an integer $n(\xi) > 0$ such that
$\text{Arg} (w - t) = \frac\pi3 + O(\xi)$ for all $n>n(\xi)$ and
uniformly for $w$ on that section of $D_n$ from $d_{1,n}$
to $d_{2,n}$. Looking at figure \ref{figCongnGnCase1Cones}, this
contour is contained in that section of the right cone in
$B(t, \xi) \setminus B(t, n^{-\theta} q_n)$. We will show the following:
\begin{enumerate}
\item[(vii)]
Consider that section of the right cone discussed above. Then,
we can choose the above $\xi$ and $n(\xi)$ such that the direction
of steepest descent for $f_n$ at $w$ equals $\frac\pi3 + O(\xi)$
for all $n>n(\xi)$ and uniformly for $w$ in that section.
\end{enumerate}
Finally recall (see lemma \ref{lemConAscDesCase1} and definition \ref{defgnGnCase1})
that $D_n$ follows the uniquely defined directions of steepest descent for
$f_n$. The above observations imply that the length of that section
of $D_n$ from $d_{1,n}$ to $d_{2,n}$ is of order $O(\xi)$ for all
$n>n(\xi)$. Moreover, definition \ref{defgnGnCase1} trivially
implies that the remaining sections of $\g_{1,n}^+$ are of order
at most $O(1)$ for all $n$ sufficiently large. Therefore we can fix
$\xi>0$ sufficiently small such that $|\g_{1,n}^+| = O(1)$. Similarly, we can show that
$|\G_{1,n}^+| = O(1)$. This proves (6).

Consider (i). First recall (see lemma \ref{lemConAscDesCase1}) that
$D_n$ is a contour of steepest descent for $f_n$ which starts at $t_{1,n}$,
and (see equation (\ref{eqdn})) $d_{1,n} \in \partial B(t,n^{-\theta} q_n)$
denotes that point at which $D_n$ `exits' $B(t,n^{-\theta} q_n)$. Thus,
letting $R_n$ and $I_n$ respectively denote the real and imaginary-parts of
$f_n$, parts (1,2) of lemma \ref{lemDesAsc} give,
\begin{equation*}
R_n(d_{1,n}) < R_n(t_{1,n})
\hspace{0.5cm} \text{and} \hspace{0.5cm}
I_n(d_{1,n}) = I_n(t_{1,n}).
\end{equation*}
Thus, since $t_{1,n} = t + O(n^{-\frac13})$ (see part (4) of lemma \ref{lemNonAsyRoots}),
and since $f_n(t) \in \R$, part (1) of corollary
\ref{corTay} gives,
\begin{equation}
\label{eqlemgnGnCase11}
R_n(d_{1,n}) < f_n(t) + O(n^{-1})
\hspace{0.5cm} \text{and} \hspace{0.5cm}
I_n(d_{1,n}) = O(n^{-1}).
\end{equation}
Similar considerations for the contour, $A_n$, of steepest ascent
for $f_n$ give,
\begin{equation}
\label{eqlemgnGnCase12}
R_n(a_{1,n}) > f_n(t) + O(n^{-1})
\hspace{0.5cm} \text{and} \hspace{0.5cm}
I_n(a_{1,n}) = O(n^{-1}).
\end{equation}

Next note, since $\theta \in (\frac14,\frac13)$, part (2) of
lemma \ref{lemTay} (take $\xi_n = n^{-\theta}$) gives,
\begin{equation*}
f_n(t + n^{-\theta} q_n e^{i \alpha})
= f_n(t) + \tfrac13 n^{-3\theta} e^{3i \alpha}
+ O(n^{-\frac13-2\theta}),
\end{equation*}
uniformly for $\alpha \subset (-\pi,\pi]$. Therefore, since $f_n(t) \in \R$,
\begin{align}
\label{eqlemgnGnCase13}
R_n(t + n^{-\theta} q_n e^{i \alpha})
&= f_n(t) + \tfrac13 n^{-3\theta} \cos(3 \alpha)
+ O(n^{-\frac13-2\theta}), \\
\label{eqlemgnGnCase14}
I_n(t + n^{-\theta} q_n e^{i \alpha})
&= \tfrac13 n^{-3\theta} \sin(3 \alpha)
+ O(n^{-\frac13-2\theta}),
\end{align}
uniformly for $\alpha \subset (-\pi,\pi]$.

Equation (\ref{eqlemgnGnCase14}), and the second parts of
equations (\ref{eqlemgnGnCase11}, \ref{eqlemgnGnCase12}), give
\begin{equation}
\label{eqlemgnGnCase15}
\sin(3 \text{Arg}(d_{1,n}-t)) = O(n^{-\frac13+\theta})
\hspace{0.5cm} \text{and} \hspace{0.5cm}
\sin(3 \text{Arg}(a_{1,n}-t)) = O(n^{-\frac13+\theta}).
\end{equation}
Then, since
$\theta \in (\frac14,\frac13)$, equations
(\ref{eqlemgnGnCase13}, \ref{eqlemgnGnCase15}), and the first parts of
equations (\ref{eqlemgnGnCase11}, \ref{eqlemgnGnCase12}) imply the
following:
\begin{itemize}
\item
Either $d_{1,n} = d_{1,n}'$ or $d_{1,n} = d_{1,n}''$ where
$\text{Arg}(d_{1,n}'-t) = \frac\pi3 + O(n^{-\frac13+\theta})$
and $\text{Arg}(d_{1,n}''-t) = \pi + O(n^{-\frac13+\theta})$.
In either case, $R_n(d_{1,n}) = f_n(t) - \frac13 n^{-3\theta}
+ O(n^{-\frac13-2\theta})$.
\item
Either $a_{1,n} = a_{1,n}'$ or $a_{1,n} = a_{1,n}''$ where
$\text{Arg}(a_{1,n}'-t) = O(n^{-\frac13+\theta})$ and
$\text{Arg}(a_{1,n}''-t) = \frac{2\pi}3 + O(n^{-\frac13+\theta})$.
In either case, $R_n(a_{1,n}) = f_n(t) + \frac13 n^{-3\theta}
+ O(n^{-\frac13-2\theta})$.
\end{itemize}
These exit points, and each of the possibilities (a,b,c) of lemmas
\ref{lemfn'Case1} and \ref{lemConAscDesCase1}, are shown in figure
\ref{figConAnglesCase1}. We will show, for each possibility,
that $d_{1,n} = d_{1,n}'$ and $a_{1,n} = a_{1,n}''$.
This proves (i).

For possibilities (a) and (b), recall that $D_n$ and $A_n$ have
the behaviours described in lemma \ref{lemConAscDesCase1} and shown
in figure \ref{figConDesAscCase1}. Note that the contours do not
intersect outside $B(t,n^{-\theta} q_n)$, and it is easy to see that
these behaviours are only possible if $D_n$ `exits'
$B(t,n^{-\theta} q_n)$ at $d_{1,n}'$ and $A_n$ `exits'
at $a_{1,n}''$. Thus $d_{1,n} = d_{1,n}'$ and $a_{1,n} = a_{1,n}''$,
as required, for possibilities (a) and (b). For
possibility (c), proceeding as above, we can show that $a_{1,n}'$ and
$a_{1,n}''$ are the possible `exit' points of both $A_n$ and $A_n'$, and
$d_{1,n}'$ and $d_{1,n}''$ are the possible `exit' points of both $D_n$
and $D_n'$. Also, the end-points of the contours, and the fact that they
do not intersect outside $B(t,n^{-\theta} q_n)$, implies that one of
$\{A_n,A_n'\}$ `exits' at $a_{1,n}'$, one of $\{A_n,A_n'\}$ `exits' at
$a_{1,n}''$, one of $\{D_n,D_n'\}$ `exits' at $d_{1,n}'$, and one of
$\{D_n,D_n'\}$ `exits' at $d_{1,n}''$. Finally, we can proceed
similarly to the proof of part (vi) of lemma \ref{lemConAscDesCase1}
to show that $A_n', D_n, A_n, D_n'$ respectively `exit' at
$a_{1,n}', d_{1,n}', a_{1,n}'', d_{1,n}''$. Thus $d_{1,n} = d_{1,n}'$
and $a_{1,n} = a_{1,n}''$, as required, for possibility (c).

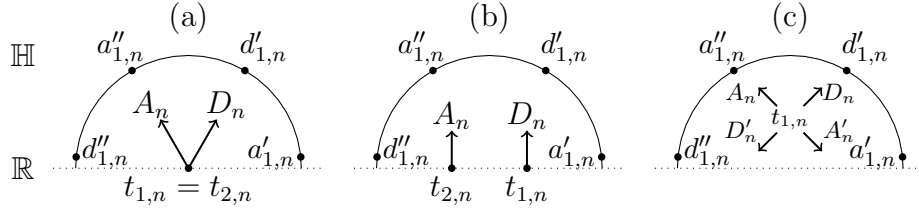
\begin{figure}[t]
\centering
\begin{tikzpicture}

\draw (-.2,1.5) node {$\mathbb{H}$};
\draw [dotted] (.2,0) --++(3.6,0);
\draw [dotted] (4.2,0) --++(3.6,0);
\draw [dotted] (8.2,0) --++(3.6,0);
\draw (-.2,0) node {$\R$};

\draw (2,2) node {(a)};
\draw [domain=0:180] plot ({2+(1.5)*cos(\x)}, {(1.5)*sin(\x)});
\fill [black] (3.493,.15) circle (.05cm);
\draw (3.1,.2) node {\small $a_{1,n}'$};
\fill [black] (2.75,1.299) circle (.05cm);
\draw (3,1.6) node {\small $d_{1,n}'$};
\fill [black] (1.25,1.299) circle (.05cm);
\draw (1.1,1.6) node {\small $a_{1,n}''$};
\fill [black] (.507,.15) circle (.05cm);
\draw (.9,.21) node {\small $d_{1,n}''$};

\fill [black] (2,0) circle (.05cm);
\draw (2,-.3) node {$t_{1,n} = t_{2,n}$};
\draw[arrows=->,line width=.75pt](2,0)--(2.375,{(.375)*sqrt(3)});
\draw (2.5,.85) node {$D_n$};
\draw[arrows=->,line width=.75pt](2,0)--(1.625,{(.375)*sqrt(3)});
\draw (1.5,.85) node {$A_n$};

\draw (6,2) node {(b)};
\draw [domain=0:180] plot ({6+(1.5)*cos(\x)}, {(1.5)*sin(\x)});
\fill [black] (7.493,.15) circle (.05cm);
\draw (7.1,.2) node {\small $a_{1,n}'$};
\fill [black] (6.75,1.299) circle (.05cm);
\draw (7,1.6) node {\small $d_{1,n}'$};
\fill [black] (5.25,1.299) circle (.05cm);
\draw (5.1,1.6) node {\small $a_{1,n}''$};
\fill [black] (4.507,.15) circle (.05cm);
\draw (4.9,.21) node {\small $d_{1,n}''$};

\fill [black] (6.5,0) circle (.05cm);
\draw (6.5,-.3) node {$t_{1,n}$};
\draw[arrows=->,line width=.75pt](6.5,0)--(6.5,.5);
\draw (6.5,.7) node {$D_n$};
\fill [black] (5.5,0) circle (.05cm);
\draw (5.5,-.3) node {$t_{2,n}$};
\draw[arrows=->,line width=.75pt](5.5,0)--(5.5,.5);
\draw (5.5,.7) node {$A_n$};

\draw (10,2) node {(c)};
\draw [domain=0:180] plot ({10+(1.5)*cos(\x)}, {(1.5)*sin(\x)});
\fill [black] (11.493,.15) circle (.05cm);
\draw (11.1,.2) node {\small $a_{1,n}'$};
\fill [black] (10.75,1.299) circle (.05cm);
\draw (11,1.6) node {\small $d_{1,n}'$};
\fill [black] (9.25,1.299) circle (.05cm);
\draw (9.1,1.6) node {\small $a_{1,n}''$};
\fill [black] (8.507,.15) circle (.05cm);
\draw (8.9,.21) node {\small $d_{1,n}''$};

\draw[arrows=->,line width=.75pt](10,.65)--(10.424,1.074);
\draw (10.65,1) node {\scriptsize $D_n$};
\draw[arrows=->,line width=.75pt](10,.65)--(9.576,1.074);
\draw (9.35,1) node {\scriptsize $A_n$};
\draw[arrows=->,line width=.75pt](10,.65)--(9.576,.226);
\draw (9.35,.5) node {\scriptsize $D_n'$};
\draw[arrows=->,line width=.75pt](10,.65)--(10.424,.226);
\draw (10.65,.5) node {\scriptsize $A_n'$};

\fill[white] (10,.65) circle (.25cm);
\draw (10,.65) node {\scriptsize $t_{1,n}$};

\end{tikzpicture}
\caption{The contours $A_n,D_n,A_n',D_n'$ of lemma \ref{lemConAscDesCase1}
in $B(t, n^{-\theta} q_n) \cap \mathbb{H}$ for each of the possibilities,
(a,b,c), of lemmas \ref{lemfn'Case1} and \ref{lemConAscDesCase1}.}
\label{figConAnglesCase1}
\end{figure}

Consider (ii). First recall (see equation
(\ref{eqdn})) that $d_{2,n}$ and $a_{2,n}$ denote the points
at which $D_n$ and $A_n$ respectively `exit' $B(t,\xi)$. Then, proceed as in
(i) to get,
\begin{align*}
R_n(d_{2,n}) < f_n(t) + O(n^{-1})
&\hspace{0.5cm} \text{and} \hspace{0.5cm}
I_n(d_{2,n}) = O(n^{-1}), \\
R_n(a_{2,n}) > f_n(t) + O(n^{-1})
&\hspace{0.5cm} \text{and} \hspace{0.5cm}
I_n(a_{2,n}) = O(n^{-1}).
\end{align*}
Thus, since $f_n(t) \to f_t(t)$,
\begin{align*}
R_n(d_{2,n}) < f_t(t) + o(1)
&\hspace{0.5cm} \text{and} \hspace{0.5cm}
I_n(d_{2,n}) = o(1), \\
R_n(a_{2,n}) > f_t(t) + o(1)
&\hspace{0.5cm} \text{and} \hspace{0.5cm}
I_n(a_{2,n}) = o(1).
\end{align*}
Next, recall that $\xi$ is any positive number for
which equations (\ref{eqxi}, \ref{eqxi1}, \ref{eqxi4}, \ref{eqxi5})
hold. Note that this can be fixed arbitrarily small. Also, for all
such $\xi>0$ chosen sufficiently small, part (2) of lemma
\ref{lemTay} implies (take $\xi_n = \xi q_n^{-1}$) that there
exists a positive integer, $n(\xi)$, such that,
\begin{equation*}
f_n(t + \xi e^{i \alpha})
= f_n(t) + \tfrac13 \xi^3 q_n^{-3} e^{3i \alpha} + O(\xi^4),
\end{equation*}
for all $n>n(\xi)$ and uniformly for $\alpha \subset (-\pi,\pi]$.
Next recall that $f_n(t) \to f_t(t)$ and $q_n \to q_t$,
where $q_t$ is some positive number. It thus follows that there exists
a choice of $n(\xi)$ such that,
\begin{equation*}
f_n(t + \xi e^{i \alpha})
= f_t(t) + \tfrac13 \xi^3 q_t^{-3} e^{3i \alpha} + O(\xi^4),
\end{equation*}
for all $n>n(\xi)$ and uniformly for $\alpha \subset (-\pi,\pi]$.
Therefore, since $f_t(t) \in \R$,
\begin{align*}
R_n(t + \xi e^{i \alpha})
&= f_t(t) + \tfrac13 \xi^3 q_t^{-3} \cos(3 \alpha) + O(\xi^4), \\
I_n(t + \xi e^{i \alpha})
&= \tfrac13 \xi^3 q_t^{-3} \sin(3 \alpha) + O(\xi^4),
\end{align*}
for all $n>n(\xi)$ and uniformly for $\alpha \subset (-\pi,\pi]$.
Then, for all $\xi>0$ sufficiently small, it follows from similar
arguments to those used in part (i) that there exists a choice of
$n(\xi)$ such that we have the following possibilities for $d_{2,n}$
and $a_{2,n}$ for all $n>n(\xi)$:
\begin{itemize}
\item
Either $d_{2,n} = d_{2,n}'$ or $d_{2,n} = d_{2,n}''$
where $\text{Arg}(d_{2,n}'-t) = \frac\pi3 + O(\xi)$
and $\text{Arg}(d_{2,n}''-t) = \pi + O(\xi)$. In
either case, $R_n(d_{2,n}) = f_t(t) - \frac13 \xi^3 q_t^{-3} + O(\xi^4)$.
\item
Either $a_{2,n} = a_{2,n}'$ or $a_{2,n} = a_{2,n}''$
where $\text{Arg}(a_{2,n}'-t) = O(\xi)$ and
$\text{Arg}(a_{2,n}''-t) = \frac{2\pi}3 + O(\xi)$.
In either case, $R_n(a_{2,n}) = f_t(t) + \frac13 \xi^3 q_t^{-3} + O(\xi^4)$.
\end{itemize}
(ii) then follows from similar arguments to those used in part (i), above.

Consider (iii). This follows trivially from part (1) of lemma \ref{lemDesAsc},
since $D_n$ is a contour of steepest descent for $f_n$. Consider (iv).
First note, proceeding similarly to the proof of part
(ii) above, we can show that we can choose $\xi$ and $n(\xi)$
such that,
\begin{equation*}
\text{Arg}(w-t) = \tfrac\pi3 + O(\xi)
\hspace{0.5cm} \text{and} \hspace{0.5cm}
R_n(w) = f_t(t) - \tfrac13 \xi^3 q_t^{-3} + O(\xi^4),
\end{equation*}
for all $n>n(\xi)$ and uniformly for $w$ on the shortest arc of
$\partial B(t, \xi)$ from $d_{2,n}$ to $d_{2,t}$.
Also recall (see proof of part (i)) that
$R_n(d_{1,n}) = f_n(t) - \frac13 n^{-3\theta} + O(n^{-\frac13-2\theta})$.
Therefore we can choose $\xi$ and
$n(\xi)$ such that $R_n(w) \le R_n(d_{1,n})$ for all $n>n(\xi)$ and $w$ on
this arc. This proves (iv).

Consider (v). Note, that section of $D_t$ from $d_{2,t}$ to $d_{3,t}$ is
independent of $n$ and is contained in $\C_\xi$, and so part (2) of lemma
\ref{lemfnftConv} implies that $R_n(w) = R_t(w) + o(1)$
uniformly for all $w$ on this section. Thus, for all $\xi>0$, there
exists an integer $n(\xi) > 0$ such that $R_n(w) = R_t(w) + O(\xi^4)$ for
all $n>n(\xi)$ and uniformly for $w$ on this section. Next recall
that $D_t$ is a contour of steepest descent for $f_t$. Part (1) of lemma
\ref{lemDesAsc} thus gives $R_t(w) \le R_t(d_{2,t})$ for all $w$ on this
section. Finally, proceeding similarly to part (ii) above, we can show
that for all $\xi>0$ sufficiently small, we can choose the above $n(\xi) > 0$
such that the following is satisfied for all $n>n(\xi)$:
$R_t(d_{2,t}) = f_t(t) - \frac13 \xi^3 q_t^{-3} + O(\xi^4)$.
Combined, these observations give,
\begin{equation}
\label{eqlemgnGnCase16}
R_n(w)
= R_t(w) + O(\xi^4)
\le R_t(d_{2,t}) + O(\xi^4)
= f_t(t) - \tfrac13 \xi^3 q_t^{-3} + O(\xi^4),
\end{equation}
for all $n>n(\xi)$ and uniformly for $w$ on that section of $D_t$
from $d_{2,t}$ to $d_{3,t}$.  Finally recall (see proof of part (i)) that
$R_n(d_{1,n}) = f_n(t) - \frac13 n^{-3\theta} + O(n^{-\frac13-2\theta})$.
Therefore we can choose $\xi$ and $n(\xi)$ such that
$R_n(w) \le R_n(d_{1,n})$ for all $n>n(\xi)$ and $w$ on this section.

Consider (vi). Note that we trivially have $|w-x| \le |d_{3,t}-x|$
for all $w$ on the vertical straight line from $d_{3,t}$ to
$\text{Re}(d_{3,t})$ and $x \in \R$. Equation (\ref{eqRn}) then gives,
\begin{equation*}
R_n(w)
\le \frac1n \sum_{x \in S_{1,n}} \log |d_{3,t}-x|
- \frac1n \sum_{x \in S_{2,n}} \log |w-x|
+ \frac1n \sum_{x \in S_{3,n}} \log |d_{3,t}-x|,
\end{equation*}
for all $w$ on the vertical line. Next recall that
$\underline{S_{2,n}} \ge \chi+\eta-1 + o(1)$ (see equations
(\ref{equnrnvnsn}, \ref{eqfn2})) and $\chi+\eta-1 \ge \overline{S_3}$
(see equation (\ref{eqS1S2S3In})).
Equation (\ref{eqd3ta3t2}) thus
implies that $\underline{S_{2,n}} > \text{Re}(d_{3,t}) + \frac12 c$ for
all $\xi>0$ sufficiently small and $n$ sufficiently large
chosen independently, where $c = c(t) > 0$ is some constant
independent of $\xi$ and $n$. Equation (\ref{eqd3ta3t1}) and the above
expression thus give,
\begin{equation*}
R_n(w)
\le \frac1n \sum_{x \in S_{1,n}} \log |d_{3,t}-x|
- \frac1n \sum_{x \in S_{2,n}} (\log |d_{3,t}-x| + O(\xi^4))
+ \frac1n \sum_{x \in S_{3,n}} \log |d_{3,t}-x|,
\end{equation*}
uniformly for $w$ on the vertical line. Then, equation (\ref{eqRn}) gives
$R_n(w) \le R_n(d_{3,t}) + O(\xi^4)$ for all such $n,w$. Next
note, substituting $d_{3,t}$ in equation (\ref{eqlemgnGnCase16})
gives $R_n(d_{3,t}) \le f_t(t) - \frac13 \xi^3 q_t^{-3} + O(\xi^4)$
for all $n>n(\xi)$. Combined, these give
$R_n(w) \le f_t(t) - \frac13 \xi^3 q_t^{-3} + O(\xi^4)$
for all $n>n(\xi)$ and uniformly for $w$ on the vertical
line. Finally recall (see proof of part (i)) that
$R_n(d_{1,n}) = f_n(t) - \frac13 n^{-3\theta} + O(n^{-\frac13-2\theta})$.
Therefore we can choose $\xi$ and
$n(\xi)$ such that $R_n(w) \le R_n(d_{1,n})$ for all $n>n(\xi)$ and
$w$ on the vertical line. This proves (vi).

Consider (vii). Fixing a constant $c>0$, and recalling that
$\{q_n\}_{n\ge1}$ is a convergent sequence with a positive limit,
note that it is sufficient to show the following:
\begin{itemize}
\item
We can fix the $\xi$ sufficiently small, and choose the $n(\xi)$, such
that the direction of steepest descent for $f_n$ at $t + r q_n e^{i \alpha}$ equals
$\frac\pi3 + O(\xi)$ for all $n>n(\xi)$ and uniformly for
$r \in [n^{-\theta}, \xi q_n^{-1}]$ and $\alpha \in (\frac{\pi}3 - c \xi, \frac{\pi}3 + c \xi)$.
\end{itemize}
To see this first note, for all such $r,\alpha$, part (3) of lemma
\ref{lemDesAsc} implies that the unique direction of steepest descent
for $f_n$ at $t + r q_n e^{i \alpha}$ is
$\pi - \text{Arg}(f_n'(t + r q_n e^{i \alpha}))$. Then, we can proceed as
in part (2) of lemma \ref{lemTay} to show that we can choose the
above $n(\xi)>0$ such that,
\begin{align*}
f_n'(t + r q_n e^{i \alpha})
&= r^2 q_n^{-1} e^{2 i \alpha}
+ O(n^{-\frac23} + n^{-\frac13} r + n^{-\frac13} r^2 + r^3) \\
&= r^2 q_n^{-1} e^{2 i \alpha} \left(1 + O(n^{-\frac23} r^{-2}
+ n^{-\frac13} r^{-1} + n^{-\frac13} + r) \right),
\end{align*}
for all $n>n(\xi)$, and uniformly for all such $r,\alpha$. Also note,
since $\theta < \frac13$ and $r \in [n^{-\theta}, \xi q_n^{-1}]$, that
$n^{-\frac23} r^{-2} + n^{-\frac13} r^{-1} + n^{-\frac13} = o(1)$,
and $r = O(\xi)$. Therefore
we can choose the above $\xi$ sufficiently small, and we can choose the $n(\xi)>0$,
such that $f_n'(t + r q_n e^{i \alpha}) = r^2 q_n^{-1} e^{2 i \alpha} ( 1 + O(\xi) )$
for all $n>n(\xi)$ and uniformly for all such $r,\alpha$. Finally recall
that  $\alpha \in (\frac{\pi}3 - c \xi, \frac{\pi}3 + c \xi)$ for some
$c>0$. Thus we can choose the above $\xi$ sufficiently small, and we can
choose the $n(\xi)>0$, such that
\begin{equation*}
\text{Arg}(f_n'(t + r q_n e^{i \alpha}))
= 2 \alpha + O(\xi)
= \tfrac{2\pi}3 + O(\xi),
\end{equation*}
for all $n>n(\xi)$, and uniformly for all such $r,\alpha$.
Thus the direction of steepest descent
for $f_n$ at $t + r q_n e^{i \alpha}$ equals
$\pi - \text{Arg}(f_n'(t + r q_n e^{i \alpha})) = \frac\pi3 + O(\xi)$
for all $n>n(\xi)$ and uniformly for all such $r,\alpha$, as required.
\end{proof}

\begin{rem}
\label{remgnPn}
Note, it may happen that $\text{Re} (d_{3,t}) \in P_n$ (see equation
(\ref{eqPnHn})). If this happens, for simplicity, we deform
$\g_{1,n}^+$ (see definition \ref{defgnGnCase1}) in neighbourhoods of
$\text{Re} (d_{3,t})$ to avoid this. Note, since such deformations are
arbitrarily small, they do not significantly affect the proof of the
previous lemma.
\end{rem}

\subsection{Lemma \ref{lemDesAsc1-12} for case (2) of lemma \ref{lemCases}}
\label{secContCase2}

Assume the conditions of lemma \ref{lemDesAsc1-12}. Additionally
assume that case (2) of lemma \ref{lemCases} is satisfied. Fix
$\xi>0$ sufficiently small such that equations (\ref{eqxi},
\ref{eqxi1}, \ref{eqxi4}, \ref{eqxi5}) are satisfied.
Note, many of the arguments in this section are similar to those
used in section \ref{secContCase1} for case (1). Therefore, we
shall not go into as much detail here, but we shall
highlight the differences where necessary.

We begin by consider the roots of the functions $f_t'$, $f_n'$
and $\tilde{f}_n'$ in this case. We consider $f_n'$ and state
that $\tilde{f}_n'$ can be treated similarly. Recall the definitions
given in equations (\ref{eqS1S2S3}, \ref{eqf'domain2}, \ref{eqIntervalt},
\ref{eqfn2}, \ref{eqfn'domain}, \ref{eqIntervaln}), and
the properties discussed in equations (\ref{eqS1S2S3In}, \ref{eqS1nS2nS3nIn}).
These, and case (2) of lemma \ref{lemCases}, give the following:
\begin{itemize}
\item
$L_t$ is an open interval with $t \in L_t \in \{K_1^{(1)},K_2^{(1)},\ldots\}$,
$\{\underline{L_t}, \overline{L_t}\} \subset \supp (\mu)$,
and $\overline{S_1} > \overline{L_t} > t > \underline{L_t} > \underline{S_1} \ge \chi$.
Moreover, $f_t'(t) = f_t''(t) = 0$ and $f_t'''(t) > 0$.
\item
$L_n$ is an open interval with
$t \in L_n \in \{K_{1,n}^{(1)},K_{2,n}^{(1)},\ldots\}$,
$\underline{L_n}$ and $\overline{L_n}$ are two consecutive elements of $S_{1,n}$,
and $\underline{L_n} \to \underline{L_t}$ and $\overline{L_n} \to \overline{L_t}$.
\end{itemize}
Moreover:
\begin{lem}
\label{lemfn'Case2}
Assume the above conditions. Then, one of the following is satisfied:
\begin{enumerate}
\item[(A)]
$f_t'$ has $1$ root in $(t,\overline{L_t})$, denoted by $s_t$.
\item[(B)]
$f_t'$ has $0$ roots in $(t,\overline{L_t})$.
\end{enumerate}
Moreover, whenever possibility (A) is satisfied:
\begin{enumerate}
\item[(A1)]
$f_t'(s) > 0$ for all $s \in (\underline{L_t},t)$,
$f_t'(t) = f_t''(t) = 0$ and $f_t'''(t) > 0$,
$f_t'(s) > 0$ for all $s \in (t,s_t)$,
$f_t'(s_t) = 0$ and $f_t''(s_t) < 0$,
and $f_t'(s) < 0$ for all $s \in (s_t,\overline{L_t})$.
\item[(A2)]
$f_t'$ has $0$ roots in each of $\{\C \setminus \R, J_1, J_2, J_3, J_4\}$.
\item[(A3)]
$f_t'$ has at most $1$ root in each of
$\cup_{i=1}^3 \{K_1^{(i)},K_2^{(i)},\ldots\} \setminus \{L_t\}$.
\end{enumerate}
Also, whenever possibility (B) is satisfied:
\begin{enumerate}
\item[(B1)]
$f_t'(s) > 0$ for all $s \in (\underline{L_t},t)$,
$f_t'(t) = f_t''(t) = 0$ and $f_t'''(t) > 0$,
$f_t'(s) > 0$ for all $s \in (t,\overline{L_t})$.
\item[(B2)]
$f_t'$ has $0$ roots in each of $\{\C \setminus \R, J_1, J_2, J_3, J_4\}$.
\item[(B3)]
$f_t'$ has at most $1$ root in each of
$\cup_{i=1}^3 \{K_1^{(i)},K_2^{(i)},\ldots\} \setminus \{L_t\}$.
\end{enumerate}

Next, fixing $\xi > 0$ as above, then $(t - 4\xi, t + 4\xi) \subset L_t$,
and $(t - 2\xi, t + 2\xi) \subset L_n$. Also, $f_n'$ has $2$
roots in $B(t,\xi)$ (denoted by $\{t_{1,n},t_{2,n}\}$), $0$ roots in
$(\underline{L_n},t-\xi]$, and $1$ root in $[t+\xi, \overline{L_n})$ (denoted
$s_n$). Moreover, whenever possibility (a) of
lemma \ref{lemRootsNonAsy1} is satisfied:
\begin{enumerate}
\item[(a1)]
$t_{1,n} \in (t-\xi,t+\xi)$ and $t_{1,n} = t_{2,n}$.
Moreover $f_n'(s) > 0$ for all $s \in (\underline{L_n},t_{1,n})$,
$f_n'(t_{1,n}) = f_n''(t_{1,n}) = 0$ and $f_n'''(t_{1,n}) > 0$,
$f_n'(s) > 0$ for all $s \in (t_{1,n},s_n)$,
$f_n'(s_n) = 0$ and $f_n''(s_n) < 0$,
and $f_n'(s) < 0$ for all $s \in (s_n,\overline{L_n})$.
\item[(a2)]
$f_n'$ has $0$ roots in each of
$\{\C \setminus \R, J_{1,n}, J_{2,n},J_{3,n},J_{4,n}\}$.
\item[(a3)]
$f_n'$ has $1$ root in each of
$\cup_{i=1}^3 \{K_{1,n}^{(i)},K_{2,n}^{(i)},\ldots\} \setminus \{L_n\}$.
\end{enumerate}
Moreover, whenever possibility (b) is satisfied:
\begin{enumerate}
\item[(b1)]
$\{t_{1,n},t_{2,n}\} \subset (t-\xi,t+\xi)$ and $t_{1,n} > t_{2,n}$.
Moreover $f_n'(s) > 0$ for all $s \in (\underline{L_n},t_{2,n})$,
$f_n'(t_{2,n}) = 0$ and $f_n''(t_{2,n}) < 0$,
$f_n'(s) < 0$ for all $s \in (t_{2,n},t_{1,n})$,
$f_n'(t_{1,n}) = 0$ and $f_n''(t_{1,n}) > 0$, 
$f_n'(s) > 0$ for all $s \in (t_{1,n},s_n)$,
$f_n'(s_n) = 0$ and $f_n''(s_n) < 0$,
and $f_n'(s) < 0$ for all $s \in (s_n,\overline{L_n})$.
\item[(b2)]
$f_n'$ has $0$ roots in each of
$\{\C \setminus \R, J_{1,n}, J_{2,n},J_{3,n},J_{4,n}\}$.
\item[(b3)]
$f_n'$ has $1$ root in each of
$\cup_{i=1}^3 \{K_{1,n}^{(i)},K_{2,n}^{(i)},\ldots\} \setminus \{L_n\}$.
\end{enumerate}
Finally, whenever possibility (c) is satisfied:
\begin{enumerate}
\item[(c1)]
$t_{1,n} \in B(t,\xi) \cap \mathbb{H}$ and $t_{2,n}$ is the complex conjugate of $t_{1,n}$.
Moreover, $f_n'(s) > 0$ for all $s \in (\underline{L_n},s_n)$,
$f_n'(s_n) = 0$ and $f_n''(s_n) < 0$,
$f_n'(s) < 0$ for all $s \in (s_n, \overline{L_n})$.
\item[(c2)]
$f_n'$ has $0$ roots in each of
$\{\C \setminus (\R \cup \{t_{1,n},t_{2,n}\}), J_{1,n}, J_{2,n},J_{3,n},J_{4,n}\}$
\item[(c3)]
$f_n'$ has $1$ root in each of
$\cup_{i=1}^3 \{K_{1,n}^{(i)},K_{2,n}^{(i)},\ldots\} \setminus \{L_n\}$.
\end{enumerate}
\end{lem}

\begin{proof}
Consider $f_t'$. First recall, since case (2) of lemma \ref{lemCases}
is satisfied, that $t \in L_t \in \{K_1^{(1)},K_2^{(1)},\ldots\}$,
and $f_t'(t) = f_t''(t) = 0$ and $f_t'''(t) > 0$. Part (1)
of lemma \ref{lemf'} then implies that $f_t'$ has at most $1$ root in
$L_t \setminus \{t\}$. Thus we have three possibilities:
\begin{itemize}
\item
$f_t'$ has $0$ roots in $(\underline{L_t},t)$,
and $1$ root in $(t,\overline{L_t})$ (denoted $s_t$).
\item
$f_t'$ has $0$ roots in $(\underline{L_t},t)$,
and $0$ roots in $(t,\overline{L_t})$.
\item
$f_t'$ has $1$ root in $(\underline{L_t},t)$ (denoted $s_t$),
and $0$ roots in $(t,\overline{L_t})$.
\end{itemize}
Whenever the first possibility is satisfied, we will show that possibility
(A) and parts (A1,A2,A3) are also satisfied. Moreover, whenever the second
possibility is satisfied, we can similarly show that possibility (B) and parts
(B1,B2,B3) are also satisfied. Finally, we will show that the third
possibility is never satisfied.

Suppose that the first possibility is satisfied. Possibility (A) is then
trivially satisfied. Moreover, $f_t'$ has $0$ roots in
$(\underline{L_t}, t)$, $f_t'(t) = f_t''(t) = 0$ and $f_t'''(t) > 0$,
$f_t'$ has $0$ roots in $(t,s_t)$, $f_t'(s_t) = 0$ and $f_t''(s_t) \neq 0$,
and $f_t'$ has $0$ roots in $(s_t,\overline{L_t})$. Thus, since
$(f_t') |_{L_t}$ is real-valued and continuous (see equation (\ref{eqft'2})),
the above observations are only possible if part (A1) is satisfied. Moreover,
since $L_t \in \{K_1^{(1)},K_2^{(1)},\ldots\}$, parts (2,3) of
lemma \ref{lemf'} imply parts (A2,A3). Thus, whenever the first possibility
is satisfied, possibility (A) and parts (A1,A2,A3) are also satisfied.

Suppose that the third possibility is satisfied. First, recall that
$s_t \in (\underline{L_t},t)$
is a root of $f_t'$ of multiplicity $1$. Thus, we can fix an $\e>0$ sufficiently
small such that $\underline{L_t} + \e < s_t < t - \e$, and $s_t$ is the unique
root of $f_t'$ in $B(s_t,\e)$. Next, we can proceed similarly to part (iii) in
the proof of lemma \ref{lemRootsNonAsy1} to show that $f_n'$ has $1$ root in
$B(s_t,\e)$. Also, since roots of $f_n'$ occur
in complex conjugate pairs (see equation (\ref{eqfn'})), this root must be
contained in $(s_t-\e,s_t+\e)$. Finally, note that equation (\ref{eqIntervaln})
implies that we can choose $\e$ and $\xi$ such that
$(s_t-\e,s_t+\e) \subset (\underline{L_n}, t-\xi]$.
Therefore $f_n'$ has a root in $(\underline{L_n}, t-\xi]$. However, as we will
shortly show below, such a root does not exist. Thus we have a contradiction,
and so the third possibility is never satisfied.

Consider $f_n'$. First recall that $(t - 4\xi, t + 4\xi) \subset L_t$,
and $(t - 2\xi, t + 2\xi) \subset L_n$
comes from equations (\ref{eqxi}, \ref{eqIntervaln}). Also, part (1) of
lemma \ref{lemRootsNonAsy1} implies that $f_n'$ has $2$ roots
in $B(t,\xi)$.

\begin{figure}[t]
\centering
\begin{tikzpicture}

\draw [dotted] (-.5,8) --++(14,0);
\draw (-.5,8) node {$\R$};
\draw (-.5,8.5) node {$\mathbb{H}$};
\draw (-1.25,8) node {(A)};

\draw (0,8.1) --++(0,-.2);
\draw (0,8) --++(1.5,0);
\draw (1.5,8.1) --++(0,-.2);
\draw (3,8.1) --++(0,-.2);
\draw (3,8) --++(1.5,0);
\draw (4.5,8.1) --++(0,-.2);
\draw (6,8.1) --++(0,-.2);
\draw (6,8) --++(1.5,0);
\draw (7.5,8.1) --++(0,-.2);
\draw (11.5,8.1) --++(0,-.2);
\draw (11.5,8) --++(1.5,0);
\draw (13,8.1) --++(0,-.2);

\draw (0,7.6) node {\scriptsize $\underline{S_3}$};
\draw (.75,7.6) node {\scriptsize $<$};
\draw (1.5,7.6) node {\scriptsize $\overline{S_3}$};
\draw (2.25,7.6) node {\scriptsize $\le$};
\draw (3,7.6) node {\scriptsize $\underline{S_2}$};
\draw (3.75,7.6) node {\scriptsize $<$};
\draw (4.5,7.6) node {\scriptsize $\overline{S_2}$};
\draw (5.25,7.6) node {\scriptsize $\le$};
\draw (6,7.6) node {\scriptsize $\underline{S_1}$};
\draw (6.75,7.6) node {\scriptsize $<$};
\draw (7.5,7.6) node {\scriptsize $\underline{L_t}$};
\draw (11.5,7.6) node {\scriptsize $\overline{L_t}$};
\draw (12.25,7.6) node {\scriptsize $<$};
\draw (13,7.6) node {\scriptsize $\overline{S_1}$};

\draw (9.5,8) node {$\ast$};
\draw (9.5,7.75) node {\scriptsize $t$};
\draw (10.5,8) node {$\times$};
\draw (10.5,7.75) node {\scriptsize $s_t$};

\draw [dotted] (-.5,6) --++(14,0);
\draw (-.5,6) node {$\R$};
\draw (-.5,6.5) node {$\mathbb{H}$};
\draw (-1.25,6) node {(B)};

\draw (0,6.1) --++(0,-.2);
\draw (0,6) --++(1.5,0);
\draw (1.5,6.1) --++(0,-.2);
\draw (3,6.1) --++(0,-.2);
\draw (3,6) --++(1.5,0);
\draw (4.5,6.1) --++(0,-.2);
\draw (6,6.1) --++(0,-.2);
\draw (6,6) --++(1.5,0);
\draw (7.5,6.1) --++(0,-.2);
\draw (11.5,6.1) --++(0,-.2);
\draw (11.5,6) --++(1.5,0);
\draw (13,6.1) --++(0,-.2);

\draw (0,5.6) node {\scriptsize $\underline{S_3}$};
\draw (.75,5.6) node {\scriptsize $<$};
\draw (1.5,5.6) node {\scriptsize $\overline{S_3}$};
\draw (2.25,5.6) node {\scriptsize $\le$};
\draw (3,5.6) node {\scriptsize $\underline{S_2}$};
\draw (3.75,5.6) node {\scriptsize $<$};
\draw (4.5,5.6) node {\scriptsize $\overline{S_2}$};
\draw (5.25,5.6) node {\scriptsize $\le$};
\draw (6,5.6) node {\scriptsize $\underline{S_1}$};
\draw (6.75,5.6) node {\scriptsize $<$};
\draw (7.5,5.6) node {\scriptsize $\underline{L_t}$};
\draw (11.5,5.6) node {\scriptsize $\overline{L_t}$};
\draw (12.25,5.6) node {\scriptsize $<$};
\draw (13,5.6) node {\scriptsize $\overline{S_1}$};

\draw (9.5,6) node {$\ast$};
\draw (9.5,5.75) node {\scriptsize $t$};

\draw [dotted] (-.5,4) --++(14,0);
\draw (-.5,4) node {$\R$};
\draw (-.5,4.5) node {$\mathbb{H}$};
\draw (-1.25,4) node {(a)};

\draw (0,4) node {$\bullet$};
\draw (.375,4) node {$\times$};
\draw (.75,4) node {$\bullet$};
\draw (1.125,4) node {$\times$};
\draw (1.5,4) node {$\bullet$};
\draw (3,4) node {$\bullet$};
\draw (3.375,4) node {$\times$};
\draw (3.75,4) node {$\bullet$};
\draw (4.125,4) node {$\times$};
\draw (4.5,4) node {$\bullet$};
\draw (6,4) node {$\bullet$};
\draw (6.375,4) node {$\times$};
\draw (6.75,4) node {$\bullet$};
\draw (7.125,4) node {$\times$};
\draw (7.5,4) node {$\bullet$};
\draw (11.5,4) node {$\bullet$};
\draw (11.875,4) node {$\times$};
\draw (12.25,4) node {$\bullet$};
\draw (12.625,4) node {$\times$};
\draw (13,4) node {$\bullet$};

\draw (0,3.6) node {\scriptsize $\underline{S_{3,n}}$};
\draw (.75,3.6) node {\scriptsize $<$};
\draw (1.5,3.6) node {\scriptsize $\overline{S_{3,n}}$};
\draw (2.25,3.6) node {\scriptsize $<$};
\draw (3,3.6) node {\scriptsize $\underline{S_{2,n}}$};
\draw (3.75,3.6) node {\scriptsize $<$};
\draw (4.5,3.6) node {\scriptsize $\overline{S_{2,n}}$};
\draw (5.25,3.6) node {\scriptsize $<$};
\draw (6,3.6) node {\scriptsize $\underline{S_{1,n}}$};
\draw (6.75,3.6) node {\scriptsize $<$};
\draw (7.5,3.6) node {\scriptsize $\underline{L_n}$};
\draw (11.5,3.6) node {\scriptsize $\overline{L_n}$};
\draw (12.25,3.6) node {\scriptsize $<$};
\draw (13,3.6) node {\scriptsize $\overline{S_{1,n}}$};

\draw [dotted] (9.5,4) circle (.75cm);
\draw (8.6,4.75) node {\scriptsize $B(t,\xi)$};
\draw (9.5,4) node {$\ast$};
\draw (9.5,3.75) node {\scriptsize $t_{1,n}$};
\draw (10.5,4) node {\scriptsize $\times$};
\draw (10.5,3.75) node {\scriptsize $s_n$};

\draw [dotted] (-.5,2) --++(14,0);
\draw (-.5,2) node {$\R$};
\draw (-.5,2.5) node {$\mathbb{H}$};
\draw (-1.25,2) node {(b)};

\draw (0,2) node {$\bullet$};
\draw (.375,2) node {$\times$};
\draw (.75,2) node {$\bullet$};
\draw (1.125,2) node {$\times$};
\draw (1.5,2) node {$\bullet$};
\draw (3,2) node {$\bullet$};
\draw (3.375,2) node {$\times$};
\draw (3.75,2) node {$\bullet$};
\draw (4.125,2) node {$\times$};
\draw (4.5,2) node {$\bullet$};
\draw (6,2) node {$\bullet$};
\draw (6.375,2) node {$\times$};
\draw (6.75,2) node {$\bullet$};
\draw (7.125,2) node {$\times$};
\draw (7.5,2) node {$\bullet$};
\draw (11.5,2) node {$\bullet$};
\draw (11.875,2) node {$\times$};
\draw (12.25,2) node {$\bullet$};
\draw (12.625,2) node {$\times$};
\draw (13,2) node {$\bullet$};

\draw (0,1.6) node {\scriptsize $\underline{S_{3,n}}$};
\draw (.75,1.6) node {\scriptsize $<$};
\draw (1.5,1.6) node {\scriptsize $\overline{S_{3,n}}$};
\draw (2.25,1.6) node {\scriptsize $<$};
\draw (3,1.6) node {\scriptsize $\underline{S_{2,n}}$};
\draw (3.75,1.6) node {\scriptsize $<$};
\draw (4.5,1.6) node {\scriptsize $\overline{S_{2,n}}$};
\draw (5.25,1.6) node {\scriptsize $<$};
\draw (6,1.6) node {\scriptsize $\underline{S_{1,n}}$};
\draw (6.75,1.6) node {\scriptsize $<$};
\draw (7.5,1.6) node {\scriptsize $\underline{L_n}$};
\draw (11.5,1.6) node {\scriptsize $\overline{L_n}$};
\draw (12.25,1.6) node {\scriptsize $<$};
\draw (13,1.6) node {\scriptsize $\overline{S_{1,n}}$};

\draw [dotted] (9.5,2) circle (.75cm);
\draw (8.6,2.75) node {\scriptsize $B(t,\xi)$};
\draw (9.75,2) node {\scriptsize $\times$};
\draw (9.75,2.25) node {\scriptsize $t_{1,n}$};
\draw (9.25,2) node {\scriptsize $\times$};
\draw (9.25,1.75) node {\scriptsize $t_{2,n}$};
\draw (10.5,2) node {\scriptsize $\times$};
\draw (10.5,1.75) node {\scriptsize $s_n$};

\draw [dotted] (-.5,0) --++(14,0);
\draw (-.5,0) node {$\R$};
\draw (-.5,.5) node {$\mathbb{H}$};
\draw (-1.25,0) node {(c)};

\draw (0,0) node {$\bullet$};
\draw (.375,0) node {$\times$};
\draw (.75,0) node {$\bullet$};
\draw (1.125,0) node {$\times$};
\draw (1.5,0) node {$\bullet$};
\draw (3,0) node {$\bullet$};
\draw (3.375,0) node {$\times$};
\draw (3.75,0) node {$\bullet$};
\draw (4.125,0) node {$\times$};
\draw (4.5,0) node {$\bullet$};
\draw (6,0) node {$\bullet$};
\draw (6.375,0) node {$\times$};
\draw (6.75,0) node {$\bullet$};
\draw (7.125,0) node {$\times$};
\draw (7.5,0) node {$\bullet$};
\draw (11.5,0) node {$\bullet$};
\draw (11.875,0) node {$\times$};
\draw (12.25,0) node {$\bullet$};
\draw (12.625,0) node {$\times$};
\draw (13,0) node {$\bullet$};

\draw (0,-.4) node {\scriptsize $\underline{S_{3,n}}$};
\draw (.75,-.4) node {\scriptsize $<$};
\draw (1.5,-.4) node {\scriptsize $\overline{S_{3,n}}$};
\draw (2.25,-.4) node {\scriptsize $<$};
\draw (3,-.4) node {\scriptsize $\underline{S_{2,n}}$};
\draw (3.75,-.4) node {\scriptsize $<$};
\draw (4.5,-.4) node {\scriptsize $\overline{S_{2,n}}$};
\draw (5.25,-.4) node {\scriptsize $<$};
\draw (6,-.4) node {\scriptsize $\underline{S_{1,n}}$};
\draw (6.75,-.4) node {\scriptsize $<$};
\draw (7.5,-.4) node {\scriptsize $\underline{L_n}$};
\draw (11.5,-.4) node {\scriptsize $\overline{L_n}$};
\draw (12.25,-.4) node {\scriptsize $<$};
\draw (13,-.4) node {\scriptsize $\overline{S_{1,n}}$};

\draw [dotted] (9.5,0) circle (.75cm);
\draw (8.6,.75) node {\scriptsize $B(t,\xi)$};
\draw (9.5,.3) node {\scriptsize $\times \; t_{1,n}$};
\draw (9.5,-.3) node {\scriptsize $\times \; t_{2,n}$};
\draw (10.5,0) node {\scriptsize $\times$};
\draw (10.5,-.25) node {\scriptsize $s_n$};

\end{tikzpicture}
\caption{(A,B): $T_t$, the set of roots of $f_t'$, as described by lemma
\ref{lemfn'Case1}, for possibilities (A,B). Note, for each $i \in \{1,2,3\}$,
the subintervals of $[\underline{S_i}, \overline{S_i}] \setminus S_i$ contain
at most $1$ root (except $L_t$). (a,b,c): $T_n$, the set
of roots of $f_t'$, as described by lemma \ref{lemfn'Case1} for possibilities (a,b,c).
Roots of multiplicity $1$ and $2$ are represented by $\times$ and $\ast$
respectively, and elements of $S_n = S_{1,n} \cup S_{2,n} \cup S_{3,n}$
are represented by $\bullet$.}
\label{figSnC_xi2}
\end{figure}
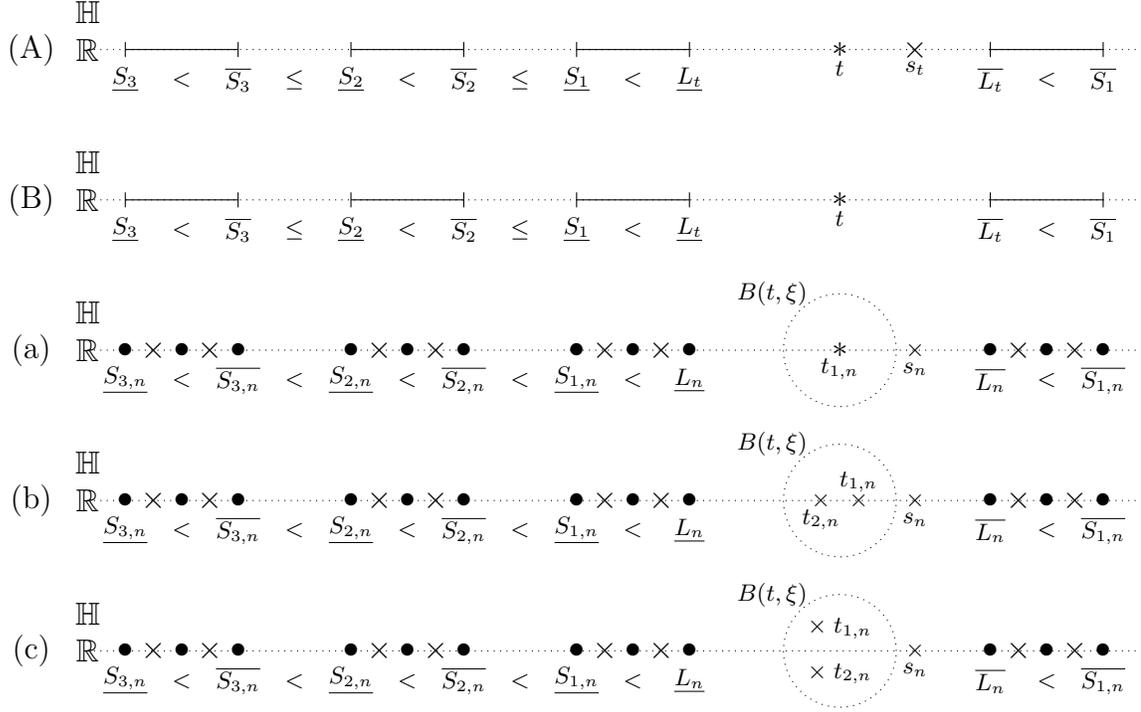

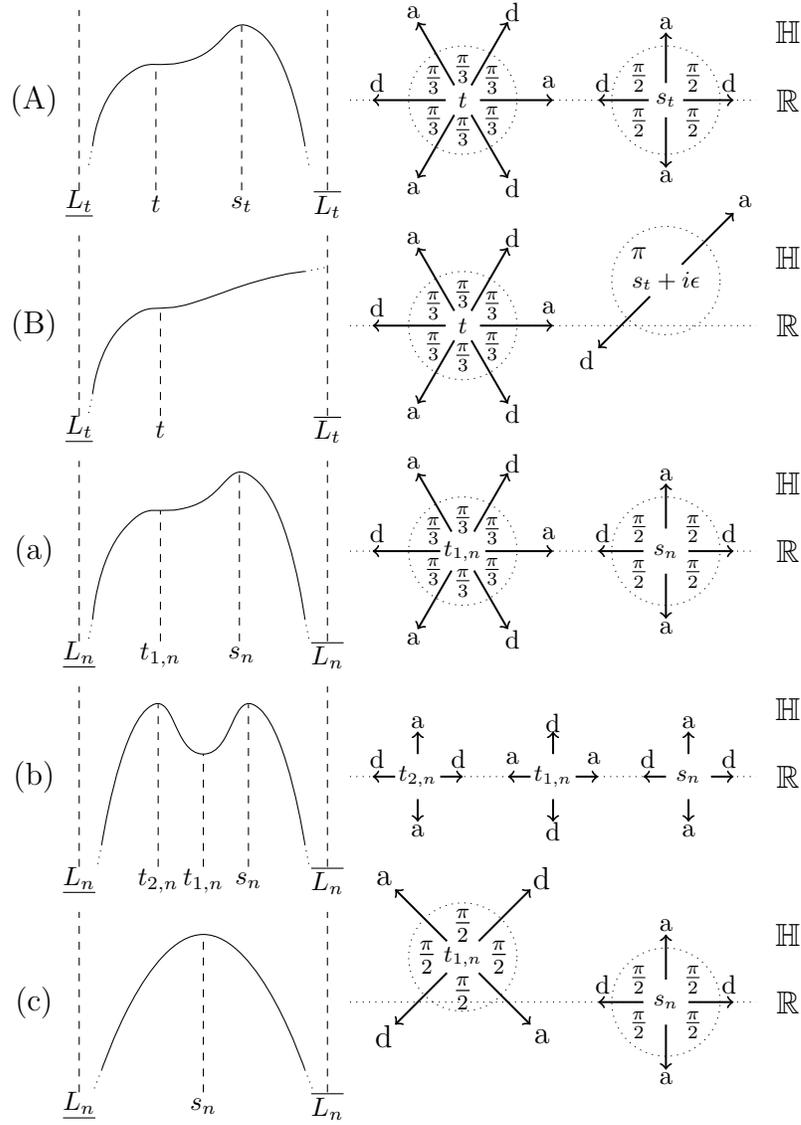
\begin{figure}[t]
\centering
\begin{tikzpicture}[scale=0.6]

\draw [dashed] (-.5,18) --++(0,4);
\draw (-1.5,20) node {(A)};
\draw (-.5,17.7) node {\footnotesize $\underline{L_t}$};
\draw [dashed] (5,18) --++ (0,4);
\draw (5,17.7) node {\footnotesize $\overline{L_t}$};

\draw plot [smooth, tension=1] coordinates
{ (-.2,19) (.5,20.5) (2,20.9) (3.5,21.5) (4.5,19)};
\draw [dotted] (-.2,19) --++(-.1,-.5);
\draw [dotted] (4.5,19) --++(.1,-.5);

\draw [dashed] (1.2,18) --++(0,2.8);
\draw (1.2,17.7) node {\footnotesize $t$};
\draw [dashed] (3.1,18) --++(0,3.65);
\draw (3.1,17.7) node {\footnotesize $s_t$};

\draw [dashed] (-.5,13) --++(0,4);
\draw (-1.5,15) node {(B)};
\draw (-.5,12.7) node {\footnotesize $\underline{L_t}$};
\draw [dashed] (5,13) --++ (0,4);
\draw (5,12.7) node {\footnotesize $\overline{L_t}$};

\draw plot [smooth, tension=1] coordinates
{ (-.2,13.5) (.5,15.05) (2,15.5) (3.5,16) (4.5,16.2)};
\draw [dotted] (-.2,13.5) --++(-.1,-.5);
\draw [dotted] (4.5,16.2) --++(.5,.1);

\draw [dashed] (1.3,13) --++(0,2.3);
\draw (1.3,12.7) node {\footnotesize $t$};

\draw [dashed] (-.5,8) --++(0,4);
\draw (-1.5,10) node {(a)};
\draw (-.5,7.7) node {\footnotesize $\underline{L_n}$};
\draw [dashed] (5,8) --++ (0,4);
\draw (5,7.7) node {\footnotesize $\overline{L_n}$};

\draw plot [smooth, tension=1] coordinates
{ (-.2,8.5) (.5,10.5) (2,11) (3.5,11.5) (4.5,8.5)};
\draw [dotted] (-.2,8.5) --++(-.1,-.5);
\draw [dotted] (4.5,8.5) --++(.1,-.5);

\draw [dashed] (1.3,8) --++(0,2.9);
\draw (1.3,7.7) node {\footnotesize $t_{1,n}$};
\draw [dashed] (3.05,8) --++(0,3.75);
\draw (3.1,7.7) node {\footnotesize $s_n$};

\draw [dashed] (-.5,3) --++(0,4);
\draw (-1.5,5) node {(b)};
\draw (-.5,2.7) node {\footnotesize $\underline{L_n}$};
\draw [dashed] (5,3) --++ (0,4);
\draw (5,2.7) node {\footnotesize $\overline{L_n}$};

\draw plot [smooth, tension=1] coordinates
{ (0,3.5) (1,6.5) (2.25,5.5) (3.5,6.5) (4.5,3.5)};
\draw [dotted] (0,3.5) --++(-.1,-.5);
\draw [dotted] (4.5,3.5) --++(.1,-.5);

\draw [dashed] (1.25,3) --++(0,3.6);
\draw (1.25,2.7) node {\footnotesize $t_{2,n}$};
\draw [dashed] (2.25,3) --++(0,2.5);
\draw (2.25,2.7) node {\footnotesize $t_{1,n}$};
\draw [dashed] (3.25,3) --++(0,3.6);
\draw (3.25,2.7) node {\footnotesize $s_n$};

\draw [dashed] (-.5,-2) --++(0,4);
\draw (-1.5,0) node {(c)};
\draw (-.5,-2.3) node {\footnotesize $\underline{L_n}$};
\draw [dashed] (5,-2) --++ (0,4);
\draw (5,-2.3) node {\footnotesize $\overline{L_n}$};

\draw plot [smooth, tension=1] coordinates
{ (0,-1.5) (2.25,1.5) (4.5,-1.5)};
\draw [dotted] (0,-1.5) --++(-.2,-.5);
\draw [dotted] (4.5,-1.5) --++(.2,-.5);

\draw [dashed] (2.25,-2) --++(0,3.5);
\draw (2.25,-2.3) node {\footnotesize $s_n$};

\draw (15.2,21.5) node {$\mathbb{H}$};
\draw [dotted] (5.5,20) --++(9,0);
\draw (15.2,20) node {$\R$};

\draw[arrows=->,line width=.75pt](8,20)--(10,20);
\draw (9.9,20.4) node {\footnotesize a};
\draw[arrows=->,line width=.75pt](8,20)--(9,{20+sqrt(3)});
\draw (9.1,21.95) node {\footnotesize d};
\draw[arrows=->,line width=.75pt](8,20)--(7,{20+sqrt(3)});
\draw (6.9,21.95) node {\footnotesize a};
\draw[arrows=->,line width=.75pt](8,20)--(6,20);
\draw (6.1,20.4) node {\footnotesize d};
\draw[arrows=->,line width=.75pt](8,20)--(7,{20-sqrt(3)});
\draw (6.9,18.05) node {\footnotesize a};
\draw[arrows=->,line width=.75pt](8,20)--(9,{20-sqrt(3)});
\draw (9.1,18.05) node {\footnotesize d};

\fill[white] (8,20) circle (.38cm);
\draw (8,20) node {\scriptsize $t$};

\draw [dotted] (8,20) circle (1.2cm);
\draw (8.65,20.45) node {\footnotesize $\frac\pi3$};
\draw (8,20.7) node {\footnotesize $\frac\pi3$};
\draw (7.35,20.45) node {\footnotesize $\frac\pi3$};
\draw (8,19.3) node {\footnotesize $\frac\pi3$};
\draw (7.35,19.55) node {\footnotesize $\frac\pi3$};
\draw (8.65,19.55) node {\footnotesize $\frac\pi3$};

\draw[arrows=->,line width=.75pt](12.5,20)--(14,20);
\draw (13.9,20.4) node {\footnotesize d};
\draw[arrows=->,line width=.75pt](12.5,20)--(12.5,21.5);
\draw (12.5,21.7) node {\footnotesize a};
\draw[arrows=->,line width=.75pt](12.5,20)--(11,20);
\draw (11.1,20.4) node {\footnotesize d};
\draw[arrows=->,line width=.75pt](12.5,20)--(12.5,18.5);
\draw (12.5,18.3) node {\footnotesize a};

\fill[white] (12.5,20) circle (.38cm);
\draw (12.5,20) node {\scriptsize $s_t$};

\draw [dotted] (12.5,20) circle (1.2cm);
\draw (13.1,20.5) node {\footnotesize $\frac\pi2$};
\draw (11.9,20.5) node {\footnotesize $\frac\pi2$};
\draw (13.1,19.5) node {\footnotesize $\frac\pi2$};
\draw (11.9,19.5) node {\footnotesize $\frac\pi2$};

\draw (15.2,16.5) node {$\mathbb{H}$};
\draw [dotted] (5.5,15) --++(9,0);
\draw (15.2,15) node {$\R$};

\draw[arrows=->,line width=.75pt](8,15)--(10,15);
\draw (9.9,15.4) node {\footnotesize a};
\draw[arrows=->,line width=.75pt](8,15)--(9,{15+sqrt(3)});
\draw (9.1,16.95) node {\footnotesize d};
\draw[arrows=->,line width=.75pt](8,15)--(7,{15+sqrt(3)});
\draw (6.9,16.95) node {\footnotesize a};
\draw[arrows=->,line width=.75pt](8,15)--(6,15);
\draw (6.1,15.4) node {\footnotesize d};
\draw[arrows=->,line width=.75pt](8,15)--(7,{15-sqrt(3)});
\draw (6.9,13.05) node {\footnotesize a};
\draw[arrows=->,line width=.75pt](8,15)--(9,{15-sqrt(3)});
\draw (9.1,13.05) node {\footnotesize d};

\fill[white] (8,15) circle (.38cm);
\draw (8,15) node {\scriptsize $t$};

\draw [dotted] (8,15) circle (1.2cm);
\draw (8.65,15.45) node {\footnotesize $\frac\pi3$};
\draw (8,15.7) node {\footnotesize $\frac\pi3$};
\draw (7.35,15.45) node {\footnotesize $\frac\pi3$};
\draw (8,14.3) node {\footnotesize $\frac\pi3$};
\draw (7.35,14.55) node {\footnotesize $\frac\pi3$};
\draw (8.65,14.55) node {\footnotesize $\frac\pi3$};

\draw[arrows=->,line width=.75pt](12.5,16)--(12.5+sqrt{9/2},{16+sqrt{9/2}});
\draw (14.25,17.75) node {\footnotesize a};
\draw[arrows=->,line width=.75pt](12.5,16)--(12.5-sqrt{9/2},{16-sqrt{9/2}});
\draw (10.75,14.25) node {\footnotesize d};

\fill[white] (12.5,16) circle (.5cm);
\draw (12.5,16) node {\scriptsize $s_t+i\e$};

\draw [dotted] (12.5,16) circle (1.2cm);
\draw (11.9,16.6) node {\footnotesize $\pi$};

\draw (15.2,11.5) node {$\mathbb{H}$};
\draw [dotted] (5.5,10) --++(9,0);
\draw (15.2,10) node {$\R$};

\draw[arrows=->,line width=.75pt](8,10)--(10,10);
\draw (9.9,10.4) node {\footnotesize a};
\draw[arrows=->,line width=.75pt](8,10)--(9,{10+sqrt(3)});
\draw (9.1,11.95) node {\footnotesize d};
\draw[arrows=->,line width=.75pt](8,10)--(7,{10+sqrt(3)});
\draw (6.9,11.95) node {\footnotesize a};
\draw[arrows=->,line width=.75pt](8,10)--(6,10);
\draw (6.1,10.4) node {\footnotesize d};
\draw[arrows=->,line width=.75pt](8,10)--(7,{10-sqrt(3)});
\draw (6.9,8.05) node {\footnotesize a};
\draw[arrows=->,line width=.75pt](8,10)--(9,{10-sqrt(3)});
\draw (9.1,8.05) node {\footnotesize d};

\fill[white] (8,10) circle (.5cm);
\draw (8,10) node {\scriptsize $t_{1,n}$};

\draw [dotted] (8,10) circle (1.2cm);
\draw (8.65,10.45) node {\footnotesize $\frac\pi3$};
\draw (8,10.7) node {\footnotesize $\frac\pi3$};
\draw (7.35,10.45) node {\footnotesize $\frac\pi3$};
\draw (8,9.3) node {\footnotesize $\frac\pi3$};
\draw (7.35,9.55) node {\footnotesize $\frac\pi3$};
\draw (8.65,9.55) node {\footnotesize $\frac\pi3$};

\draw[arrows=->,line width=.75pt](12.5,10)--(14,10);
\draw (13.9,10.4) node {\footnotesize d};
\draw[arrows=->,line width=.75pt](12.5,10)--(12.5,11.5);
\draw (12.5,11.7) node {\footnotesize a};
\draw[arrows=->,line width=.75pt](12.5,10)--(11,10);
\draw (11.1,10.4) node {\footnotesize d};
\draw[arrows=->,line width=.75pt](12.5,10)--(12.5,8.5);
\draw (12.5,8.3) node {\footnotesize a};

\fill[white] (12.5,10) circle (.5cm);
\draw (12.5,10) node {\scriptsize $s_n$};

\draw [dotted] (12.5,10) circle (1.2cm);
\draw (13.1,10.5) node {\footnotesize $\frac\pi2$};
\draw (11.9,10.5) node {\footnotesize $\frac\pi2$};
\draw (13.1,9.5) node {\footnotesize $\frac\pi2$};
\draw (11.9,9.5) node {\footnotesize $\frac\pi2$};

\draw (15.2,6.5) node {$\mathbb{H}$};
\draw [dotted] (5.5,5) --++(9,0);
\draw (15.2,5) node {$\R$};

\draw[arrows=->,line width=.75pt](7,5)--(8,5);
\draw (7.9,5.4) node {\footnotesize d};
\draw[arrows=->,line width=.75pt](7,5)--(7,6);
\draw (7,6.2) node {\footnotesize a};
\draw[arrows=->,line width=.75pt](7,5)--(6,5);
\draw (6.1,5.4) node {\footnotesize d};
\draw[arrows=->,line width=.75pt](7,5)--(7,4);
\draw (7,3.8) node {\footnotesize a};

\fill[white] (7,5) circle (.5cm);
\draw (7,5) node {\scriptsize $t_{2,n}$};

\draw[arrows=->,line width=.75pt](10,5)--(11,5);
\draw (10.9,5.4) node {\footnotesize a};
\draw[arrows=->,line width=.75pt](10,5)--(10,6);
\draw (10,6.2) node {\footnotesize d};
\draw[arrows=->,line width=.75pt](10,5)--(9,5);
\draw (9.1,5.4) node {\footnotesize a};
\draw[arrows=->,line width=.75pt](10,5)--(10,4);
\draw (10,3.8) node {\footnotesize d};

\fill[white] (10,5) circle (.5cm);
\draw (10,5) node {\scriptsize $t_{1,n}$};

\draw[arrows=->,line width=.75pt](13,5)--(14,5);
\draw (13.9,5.4) node {\footnotesize d};
\draw[arrows=->,line width=.75pt](13,5)--(13,6);
\draw (13,6.2) node {\footnotesize a};
\draw[arrows=->,line width=.75pt](13,5)--(12,5);
\draw (12.1,5.4) node {\footnotesize d};
\draw[arrows=->,line width=.75pt](13,5)--(13,4);
\draw (13,3.8) node {\footnotesize a};

\fill[white] (13,5) circle (.5cm);
\draw (13,5) node {\scriptsize $s_n$};

\draw (15.2,1.5) node {$\mathbb{H}$};
\draw [dotted] (5.5,0) --++(9,0);
\draw (15.2,0) node {$\R$};

\draw[arrows=->,line width=.75pt](12.5,0)--(14,0);
\draw (13.9,.4) node {\footnotesize d};
\draw[arrows=->,line width=.75pt](12.5,0)--(12.5,1.5);
\draw (12.5,1.7) node {\footnotesize a};
\draw[arrows=->,line width=.75pt](12.5,0)--(11,0);
\draw (11.1,.4) node {\footnotesize d};
\draw[arrows=->,line width=.75pt](12.5,0)--(12.5,-1.5);
\draw (12.5,-1.7) node {\footnotesize a};

\fill[white] (12.5,0) circle (.5cm);
\draw (12.5,0) node {\scriptsize $s_n$};

\draw [dotted] (12.5,0) circle (1.2cm);
\draw (13.1,.5) node {\footnotesize $\frac\pi2$};
\draw (11.9,.5) node {\footnotesize $\frac\pi2$};
\draw (13.1,-.5) node {\footnotesize $\frac\pi2$};
\draw (11.9,-.5) node {\footnotesize $\frac\pi2$};

\draw[arrows=->,line width=.75pt](8,1)--(8+sqrt{9/2},{1+sqrt{9/2}});
\draw (9.75,2.75) node {d};
\draw[arrows=->,line width=.75pt](8,1)--(8-sqrt{9/2},{1+sqrt{9/2}});
\draw (6.25,2.75) node {a};
\draw[arrows=->,line width=.75pt](8,1)--(8-sqrt{9/2},{1-sqrt{9/2}});
\draw (6.25,-.75) node {d};
\draw[arrows=->,line width=.75pt](8,1)--(8+sqrt{9/2},{1-sqrt{9/2}});
\draw (9.75,-.75) node {a};

\fill[white] (8,1) circle (.5cm);
\draw (8,1) node {\scriptsize $t_{1,n}$};

\draw [dotted] (8,1) circle (1.2cm);
\draw (8.8,1) node {$\frac\pi2$};
\draw (8,1.7) node {$\frac\pi2$};
\draw (7.2,1) node {$\frac\pi2$};
\draw (8,.2) node {$\frac\pi2$};

\end{tikzpicture}
\caption{
(A,B), left: $R_t |_{L_t}$ for possibilities (A,B) of lemma \ref{lemfn'Case2}.
(a,b,c), left: $R_n |_{L_n}$ for possibilities (a,b,c) of lemma \ref{lemfn'Case2}.
(A,B), right: The associated directions of steepest decent/ascent in $\C$ for
$f_t$ for possibilities (A,B).
(a,b,c), right: The associated directions of steepest decent/ascent in $\C$ for
$f_n$ for possibilities (a,b,c).}
\label{figConDirCase2}
\end{figure}

Suppose that possibility (a) of lemma \ref{lemRootsNonAsy1} is satisfied,
i.e., that $t_{1,n} \in (t-\xi,t+\xi) \subset L_n$,
and $t_{1,n} = t_{2,n}$ is root of $f_n'$ of multiplicity $2$. Thus
$f_n'(t_{1,n}) = f_n''(t_{1,n}) = 0$ and $f_n'''(t_{1,n}) \neq 0$.
Indeed, $f_n'''(t_{1,n}) \to f_t'''(t) > 0$
(see part (4) of lemma \ref{lemNonAsyRoots}, and lemma
\ref{lemfnftConv}). Next note, equation (\ref{eqfn'}) implies that
$(f_n') |_{L_n}$ is real-valued and continuous. Moreover, recalling
that $\underline{L_n}$ and $\overline{L_n}$ are two consecutive elements
of $S_{1,n}$, this equation gives
\begin{equation*}
\lim_{w \in \R, w \downarrow \underline{L_n}} f_n'(w) = + \infty
\hspace{.5cm} \text{and} \hspace{.5cm}
\lim_{w \in \R, w \uparrow \overline{L_n}} f_n'(w) = - \infty.
\end{equation*}
Also, recalling that $L_n \in \{K_{1,n}^{(1)},K_{2,n}^{(1)},\ldots\}$,
part (1) of lemma \ref{lemRootsNonAsy1} implies that $f_n'$ has $0$
roots in $(t-\xi,t+\xi) \setminus \{t_{1,n}\}$, and $1$ root
(denoted $s_n$) in $L_n \setminus (t-\xi,t+\xi)
= (\underline{L_n}, t-\xi] \cup [t+\xi,\overline{L_n})$.
The above observations are only possible if $s_n \in [t+\xi,\overline{L_n})$
and part (a1) is satisfied. Finally, since $t_{1,n} \in (t-\xi, t+\xi)$
and $t_{1,n} = t_{2,n}$ (see possibility (a) of lemma \ref{lemRootsNonAsy1}),
and since $L_n \in \{K_{1,n}^{(1)},K_{2,n}^{(1)},\ldots\}$, parts (1,2,3)
of lemma \ref{lemRootsNonAsy1} imply parts (a2,a3). Thus, whenever
possibility (a) of lemma \ref{lemRootsNonAsy1} is satisfied, we have
parts (a1,a2,a3). Similarly, whenever possibilities (b) and (c)
are satisfied, we have parts (b1,b2,b3) and (c1,c2,c3) respectively.
\end{proof}

As in section \ref{secContCase1}, let $T_t \subset \C \setminus S$ and
$T_n \subset \C \setminus S_n$ respectively denote the set of roots of
$f_t'$ and the set of roots of $f_n'$. The previous lemma discusses the
locations of the elements of these discrete sets, and this is depicted
in figure \ref{figSnC_xi2}. Next recall, similarly to section
\ref{secContCase1}, equation (\ref{eqRnRes}) gives
$(R_t |_{L_t})' = (f_t') |_{L_t}$ and $(R_n |_{L_n})' = (f_n') |_{L_n}$,
and similarly for the higher order derivatives. Parts (A1,B1,a1,b1,c1)
of lemma \ref{lemfn'Case2} then imply that $R_t |_{L_t}$ and
$R_n |_{L_t}$ have those behaviours shown on the left of figure
\ref{figConDirCase2} for the various possibilities. Part (3) of lemma
\ref{lemDesAsc} also shows that $f_t$ and $f_n$ have those
directions of steepest descent/ascent shown on the right of figure
\ref{figConDirCase2}. For possibility (B), in anticipation of the
following lemma, we also display the directions of steepest
descent/ascent in the neighbourhood of an arbitrary point
$s_t + i \e \in \mathbb{H}$, where $s_t \in (t, \overline{L_t})$ and $\e>0$. Note,
part (B2) of lemma \ref{lemfn'Case2} implies that $f_t'(s_t+i\e) \neq 0$
for any such $s_t+i\e$, and so part (3) of lemma \ref{lemDesAsc}
implies that there is a unique direction of steepest descent and a
unique direction of steepest ascent as shown. The following lemma
examines the resulting contours of steepest descent/ascent:
\begin{lem}
\label{lemConAscDesCase2}
Whenever possibility (A) of lemma \ref{lemfn'Case2} is satisfied,
$s_t \in [t+\xi,\overline{L_t})$ (see equation (\ref{eqxi}) and
part (A1) of lemma \ref{lemfn'Case2}). Moreover, there exists simple
contours, $D_t, A_t, A_t''$, as shown in figure \ref{figConDesAscCase2}
with the following properties:
\begin{enumerate}
\item[(A1)]
$D_t, A_t, A_t''$ start at $t, t, s_t$ respectively, enter $\mathbb{H}$
in the directions $\frac\pi3, \frac{2\pi}3, \frac\pi2$ respectively,
and end in the intervals shown or are unbounded.
\item[(A2)]
$D_t$ is a contour of steepest descent for $f_t$, and
$A_t, A_t''$ are contours of steepest ascent for $f_t$.
\item[(A3)]
$D_t$ and $A_t$ intersect at $t$, and $D_t, A_t, A_t''$ do not otherwise intersect.
\end{enumerate}
When possibility (B) is satisfied, fix an arbitrary
$s_t \in (t, \overline{L_t})$. Moreover, fix the $\xi>0$
sufficiently small such that $s_t \in [t+\xi,\overline{L_t})$,
and fix an arbitrary $\e > 0$.
Then, when the $\e>0$ is fixed sufficiently small, there exists simple contours,
$D_t, A_t, D_t'', A_t''$, as shown in figure \ref{figConDesAscCase2}
with the following properties:
\begin{enumerate}
\item[(B1)]
$D_t$ and $A_t$ both start at $t$, enter $\mathbb{H}$
in the directions $\frac\pi3$ and $\frac{2\pi}3$ respectively,
and end in the intervals shown. $D_t''$ and
$A_t''$ both start at $s_t+i\e \in \mathbb{H}$, leave $s_t+i\e$
in opposite directions, and end in the interior of the
intervals shown or are unbounded.
\item[(B2)]
$D_t,D_t''$ are contours of steepest descent for $f_t$, and
$A_t, A_t''$ are contours of steepest ascent for $f_t$.
\item[(B3)]
$D_t$ and $A_t$ intersect at $t$,
$D_t''$ and $A_t''$ intersect at $s_t+i\e$,
and $D_t, A_t, D_t'', A_t''$ do not otherwise intersect.
\end{enumerate}

Also, whenever possibility (a) is
satisfied, there exists simple contours, $D_n, A_n, A_n''$, as shown
in figure \ref{figConDesAscCase2} with the following properties:
\begin{enumerate}
\item[(a1)]
$D_n, A_n, A_n''$ start at $t_{1,n}, t_{1,n}, s_n$ respectively, enter $\mathbb{H}$
in the directions $\frac\pi3, \frac{2\pi}3, \frac\pi2$ respectively,
and end in the intervals shown or are unbounded.
\item[(a2)]
$D_n$ is a contour of steepest descent for $f_n$, and
$A_n, A_n''$ are contours of steepest ascent for $f_n$.
\item[(a3)]
$D_n$ and $A_n$ intersect at $t_{1,n}$,
and $D_n, A_n, A_n''$ do not otherwise intersect.
\end{enumerate}
Next, whenever possibility (b) is
satisfied, there exists simple contours, $D_n,A_n,A_n''$, as shown
in figure \ref{figConDesAscCase2} with the following properties:
\begin{enumerate}
\item[(b1)]
$D_n, A_n, A_n''$ start at $t_{1,n}, t_{2,n}, s_n$ respectively,
all enter $\mathbb{H}$ in the direction $\frac\pi2$,
and end in the intervals shown or are unbounded.
\item[(b2)]
$D_n$ is a contour of steepest descent for $f_n$, and
$A_n, A_n''$ are contours of steepest ascent for $f_n$.
\item[(b3)]
$D_n, A_n, A_n''$ do not intersect.
\end{enumerate}
Next, whenever possibility (c) is
satisfied, there exists simple contours, $D_n, A_n, D_n', A_n', A_n''$, as shown
in figure \ref{figConDesAscCase2} with the following properties:
\begin{enumerate}
\item[(c1)]
$D_n,A_n,D_n',A_n'$ all start at $t_{1,n} \in \mathbb{H}$, leave $t_{1,n}$
in orthogonal directions in the counter-clockwise order $D_n,A_n,D_n',A_n'$,
and end in the intervals shown or are unbounded. $A_n''$ starts at $s_n$,
enters $\mathbb{H}$ in the direction $\frac\pi2$, and is unbounded.
\item[(c2)]
$D_n,D_n'$ are contours of steepest descent for $f_n$,
and $A_n,A_n',A_n''$ are contours of steepest ascent for $f_n$.
\item[(c3)]
$D_n,A_n,D_n',A_n'$ intersect at $t_{1,n}$,
and $D_n,A_n,D_n',A_n',A_n''$ does not otherwise intersect.
\end{enumerate}
\end{lem}

\begin{figure}[t]
\centering
\begin{tikzpicture};

\draw (-1.2,15.5) node {$\mathbb{H}$};
\draw [dotted] (-1,14) --++(13,0);
\draw (-2,15.55) node {(A)};
\draw (-1.2,14) node {$\R$};
\draw (-.5,14) node {\scriptsize $)$};
\draw (-.5,13.7) node {\scriptsize $\underline{S_3}$};
\draw (1,14) node {\scriptsize $)$};
\draw (1,13.7) node {\scriptsize $\overline{S_3}$};
\draw (2.5,14) node {\scriptsize $($};
\draw (2.5,13.7) node {\scriptsize $\underline{S_2}$};
\draw (4,14) node {\scriptsize $)$};
\draw (4,13.7) node {\scriptsize $\overline{S_2}$};
\draw (5.5,13.9) --++(0,.2);
\draw (5.5,13.7) node {\scriptsize $\underline{S_1}$};
\draw (7,14) node {\scriptsize $($};
\draw (7,13.7) node {\scriptsize $\underline{L_t}$};
\draw (8,13.7) node {$t$};
\draw (9,13.7) node {$s_t$};
\draw (10,14) node {\scriptsize $)$};
\draw (10,13.7) node {\scriptsize $\overline{L_t}$};
\draw (11.5,13.9) --++(0,.2);
\draw (11.5,13.7) node {\scriptsize $\overline{S_1}$};

\draw plot [smooth, tension=.9] coordinates
{ (8,14) (7.5,15.75) (2,16) (.25,14) };
\draw[arrows=->,line width=1pt](4.01,16.24)--(4,16.24);
\draw (4,15.9) node {$D_t$};
\draw plot [smooth, tension=1.4] coordinates
{ (8,14) (5,15.5) (3.25,14) };
\draw[arrows=->,line width=1pt](5.01,15.5)--(5,15.5);
\draw (5,15.15) node {$A_t$};
\draw plot [smooth, tension=1] coordinates
{ (9,14) (9.25,14.5) (10,15) (10.5,15.75) };
\draw[arrows=->,line width=1pt](10,15)--(10.01,15.01);
\draw [dotted] (10.5,15.75) --++(.25,.5);
\draw (9.9,16.3) node {\scriptsize unbounded};
\draw (9.8,15.3) node {$A_t''$};

\draw (-1.2,12) node {$\mathbb{H}$};
\draw (-2,12) node {(B)};
\draw [dotted] (-1,10.5) --++(13,0);
\draw (-1.2,10.5) node {$\R$};
\draw (-.5,10.5) node {\scriptsize $)$};
\draw (-.5,10.2) node {\scriptsize $\underline{S_3}$};
\draw (1,10.5) node {\scriptsize $)$};
\draw (1,10.2) node {\scriptsize $\overline{S_3}$};
\draw (2.5,10.5) node {\scriptsize $($};
\draw (2.5,10.2) node {\scriptsize $\underline{S_2}$};
\draw (4,10.5) node {\scriptsize $)$};
\draw (4,10.2) node {\scriptsize $\overline{S_2}$};
\draw (5.5,10.4) --++(0,.2);
\draw (5.5,10.2) node {\scriptsize $\underline{S_1}$};
\draw (7,10.5) node {\scriptsize $($};
\draw (7,10.2) node {\scriptsize $\underline{L_t}$};
\draw (8,10.2) node {$t$};
\draw [dashed] (9,10.5) --++ (0,.5);
\draw (9,10.2) node {$s_t$};
\draw (9.25,10.73) node {$\e$};
\draw[arrows=<->,line width=.5pt](9.1,10.5)--(9.1,11);
\draw (10,10.5) node {\scriptsize $)$};
\draw (10,10.2) node {\scriptsize $\overline{L_t}$};
\draw (11.5,10.4) --++(0,.2);
\draw (11.5,10.2) node {\scriptsize $\overline{S_1}$};

\draw plot [smooth, tension=.9] coordinates
{ (8,10.5) (7.5,12.25) (2,12.5) (.25,10.5) };
\draw[arrows=->,line width=1pt](4.01,12.74)--(4,12.74);
\draw (4,12.4) node {$D_t$};
\draw plot [smooth, tension=1.4] coordinates
{ (8,10.5) (5,12) (3.25,10.5) };
\draw[arrows=->,line width=1pt](5.01,12)--(5,12);
\draw (5,11.65) node {$A_t$};
\filldraw [dotted] (9,11) circle (.05cm);
\draw plot [smooth, tension=1] coordinates
{ (9,11) (9.25,11.5) (10,12) (10.5,12.75) };
\draw[arrows=->,line width=1pt](10,12)--(10.01,12.01);
\draw [dotted] (10.5,12.75) --++(.25,.5);
\draw (9.7,13.1) node {\scriptsize unbounded};
\draw (10.2,11.7) node {$A_t''$};
\draw plot [smooth, tension=.9] coordinates
{ (9,11) (8.5,10.8) (7.5,12.75) (2,13) (0,10.5) };
\draw[arrows=->,line width=1pt](4.11,13.31)--(4.1,13.31);
\draw (8.7,12.2) node {$D_t''$};

\draw (-1.2,8.5) node {$\mathbb{H}$};
\draw (-2,8.5) node {(a)};
\draw [dotted] (-1,7) --++(13,0);
\draw (-1.2,7) node {$\R$};
\draw (-.5,6.9) --++(0,.2);
\draw (-.5,6.7) node {\scriptsize $\underline{S_{3,n}}$};
\draw (1,6.9) --++(0,.2);
\draw (1,6.7) node {\scriptsize $\overline{S_{3,n}}$};
\draw (2.5,6.9) --++(0,.2);
\draw (2.5,6.7) node {\scriptsize $\underline{S_{2,n}}$};
\draw (4,6.9) --++(0,.2);
\draw (4,6.7) node {\scriptsize $\overline{S_{2,n}}$};
\draw (5.5,6.9) --++(0,.2);
\draw (5.5,6.7) node {\scriptsize $\underline{S_{1,n}}$};
\draw (7,6.9) --++(0,.2);
\draw (7,6.7) node {\scriptsize $\underline{L_n}$};
\draw (8,6.7) node {$t_{1,n}$};
\draw (9,6.7) node {$s_n$};
\draw (10,6.9) --++(0,.2);
\draw (10,6.7) node {\scriptsize $\overline{L_n}$};
\draw (11.5,6.9) --++(0,.2);
\draw (11.5,6.7) node {\scriptsize $\overline{S_{1,n}}$};

\draw plot [smooth, tension=.9] coordinates
{ (8,7) (7.5,8.75) (2,9) (.25,7) };
\draw[arrows=->,line width=1pt](4.01,9.24)--(4,9.24);
\draw (4,8.9) node {$D_n$};
\draw plot [smooth, tension=1.4] coordinates
{ (8,7) (5,8.5) (3.25,7) };
\draw[arrows=->,line width=1pt](5.01,8.5)--(5,8.5);
\draw (5,8.15) node {$A_n$};
\draw plot [smooth, tension=1] coordinates
{ (9,7) (9.25,7.5) (10,8) (10.5,8.75) };
\draw[arrows=->,line width=1pt](10,8)--(10.01,8.01);
\draw [dotted] (10.5,8.75) --++(.25,.5);
\draw (9.9,9.3) node {\scriptsize unbounded};
\draw (9.8,8.3) node {$A_n''$};

\draw (-1.2,5) node {$\mathbb{H}$};
\draw (-2,5) node {(b)};
\draw [dotted] (-1,3.5) --++(13,0);
\draw (-1.2,3.5) node {$\R$};
\draw (-.5,3.4) --++(0,.2);
\draw (-.5,3.2) node {\scriptsize $\underline{S_{3,n}}$};;
\draw (1,3.4) --++(0,.2);
\draw (1,3.2) node {\scriptsize $\overline{S_{3,n}}$};
\draw (2.5,3.4) --++(0,.2);
\draw (2.5,3.2) node {\scriptsize $\underline{S_{2,n}}$};
\draw (4,3.4) --++(0,.2);
\draw (4,3.2) node {\scriptsize $\overline{S_{2,n}}$};
\draw (5.5,3.4) --++(0,.2);
\draw (5.5,3.2) node {\scriptsize $\underline{S_{1,n}}$};
\draw (7,3.4) --++(0,.2);
\draw (7,3.2) node {\scriptsize $\underline{L_n}$};
\draw (7.75,3.2) node {\scriptsize $t_{2,n}$};
\draw (8.5,3.2) node {\scriptsize $t_{1,n}$};
\draw (9.25,3.2) node {\scriptsize $s_n$};
\draw (10,3.4) --++(0,.2);
\draw (10,3.2) node {\scriptsize $\overline{L_n}$};
\draw (11.5,3.4) --++(0,.2);
\draw (11.5,3.2) node {\scriptsize $\overline{S_{1,n}}$};

\draw plot [smooth, tension=1.8] coordinates
{ (8.5,3.5) (4.375,6) (.25,3.5) };
\draw[arrows=->,line width=1pt](4.375,6)--(4.37,6);
\draw (4.375,5.7) node {$D_n$};
\draw plot [smooth, tension=1.9] coordinates
{ (7.75,3.5) (5.5,5) (3.25,3.5) };
\draw[arrows=->,line width=1pt](5.51,5)--(5.5,5);
\draw (5.5,4.6) node {$A_n$};
\draw plot [smooth, tension=1] coordinates
{ (9.25,3.5) (9.5,4) (10.25,4.5) (10.75,5.25) };
\draw[arrows=->,line width=1pt](10.25,4.5)--(10.26,4.51);
\draw [dotted] (10.75,5.25) --++(.25,.5);
\draw (10,5.7) node {\scriptsize unbounded};
\draw (10.05,4.8) node {$A_n''$};

\draw (-1.2,1.5) node {$\mathbb{H}$};
\draw (-2,1.5) node {(c)};
\draw [dotted] (-1,0) --++(13,0);
\draw (-1.2,0) node {$\R$};
\draw (-.5,-.1) --++(0,.2);
\draw (-.5,-.3) node {\scriptsize $\underline{S_{3,n}}$};
\draw (1,-.1) --++(0,.2);
\draw (1,-.3) node {\scriptsize $\overline{S_{3,n}}$};
\draw (2.5,-.1) --++(0,.2);
\draw (2.5,-.3) node {\scriptsize $\underline{S_{2,n}}$};
\draw (4,-.1) --++(0,.2);
\draw (4,-.3) node {\scriptsize $\overline{S_{2,n}}$};
\draw (5.5,-.1) --++(0,.2);
\draw (5.5,-.3) node {\scriptsize $\underline{S_{1,n}}$};
\draw (7,-.1) --++(0,.2);
\draw (7,-.3) node {\scriptsize $\underline{L_n}$};
\draw (8.4,-.3) node {$t_{1,n}$};
\draw[arrows=->,line width=.5pt](8.5,-.2)--(8.5,.4);
\draw (9.5,-.3) node {$s_n$};
\draw (10,-.1) --++(0,.2);
\draw (10,-.3) node {\scriptsize $\overline{L_n}$};
\draw (11.5,-.1) --++(0,.2);
\draw (11.5,-.3) node {\scriptsize $\overline{S_{1,n}}$};

\draw plot [smooth, tension=.9] coordinates
{ (8.5,0.5) (8.2,2.5) (2,2.5) (.25,0) };
\draw[arrows=->,line width=1pt](5.21,2.92)--(5.2,2.92);
\draw (5.2,2.6) node {$D_n$};
\draw plot [smooth, tension=1.3] coordinates
{ (8.5,0.5) (5,2) (3.25,0) };
\draw[arrows=->,line width=1pt](5.41,2.01)--(5.4,2.01);
\draw (5.4,1.65) node {$A_n$};
\draw plot [smooth, tension=.95] coordinates
{ (8.5,0.5) (8,.2) (7.3,.2) (6.8,0)};
\draw[arrows=->,line width=1pt](7.4,.21)--(7.39,.21);
\draw (7.45,.5) node {$D_n'$};
\draw plot [smooth, tension=.95] coordinates
{ (8.5,0.5) (9.2,.4) (10,1.75)};
\draw [dotted] (10,1.75) --++(.25,.5);
\draw[arrows=->,line width=1pt](9.9,1.5)--(9.91,1.52);
\draw (9.5,1.5) node {$A_n'$};
\draw plot [smooth, tension=1] coordinates
{ (9.5,0) (9.75,.5) (10.5,1) (11,1.75) };
\draw[arrows=->,line width=1pt](10.5,1)--(10.51,1.01);
\draw (10.75,2.4) node {\scriptsize unbounded};
\draw [dotted] (11,1.75) --++(.25,.5);
\draw (10.7,.7) node {$A_n''$};

\end{tikzpicture}
\caption{The contours of lemma \ref{lemConAscDesCase2}.}
\label{figConDesAscCase2}
\end{figure}
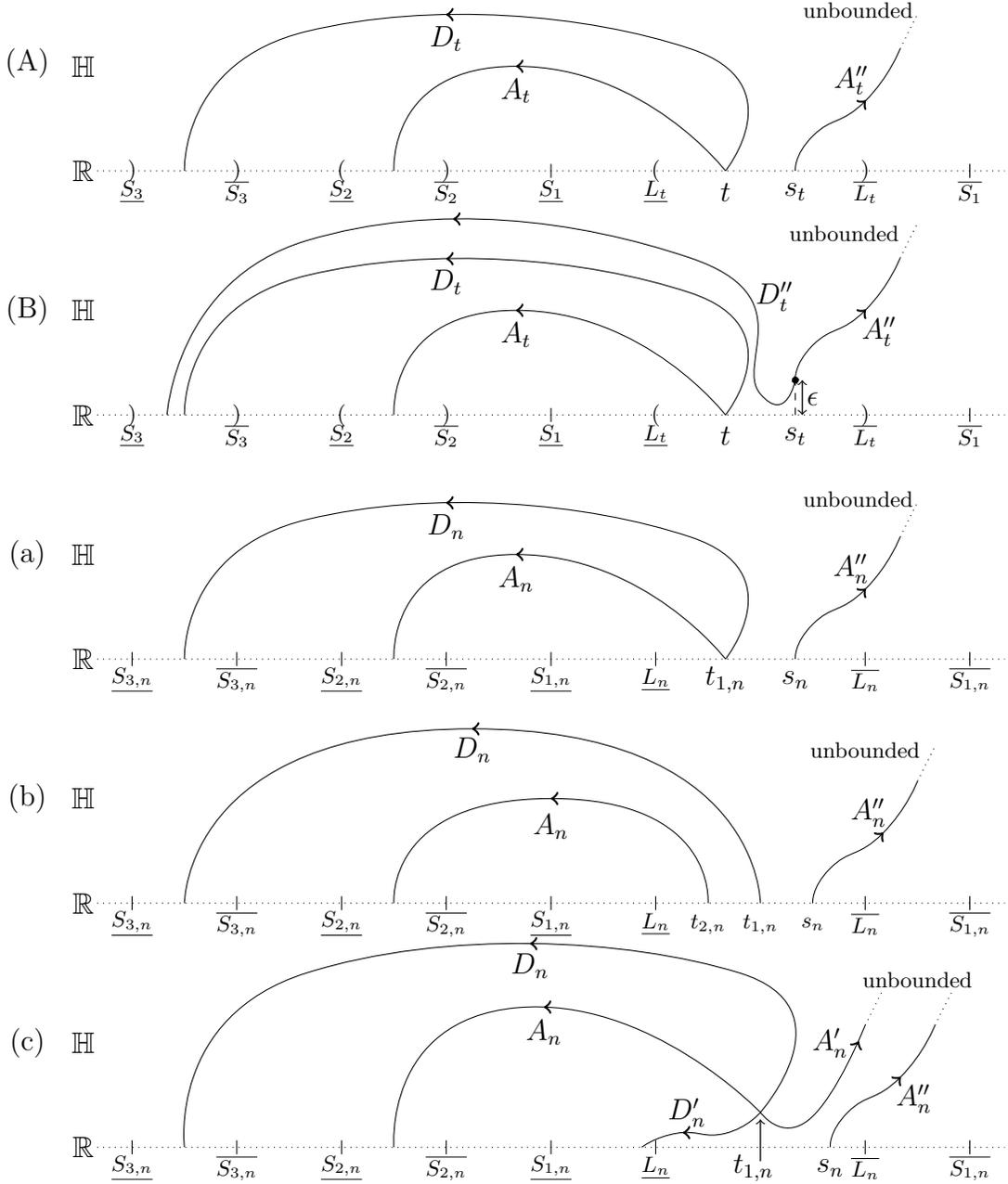

\begin{proof}
Many of the arguments here are similar to those used in lemma
\ref{lemConAscDesCase1}, so we do not go into as much detail here.

Consider $f_n$. First, for possibilities (a,b), define $D_n$ and $A_n$
as in lemma \ref{lemConAscDesCase1}. Also, for possibility (c), define
$D_n,A_n,D_n',A_n'$ as in lemma \ref{lemConAscDesCase1}. Moreover, for
possibilities (a,b,c), let $A_n''$ denote the contour of steepest ascent for
$f_n$ which starts at $s_n \in L_n$, and which enters $\mathbb{H}$ in the
direction $\frac\pi2$. Then, proceeding similarly to the proof of lemma
\ref{lemConAscDesCase1}, we can show that each of $D_n,D_n'$ end in
$S_n \cup T_n$, each of $A_n,A_n',A_n''$ either end in $S_n \cup T_n$ or are
unbounded, and the contours are otherwise in $\mathbb{H}$. We will show:
\begin{enumerate}
\item[(i)]
There exists a $\delta \in (0,\xi)$ for which
$A_n''$ does not intersect $\text{cl}(B(t,\delta))$.
\item[(ii)]
For possibility (a), $D_n$ and $A_n$ do not intersect except at
$t_{1,n}$, and $D_n$ and $A_n''$ never intersect. For possibility (b),
$D_n$ and $A_n$ never intersect, and $D_n$ and $A_n''$ never intersect. For possibility
(c), $D_n,A_n,D_n',A_n'$ do not intersect except at $t_{1,n}$, $D_n$
and $A_n''$ never intersect, and $D_n'$ and $A_n''$ never intersect. For all
possibilities, the contours are simple.
\end{enumerate}
Then, we can investigate the possible end-points of $D_n,A_n,D_n',A_n',A_n''$
using arguments by contradiction similar to those used in the proof of
lemma \ref{lemConAscDesCase1}. We omit the details, and simply state that
this investigation gives the required results.

Consider (i). First recall that $A_n''$ starts at $s_n$, and $R_n$
strictly increases along $A_n''$. It is thus sufficient to show
that there exists a $\delta \in (0,\xi)$ for which the following is satisfied:
\begin{equation}
\label{eq-lemConAscDesCase2-1}
R_n(w) < R_n(s_n)
\text{ for all } w \in \text{cl}(B(t,\delta)).
\end{equation}
To see this first note, part (2) of lemma \ref{lemfnftConv} implies that
$R_n(w) = R_t(w) + o(1)$ uniformly for $w \in B(t,\xi)$. Also, since $R_t$
is continuous, $R_t(w) = R_t(t) + O(\delta) = f_t(t) + O(\delta)$ for all $\delta>0$
sufficiently small and uniformly for $w \in \text{cl}(B(t,\delta))$. Next
recall (see statement of this lemma) that $s_t \in [t+\xi, \overline{L_t})$
for possibilities (A,B), and so parts (A1,B1) of lemma
\ref{lemfn'Case2} imply that $f_t(t) < f_t(t+\xi)$. Moreover, part (2) of
lemma \ref{lemfnftConv} implies that $f_t(t+\xi) = f_n(t+\xi) + o(1)$.
Finally, parts (a1,b1,c1) of lemma \ref{lemfn'Case2} imply that
$f_n(t+\xi) < f_n(s_n) = R_n(s_n)$. Combined, the above imply equation
(\ref{eq-lemConAscDesCase2-1}). This proves part (i).

Consider (ii) for possibility (a). Recall that $D_n$ and $A_n$ start at
$t_{1,n}$, $R_n$ strictly decreases along $D_n$, and $R_n$ strictly
increases along $A_n$. Also recall that $D_n$ starts at $t_{1,n} \in L_n$
and $A_n''$ starts at $s_n \in L_n$, $R_n$ strictly decreases along $D_n$,
$R_n$ strictly increases along $A_n''$, $s_n > t_{1,n}$, and $R_n$
strictly increases as we move from $t_{1,n}$ to $s_n$ along $L_n$ (see
left of figure \ref{figConDirCase2}). A contradiction argument then proves
part (ii) for possibility (a). Part (ii) for possibility (b) follows similarly.

Consider (ii) for possibility (c). This follows from similar arguments
to those used in the proof of lemma \ref{lemConAscDesCase1}, except we
need to additionally argue that $D_n$ never intersects $A_n''$, and $D_n'$
never intersects $A_n''$. Recall that $D_n$ and $D_n'$ start at $t_{1,n}$,
$A_n''$ starts at $s_n$, $R_n$ strictly decreases along $D_n$ and $D_n'$,
$R_n$ strictly increases along $A_n''$, and $R_n(t_{1,n}) < R_n(s_n)$
(recall that $t_{1,n} \to t$, and apply part (i) above). A
contradiction argument then proves that $D_n$ and $D_n'$ never intersect
$A_n''$. This proves part (ii) for possibility (c).

Consider $f_t$. First, define $D_t$ and $A_t$ as in lemma
\ref{lemConAscDesCase1}. Next, for possibility (A) of
lemma \ref{lemfn'Case2}, let $A_t''$ denote the contour of
steepest ascent for $f_t$ which starts at $s_t \in L_t$, and
which enters $\mathbb{H}$ in the direction $\frac\pi2$. For
possibility (B), let $D_t''$ and $A_t''$ denote the contours
of steepest descent and ascent (respectively) which start at
$s_t+i\e \in \mathbb{H}$.  Then, proceeding similarly to the proof
of lemma \ref{lemConAscDesCase1}, we can show that each of
$D_t,D_t''$ end in $S \cup T_t$, each of $A_t,A_t''$ either
end in $S \cup T_t$ or are unbounded, and the contours are
otherwise in $\mathbb{H}$. We will show:
\begin{enumerate}
\item[(iii)]
For possibility (A), $D_t$ and $A_t$ do not intersect except at
$t$, and $D_t$ and $A_t''$ never intersect. For possibility (B),
$D_t$ and $A_t$ do not intersect except at $t$, $D_t''$ and $A_t''$ do
not intersect except at $s_t+i\e$, and we can fix the $\e>0$ sufficiently
small such that $D_t'',A_t''$ never intersect $D_t,A_t$. For both
possibilities, the contours are simple.
\item[(iv)]
For possibilities (A,B), $D_t$ ends in $[\underline{S_3}, \overline{S_3})$,
and $A_t$ ends in $(\underline{S_2},\overline{S_2})$.
\item[(v)]
For possibility (A), $A_t''$ is unbounded. For possibility (B),
fixing the $\e>0$ sufficiently small as above, $A_t''$ is unbounded
and $D_t''$ ends in $[\underline{S_3}, d_t)$, where $d_t$ denotes
the end-point of $D_t$ (thus, necessarily, we must have
$d_t \in (\underline{S_3}, \overline{S_3})$).
\end{enumerate}
These prove the required results.

Consider (iii) for possibility (A). This follows from similar arguments
to those used to show part (ii) for possibility (a), above. Consider
(iii) for possibility (B). Similar arguments show that
$D_t$ and $A_t$ do not intersect except at $t$, $D_t''$ and $A_t''$ do
not intersect except at $s_t+i\e$, and the contours are simple. It thus
remains to show that  we can fix the $\e>0$ sufficiently
small such that $D_t'',A_t''$ never intersect $D_t,A_t$. Recalling
that $\text{Im}(f_t(w)) = \text{Im}(f_t(t))$ and
$\text{Im}(f_t(z)) = \text{Im}(f_t(s_t+i\e))$ for all $w$ on
$D_t \cup A_t$ and $z$ on $D_t'' \cup A_t''$ (see part (2) of
lemma \ref{lemDesAsc}), and $\text{Im}(f_t(t)) = \text{Im}(f_t(s_t))$
(recall that $t$ and $s$ are both in $L_t \subset \R \setminus S$, and see
figure \ref{figImfnExt}), it is thus sufficient to show that
$\text{Im}(f_t(s_t+i\e)) \neq \text{Im}(f_t(s_t))$ for all $\e>0$
sufficiently small. To see this recall that $s_t \in (t,\overline{L_t})$,
and so part (B1) of lemma \ref{lemfn'Case2} implies that $f_t'(s_t) > 0$.
Also, a Taylor expansion gives $f_t(s_t+i\e) = f_t(s_t) + i \e f_t'(s_t) + O(\e^2)$
for all $\e>0$ sufficiently small. Therefore,
\begin{equation}
\label{eq-lemConAscDesCase2-2}
\text{Im}(f_t(s_t+i\e)) > \text{Im}(f_t(s_t)) + \tfrac12 \e f_t'(s_t)
\text{ for all } \e>0 \text{ sufficiently small.}
\end{equation}
This proves part (iii) for possibility (B).

Consider (iv). Recall that $D_t$ ends in $S \cup T_t$, and $A_t$ either ends
$S \cup T_t$ or is unbounded, where $T_t$ is the set of roots of $f_t'$.
The behaviours of $T_t$ for possibilities (A,B) are discussed in lemma
\ref{lemfn'Case2}, and these behaviours are displayed in figure
\ref{figSnC_xi2}. It follows that $S \cup T_t \subset [\underline{S_3}, \overline{S_3}]
\cup [\underline{S_2}, \overline{S_2}]
\cup [\underline{S_1}, \underline{L_t}] \cup \{t,s_t\}
\cup [\overline{L_t}, \overline{S_1}]$
for possibility (A), and $S \cup T_t \subset [\underline{S_3}, \overline{S_3}]
\cup [\underline{S_2}, \overline{S_2}]
\cup [\underline{S_1}, \underline{L_t}] \cup \{t\}
\cup [\overline{L_t}, \overline{S_1}]$ for possibility (B).
Next recall that $D_t$ and $A_t$ start at $t$, and $A_t''$ starts
at $s_t$ for possibility (A). The above observations, and part
(iii), thus imply the following for both possibilities: $D_t$ ends in
$[\underline{S_3}, \overline{S_3}] \cup [\underline{S_2}, \overline{S_2}] \cup
[\underline{S_1}, \underline{L_t}] \cup [\overline{L_t}, \overline{S_1}]$, and
$A_t$ either ends in $[\underline{S_3}, \overline{S_3}] \cup [\underline{S_2}, \overline{S_2}]
\cup [\underline{S_1}, \underline{L_t}] \cup [\overline{L_t}, \overline{S_1}]$
or is unbounded. Then, since $D_t$ and $A_t$ start at $t \in L_t$, part (iv)
follows from similar arguments to those used to show parts (xii,xiii) in the
proof of lemma \ref{lemConAscDesCase1}.

Consider (v) for possibility (A). Recall that $A_t''$ starts at
$s_t \in (t,\overline{L_t})$, and ends either in $S \cup T_t$ or is unbounded.
Then, similar considerations to those used in part (iv), above, show that
$A_t''$ either ends in $[\underline{S_3}, \overline{S_3}] \cup
[\underline{S_2}, \overline{S_2}] \cup [\underline{S_1}, \underline{L_t}]
\cup [\overline{L_t}, \overline{S_1}]$ or is unbounded. We must thus show
that $A_t''$ does not end in $[\underline{S_3}, \overline{S_3}] \cup
[\underline{S_2}, \overline{S_2}] \cup [\underline{S_1}, \underline{L_t}]
\cup [\overline{L_t}, \overline{S_1}]$. We argue by contradiction: First assume
that $A_t''$ ends in $[d_t,\overline{S_3}] \cup [\underline{S_2}, \overline{S_2}]
\cup [\underline{S_1}, \underline{L_t}]$, where
$d_t \in [\underline{S_3},\overline{S_3})$ denotes the end point of $D_t$
(see part (iv)). Then $D_t$ starts at $t \in L_t$ and ends at $d_t$,
$A_t''$ starts at $s_t \in (t,\overline{L_t})$ and ends in $[d_t,t)$, and
$D_t$ and $A_t''$ are otherwise contained in $\mathbb{H}$. $A_t''$ must
therefore intersect $D_t$, which contradicts part (iii). Next assume
that $A_t''$ ends at a point $a_t'' \in [\underline{S_3}, d_t]$.
Then $D_t$ starts at $t \in L_t$ and ends at $d_t \in [\underline{S_3},\overline{S_3})$,
$A_t''$ starts at $s_t \in (t,\overline{L_t})$ and ends at
$a_t'' \in [\underline{S_3}, d_t]$, and $D_t$ and $A_t''$ are otherwise
contained in $\mathbb{H}$ and do not intersect (see part (iii)). Consider the simple
closed contour consisting of $D_t$ and $A_t''$ and $[a_t'',d_t]$ and $[t,s_t]$.
We can proceed as in part (xii) in the proof of lemma \ref{lemConAscDesCase1}
to show that $w \mapsto \text{Im}(f_t(w))$ is a constant harmonic
function on the above simple closed contour, and so it also constant
in the domain bounded by the closed contour. Equation (\ref{eqImt})
easily shows that this is not true. Thus we have a contradiction, and
so $A_t''$ does not end in $[\underline{S_3},d_t]$.
Finally suppose that $A_t''$ ends in $[\overline{L_t}, \overline{S_1}]$.
Then $A_t''$ starts at $s_t \in (t,\overline{L_t})$, ends in
$[\overline{L_t}, \overline{S_1}]$, and is otherwise contained
in $\mathbb{H}$. We can then proceed as in part (xii) in the proof of lemma
\ref{lemConAscDesCase1} to show that this gives a contradiction,
and so $A_t''$ does not end in $[\overline{L_t}, \overline{S_1}]$.
This proves (v) for possibility (A).

Consider (v) for possibility (B). Recall that $D_t''$ and $A_t''$ start at
$s_t+i\e$ where $s_t \in (t,\overline{L_t})$, $D_t''$ ends in $S \cup T_t$,
and $A_t''$ ends either in $S \cup T_t$ or is unbounded.
Then, similar considerations to those used in part (iv), above, show that
$D_t''$ ends in $[\underline{S_3}, \overline{S_3}] \cup
[\underline{S_2}, \overline{S_2}] \cup [\underline{S_1}, \underline{L_t}]
\cup [\overline{L_t}, \overline{S_1}]$, and
$A_t''$ either ends in $[\underline{S_3}, \overline{S_3}] \cup
[\underline{S_2}, \overline{S_2}] \cup [\underline{S_1}, \underline{L_t}]
\cup [\overline{L_t}, \overline{S_1}]$ or is unbounded. Also, letting
$d_t \in [\underline{S_3},\overline{S_3})$ denote the end point of $D_t$
(see part (iv)), similar considerations to those used above (in the proof of
part (v) for possibility (A)) show that both $D_t''$ and $A_t''$ do not
end in $[d_t, \overline{S_3}] \cup [\underline{S_2}, \overline{S_2}] \cup
[\underline{S_1}, \underline{L_t}]$. It thus remains to show that $D_t''$ does not
end in $[\overline{L_t}, \overline{S_1}]$, and $A_t''$ does not end in
$[\underline{S_3}, d_t) \cup [\overline{L_t}, \overline{S_1}]$.
We argue by contradiction: First assume that $D_t''$ and $A_t''$ both
end in $[\underline{S_3}, d_t)$, at $d_t''$ and $a_t''$ respectively.
Then $D_t''$ starts at $s_t+i\e$ and ends at $d_t'' \in [\underline{S_3},d_t)$,
$A_t''$ starts at $s_t+i\e$ and ends at
$a_t'' \in [\underline{S_3}, d_t)$, and $D_t''$ and $A_t''$ are otherwise
contained in $\mathbb{H}$ and do not intersect except at $s_t+i\e$
(see part (iii)). Consider the simple closed contour consisting of $D_t''$ and
$A_t''$ and $[a_t'' \wedge d_t'', a_t'' \vee d_t'']$.
We can proceed as in part (xii) in the proof of lemma \ref{lemConAscDesCase1}
to show that $w \mapsto \text{Im}(f_t(w))$ is a constant harmonic
function on the above simple closed contour, and so it also constant
in the domain bounded by the closed contour. Equation (\ref{eqImt})
easily show that this is not true. Thus we have a contradiction, and
so $D_t''$ and $A_t''$ cannot both end in $[\underline{S_3}, d_t)$.
Next assume that $D_t''$ ends at a point
$d_t'' \in [\overline{L_t},\overline{S_1}]$. Recall (see part (2)
of lemma \ref{lemDesAsc}) that $w \mapsto \text{Im}(f_t(w))$ is
constant on $D_t''$. In particular this gives,
$\text{Im}(f_t(d_t'')) = \text{Im}(f_t(s_t+i\e))$. Also recall (see
case (2) of lemma \ref{lemCases} and equation (\ref{eqIntervalt})) that
$s_t \in L_t \subset [\underline{S_1},\overline{S_1}] \setminus S_1$ and
$[\overline{L_t},\overline{L_t}+\delta) \subset S_1$ for all $\delta>0$
sufficiently small. Equation (\ref{eqIt}) and figure \ref{figImfnExt}
then give $\text{Im}(f_t(s_t)) \ge \text{Im}(f_t(s))$ for all
$s \in [\overline{L_t},\overline{S_1}]$. Equation (\ref{eq-lemConAscDesCase2-2})
thus gives $\text{Im}(f_t(s_t+i\e)) > \text{Im}(f_t(s))$ for all
$s \in [\overline{L_t},\overline{S_1}]$. This contradicts
$\text{Im}(f_t(d_t'')) = \text{Im}(f_t(s_t+i\e))$, and so $D_t''$ does
not end in $[\overline{L_t},\overline{S_1}]$. We can similarly
show that $A_t''$ does not end in $[\overline{L_t},\overline{S_1}]$.
This proves (v) for possibility (B).
\end{proof}

As in the section \ref{secContCase1} for case (1) of lemma \ref{lemCases}, we are
now in a position to define the contours, $\gamma_{1,n}^+$ and $\Gamma_{1,n}^+$
and $\Gamma_{2,n}^+$, that satisfy lemma \ref{lemDesAsc1-12} for case (2) of
lemma \ref{lemCases}. We do not go into as much detail here as the
section \ref{secContCase1}, as the construction is quite
similar to that used for case (1), but we will highlight the differences.

Fix $\xi>0$, $\theta \in (\frac14,\frac13)$ and $\{q_n\}_{n\ge1} \subset \R$
as in the section \ref{secContCase1}. Note, since $f_t'''(t) > 0$ (see case
(2) of lemma \ref{lemCases}), lemma \ref{lemftn'} and definition
\ref{defmnpn} imply that $\{q_n\}_{n\ge1}$ converges to a positive constant.
Then, part (4) of lemma \ref{lemNonAsyRoots} implies that
$\{t_{1,n},t_{2,n}\} \subset B(t, n^{-\theta} q_n) \subset B(t,\xi)$.
Thus, $D_n$ and $A_n$ both start inside $B(t, n^{-\theta} q_n)$ and
$B(t,\xi)$. Define $d_{1,n}, d_{2,n}, a_{1,n}, a_{2,n}$ as in equation
(\ref{eqdn}). Denote the equivalent quantities for $D_t$ and $A_t$ by
$d_{1,t}, d_{2,t}, a_{1,t}, a_{2,t}$. Also, fixing $\xi>0$ sufficiently
small as above, define $\C_\xi$ is in equation (\ref{eqAnalSetftfn}).
Recall that $D_t$ ends in
$[\underline{S_3}, \overline{S_3}) \subset \C \setminus \C_\xi$, and $A_t$
ends in $(\underline{S_2}, \overline{S_2}) \subset \C \setminus \C_\xi$,
and let $d_{3,t}$ and $a_{3,t}$ denote the points where $D_t$ and $A_t$
`exit' $\C_\xi$ respectively. We choose the $\xi>0$ sufficiently small
such that equation (\ref{eqd3ta3t1}) is satisfied in this case also:
$\text{Im}(d_{3,t}) = \text{Im}(a_{3,t}) = \xi^4$. Moreover, we
choose the $\xi>0$ sufficiently small such that
$\text{Re}(d_{3,t}) \in (-\infty, \overline{S_3} - c)$ and
$\text{Re}(a_{3,t}) \in (\underline{S_2} + c, \overline{S_2} - c)$,
where $c = c(t) > 0$ is some constant which is independent of
$\xi$. Next, define $\e := 0$ whenever possibility (A) of lemmas
\ref{lemfn'Case2} and \ref{lemConAscDesCase2} is satisfied, and fix
$\e > 0$ as in lemma \ref{lemConAscDesCase2} whenever possibility (B)
is satisfied. Finally, fix $E>0$ such that $B(0,E)$ contains
$S_1 \cup S_2 \cup S_3$ and the point $s_t+i\e$, and let
$a_{4,t} \in \partial B(0,E)$ denote that point where $A_t''$ `exits'
$B(0,E)$ for the first time. We then define the following contours,
which are depicted in figure \ref{figCongnGnCase2}:
\begin{definition}
\label{defgnGnCase2}
Define $\g_{1,n}^+$ and $\G_{1,n}^+$
similarly as in definition \ref{defgnGnCase1} (they both start at
$t \in L_t = (\underline{L_t}, \overline{L_t})$,
$\g_{1,n}^+$ ends at $\text{Re}(d_{3,t}) \in (-\infty, \overline{S_3} - c)$,
$\G_{1,n}^+$ ends at $\text{Re}(a_{3,t}) \in (\underline{S_2} + c, \overline{S_2} - c)$
etc). Moreover, define $\G_{2,n}^+$ to be
the opposite (traversed in the opposite direction) of the simple contour which:
\begin{itemize}
\item
is independent of $n$,
\item
starts at $s_t \in (t, \overline{L_t})$,
\item
then traverses the straight line from $s_t$ to $s_t+i\e$,
\item
then traverses that section of $A_t''$ from $s_t+i\e$
to $a_{4,t} \in \partial B(0,E)$,
\item
then traverses (clockwise) the arc of $\partial B(t, \xi)$ from
$a_{4,t} \in \partial B(0, E)$ to $E$,
\item
then ends at $E \in (\overline{S_1}, +\infty)$.
\end{itemize}
\end{definition}

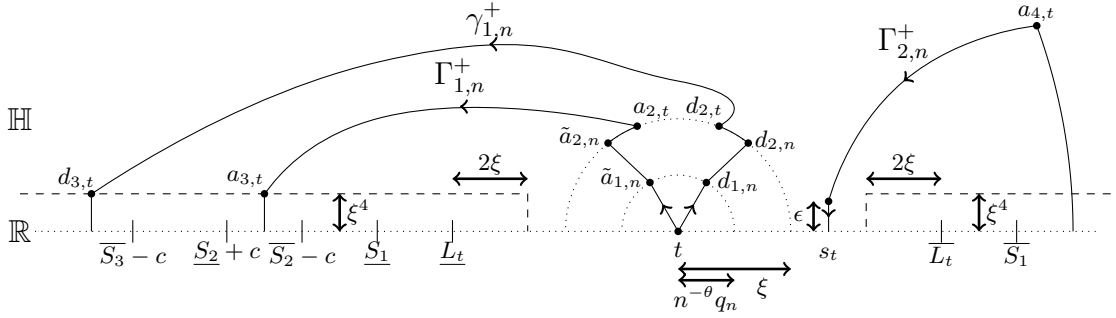
\begin{figure}[t]
\centering
\begin{tikzpicture};

\draw (8,0) --++(.375,.65);
\draw[arrows=->,line width=1pt](8.2,.34641)--(8.205,.35507);
\fill [black] (8.375,.65) circle (.05cm);
\draw (8.8,.65) node {\scriptsize $d_{1,n}$};
\draw (8.375,.65) --++(.557,.525);
\fill [black] (8.932,1.175) circle (.05cm);
\draw (9.3,1.2) node {\scriptsize $d_{2,n}$};
\draw [domain=51.567:68.755] plot ({8+(1.5)*(cos(\x))}, {(1.5)*(sin(\x))});
\fill [black] (8.543,1.4) circle (.05cm);
\draw (8.35,1.65) node {\scriptsize $d_{2,t}$};
\draw plot [smooth, tension=1] coordinates
{ (8.543,1.398) (8,2) (4,2.25) (-.5,0)};
\draw[arrows=->,line width=1pt](5.5,2.49)--(5.49,2.49);
\draw (5.5,2.8) node {$\g_{1,n}^+$};
\fill (0,0) [white] circle (.5cm);
\draw (.2,.5) --++(0,-.5);
\fill [black] (.2,.5) circle (.05cm);
\draw (0,.7) node {\scriptsize $d_{3,t}$};

\draw (8,0) --++(-.375,.65);
\draw[arrows=->,line width=1pt](7.8,.34641)--(7.795,.35507);
\fill [black] (7.625,.65) circle (.05cm);
\draw (7.25,.7) node {\scriptsize $\tilde{a}_{1,n}$};
\draw (7.625,.65) --++(-.557,.525);
\fill [black] (7.068,1.175) circle (.05cm);
\draw (6.7,1.25) node {\scriptsize $\tilde{a}_{2,n}$};
\draw [domain=111.567:128.755] plot ({8+(1.5)*(cos(\x))}, {(1.5)*(sin(\x))});
\fill [black] (7.456,1.4) circle (.05cm);
\draw (7.65,1.6) node {\scriptsize $a_{2,t}$};
\draw plot [smooth, tension=1] coordinates
{ (7.456,1.398) (4,1.5) (2.25,0)};
\draw[arrows=->,line width=1pt](5.11,1.66)--(5.1,1.66);
\draw (5.1,2.05) node {$\G_{1,n}^+$};
\fill (2.5,0) [white] circle (.5cm);
\draw (2.5,.5) --++(0,-.5);
\fill [black] (2.5,.5) circle (.05cm);
\draw (2.25,.7) node {\scriptsize $a_{3,t}$};

\fill [black] (10,0.4) circle (.05cm);
\draw (10,0) --++(0,.4);
\draw[arrows=->,line width=1pt](10,.2)--(10,.19);
\draw (10,-.3) node {\scriptsize $s_t$};
\draw[arrows=<->,line width=1pt](9.8,0)--(9.8,.4);
\draw (9.6,.2) node {\scriptsize $\e$};
\draw plot [smooth, tension=1] coordinates
{ (10,.4) (11,2) (12.7675,2.7362)};
\fill [black] (12.7675,2.7362) circle (.05cm);
\draw (12.7675,2.9) node {\scriptsize $a_{4,t}$};
\draw[arrows=->,line width=1pt](11.05,2.04)--(11,2);
\draw (11,2.5) node {$\G_{2,n}^+$};
\draw [domain=0:20] plot ({5.25+8*cos(\x)}, {8*sin(\x)});

\fill [black] (8,0) circle (.05cm);
\draw (8,-.25) node {\scriptsize $t$};
\draw [dotted, domain=0:180] plot ({8+(.75)*(cos(\x))}, {(.75)*(sin(\x))});
\draw [dotted, domain=0:180] plot ({8+(1.5)*(cos(\x))}, {(1.5)*(sin(\x))});

\draw [dashed] (-.75,.5) --++(6.75,0) --++ (0,-.5);
\draw[arrows=<->,line width=1pt](3.5,0)--(3.5,.5);
\draw (3.75,0.25) node {\scriptsize $\xi^4$};
\draw[arrows=<->,line width=1pt](5,.7)--(6,.7);
\draw (5.5,.9) node {\scriptsize $2\xi$};
\draw [dashed] (13.75,.5) --++(-3.25,0) --++ (0,-.5);
\draw[arrows=<->,line width=1pt](12,0)--(12,.5);
\draw (12.25,0.25) node {\scriptsize $\xi^4$};
\draw[arrows=<->,line width=1pt](10.5,.7)--(11.5,.7);
\draw (11,.9) node {\scriptsize $2\xi$};
\draw (12.5,-.15) --++(0,.3);
\draw (12.5,-.3) node {\scriptsize $\overline{S_1}$};
\draw (11.5,-.15) --++(0,.3);
\draw (11.5,-.3) node {\scriptsize $\overline{L_t}$};
\draw (5,-.15) --++(0,.3);
\draw (5,-.3) node {\scriptsize $\underline{L_t}$};
\draw (4,-.15) --++(0,.3);
\draw (4,-.3) node {\scriptsize $\underline{S_1}$};
\draw (3,-.15) --++(0,.3);
\draw (3,-.3) node {\scriptsize $\overline{S_2}-c$};
\draw (2,-.15) --++(0,.3);
\draw (2,-.3) node {\scriptsize $\underline{S_2}+c$};
\draw (.75,-.15) --++(0,.3);
\draw (.75,-.3) node {\scriptsize $\overline{S_3}-c$};

\draw (-.75,1.5) node {$\mathbb{H}$};
\draw [dotted] (-.55,0) --++(14.15,0);
\draw (-.75,0) node {$\R$};
\draw[arrows=<->,line width=1pt](8,-.5)--(9.5,-.5);
\draw[arrows=<->,line width=1pt](8,-.65)--(8.75,-.65);
\draw (9.125,-.75) node {\scriptsize $\xi$};
\draw (8.375,-.9) node {\scriptsize $n^{-\theta} q_n$};

\end{tikzpicture}
\caption{The contours defined in definition \ref{defgnGnCase2}.
The dashed lines represent boundaries of $\C_\xi$.
$c = c(t) > 0$ is independent of $\xi$.}
\label{figCongnGnCase2}
\end{figure}

We finally show that the above contours satisfy the requirements of lemma
\ref{lemDesAsc1-12}:
\begin{proof}[Proof of lemma \ref{lemDesAsc1-12} for case (2) of lemma \ref{lemCases}:]
The proof for case (2) of lemma \ref{lemCases} is very similar to the proof
for case (1) (see the end of section \ref{secContCase1}). Parts (1,2,3,5,6)
follow from similar arguments. Similarly for part (4), except that we need
to additionally show that $\tilde{R}_n(z) \ge \tilde{R}_n(\tilde{a}_{1,n})$
for all $z \in \G_{2,n}^+$, where $\tilde{R}_n$ denote the real-part of
$\tilde{f}_n$. We will show that there exists a sufficiently
small choice $\e,\xi$ in definition \ref{defgnGnCase2}, and a sufficiently
large choice of the $E>0$, such that the following are satisfied:
\begin{enumerate}
\item[(i)]
$\tilde{R}_n(z) > \tilde{R}_n(\tilde{a}_{1,n})$ for all $z \in \text{cl}(B(s_t,\e))$.
\item[(ii)]
$\tilde{R}_n(z) > \tilde{R}_n(\tilde{a}_{1,n})$ for all $z$ on that section of
$A_t''$ from $s_t+i\e$ to $a_{4,t} \in \partial B(0,E)$.
\item[(iii)]
$\tilde{R}_n(z) > \tilde{R}_n(\tilde{a}_{1,n})$ for all $z \in \partial B(0,E)$.
\end{enumerate}
Definition \ref{defgnGnCase2} then implies that
$\tilde{R}_n(z) > \tilde{R}_n(\tilde{a}_{1,n})$
for all $z \in \G_{2,n}^+$. This proves (4).

Consider (i) for possibility (B). First,
we fix the $\xi,\e>0$ sufficiently small such that $\text{cl}(B(s_t,\e)) \subset \C_\xi$
(see figure \ref{figCongnGnCase2} to see that this can be done). Then, part
(2) of lemma \ref{lemfnftConv} implies that $\tilde{R}_n(z) = R_t(z) + o(1)$
uniformly for $z \in \text{cl}(B(s_t,\e))$. Next note, since
$R_t$ is continuous at $s_t$, that $R_t(z) = R_t(s_t) + O(\e)$ for all $\e>0$
sufficiently small and uniformly for $z \in \text{cl}(B(s_t,\e))$. Next recall
that $s_t \in (t, \overline{L_t})$, and so
$R_t(s_t) > R_t(t)$ (see part (B1) of lemma \ref{lemfn'Case2} and figure
\ref{figConDirCase2}). Finally note, since $\tilde{a}_{1,n} = t + o(1)$,
equations (\ref{eqS1nS1WeakConv}, \ref{eqRt}, \ref{eqRn})
give $R_t(t) = \tilde{R}_n(\tilde{a}_{1,n}) + o(1)$.
Combined, the above proves part (i) for possibility (B). Part (i)
for possibility (A), where $\e=0$ by the remarks preceding definition
\ref{defgnGnCase2}, can be shown similarly.

Consider (ii). First, we fix $\xi,\e$ sufficiently small such that
$A_t'' \subset \C_\xi$ (see figure \ref{figCongnGnCase2} to see
that this can be done). Then, part (2) of lemma \ref{lemfnftConv}
implies that $\tilde{R}_n(z) = R_t(z) + o(1)$
uniformly for $z$ on that section of $A_t''$ from $s_t+i\e$ to
$a_{4,t} \in \partial B(0,E)$. Next note, part (1) of lemma \ref{lemDesAsc}
gives $R_t(z) \ge R_t(s_t+i\e)$ for $z$ on $A_t''$. Moreover, since
$R_t$ is continuous at $s_t$, $R_t(s_t+i\e) = R_t(s_t) + O(\e)$ for
all $\e>0$ sufficiently small. We can then proceed as in the proof
of part (i), above, to prove part (ii).

Consider (iii). First note, whenever $z \neq 0$, equation (\ref{eqRn}) gives,
\begin{align*}
\tilde{R}_n(z)
&= \frac1n \left( |\tilde{S}_{1,n}| - |\tilde{S}_{2,n}| + |\tilde{S}_{3,n}| \right) \log |z| \\
&+ \frac1n \sum_{x \in \tilde{S}_{1,n}} \log \left| 1 - \frac{x}z \right|
- \frac1n \sum_{x \in \tilde{S}_{2,n}} \log \left| 1 - \frac{x}z \right|
+ \frac1n \sum_{x \in \tilde{S}_{3,n}} \log \left| 1 - \frac{x}z \right|,
\end{align*}
where $\log$ is natural logarithm. Recall (see equations
(\ref{eqS1S2S3In}, \ref{eqS1nS2nS3nIn})) that
$\frac1n( |\tilde{S}_{1,n}| - |\tilde{S}_{2,n}| + |\tilde{S}_{3,n}|)
\to \eta \in (0,1)$. Also recall (see equation
(\ref{eqS1S2S3In})) that
$\sup_{x \in S_{1,n} \cup S_{2,n} \cup S_{3,n}} |x| = O(1)$.
Thus $\tilde{R}_n(z) > \frac12 \eta \log |z|$
for all $n$ and $|z|$ sufficiently large, chosen independently.
Finally recall (see the proof of part (i), above) that
$\tilde{R}_n(\tilde{a}_{1,n}) \to R_t(t)$.
Combined, the above prove part (iii).
\end{proof}

\subsection{Lemma \ref{lemDesAsc1-12} for cases (3-12) of lemma \ref{lemCases}}
\label{secContOthCases}

In sections \ref{secContCase1} and \ref{secContCase2} we proved
lemma \ref{lemDesAsc1-12} for cases (1,2) of lemma \ref{lemCases}.
In this section we discuss cases (3-6).
We will be very brief here, as no new ideas are needed.
Instead we seek to highlight some differences which may cause confusion.
Indeed, we only recall the possible behaviours of the roots
of $f_t'$ in $L_t$ for cases (1,2), and state the possible behaviours of the roots
of $f_t'$ in $L_t$ for cases (3-6). For cases (3-6), the behaviour of the
remaining roots of $f_t'$,
the behaviour of the roots of $f_n'$, the contours of steepest descent/ascent,
and the definitions of the contours $\g_{1,n}^+$, etc, then follow from
similar considerations to those used for cases (1,2).
Similar considerations also apply for cases (7-12).

First recall the situation for case (1) of lemma \ref{lemCases}
(see section \ref{secContCase1}). Here, $t > \chi$, $t \in L_t$,
$L_t = J_1 = (\overline{S_1}, +\infty)$, and $f_t'''(t) > 0$.
Moreover $f_t'(s) > 0$ for all $s \in (\underline{L_t},t)$,
$f_t'(t) = f_t''(t) = 0$ and $f_t'''(t) > 0$, and $f_t'(s) > 0$
for all $s \in (t,\overline{L_t})$.

Next recall the situation for case (2) of lemma \ref{lemCases}
(see section \ref{secContCase2}). Here, $t > \chi$, $t \in L_t$,
$L_t \in \{K_1^{(1)},K_2^{(1)},\ldots\}$, and $f_t'''(t) > 0$.
Moreover, one of the following is satisfied:
\begin{enumerate}
\item[(A)]
There exists an $s_t \in (t,\overline{L_t})$ for which:
$f_t'(s) > 0$ for all $s \in (\underline{L_t},t)$,
$f_t'(t) = f_t''(t) = 0$ and $f_t'''(t) > 0$,
$f_t'(s) > 0$ for all $s \in (t,s_t)$,
$f_t'(s_t) = 0$ and $f_t''(s_t) < 0$,
and $f_t'(s) < 0$ for all $s \in (s_t,\overline{L_t})$.
\item[(B)]
$f_t'(s) > 0$ for all $s \in (\underline{L_t},t)$,
$f_t'(t) = f_t''(t) = 0$ and $f_t'''(t) > 0$,
$f_t'(s) > 0$ for all $s \in (t,\overline{L_t})$.
\end{enumerate}

Consider case (3). Lemma \ref{lemCases} gives $t > \chi$, $t \in L_t$,
$L_t \in \{K_1^{(1)},K_2^{(1)},\ldots\}$, and $f_t'''(t) < 0$.
Moreover, we state that one of the following is satisfied:
\begin{enumerate}
\item[(A)]
There exists an $s_t \in (\underline{L_t},t)$ for which:
$f_t'(s) > 0$ for all $s \in (\underline{L_t},s_t)$,
$f_t'(s_t) = 0$ and $f_t''(s_t) < 0$,
$f_t'(s) < 0$ for all $s \in (s_t,t)$,
$f_t'(t) = f_t''(t) = 0$ and $f_t'''(t) < 0$,
and $f_t'(s) < 0$ for all $s \in (t,\overline{L_t})$.
\item[(B)]
$f_t'(s) < 0$ for all $s \in (\underline{L_t},t)$,
$f_t'(t) = f_t''(t) = 0$ and $f_t'''(t) < 0$,
$f_t'(s) < 0$ for all $s \in (t,\overline{L_t})$.
\end{enumerate}

Consider cases (4,5). Lemma \ref{lemCases} gives $t > \chi$
for case (4), and $t \in (\chi+\eta-1,\chi)$ for case (5).
Moreover, for both cases, $t \in L_t$, $\chi \in L_t$,
$L_t = J_3 = (\overline{S_2},\underline{S_1})$
and $f_t'''(t) < 0$. Moreover, we state that
$f_t'(s) < 0$ for all $s \in (\underline{L_t},t)$,
$f_t'(t) = f_t''(t) = 0$ and $f_t'''(t) < 0$,
$f_t'(s) < 0$ for all $s \in (t,\overline{L_t})$.

Consider case (6). Lemma \ref{lemCases} gives $t \in (\chi+\eta-1,\chi)$,
$t \in L_t$, $L_t \in \{K_1^{(2)},K_2^{(2)},\ldots\}$, and $f_t'''(t) < 0$.
Moreover, we state that one of the following is satisfied:
\begin{enumerate}
\item[(A)]
There exists an $s_t \in (t,\overline{L_t})$ for which:
$f_t'(s) < 0$ for all $s \in (\underline{L_t},t)$,
$f_t'(t) = f_t''(t) = 0$ and $f_t'''(t) < 0$,
$f_t'(s) < 0$ for all $s \in (t,s_t)$,
$f_t'(s_t) = 0$ and $f_t''(s_t) > 0$,
and $f_t'(s) > 0$ for all $s \in (s_t,\overline{L_t})$.
\item[(B)]
$f_t'(s) < 0$ for all $s \in (\underline{L_t},t)$,
$f_t'(t) = f_t''(t) = 0$ and $f_t'''(t) < 0$, and
$f_t'(s) < 0$ for all $s \in (t,\overline{L_t})$.
\end{enumerate}

\subsection{Lemma \ref{lemDesAsc1-12Rem} for cases (1-4) of lemma \ref{lemCases}}
\label{secContCase1Rem}

Note, many of the arguments in this section are similar to those used in
the previous sections. Therefore, we shall not go into as much detail here.
Assume the conditions of lemma \ref{lemDesAsc1-12Rem}. Additionally assume
that one of cases (1-4) of lemma \ref{lemCases} is satisfied. Fix $\xi>0$
sufficiently small such that equations (\ref{eqxi}, \ref{eqxi1}, \ref{eqxi4},
\ref{eqxi5}) are satisfied. Define $U_n$ and $V_n$ as in equation (\ref{eqUnVn}),
and recall equation (\ref{eqVn-Un}).

First, since one of case (1-4) of lemma \ref{lemCases} is satisfied,
and since $v > u$, we can proceed as in the proof of lemma
\ref{lemvnunIneq} to show the following:
\begin{equation*}
v_n > u_n
\hspace{.25cm} \text{and} \hspace{.25cm}
v_n + s_n - 1 > u_n + r_n + 1
\hspace{.25cm} \text{and} \hspace{.25cm}
VU^{(n)} \neq \emptyset
\hspace{.25cm} \text{and} \hspace{.25cm}
UV_{(n)} \neq \emptyset,
\end{equation*}
and
\begin{align}
\label{eqxiCase1}
&t - 2\xi
> \chi + 2\xi
> \max(VU^{(n)})
> \min(VU^{(n)})
> \chi-2\xi \\
\nonumber
&> \chi + \eta - 1 + 2\xi
> \max(UV_{(n)})
> \min(UV_{(n)})
> \chi+\eta-1 - 2\xi.
\end{align}
Also, equation (\ref{eqFnw2}) gives,
\begin{equation*}
F_n(w)
= - \frac1n \sum_{x \in VU^{(n)}} \log (w - x)
+ \frac1n \sum_{x \in UV_{(n)}} \log (w - x),
\end{equation*}
for all $w \in (\C \setminus \R) \cup (t-2\xi,t+2\xi)$, where the
branch cuts are all $(-\infty,0]$. Therefore,
\begin{equation*}
F_n'(w)
= - \frac1n \sum_{x \in VU^{(n)}} \frac1{w - x}
+ \frac1n \sum_{x \in UV_{(n)}} \frac1{w - x},
\end{equation*}
and $F_n'$ extends analytically to $\C \setminus ((VU^{(n)}) \cup (UV_{(n)})$.
Finally define $G_t$ as in equation (\ref{eqGt}). Recall that this is also
analytic in $(\C \setminus \R) \cup (t-2\xi,t+2\xi)$, and recall (see equation
(\ref{eqGt'})) that $G_t'$ extends analytically to $\C \setminus \{\chi,\chi+\eta-1\}$.

We now consider the roots of $G_t'$ and $F_n'$:
\begin{lem}
\label{lemFn'Case1}
Assume the above conditions. Then, $t \in (\chi, +\infty)$. Indeed,
$t- 4\xi > \chi$. Moreover:
\begin{enumerate}
\item
$G_t'(s) < 0$ for all $s \in (\chi,t)$,
$G_t'(t) = 0$ and $G_t''(t) > 0$, and
$G_t'(s) > 0$ for all $s \in (t,+\infty)$.
\item
$G_t'$ has no other roots in $\C \setminus \{\chi,\chi+\eta-1\}$.
\end{enumerate}

Next, note $t \in (\max (VU^{(n)}), +\infty)$.
Indeed, $t - 2\xi > \chi + 2\xi > \max (VU^{(n)})$.
Also $F_n'$ has $1$ root in $(t-\xi,t+\xi)$.
We denote this by $w_n$ as in lemma \ref{lemFn'}.
Finally:
\begin{enumerate}
\setcounter{enumi}{2}
\item
$F_n'(s) < 0$ for all $s \in (\max (VU^{(n)}),w_n)$,
$F_n'(w_n) = 0$ and $F_n''(w_n) > 0$, and
$F_n'(s) > 0$ for all $s \in (w_n,+\infty)$.
\item
$F_n'$ has $1$ root in each interval
of the form $(x,y)$, when $x$ and $y$ are any two consecutive elements of
either $VU^{(n)}$ or $UV_{(n)}$.
\item
$F_n'$ has no other roots, other than those listed above.
\end{enumerate}
\end{lem}

\begin{proof}
Consider $G_t'$. Equation (\ref{eqxiCase1}) gives $t- 4\xi > \chi$.
Also, lemma \ref{lemGt'} implies that $G_t'$ has a root of multiplicity $1$
at $t$, has no other roots in $\C \setminus \{\chi,\chi+\eta-1\}$, and that,
\begin{equation*}
G_t''(t) = \frac{e^{C(t)} C'(t)^2}{e^{C(t)} - 1}.
\end{equation*}
Thus, since one of cases (1-4) of lemma \ref{lemCases} is satisfied,
lemma \ref{lemAnalExt} gives $G_t''(t) > 0$. This proves parts (1,2).

Consider $F_n'$. First note, equation (\ref{eqxiCase1}) implies
that $t - 2\xi > \chi + 2\xi > \max (VU^{(n)})$. Next note, part (1) of
lemma \ref{lemFn'} implies
that $F_n'$ has $1$ root in $(t-\xi,t+\xi)$. We denote this by $w_n$, and we recall that
$w_n \to t$ (see part (5) of lemma \ref{lemFn'}). Moreover,
part (2) of lemma \ref{lemGtGtnFn} gives
\begin{equation*}
n^\frac13 F_n''(w) = m_n (v-u) G_{t,n}''(w) + O(n^{-\frac13}),
\end{equation*}
uniformly for $w \in B(t, \xi)$. Recall that $v - u > 0$ by assumption.
Also recall that $\{m_n\}_{n\ge1}$ is a convergent sequence of
real-numbers with a non-zero limit (see definition \ref{defmnpn}),
$m_n (e^{C_n(t)} - 1)/e^{C_n(t)} > 0$ (see definition \ref{defmnpn}),
and $e^{C_n(t)} - 1 \to e^{C(t)} - 1 > 0$ (see cases (1-4) of lemma
\ref{lemCases}, lemma \ref{lemAnalExt}, and equation (\ref{eqNonAsyEdge})).
Therefore $m_n \to m_t$ for some $m_t > 0$.
Moreover (see lemma \ref{lemGt'}), $G_{t,n}''(w) = G_t''(w) + o(1)$
uniformly for $w \in B(t, \xi)$. Finally
(see part (5) of lemma \ref{lemFn'}), $w_n \to t$. Combined, the above
observations give $n^\frac13 F_n''(w_n) \to m_t (v-u) G_t''(t) > 0$.
Therefore $F_n''(w_n) > 0$.
Parts (3,4,5) then follow from parts (2-4) of lemma \ref{lemFn'}.
\end{proof}

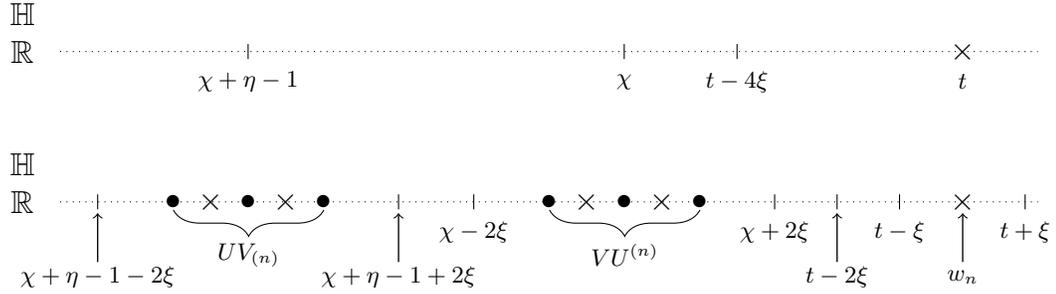
\begin{figure}[t]
\centering
\begin{tikzpicture}

\draw [dotted] (-1,2) --++(13,0);
\draw (-1.5,2) node {$\R$};
\draw (-1.5,2.5) node {$\mathbb{H}$};

\draw (1.5,2.1) --++(0,-.2);
\draw (6.5,2.1) --++(0,-.2);
\draw (8,2.1) --++(0,-.2);
\draw (11,2) node {$\times$};

\draw (1.5,1.6) node {\scriptsize $\chi+\eta-1$};
\draw (6.5,1.6) node {\scriptsize $\chi$};
\draw (8,1.6) node {\scriptsize $t-4\xi$};
\draw (11,1.6) node {\scriptsize $t$};

\draw [dotted] (-1,0) --++(13,0);
\draw (-1.5,0) node {$\R$};
\draw (-1.5,0.5) node {$\mathbb{H}$};

\draw (-.5,0.1) --++(0,-.2);
\draw (.5,0) node {$\bullet$};
\draw (1,0) node {$\times$};
\draw (1.5,0) node {$\bullet$};
\draw (2,0) node {$\times$};
\draw (2.5,0) node {$\bullet$};
\draw (3.5,0.1) --++(0,-.2);
\draw (4.5,0.1) --++(0,-.2);
\draw (5.5,0) node {$\bullet$};
\draw (6,0) node {$\times$};
\draw (6.5,0) node {$\bullet$};
\draw (7,0) node {$\times$};
\draw (7.5,0) node {$\bullet$};
\draw (8.5,0.1) --++(0,-.2);
\draw (9.33,.1) --++(0,-.2);
\draw (10.16,0.1) --++(0,-.2);
\draw (11,0) node {$\times$};
\draw (11.83,.1) --++(0,-.2);

\draw (-.5,-1) node {\scriptsize $\chi+\eta-1-2\xi$};
\draw[arrows=->,line width=.5pt](-.5,-.8)--(-.5,-.15);
\draw [decorate,decoration={brace,amplitude=10pt,mirror},xshift=0.2pt,yshift=-0.2pt]
(.5,-.1) -- (2.5,-.1) node[black,midway,yshift=-0.6cm] {\scriptsize $UV_{(n)}$};
\draw (3.5,-1) node {\scriptsize $\chi+\eta-1+2\xi$};
\draw[arrows=->,line width=.5pt](3.5,-.8)--(3.5,-.15);
\draw (4.5,-.4) node {\scriptsize $\chi-2\xi$};
\draw [decorate,decoration={brace,amplitude=10pt,mirror},xshift=0.2pt,yshift=-0.2pt]
(5.5,-.1) -- (7.5,-.1) node[black,midway,yshift=-0.6cm] {\scriptsize $VU^{(n)}$};
\draw (8.5,-.4) node {\scriptsize $\chi+2\xi$};
\draw (9.33,-1) node {\scriptsize $t-2\xi$};
\draw[arrows=->,line width=.5pt](9.33,-.8)--(9.33,-.15);
\draw (10.16,-.4) node {\scriptsize $t-\xi$};
\draw (11,-1) node {\scriptsize $w_n$};
\draw[arrows=->,line width=.5pt](11,-.8)--(11,-.15);
\draw (11.83,-.4) node {\scriptsize $t+\xi$};

\end{tikzpicture}
\caption{Top: The roots of $G_t'$. Bottom: The
roots of $F_n'$. Roots of multiplicity $1$ are represented by $\times$, and
elements of $VU^{(n)}$ and $UV_{(n)}$
are represented by $\bullet$.
\label{figRootsFn'Case1}}
\end{figure}

A depiction of equation (\ref{eqxiCase1}) and the root behaviours
described in lemma \ref{lemFn'Case1} is given in figure \ref{figRootsFn'Case1}.
Next we state (without proof) a result which follows from similar arguments
to those used to show lemma \ref{lemConAscDesCase1}:
\begin{lem}
\label{lemConAscDesFnCase1}
There exists simple contours, $\mathcal{D}_t$ and $\mathcal{D}_n$,
as shown in figure
\ref{figConDesAscFnCase1} with the following properties:
\begin{itemize}
\item
$\mathcal{D}_t$ starts at $t$, enter $\mathbb{H}$
in the direction $\frac\pi2$, ends at $\chi+\eta-1$,
and is otherwise contained in $\mathbb{H}$. Moreover,
$\mathcal{D}_t$ is a contour of steepest descent for $G_t$.
\item
$\mathcal{D}_n$ starts at $w_n$, enters $\mathbb{H}$
in the direction $\frac\pi2$, ends in the interval shown,
and is otherwise contained in $\mathbb{H}$. Moreover,
$\mathcal{D}_n$ is a contour of steepest descent for $F_n$.
\end{itemize}
\end{lem}

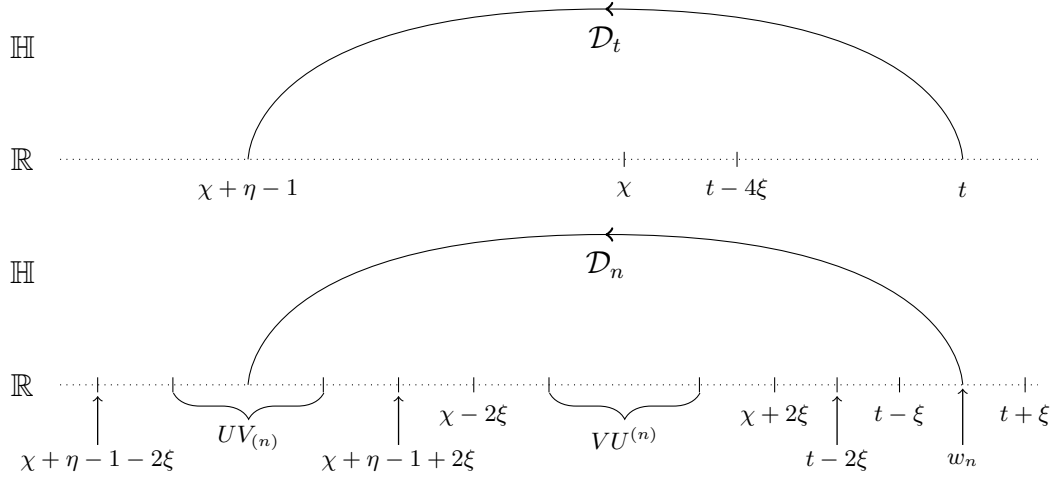
\begin{figure}[t]
\centering
\begin{tikzpicture};

\draw [dotted] (-1,3) --++(13,0);
\draw (-1.5,3) node {$\R$};
\draw (-1.5,4.5) node {$\mathbb{H}$};

\draw (6.5,3.1) --++(0,-.2);
\draw (8,3.1) --++(0,-.2);

\draw (1.5,2.6) node {\scriptsize $\chi+\eta-1$};
\draw (6.5,2.6) node {\scriptsize $\chi$};
\draw (8,2.6) node {\scriptsize $t-4\xi$};
\draw (11,2.6) node {\scriptsize $t$};

\draw plot [smooth, tension=1.8] coordinates
{ (11,3) (6.25,5) (1.5,3) };
\draw[arrows=->,line width=1pt](6.26,5)--(6.25,5);
\draw (6.25,4.6) node {$\mathcal{D}_t$};

\draw [dotted] (-1,0) --++(13,0);
\draw (-1.5,0) node {$\R$};
\draw (-1.5,1.5) node {$\mathbb{H}$};

\draw (-.5,0.1) --++(0,-.2);
\draw (.5,0.1) --++(0,-.2);
\draw (3.5,0.1) --++(0,-.2);
\draw (2.5,0.1) --++(0,-.2);
\draw (4.5,0.1) --++(0,-.2);
\draw (5.5,0.1) --++(0,-.2);
\draw (7.5,0.1) --++(0,-.2);
\draw (8.5,0.1) --++(0,-.2);
\draw (9.33,.1) --++(0,-.2);
\draw (10.16,0.1) --++(0,-.2);
\draw (11.83,.1) --++(0,-.2);

\draw (-.5,-1) node {\scriptsize $\chi+\eta-1-2\xi$};
\draw[arrows=->,line width=.5pt](-.5,-.8)--(-.5,-.15);
\draw [decorate,decoration={brace,amplitude=10pt,mirror},xshift=0.2pt,yshift=-0.2pt]
(.5,-.1) -- (2.5,-.1) node[black,midway,yshift=-0.6cm] {\scriptsize $UV_{(n)}$};
\draw (3.5,-1) node {\scriptsize $\chi+\eta-1+2\xi$};
\draw[arrows=->,line width=.5pt](3.5,-.8)--(3.5,-.15);
\draw (4.5,-.4) node {\scriptsize $\chi-2\xi$};
\draw [decorate,decoration={brace,amplitude=10pt,mirror},xshift=0.2pt,yshift=-0.2pt]
(5.5,-.1) -- (7.5,-.1) node[black,midway,yshift=-0.6cm] {\scriptsize $VU^{(n)}$};
\draw (8.5,-.4) node {\scriptsize $\chi+2\xi$};
\draw (9.33,-1) node {\scriptsize $t-2\xi$};
\draw[arrows=->,line width=.5pt](9.33,-.8)--(9.33,-.15);
\draw (10.16,-.4) node {\scriptsize $t-\xi$};
\draw (11,-1) node {\scriptsize $w_n$};
\draw[arrows=->,line width=.5pt](11,-.8)--(11,-.05);
\draw (11.83,-.4) node {\scriptsize $t+\xi$};

\draw plot [smooth, tension=1.8] coordinates
{ (11,0) (6.25,2) (1.5,0) };
\draw[arrows=->,line width=1pt](6.26,2)--(6.25,2);
\draw (6.25,1.6) node {$\mathcal{D}_n$};

\end{tikzpicture}
\caption{The contours of lemma \ref{lemConAscDesFnCase1}.}
\label{figConDesAscFnCase1}
\end{figure}

Similarly to the previous sections, we are now in a position to define the
contour, $\kappa_n^+$, that satisfies lemma \ref{lemDesAsc1-12Rem} for cases
(1-4) of lemma \ref{lemCases}. As in those sections, recalling (see part (5)
of lemma \ref{lemFn'}) that $w_n = t + O(n^{-\frac13})$, then
$w_n \in B(t, n^{-\theta} |q_n|) \subset B(t,\xi)$. Also, similarly to the
previous section, we let $D_{1,n}$ and $D_{2,n}$ denote the points where
$\mathcal{D}_n$ `exits' $B(t, n^{-\theta} |q_n|)$ and $B(t, \xi)$ respectively.
We denote the equivalent quantities for $\mathcal{D}_t$ by $D_{1,t}, D_{2,t}$.
Finally, define $\C_\xi$ as in equation (\ref{eqAnalSetftfn}), and let
$D_{3,t}$ denote the point where $\mathcal{D}_t$ `exits' $\C_\xi$. Similarly
to the previous sections, note that it is always possible to choose the $\xi>0$
sufficiently small such that,
\begin{equation*}
\text{Im}(D_{3,t}) = \xi^4
\hspace{.5cm} \text{and} \hspace{.5cm}
\text{Re}(D_{3,t}) \in (-\infty, \chi - c),
\end{equation*}
where $c = c(t) > 0$ is some constant which is independent of $\xi$.
Finally define:
\begin{definition}
\label{defknCase1}
Define $\kappa_n^+$ to be the simple contour which:
\begin{itemize}
\item
starts at $t \in (\chi, +\infty)$,
\item
then traverses the straight line from $t$ to
$D_{1,n} \in \partial B(t, n^{-\theta} |q_n|)$,
\item
then traverses that section of $\mathcal{D}_n$
from $D_{1,n} \in \partial B(t, n^{-\theta} |q_n|)$
to $D_{2,n} \in \partial B(t, \xi)$,
\item
then traverses the shortest arc of $\partial B(t, \xi)$ from
$D_{2,n} \in \partial B(t, \xi)$ to $D_{2,t} \in \partial B(t, \xi)$,
\item
then traverses that section of $\mathcal{D}_t$ from $D_{2,t}$ to $D_{3,t}$,
\item
then traverses the straight line from $D_{3,t}$ to $\text{Re}(D_{3,t})$,
\item
then ends at $\text{Re}(D_{3,t}) \in (-\infty, \chi-c)$,
\end{itemize}
where $c = c(t) > 0$ is some constant which is independent of $\xi$.
\end{definition}

We finally show that the above contour satisfies the requirements of lemma
\ref{lemDesAsc1-12Rem}:
\begin{proof}[Proof of lemma \ref{lemDesAsc1-12Rem} for case (1-4) of lemma \ref{lemCases}:]
Note, lemma \ref{lemConAscDesFnCase1} and figure \ref{figConDesAscFnCase1} and
definition \ref{defknCase1} imply that $\kappa_n^+$ starts at $t$, ends in
the interior of the intervals shown in figure \ref{figDesAscRem}, is otherwise contained
in $\mathbb{H}$, and is independent of $n$ outside $\text{cl}(B(t,\xi))$. Moreover,
using similar arguments to those used in the proof of lemma \ref{lemDesAsc1-12}
(see end of section \ref{secContCase1}), we can show that
$\text{Re}(F_n(w)) \le \text{Re}(F_n(D_{1,n}))$ for all
$w \in \kappa_n^+ \setminus B(t,n^{-\theta} |q_n|)$,
and that $|\kappa_n^+| = O(1)$. It thus remains to show that
$\text{Arg} (D_{1,n} - t) = \frac\pi2 + O(n^{-\frac13+\theta})$.

To see the above, first recall (see lemma \ref{lemConAscDesFnCase1}), that
$\mathcal{D}_n$ is a contour of steepest descent for $F_n$ which starts at
$w_n \in (t-\xi,t+\xi)$, and which `exits' $B(t,n^{-\theta} |q_n|)$ at $D_{1,n}$.
Thus parts (1,2) of lemma \ref{lemDesAsc} give,
\begin{equation*}
\text{Re}(F_n(D_{1,n})) < \text{Re}(F_n(w_n))
\hspace{0.5cm} \text{and} \hspace{0.5cm}
\text{Im}(F_n(D_{1,n})) = \text{Im}(F_n(w_n)).
\end{equation*}
Thus, since $w_n = t + O(n^{-\frac13})$ (see part (5) of lemma \ref{lemFn'}),
and since $F_n(t) \in \R$, and since $F_n = f_n - \tilde{f}_n$ (see
equation (\ref{eqFnw})), parts (1,2) of corollary \ref{corTay} give,
\begin{equation*}
\text{Re}(F_n(D_{1,n})) < F_n(t) + O(n^{-1})
\hspace{0.5cm} \text{and} \hspace{0.5cm}
\text{Im}(F_n(D_{1,n})) = O(n^{-1}).
\end{equation*}
Next note, since $\theta \in (\frac14,\frac13)$, and since $F_n = f_n - \tilde{f}_n$,
parts (2,3) of lemma \ref{lemTay} (take $\xi_n = n^{-\theta}$) give,
\begin{equation*}
F_n(t + n^{-\theta} q_n e^{i \alpha})
= F_n(t) + n^{-\frac13-2\theta} (v-u) e^{2i \alpha}
+ O(n^{-\frac23-\theta}),
\end{equation*}
uniformly for $\alpha \subset (-\pi,\pi]$. Therefore, since $F_n(t) \in \R$,
\begin{align*}
\text{Re}(F_n(t + n^{-\theta} q_n e^{i \alpha}))
&= F_n(t) + n^{-\frac13-2\theta} (v-u) \cos(2 \alpha)
+ O(n^{-\frac23-\theta}), \\
\text{Im}(F_n(t + n^{-\theta} q_n e^{i \alpha}))
&= n^{-\frac13-2\theta} (v-u) \sin(2 \alpha)
+ O(n^{-\frac23-\theta}),
\end{align*}
uniformly for $\alpha \subset (-\pi,\pi]$. Finally recall that $v>u$.
We can then proceed similarly to the proof of part (i)
of lemma \ref{lemDesAsc1-12} (see end of section \ref{secContCase1})
to show that $\text{Arg} (D_{1,n} - t) = \frac\pi2 + O(n^{-\frac13+\theta})$,
as required.
\end{proof}

\subsection{Lemma \ref{lemDesAsc1-12Rem} for cases (5-12) of lemma \ref{lemCases}}

In section \ref{secContCase1Rem} we proved
lemma \ref{lemDesAsc1-12Rem} for cases (1-4) of lemma \ref{lemCases}.
In this section we discuss cases (5,6). As in section \ref{secContOthCases},
we will be very brief, as the considerations are very similar.
We only seek to highlight the differences here.
Similar considerations also apply for cases (7-12).

Assume the conditions of lemma \ref{lemDesAsc1-12Rem}, and that
one of cases (5,6) are satisfied. Note, since $v > u$, we can
proceed similarly to the proof of lemma \ref{lemvnunIneq} to get:
\begin{equation*}
v_n > u_n
\hspace{.25cm} \text{and} \hspace{.25cm}
v_n + s_n < u_n + r_n
\hspace{.25cm} \text{and} \hspace{.25cm}
VU^{(n)} \neq \emptyset
\hspace{.25cm} \text{and} \hspace{.25cm}
VU_{(n)} \neq \emptyset,
\end{equation*}
and
\begin{align*}
&\chi + 2\xi
> \max(VU^{(n)})
> \min(VU^{(n)})
> \chi - 2\xi
> t + 2\xi
> t - 2\xi \\
&> \chi + \eta - 1 + 2\xi
> \max(VU_{(n)})
> \min(VU_{(n)})
> \chi+\eta-1 - 2\xi.
\end{align*}
Also, equation (\ref{eqFnw2}) gives,
\begin{equation*}
F_n(w)
= - \frac1n \sum_{x \in VU^{(n)}} \log (w - x)
- \frac1n \sum_{x \in VU_{(n)}} \log (w - x),
\end{equation*}
for all $w \in (\C \setminus \R) \cup (t-2\xi,t+2\xi)$, where the branch cuts
in the 1st sum on the RHS are all $[0,+\infty)$, and the branch cuts in the 2nd
sum are all $(-\infty,0]$. Therefore,
\begin{equation*}
F_n'(w)
= - \frac1n \sum_{x \in VU^{(n)}} \frac1{w - x}
- \frac1n \sum_{x \in VU_{(n)}} \frac1{w - x},
\end{equation*}
and $F_n'$ extends
analytically to $\C \setminus ((VU^{(n)}) \cup (VU_{(n)})$. Finally, define
$G_t$ as in equation (\ref{eqGt}). Recall that this is also analytic in
$(\C \setminus \R) \cup (t-2\xi,t+2\xi)$, and recall (see equation (\ref{eqGt'})
that $G_t'$ extends analytically to $\C \setminus \{\chi,\chi+\eta-1\}$.
Then $t \in (\chi+\eta-1,\chi)$, and we can proceed as in the proof of lemma
\ref{lemFn'Case1} to show that $G_t'(s) < 0$ for all $s \in (\chi+\eta-1,t)$,
$G_t'(t) = 0$ and $G_t''(t) > 0$, and $G_t'(s) > 0$ for all $s \in (t,\chi)$.
Moreover, $G_t'$ has no other roots in $\C \setminus \{\chi,\chi+\eta-1\}$.
The behaviour of the roots of $F_n'$, the contours of steepest descent,
and the definitions of the contour $\kappa_n^+$, then follow from
similar considerations to the previous section.

\vspace{0.5cm}

\textbf{Acknowledgements:} This research was carried out at the Royal Institute
of Technology (KTH), Stockholm, and at Uppsala University. It was partially
supported by grant KAW 2010.0063 from the Knut and Alice Wallenberg Foundation.
Special thanks to Kurt Johansson for helpful comments and suggestions. Special
thanks also to Takis Konstantopoulos for his support.

\end{document}